\documentclass{amsart}
\usepackage{amssymb, graphics, color, enumitem}
\usepackage[all]{xy}

\usepackage{tikz}
\usepackage{mathrsfs}
\usepackage[utf8]{inputenc}
\usepackage[cyr]{aeguill}
\usepackage[a4paper,margin=2.2cm]{geometry}
\usepackage{comment}

\newtheorem{lemma}{Lemma}[subsection]
\newtheorem{proposition}[lemma]{Proposition}
\newtheorem{corollary}[lemma]{Corollary}
\newtheorem{theorem}[lemma]{Theorem}
\newtheorem{theorem_int}{Theorem}

\newtheorem{definition}[lemma]{Definition}
\newtheorem{remark}[lemma]{Remark}
\newtheorem{question}{Question}

\newtheorem{ex}[lemma]{Example}

\newtheorem{convention}[lemma]{Convension}

\newtheorem{assumption}[lemma]{Assumption}

\newcommand\bl{Bl}
\newcommand\epi{\epsilon}
\newcommand\epio{\epsilon}

\newcommand\ie{i\@.e\@. }

\newcommand\Pol{\mathrm P}
\newcommand\pa{ \partial}
\newcommand{\la}{\mathrm a}
\newcommand{\m}{\mathrm b}
\newcommand{\n}{\mathrm c}
\newcommand{\K}{\mathcal K}

\newcommand\bbC{\mathbb C}

\newcommand\bbG{\mathbb G}
\newcommand\bbH{\mathbb H}

\newcommand\bbN{\mathbb N}
\newcommand\bbP{\mathbb P}

\newcommand\bbR{\mathbb R}
\newcommand\bbS{\mathbb S}
\newcommand\bbT{\mathbb T}

\newcommand\bbZ{\mathbb Z}

\newcommand{\PP}{\mathbb P}

\newcommand\be{Bl}
\renewcommand\Re{\operatorname{Re}}

\usepackage{color}

\newcommand{\lrp}[1]{\left( {#1} \right)}

\newcommand\edp{\mathrm{x}_0}
\newcommand{\Sc}{\mathcal{Z}}

\newcommand\bX{\overline{X}}

\newcommand\hH{\widehat{H}}

\newcommand\hD{\widehat{D}}

\newcommand\CI{\mathcal{C}^{\infty}}

\newcommand\Diff{\operatorname{Diff}}
\newcommand\Ric{\operatorname{Rc}}
\newcommand\Hom{\operatorname{Hom}}

\newcommand\ord{\operatorname{ord}}
\newcommand\cC{\mathcal{C}}
\newcommand\cG{\mathcal{G}}

\newcommand\cA{\mathcal{A}}
\newcommand\cE{\mathcal{E}}
\newcommand\cD{\mathcal{D}}
\newcommand\chD{\widehat{\mathcal{D}}}
\newcommand\cF{\mathcal{F}}
\newcommand\cL{\mathcal{L}}
\newcommand\cR{\mathcal{R}}

\newcommand\cV{\mathcal{V}}
\newcommand\cU{\mathcal{U}}
\newcommand\cQ{\mathcal{Q}}

\newcommand\pr{\operatorname{pr}}

\newcommand\phg{\operatorname{phg}}
\newcommand\Spec{\operatorname{Spec}}
\newcommand\Id{\operatorname{Id}}
\newcommand\rank{\operatorname{rk}}

\newcommand\homega{\widehat{\omega}}

\newcommand\halpha{\hat{\alpha}}

\newcommand\bM{\overline{M}}

\newcommand\hP{\widehat{P}}

\renewcommand\sc{\operatorname{sc}}

\newcommand\mf{\operatorname{mf}}

\newcommand\cH{\mathcal{H}}

\newcommand\cZ{\mathcal{Z}}

\newcommand\ind{\operatorname{ind}}

\newcommand\cN{\mathcal{N}}
\newcommand\cM{\mathcal{M}}

\newcommand\hV{\widehat{V}}
\newcommand\cP{\mathcal{P}}

\newcommand\ALE{\operatorname{ALE}}

\newcommand\cB{\mathcal{B}}

\newcommand\cO{\mathcal{O}}
\newcommand\cS{\mathcal{S}}

\newcommand\Aut{\operatorname{Aut}}
\newcommand\ext{\operatorname{ext}}
\newcommand\Skew{\operatorname{skew}}
\newcommand\rf{\operatorname{rf}}
\newcommand\lf{\operatorname{lf}}
\newcommand\fb{\operatorname{bf}}
\newcommand\Euc{\operatorname{Euc}}
\newcommand\coker{\operatorname{coker}}
\newcommand\hM{\widehat{M}}
\newcommand\chM{\widehat{\mathcal{M}}}
\newcommand\gt{\mathfrak{t}}
\newcommand\cbl{\mathcal{B}\ell}
\newcommand\hbeta{\hat{\beta}}

\begin{document}
\title[Gluing construction of  $Z$-extremal metrics and  BHE $3$-folds]
{A gluing construction of $Z$-extremal K\"ahler metrics and new examples of Bismut Hermitian-Einstein $3$-folds}


\author{Vestislav Apostolov}
\address{D\'epartement de Math\'ematiques, Universit\'e du Qu\'ebec \`a Montr\'eal, and, Institute of Mathematics and Informatics, Bulgarian Academy of Sciences }
\email{apostolov.vestislav@uqam.ca}

\author{Abdellah Lahdili}
\address{D\'epartement de Math\'ematiques, Universit\'e du Qu\'ebec \`a Montr\'eal}
\email{lahdili.abdellah@gmail.com}

\author{Kuan-Hui Lee}
\address{Department of mathematics and statistic, McGill University}
\email{kuan-hui.lee@mcgill.ca}

\author{Fr\'ed\'eric Rochon}
\address{D\'epartement de Math\'ematiques, Universit\'e du Qu\'ebec \`a Montr\'eal}
\email{rochon.frederic@uqam.ca}

%\thanks{V.A. is  grateful to X.X. Chen, T. Collins, G. Grantcharov, J. Streets and S. Sun  for their interest and discussions. He was supported in part by the NSERC Discovery Grant RGPIN-2023-03316. K.-H. L.was supported in part by a CIRGET post-doctoral fellowship and a FRQ Team Grant. F.R. was supported in part by the NSERC Discovery Grant RGPIN-2025-04223.}

\begin{abstract} We present a  gluing construction for a class of 6th order geometric PDE's on a K\"ahler manifold or orbifold, arising in the study of Bismut Hermitian Einstein complex $3$-folds, the supersymmetric supergravity models in physics, and in Dervan's  Z-charge stability theory of K\"ahler manifolds. As an application, we obtain a desingularization theorem for orbifolds with crepant singularities admitting a solution of the PDE,  as well as a blowing-up of points theorem  for smooth complex $3$-folds admitting such solutions. These results yield new examples of regular Bismut Hermitian Einstein structures defined on $3(S^2\times S^4)\sharp 4(S^3\times S^3)$, $5(S^2\times S^4)\sharp 6(S^3\times S^3)$ and $9(S^2\times S^3)\times S^1$, as well as infinite topological types of smooth compact $7$-dimensional manifolds $Y_7$ providing solutions to the supersymmetric ${\rm AdS}_3 \times Y_7$  type IIB  supergravity models.
\end{abstract}

%\subjectclass{53C21, 53C25, 53C55}

\maketitle

\tableofcontents

\numberwithin{equation}{section}

\section{Introduction}

\subsection{Motivation} Non-K\"ahler,  Calabi--Yau type Hermitian  manifolds are subject of intense recent interest in mathematical physics and complex geometry (see e.g.  \cite{fernandez2014non, finosurvey, JFS,garciafern2018canonical,ivanov2010heterotic,phong2019geometric,picard2024strominger,st-geom, tosatti2015non}).  A prominent special class consists of Hermitian manifolds  whose fundamental $2$-form is $\partial \bar\partial$-closed and  whose {Bismut--Ricci form} is zero: such Hermitian structures are known in physics literature as \emph{Calabi-Yau geometries with strong torsion}~\cite{StromingerSST}  and in mathematics literature as \emph{Bismut--Hermitian--Einstein} (BHE for short) structures~\cite{JFS, GRFBook,ABLS}. They solve the Hermite--Einstein equations on a certain holomorphic vector bundle associated to the Hermitian manifold~\cite{JFS}, arise as fixed points of the pluriclosed flow~\cite{PCF,PCFReg},  and are also related to solutions of the supersymmetric ${\rm AdS}_3 \times Y_7$ and ${\rm AdS}_2 \times Y_9$ of type IIB  and type D=11 supergravity models introduced in ~\cite{CGMS, GK}. The main motivation of this article is to develop an existence theory for (non-K\"ahler) compact BHE  Hermitian manifolds in complex dimension $3$, as a sequel to~\cite{ABLS, ALL2026}  where the main geometric features of this class of manifolds have been obtained.

\smallskip
It was established in \cite{ABLS} that any compact, non-K\"ahler BHE manifold $(N, J_N, g_N)$ of complex dimension $m+1$ admits a pair of commuting Killing fields $\{V, W\}$ of unit norm, defining a balanced transversal Hermitian structure of complex dimension $m$, which satisfies a number of geometric relations. Conversely, any balanced complex $m$-dimensional manifold $(M, g, J)$ verifying these relations locally gives rise to an $(m+1)$-dimensional BHE manifold, obtained as a principal $2$-dimensional torus bundle over $M$. This correspondence can be made global if we assume that  the Killing fields $\{V, W\}$ generate a $2$-dimensional torus $T^2$ inside the isometry group of $(N, g_N)$, in which case the transversal geometry descends to  $M:=N/T^2$,  defining a complex $m$-dimensional balanced orbifold. We refer to such a situation as  a \emph{quasi regular BHE structure}. We say that the BHE structure is \emph{regular} if, furthermore, $M$ is smooth, i.e. the $T^2$-action is free.

When $m=1,2$, the above correspondence specializes to K\"ahler geometries on $M$. In the case $m=1$, Gauduchon--Ivanov~\cite{GauduchonIvanov} showed that $M$ must be a Riemann surface with positive  constant Gauss curvature, thus leading to a remarkable rigidity result for compact non-K\"ahler BHE complex surfaces: 
\begin{theorem_int}\label{thm:m=2}\cite{GauduchonIvanov} Suppose $(N, J_N, g_N)$ is a compact non-K\"ahler BHE complex surface. Then $(N, J_N, g_N)$  is a Hopf surface isometrically covered by $\bbR \times S^3$.
\end{theorem_int}
When $m=2$, the results in \cite{ABLS} identify the orbifold K\"ahler geometry in the quasi-regular case,  and further reduce the existence of such geometries to solving a PDE problem on $M$, reminiscent of the famous Calabi problem in K\"ahler geometry: it seeks for a K\"ahler metric $g_{\varphi}$ with K\"ahler  form $\omega_{\varphi}= \omega_0 + dd^c\varphi$ in a given K\"ahler deRham cohomology class $2\pi \alpha:=[\omega_0] \in H^{1,1}(M, \bbR)$,  which satisfies  
\begin{equation}\label{PDE-0}
\Sc^{\beta}(\omega_{\varphi}):=\Delta_{{\varphi}}R_{\varphi} + \frac{R_{\varphi}^2}{2} -\|{\rm Rc}_{\varphi}\|_{{\varphi}}^2-\|\theta_{\varphi}\|_{\varphi}^2 =0, \qquad R_{\varphi}>0,\end{equation}
where $R_{\varphi}$, ${\rm Rc}_{\varphi}$, $\Delta_{{\varphi}}$ and $\| \cdot \|_{{\varphi}}$ denote, respectively, the scalar curvature, Ricci tensor, Laplacian,  and tensorial norm of $g_{\varphi}$, and $\theta_{\varphi}$ denotes the harmonic representative with respect to $g_{\varphi}$ of a given deRham class $2\pi\beta \in H^{1,1}(M, \bbR)$, orthogonal to $\alpha$ with respect to the global cup product of $H^2(M, \bbR)$. 
Conversely, given an orbifold K\"ahler surface $M$ solving \eqref{PDE-0} and a smooth manifold $N$ which is a principal $T^2$ orbifold bundle over $M$ whose Euler class belongs to ${\rm span}_{\bbR}\{c_1(M), \beta\}$, the construction from \cite{ABLS} allows one to obtain a quasi regular  BHE structure on $N$. If, furthermore,  $M$ is a \emph{smooth} K\"ahler surface, we obtain a \emph{regular} BHE structure on $N$.

\smallskip
Solutions of \eqref{PDE-0} are still  scarce. When $M$ is a \emph{smooth} compact  K\"ahler surface, prior to this work, the only solutions known to us appear on $M= \bbP^1\times \bbP^1$  with $(\omega, \theta_{\omega})$ being suitable combinations of Fubini--Study metrics on the factors, and  on $M=\bbP(E) \to C$  with $C$ being an elliptic curve  and $E$ a polystable rank $2$ vector bundle over $C$, in which case $(\omega, \theta_{\omega})$ again correspond to forms which decompose as a linear combination of the flat metric on $\bbC$ and the Fubini-Study metric on $\bbP^1$ when pulled back to the universal cover $\bbC \times \bbP^1$ of $\bbP(E)$. These special solutions underline the so-called \emph{Samelson} BHE locally homogeneous 3-folds, covered respectively by the Lie groups
\begin{equation}\label{samelson}
S^3 \times S^3, \qquad \bbR^3\times S^3.
\end{equation}
Furthermore, in \cite{ALL2026}, a family of orbifold K\"ahler surfaces satisfying \eqref{PDE-0} were constructed, leading to the following result
\begin{theorem_int}\label{thm:previous}\cite{ABLS,ALL2026} Suppose $(N, J_N, g_N)$ is a compact non-K\"ahler BHE complex $3$-fold.
\begin{enumerate}
\item[(a)] The complex dimension $h^{1,1}_{\rm BC}(N)$ of Bott--Chern cohomology group $H^{1,1}_{\rm BC}(N, \bbC)$ is at least $2$-dimensional. 
\item[(b)] If $(N, J_N, g_N)$ is \emph{regular} and $h^{1,1}_{\rm BC}(N)=2$, then $(N, J_N, g_N)$ locally homogeneous and is covered by the Samelson's geometries \eqref{samelson}.
\item[(c)] There are (non-regular) quasi-regular,  non-K\"ahler BHE structures on $S^3\times S^3$ and $S^1\times S^2 \times S^3$, such that  $h^{1,1}(N)=2$ but which are not locally homogeneous.
\end{enumerate}
\end{theorem_int}
The above result disproves the rigidity phenomenon of Theorem~\ref{thm:m=2} for BHE 3-folds. It, however,  leaves open the possibility that the Samelson geometries are the only \emph{regular} BHE $3$-folds. 
\begin{question}\label{q:main} Are there \emph{regular} compact non-K\"ahler BHE $3$-folds which are not isometrically covered by the Samelson geometries \eqref{samelson}?
\end{question}
%By the correspondence in \cite{ABLS} and Theorem~\ref{thm:previous}-(b), Question~\ref{q:main} is equivalent to ask whether or not there exists a smooth  compact K\"ahler surface $M$ with second Betti number $b_2(M) \geq 3$, which admits a K\"ahler metric solving \eqref{PDE-0}. 
The main motivation of this work is to give a positive answer to  Question~\ref{q:main}, by establishing the following existence result.
\begin{theorem_int}\label{thm:main} The following smooth $6$-manifolds admit non-K\"ahler regular BHE structures:
\begin{itemize}
\item[(a)]  $N=3(S^2\times S^4) \sharp 4(S^3 \times S^3)$ with $h^{1,1}_{\rm BC}(N)=6$;  
\item[(b)] $N=5(S^2\times S^4) \sharp 6(S^3 \times S^3)$ with $h^{1,1}_{\rm BC}(N)=8$;  
%\item[(c)] $N=6(S^2\times S^4) \sharp 7(S^3 \times S^3)$ with $h^{1,1}_{\rm BC}(N)=9$
\item[(c)] $N=S^1 \times 9(S^2 \times S^3)$ with $h^{1,1}_{\rm BC}(N)=10$. 
%In particular, such structures exist on  $N=3(S^2\times S^4) \sharp 4(S^3 \times S^3)$.
%\item[(b)] The smooth manifold $N=S^1 \times \sharp 9(S^2 \times S^3)$ admits regular BHE structures with $h^{1,1}_{\rm BC}(N)=10$.
\end{itemize}
\end{theorem_int}
In the above statement and henceforth, $kN' \sharp \ell N''$ denotes the (smooth) connected sum of $k$-copies
of the smooth manifold $N'$ with $\ell$-copies of the smooth manifold $N''$.

\bigskip
Recently, it was shown in \cite[Theorem 1.3]{streets-topo} that the universal cover of a compact regular non-K\"ahler BHE 3-fold must be one of the following $6$-manifolds:
\begin{equation*}\label{eq:streets}
(k-1)(S^2 \times S^4)\sharp k (S^3\times S^3), \, 1\leq k \leq 8, \qquad \bbR\times 9(S^2\times S^3), \qquad  \bbR^3\times S^3.  
\end{equation*}
In view of  Theorem~\ref{thm:main} and Samelson's examples \eqref{samelson}, it remains an open problem whether or not $N=(k-1)(S^2 \times S^4)\sharp k (S^3\times S^3)$ does admit regular BHE structures for $k=2,3, 5, 7$.

\subsection{General Approach} We shall study the PDE \eqref{PDE-0} in arbitrary complex dimensions $m \geq 2$, independent of its relevance to the BHE construction when $m=2$.
It was observed in \cite{ALL2026} that on a given compact K\"ahler manifold or orbifold $M$ of complex dimension $m\geq 2$, the geometric quantity $\Sc^{\beta}$ in \eqref{PDE-0} has an infinite dimensional momentum map interpretation, similar to the Donaldson--Fujiki setup~\cite{donaldson-moment, fujiki-moment} for the scalar curvature of a K\"ahler manifold. 
One of the consequences of this discovery led to a natural 
extension of \eqref{PDE-0},  similar to Calabi's introduction of the notion of extremal K\"ahler metrics as a generalization of constant scalar curvature K\"ahler metrics,   removing the obstructions coming from the relevant Futaki invariant~\cite{ALL2026}. Indeed, following \cite{ALL2026}, for given real $(1,1)$-deRham classes $\alpha, \beta \in H^{1,1}(M, \bbR)$ on $M$ such that 
\[ \alpha >0,  \qquad  \beta \cdot \alpha^{m-1}[M]=0,\] and a fixed maximal compact torus $\bbT$ in the group ${\Aut}_{\rm red}(M)$ of \emph{reduced automorphisms} of $M$, one wants to find a $\bbT$-invariant K\"ahler metric $\omega_{\varphi}=\omega_0 + dd^c \varphi$  solving
\begin{equation}\label{PDE}
\Sc^{\beta}(\omega_{\varphi}) = c_{\alpha, \beta} + \langle \xi^{\rm ext}_{\alpha, \beta}, \mu_{\varphi}\rangle,
\end{equation}
where $c_{\alpha, \beta}:= 2m(m+1)\left(\frac{(c_1^2(M) + \beta^2)\cdot \alpha^{m-2}}{\alpha^{m}}\right)$ is a topological constant, $\xi^{\rm ext}_{\alpha, \beta} \in {\rm Lie}(\bbT)$ is  determined by $(\alpha, \beta, \bbT)$,  and $\mu_{\varphi} := \mu_{\omega_0} + d^c \varphi$ is the normalized $\bbT$-momentum map of $\omega_{\varphi}$ of zero mean. Solutions to \eqref{PDE-0} appear when \eqref{PDE}  has a solution  and the a priori data $(c_{\alpha, \beta}, \xi_{\alpha, \beta}^{\rm ex})$ associated to $(\alpha, \beta)$ is zero. 
%The extension \eqref{PDE} is motivated from similar considerations which  led Calabi to introduce the notion of extremal K\"ahler metrics as a generalization of constant scalar curvature K\"ahler metrics: it allows to remove the obstructions coming from the relevant Futaki invariant~\cite{ALL2026}.   
In our geometric situation,  \eqref{PDE-0} would thus correspond to the case of K\"ahler metrics of zero scalar curvature in the Calabi setup. 

We shall refer in this paper to the quantity $\Sc^{\beta}(\omega)$ as the \emph{$\Sc^{\beta}$-scalar curvature} of $\omega$,  and to a solution of \eqref{PDE} as a \emph{$\Sc^{\beta}$-extremal} K\"ahler metric. Notice that \eqref{PDE} is a non-linear 6th-order PDE with respect to the $\bbT$-invariant K\"ahler potential $\varphi$, with an additional pseudodifferential term corresponding to the quantity $\|\theta_{\varphi}\|^2_{\varphi}$. From an analytical viewpoint, this is a harder geometric PDE to solve than the  Calabi problem. When the pseudodifferential term vanishes, K\"ahler metrics with constant $\Sc^0$-scalar curvature have been already studied in K\"ahler geometry,  in the framework of Fano manifolds. This is the case of $\alpha = c_1(M)$ and $\beta=0$ when the  solutions to \eqref{PDE} are precisely the critical points of the energy functional ${E}_1$ introduced in \cite{Chen-Tian1},   and further studied in \cite{Chen-Wang, Song-Weinkove}. Notice that a K\"ahler--Einstein metric in $c_1(M)$ is an example of $\cZ^0$-extremal metric.  K\"ahler metrics with constant $\Sc^0$-scalar curvature also appear in the recent works~\cite{Dervan, Dervan-Hallem} as a special case of a more general notion of \emph{$Z$-critical K\"ahler metrics}. In particular, a suitable stability condition, similar to K-stability in the constant scalar curvature case, is introduced in~\cite{Dervan} (see also \cite[Appendix A]{ALL2026}) and is  related to the existence problem \eqref{PDE}. Finally, solutions of \eqref{PDE} with $\beta=0$ and $c_{\alpha, \beta}=0= \xi_{\alpha, \beta}^{\rm ex}$ have been studied in complex dimensions $m=3,4$ in the framework of supersymmetric ${\rm AdS}_3 \times Y_7$  and ${\rm AdS}_2 \times Y_9$ of type IIB  and type D=11 supergravity~\cite{GK,CGMS}, and in complex dimensions $m=2,3$ in the framework of $\nabla$-Einstein Sasaki manifolds~\cite{fino-et-al}. In this special case (i.e. $\beta=0$ and $c_{\alpha,0}=0=\xi_{\alpha, 0}^{\rm ext}$),  the PDE \eqref{PDE} is also referred to as  the \emph{Box equation}. Therefore, besides the solutions of \eqref{PDE-0} when $m=2$ underlying the existence results for BHE structures of Theorem~\ref{thm:main}, the  general existence results for \eqref{PDE}  we present in this paper are applicable to all these related geometric problems.

\subsection{Desingularization and Blowing-up Theorems} Thus motivated, our main technical result in this paper is a general existence theorem for solutions of \eqref{PDE} defined on smooth compact K\"ahler manifolds $\widehat{M}$, obtained by a desingularization of a known solution on a (certain type of) compact  K\"ahler orbifold $M$. More precisely, 
we consider the following geometric setting.  
%\vesti{Can we introduce a family $\omega_{\la}$ of $\Sc^{\beta}$-extremal metrics with $|\la-\la_0|< \delta_0$,  and then claim existence on the resolution for $|\la-\la_0| <\delta_0'<\delta_0$ and $0< \epi < \epi_0$?} 
Let $(M, \omega)$ be a compact complex $m$-dimensional K\"ahler orbifold $(m\geq 2)$ with only isolated singular points $\{p_1, \ldots, p_\ell\}$ modelled on $\bbC^m/\Gamma_i$ ($\Gamma_i \subset U(m)$),  such that $\omega$ is a $\Sc^{\beta}$-extremal K\"ahler metric with positive scalar curvature $R_{\omega}>0$. Here we assume that $(M, \omega)$ is invariant under the Hamiltonian isometric action of a maximal torus $\bbT$ which, by \cite{ALL2026},  is also a maximal torus in the group $\Aut_{\rm red}(M)$. Notice that, by the usual definition of  smooth orbifold tensors (see e.g.~\cite{Boyer-Galicki-book}), each point $p_i$ is fixed under the $\bbT$-action and there is an induced \emph{linear} $\bbT$-action on $\bbC^m$ commuting with $\Gamma_i$. We assume, furthermore,  that $\bl: \widehat M \to M$ is a  $\bbT$-equivariant smooth K\"ahler resolution of $M$, i.e. $\widehat{M}$  is a smooth K\"ahler manifold with $\bbT \subset \Aut_{\rm red}(\widehat{M})$ and $\bl$ is a bimeromorphic proper $\bbT$-equivariant complex analytic morphism such that 
$\bl : \widehat{M}\setminus (E_1 \cup \cdots \cup E_{\ell}) \to M\setminus\{p_1, \ldots, p_{\ell}\}$ is a biholomorphism, where $E_i:= \bl^{-1}(p_i), \, i=1, \ldots, \ell$ denote the exceptional divisors. In particular,  for each $i=1, \ldots, \ell$, $\bbC^m/\Gamma_i$ admits a $\bbT$-equivariant K\"ahler resolution 
\[ \bl_i : X_i \to \bbC^m/\Gamma_i.\]
In the above setup, our main result gives a sufficent condition for $\widehat M$ to admit an $\Sc^{\hat\beta}$-extremal K\"ahler metric  in the K\"ahler classes {$\hat \alpha_{\epi}: = \bl^*[\omega] +\epi^2\left([\omega_1] +\cdots + [\omega_{\ell}]\right)$ for all sufficiently  small $\epi$, where $\hat\beta=\bl^*(\beta)$ is the pullback of $\beta$ to $\widehat M$ and the $(1,1)$-class $[\omega_i]$ on $X_i$ is seen as a $(1,1)$-class on the resolution $\widehat{M}$ (see Lemma~\ref{fs.6})}.
\begin{theorem_int}[Desingularization Theorem]\label{thm:desingularization} Let $(M, \omega, \bbT)$ be a $\Sc^{\beta}$-extremal compact K\"ahler orbifold  as above and $ \bl : \widehat{M}\to M$ a $\bbT$-equivariant K\"ahler resolution with induced local $\bbT$-equivariant K\"ahler resolutions 
$\bl_i : X_i \to \bbC^m/\Gamma_i$. %(a) Suppose that each $X_i$ admits a $\bbT$-invariant ALE K\"ahler metric with zero $\Sc^0$-scalar curvature and non-negative scalar curvature. Then  for all sufficiently small real numbers $\epi_i>0, \, i=1, \ldots, \ell$ the K\"ahler class $\hat \alpha:= \bl^*[\omega]- \sum_{i=1}^{\ell} \epi_i[E_i]$  on $\widehat M$ admits a $\bbT$-invariant  $\Sc^{\hat \beta}$-extremal K\"ahler metric.
Suppose that each  $X_i$ admits a $\mathbb{T}$-invariant asymptotically locally Euclidean (ALE) Ricci-flat K\"ahler metric $\omega_i$. Then  there exists $\epi_0>0$ such that 
for any  $0<\epi< \epi_0$, $\widehat M$ admits a   $\bbT$-invariant  $\Sc^{\hat \beta}$-extremal K\"ahler metric $\hat\omega_{\epsilon}$ 
with everywhere positive scalar curvature, belonging to the K\"ahler class  $\hat \alpha_{\epi}= \bl^*[\omega] + \epi^2\left(\sum_{i=1}^{\ell} [\omega_i]\right),$
where  $\hat \beta := \bl^* (\beta)$.  {Furthermore, the family $\hat \omega_{\epsilon}$ admits  a polyhomogeneous expansion as $\epsilon\searrow 0$  on a suitable manifold with corners.  In particular, on any compact set outside the exceptional divisor of $\bl: \widehat M \to M$, $\hat{\omega}_{\epsilon}$ $C^{\infty}$-converges as $\epsilon\searrow 0$  to $\bl^*(\omega)$.} 
\end{theorem_int}
In the above statement, the ALE condition for the Ricci-flat K\"ahler metrics $\omega_i$ on $X_i$ is  introduced as in \cite{Arezzo-Pacard2006},  in terms of the existence of a K\"ahler potential $\Phi_i$ defined outside a compact subset of $X_i$
and satisfying there 
\[\omega_i =  dd^c \Phi_i, \qquad \Phi_i= \begin{cases}\bl_i^*\left(\frac{1}{4}\|z\|^2 +  \cO(\|z\|^{-2})\right) \, \text{when}\,  m=2,\\ \bl_i^*\left(\frac{1}{4}\|z\|^2  + a\|z\|^{4-2m} + \cO(\|z\|^{3-2m})\right) \, \text{when}\,  m\geq 3, \end{cases} \] 
where $a \in \bbR$ is a real constant and $\|z\|$ is the Euclidean norm of $\bbC^m$.   The existence of a Ricci-flat metric $\omega_i$ on $X_i$ with such  asymptotic is a non-trivial problem. It has been solved by D. Joyce~\cite{Joyce} in the case when $X_i$ is a \emph{crepant} resolution of $\bbC^m/\Gamma_i$: in complex dimensions $m=2,3$,  $\bbC^m/\Gamma_i$ admits a crepant resolution as soon as $\Gamma_i \subset SU(m)$ but in higher dimension it is not known in general which singularities admit crepant resolutions.

The method of proof of Theorem~\ref{thm:desingularization} that we develop in this paper allows us  to relax the Ricci-flat assumption on $\omega_i$ with the weaker condition that $\omega_i$ is a $\bbT$-invariant ALE K\"ahler metric with zero $\Sc^0$-scalar curvature and non-negative scalar  curvature, as well as to work with  weaker decays of the potential of $\omega_i$ at infinity, see Theorem~\ref{thm:general} below for the precise statement. However,  under these more general assumptions, we cannot conclude that the scalar curvature of the produced $\Sc^{\hat \beta}$-extremal metric on $\widehat{M}$ is everywhere positive: the Ricci-flat assumption for $(X_i, \omega_i)$ in Theorem~\ref{thm:desingularization} is an instance where we can effectively control the sign of scalar curvature of the desingularized metrics over the exceptional divisors, see Corollary~\ref{c:Ricci-flat ALE}. Another such situation appears if the  resolutions  $\bl_i : X_i \to \bbC^m/\Gamma_i$ admit $\bbT$-invariant ALE K\"ahler metrics $\omega_i$ with zero $\Sc^0$-scalar curvature and \emph{positive} scalar curvature of a certain decay at infinity, see Corollary~\ref{psc.8}. While we do not know of any example of such ALE K\"ahler spaces in complex dimension $m=2$,  we use the Calabi ansatz to construct in Proposition~\ref{p:new ALE} explicit examples on  $\cO_{\bbP^{m-1}}(-(m-2))\to \bbC^m/\bbZ_{m-2}$ for each $m \geq 3$.  In complex dimension $m=3$, this leads to an asymptotically Euclidean (AE) model on $\cO_{\bbP^2}(-1) \to \bbC^3$  which, similarly to the Burns--Simanca metric~\cite{burns-simanca} in the constant scalar curvature case, can be used to obtain a general blowing up  result of K\"ahler $3$-folds with constant $\Sc^{\beta}$-scalar curvature.
{\begin{theorem_int}[Blowing-up Theorem]\label{thm:blow-up} Let $(M, \omega_\la), \la \in B(\la_0, \delta_0) \subset \bbR^r$ be a smooth family of $\Sc^{\beta_\la}$-constant K\"ahler metrics with positive scalar curvature defined on a smooth compact complex $3$-dimensional manifold $M$ with $\Aut_{\rm red}(M)=\{1\}$. Let  $p_1, \ldots, p_\ell\in M$  be a finite set of points. Then, there exists positive reals $\epi_0$ and $\delta_0' < \delta_0$ such that for any  $0< \epi<\epi_0$  and  any $\|\la-\la_0\|< \delta_0'$,  the blow-up $Bl_{p_1, \ldots, p_\ell}: \widehat{M} \to M$ of $M$ at the points $p_1, \ldots, p_\ell$  admits a smooth family $\hat \omega_{\la, \epi}$ of $\cZ^{\hat\beta_\la}$-constant K\"ahler metrics with positive scalar curvature  in the K\"ahler class {$\hat\alpha_{\la,\epi} := Bl_{p_1, \ldots, p_\ell}^*([\omega_{\la}]) - \epi^2\left(\sum_{i=1}^{\ell}[E_i]\right),$ where $[E_i] \in H^{1,1}(\widehat{M}, \bbR)$ are the Poincar\'e dual of the exceptional divisors  and $\hat\beta_{\la} = Bl_{p_1, \ldots, p_\ell}^*(\beta_{\la})$.} { 
Furthermore, the family $\hat \omega_{\la, \epsilon}$ admits  a polyhomogeneous expansion as $\epsilon\searrow 0$  on a suitable manifold with corners and, on any compact set outside $\cup_{i=1}^{\ell} E_i$,  $C^{\infty}$-converges  as $\epsilon\searrow 0$  to $\bl_{p_1,\ldots, p_{\ell}}^*(\omega_\la)$.} 
\end{theorem_int}}

Theorems~\ref{thm:desingularization} and \ref{thm:blow-up} above can be thought of as analogs of the celebrated Arezzo--Pacard theorem~\cite{Arezzo-Pacard2006} and its  numerous extensions~\cite{Arezzo-Pacard2009, arezzo-pacard-singer, BJT, DS, dyrefelt3, hallam, Najafpour, rollin-singer2, s-sz, sz1,sz2, tipler}, obtained in the case of Calabi extremal K\"ahler manifolds and orbifolds. Our method of proof,  similarly to e.g. \cite{Arezzo-Pacard2006, Arezzo-Pacard2009, arezzo-pacard-singer, rollin-singer2, sz1, sz2},  uses gluing techniques. There are, however, some fundamental differences compared to these previous works. 

First, the geometric problem \eqref{PDE} studied here is much more rigid than the Calabi extremal  condition studied in the above mentioned works. To see this, suppose that $m=2$ and $\bl: \widehat{M} \to M$ is  an equivariant resolution of the $\Sc^{\beta}$-extremal K\"ahler orbifold $M$. If a family of $\Sc^{\hat \beta}$-extremal K\"ahler metrics $\hat \omega_{\epi}$ existed on $\widehat{M}$, and as in Theorem~\ref{thm:desingularization},  uniformly  
$C^{\infty}$ converged to $\bl^*(\omega)$  (as $\epi \to 0$) on each compact subset away from a divisor of $\widehat{M}$, then  (by integrating \eqref{PDE}) we would get
\begin{equation}\label{chern-constraint} c_1^2(\widehat{M}) = c_1^2(M).\end{equation}
Therefore, one can not expect to desingularize any $\Sc^{\beta}$-extremal orbifold as the above topological condition imposes strong rigidity both on the admissible singularities and on their resolutions. Similarly, we cannot expect to extend Theorem~\ref{thm:blow-up} in complex dimension $2$, see also Remark~\ref{nbu.1} below.  From this point of view, Theorem~\ref{thm:desingularization} is closer in spirit to the Kummer type constructions of Calabi--Yau metrics on K3-surfaces~\cite{kobayashi-K3, lebrun-singer, donaldson-K3}.  

The second  novel feature of the $\Sc^{\beta}$-extremal problem is the appearance of the term $\|\theta_{\varphi}\|^2_{\varphi}$ in the PDE \eqref{PDE}: This leads us to consider pseudodifferential operators in the PDE analysis, which in turn introduces an additional analytic difficulty compared to the methods in \cite{Arezzo-Pacard2006}. We solve this by incorporating the  $b$-surgery calculus of Mazzeo--Melrose \cite{Mazzeo-MelroseETA}.

\subsection{Applications} We first apply  Theorem~\ref{thm:desingularization} to obtain many  smooth examples of $\Sc^{\beta}$-extremal K\"ahler complex surfaces.  Our general strategy is to take orbifold quotients of known smooth solutions and use Theorem~\ref{thm:desingularization} to produce new examples on the corresponding resolutions. To this end, similarly to the constructions of constant scalar K\"ahler metrics by Rollin--Singer in \cite{rollin-singer1, rollin-singer2},  one can consider orbifold ``ruled'' K\"ahler surfaces $M$ obtained as global isometric quotients of $\widetilde{M}=\widetilde{\Sigma}\times \bbP^1$ where $\widetilde{\Sigma}=\bbP^1, \bbC$ or $\bbH$  is a simply connected Riemann surface, and  $\widetilde{M}$ is endowed with a product K\"ahler metric $\tilde\omega=\m\omega_{\tilde\Sigma} + \n \omega_{\bbP^1}, \, \, \m>0, \, \n>0$ of  constant Gauss curvature metrics on each factor. Notice that $\tilde\omega$ satisfies $\Sc^{\tilde \beta}(\tilde\omega)=const$ with respect to the primitive  closed $(1,1)$-form $\tilde\theta_{\la,\m,\n}=\la(-\n \omega_{\widetilde{\Sigma}} + \m \omega_{\bbP^1})$ on $\widetilde{M}$. If $(\tilde \omega_{\m, \n}, \tilde\theta_{\la, \m, \n})$  descends to $M$, we thus obtain a $\Sc^{\beta}$-extremal compact orbifold $(M, \omega_{\m, \n}, \theta_{\la, \m, \n})$.  It is shown in \cite{rollin-singer2} that such orbifolds $M$ arise from stable parabolic structures on smooth ruled surfaces. In order to apply Theorem~\ref{thm:desingularization} to these orbifolds, we need to further make sure that all the orbifold singularities carry Ricci-flat K\"ahler ALE resolutions:  in complex dimension $2$,  these are well understood  by the work of Kronheimer~\cite{Kronheimer1989}.  
In our case, this amounts to consider parabolic structures with weights $1/2$ %and $1/4$, 
so that $M$ has only cyclic singularities $A_1= \bbC^2/\langle \pm {\rm Id} \rangle$. %$=\frac{1}{2}(1,1)$, $A_3=\frac{1}{4}(1, 3)$ and $\frac{1}{4}(1,1)$,   where we use the notation $\frac{1}{q}(s, r)$ to denote the quotient of $\bbC^2$ by the cyclic group $\Gamma \subset U(2)$ generated by ${\rm diag}\left(e^{\frac{2i\pi s}{q}}, e^{\frac{2i\pi r}{q}} \right)$, $q, r, s \in \bbN$. 
We can then apply Theorem~\ref{thm:desingularization} with the Calabi--Eguchi--Hanson metric~\cite{Calabi,EH} on $\cO_{\bbP^1}(-2)\to A_1$.
%for the $A_1$ singularities, the multi Eguchi-Hanson metric with $3$ poles for $A_3$, and the $\bbZ_2$-quotient of the Eguchi--Hanson metric in the case of $\frac{1}{4}(1, 1)$. 
Implementing this strategy to the smooth ruled surface $\pi_2: \bbP^1 \times \bbP^1\to \bbP^1$ (where $\pi_2$ is the projection to the second factor) we get an existence result for K\"ahler metrics with constant $\Sc^{\beta}$-scalar curvature and positive scalar curvature on a large family of rational complex surfaces. To state it,   
%let\[ \widehat{M}_k := \text{k-fold blow-up of} \,  \, \bbP^1 \times \bbP^2,\] which,  for $k\geq 1$,  can also be viewed as a $(k+1)$-fold blow-up of $\bbP^2$. \begin{theorem}\label{thm:secondary} The smooth manifold $\widehat{M}_k$  admits (families of) K\"ahler metrics  with constant $\Sc^{\hat\beta}$-scalar curvature and positive scalar curvature. For $k=4,6,7$, the sign of $\Sc^{\hat\beta}$ of these K\"ahler metrics is positive negative or zero. For $k=8$, the  sign of $\Sc^{\hat\beta}$ is negative or zero,  and for $k\geq 9$ it is negative.\end{theorem}
for any point $p\in \bbP^1\times \bbP^1$, we denote by $\bl^2_p(\bbP^1\times\bbP^1)$ the two step iterated blow up of $p$, consisting of first blowing up $p$ and then blowing up the point of intersection of the exceptional divisor with the direct image of the fiber of $\pi_2$ through $p$. If $p_1, \ldots, p_k \in \bbP^1\times \bbP^1$ are points with pairwise distinct projection on each factor, we denote by  \[\widehat{M}_{2k}:=\bl^2_{p_1, \ldots, p_k}(\bbP^1 \times \bbP^1)\] the two step iterated blow up of $\bbP^1\times \bbP^1$ at each of the points $p_1, \ldots, p_k$. Clearly $\widehat{M}_{2k}$ is a rational complex surface diffeomorphic to $\bbP^2  \sharp (2k+1)\bar \bbP^2$. 
\begin{theorem_int}\label{thm:secondary} Let $\{p_1, \ldots, p_k\} \in \bbP^1\times \bbP^1$ be any collection of  points such that the projections to either factor are pairwise distinct. For $k\geq 2$, the two step iterated blow-up $\widehat{M}_{2k}$  of $\bbP^1\times \bbP^1$ at $p_1, \ldots, p_k$  admits (families of) K\"ahler metrics with constant $\Sc^{\hat\beta}$-scalar curvature and positive scalar curvature. When $k=2,3$, the sign of $\Sc^{\hat\beta}$ of these metrics is positive negative or zero. For $k=4$, this sign is negative or zero,  and for $k\geq 5$ the sign of $\Sc^{\hat\beta}$ is negative.\end{theorem_int}
The examples of zero $\Sc^{\hat\beta}$-scalar curvature on $\widehat{M}_4$ and $\widehat{M}_6$  lead, via the results in \cite{ABLS}, to the existence of BHE structures on the manifolds described in parts (a), (b)  of Theorem~\ref{thm:main}, respectively, whereas the examples of zero $\cZ^{0}$-scalar curvature on $\widehat{M}_{8}$ lead to the BHE structures of Theorem~\ref{thm:main}-(c). Notice that the product of $\widehat{M}_{8}$  with a flat complex elliptic curve yields  new solutions (in complex dimension $3$) of the \emph{Box equation} of GK geometry studied in \cite{GK,CGMS}.

\bigskip 
Beyond the examples in Theorem~\ref{thm:secondary},  we can use Theorem~\ref{thm:desingularization} to obtain  smooth K\"ahler surfaces with  positive constant $\Sc^{0}$-scalar curvature and positive scalar curvature from Fano K\"ahler--Einstein orbifold surfaces with only crepant isolated singularities. The following  such orbifold examples are obtained in \cite{OSS2016}: 
\begin{itemize}
\item $M=\bbP^2/\bbZ_3$ where $\bbZ_3$ acts by 
$\zeta \cdot [z_0, z_1, z_2] = [z_0, \zeta z_1, \zeta^2 z_2]$ ($\zeta=e^{2\pi i/3}$). The resolution  $\widehat{M}$ is a toric  surface with $b_2(\widehat{M})=7$ which admits toric K\"ahler metrics with positive constant $\cZ^0$-scalar curvature and positive scalar curvature in suitable K\"ahler classes arbitrary close to the boundary of the K\"ahler cone;
\item $M=(Bl_{p_0,p_1,p_2}(\bbP^2))/\bbZ_2$ is the Cayley cubic, viewed as an orbifold quotient by the Cremona involution of the K\"ahler--Einstein del Pezzo surface obtained by blowing up 3 points of $\bbP^2$, see \cite{OSS2016}. The minimal resolution $\widehat{M}$ of $M$ is then a complex surface diffeomorphic to the blow-up of $\bbP^2$ at $6$ points, which has no non-trivial holomorphic vector fields  and admits K\"ahler metrics with positive constant $\cZ^0$-scalar curvature and positive scalar curvature in suitable K\"ahler classes arbitrary close to the boundary of the K\"ahler cone.
\end{itemize}

\smallskip
Similarly, Theorem~\ref{thm:blow-up} leads to 
\begin{theorem_int}\label{thm:boxed} (a) Let $S$ be a smooth K\"ahler-Einstein delPezzo complex surface with $\Aut_{\circ}(S)=1$,  and $\Sigma$  a compact complex curve of genus ${\bf g}\geq 2$.  Then the blow up of $M:= S \times \Sigma$ at any collection of points admits a K\"ahler metric with $\Sc^{0}=0$.

(b) Let $M=\bbP(E) \to \Sigma$ be the total space of a projectively flat   Hermitian rank $3$ vector bundle $E$ over a compact complex curve $\Sigma$ of genus ${\bf g} \geq 2$. Then the blow up of $M$ at any collection of points admits a K\"ahler metric with $\Sc^{0}=0$.
\end{theorem_int}
Theorem~\ref{thm:boxed} provides a robust existence result for complex $3$-dimensional solutions of the \emph{Box equation} (i.e. solutions of \eqref{PDE-0} with $\theta_{\varphi}=0$) appearing in the so-called  GK geometry \cite{GK,CGMS}; multiplying these examples by a compact elliptic curve further leads to infinitely many complex $4$-dimensional examples.

\subsection{Outlook} {  In this paper, we develop a gluing method,  based on \cite{Mazzeo-MelroseETA}, in order to obtain existence results for $\Sc^{\beta}$-extremal metrics. While the single surgery space of Mazzeo and Melrose and a related adiabatic version have been used  before (see \cite{Benabida, Najafpour,Schroers-Singer,Zhu2026}) to tackle the existence of other special K\"ahler metrics (Ricci-flat,cscK, etc),  we go one step further by making use of the $b$-surgery \emph{pseudodifferential} calculus of \cite{Mazzeo-MelroseETA}.  This is particularly well-adapted to our present situation,  where the non-linear geometric PDE has a linearization which is a self-adjoint pseudodifferential operator, rather than a differential operator. On the other hand,  our method can possibly be used to extend our main results in a few  complementary directions, an endeavor which,  due to space limitation,  we postpone to future work. Below we list some of these directions.

\begin{itemize}

\item It is plausible that with some extra work our methods extend to \emph{partial} orbifold resolutions of a $\Sc^{\beta}$-extremal orbifold $M$. This can lead to new quasi-regular examples of BHE 3-folds, obtained for instance from the explicit orbifold construction in \cite{ALL2026} by resolving only the isolated $A_k$-singularities.
\item It is interesting to obtain a similar result to Theorem~\ref{thm:desingularization} but considering smoothings of $M$ instead of resolutions. This can  lead to further examples of regular BHE 3-folds as follows: the orbifold $M_0=(\bbP^1\times \bbP^1)/D_4$ with $D_4$ being the dihedral group of order $8$ admits an orbifold solution of \eqref{PDE-0} and has two $A_1$, one $A_3$ and one $\frac{1}{4}(1,1)$ T-singularity; by the results in \cite{Suvaina,OSS2016}, $M_0$ admits a global smoothing to a family of smooth K\"ahler surfaces $M_t, \, t\neq 0$ with $c_1^2(M_t)=1$. If we can show the existence of smooth solutions of \eqref{PDE-0} on $M_t$, this would lead to  regular examples of BHE $3$-folds on $N=6(S^2\times S^4) \sharp 7(S^3 \times S^3)$.
%\item Some of the new smooth solutions of \eqref{PDE-0} we obtained in this paper, as well as the orbifold solutions in \cite{ALL2026},  are  toric, so  the examples on $S^3\times S^3$ and in Theorem~\ref{thm:main}-(a) have a $T^4$-symmetry. The underlying complex structures $J_N$ are instances of LVMB complex manifolds. It is natural to study the moduli space of these BHE solutions within this class of manifolds.
\item In \cite{Arezzo-Pacard2009,arezzo-pacard-singer,sz1, sz2}, the existence results for  Calabi extremal K\"ahler metrics are obtained by relaxing the assumption $\Aut_{\rm red}(M)=\{1\}$ to certain conditions on the blown-up points and their stabilizer groups. 
%Furthermore, \cite{Arezzo-Pacard2006, Arezzo-Pacard2009} proves the existence of extremal K\"ahler metrics on blowups and resolutions in K\"ahler classes obtained by varying with independent scales the exceptional classes $[E_i]$. 
We expect similar generalizations to hold true for  Theorem~\ref{thm:blow-up}.
\item We emphasized in  Theorems~\ref{thm:blow-up}, \ref{thm:secondary} and ~\ref{thm:boxed} applications of our techniques in complex dimensions $2$ and $3$.  One can use Theorem~\ref{thm:general} combined with the ALE examples of solutions of \eqref{PDE-0} on $\cO_{\PP^{m-1}}(-(m-2)), \, m\geq 4$ presented in Proposition~\ref{p:new ALE}  and/or the Calabi Ricci-flat ALE examples on $\cO_{\PP^{m-1}}(-m)$ to generate higher dimensional $\Sc^{\hat \beta}$-extremal metrics out of $\Sc^{\beta}$-extremal orbifolds with $\frac{1}{(m-2)}(1, \cdots, 1)$ and/or $\frac{1}{m}(1, \cdots, 1)$ quotient singularities. 
\end{itemize}
}

\subsection{Outline} The paper is organized as follows: In Section~\ref{s:2}, we explain how to apply the desingularization Theorem~\ref{thm:desingularization} in complex dimension $m=2$ in order to obtain smooth $\Sc^{\beta}$-extremal K\"ahler surfaces from orbifold examples, establishing Theorem~\ref{thm:secondary}. %The main idea here is to generate orbifold examples with by taking finite quotients with only ADC singularities of the product $\Sigma\times \PP^1$ of Riemann surfaces, or of a delPezzo K\"ahler--Einstein complex surface $S$, following respectively constructions by Rollin--Singer~\cite{rollin-singer1, rollin-singer2}  and Odaka--Spotti--Sun~\cite{OSS2016} in the constant scalar curvature case. 
In Section~\ref{s:3}, we specialize the constructions in Section~\ref{s:2} to the case when $c_{\alpha,\beta}=0$,  and use further symmetry arguments to generate solutions with $\xi_{\alpha, \beta}^{\rm ext}=0$, thus leading to the existence results stated in Theorem~\ref{thm:main}. Section~\ref{s:4} is the technical core of the paper,  and can be read in isolation from the rest of the paper. It relies on the $b$-calculus of Melrose and the $b$-surgery calculus of Mazzeo-Melrose to obtain the general gluing Theorem~\ref{thm:general}. In this  statement,  the hypotheses for the ALE models are rather weak: we only require at this stage that each quotient singularity $X_i$ admits an ALE K\"ahler metric (in a weak sense formulated in Definition~\ref{ALE.1}) which solves \eqref{PDE-0} and which has a \emph{non-negative scalar curvature}. The conclusion is then that the resolution admits a solution to \eqref{PDE} without a sign control of  the scalar curvature over the exceptional divisor.  In Corollary~\ref{c:Ricci-flat ALE},  we improve the last point in the case when the $X_i$ are Ricci-flat, and thus establish Theorem~\ref{thm:desingularization}. Another geometric situation in which we control the positivity of the scalar curvature over the exceptional divisor in Theorem~\ref{thm:general} is formulated in Corollary~\ref{c:blow-up}: here we require instead of the vanishing of the Ricci tensor that each singularity $X_i$ admits an ALE solution of \eqref{PDE-0} with $\theta_{\varphi}=0$ and with \emph{strictly positive} scalar curvature with a given decay at infinity.  In the final Section~\ref{s:5}, we establish Theorem~\ref{thm:blow-up}  by using   Corollary~\ref{c:blow-up} and providing in Proposition~\ref{p:new ALE} an explicit construction of an AE solution of \eqref{PDE-0} with positive scalar curvature on $\cO_{\PP^2}(-1)$. 

\bigskip
\noindent
{\bf Acknowledgements.} V.A. is  grateful to X.X. Chen, T. Collins, M. Garcia Fernandez, G. Grantcharov, J. Streets and S. Sun  for their interest and discussions. V.A. was supported in part by an NSERC Discovery Grant RGPIN-2023-03316. 
K.-H. L. was supported in part by a CIRGET post-doctoral fellowship and a FRQ Team Grant. F.R. was supported in part by an NSERC Discovery Grant RGPIN-2025-04223.

\bigskip
\noindent
{\bf AI Statements and Declarations.} Data sharing and AI or other authomated tools were not used  during the current study.

\section{Examples of $\Sc^{\beta}$-extremal complex surfaces: deriving Theorem~\ref{thm:secondary} from Theorem~\ref{thm:desingularization}}\label{s:2} We explain in this section how to construct orbifolds $M$ with constant $\Sc^{\beta}$ scalar curvature to which Theorem~\ref{thm:desingularization} applies. 

As mentioned in the introduction, $M$ will be of the form 
\[ (\widetilde{\Sigma} \times \bbP^1)/G \qquad \text{with}  \qquad \widetilde{\Sigma}=\bbP^1, \, \bbC, \, \bbH,\]
where $G$ is a discrete group acting semi-freely by preserving the product structure on $\tilde{M}=\widetilde{\Sigma}\times \bbP^1$ and the product K\"ahler metric
\[ \tilde \omega_{\m, \n}= \m \omega_{\tilde{\Sigma}} + \n \omega_{\bbP^1}, \, \qquad \m>0, \,  \n>0, \]
where $\omega_{\tilde{\Sigma}}$ stands for the K\"ahler metric of constant scalar curvature $s_{\widetilde{\Sigma}}=2, 0$ or $-2$ on $\widetilde{\Sigma}$,  according to  as whether $\widetilde{\Sigma}= \bbP^1, \, \bbC, \, \bbH$. We further let
\[ \tilde \theta_{\la, \m, \n} := \la\left( -\m \omega_{\tilde{\Sigma}} + \n \omega_{\bbP^1}\right), \qquad \la \in \bbR.\] %\fr{I changed the roles of b and c in the definition of $\tilde \theta_{\la, \m, \n}$, otherwise the class is not primitive I think}
It is easy to check that $\tilde\theta_{\la, \m, \n}$ is a closed $\tilde\omega_{\m, \n}$-primitive $2$-form on $\widetilde M$,  and the $\Sc^{[\tilde\theta_{\la, \m, \n}]}$-scalar curvature (see \eqref{PDE-0}) and scalar curvature of $\tilde\omega_{\m, \n}$  are
\[ \Sc^{[\tilde\theta_{\la, \m, \n}]}(\tilde \omega_{\m, \n})=\frac{2s_{\widetilde{\Sigma}}}{\m\n} -4\la^2, \qquad R_{\tilde\omega_{\m, \n}}= \left(\frac{s_{\widetilde{\Sigma}}}{\m} + \frac{2}{\n}\right).\]
%\fr{I multiplied the RHS of the first equation by 2 since we are using the definition of $\Sc$ with $\nabla R$ instead of $\frac12\nabla R$}
We can thus define an orbifold K\"ahler structure $(\omega_{\m, \n.}, \theta_{\la, \m, \n})$ on $M$ with constant $\Sc^{\beta_{\la,\m, \n}}$ curvature and constant scalar curvature, where we have set $\beta_{\la, \m, \n} : = [\theta_{\la, \m, \n}]$ for the induced $(1,1)$-class on $M$.
Letting 
\[\n s_{\widetilde{\Sigma}} + 2\m >0, \]
we can make sure the scalar curvature $R_{\omega_{\m, \n}}>0$.

 \subsection{Constructing orbifolds via parabolic ruled surfaces} In \cite{rollin-singer1,rollin-singer2}, the authors explain how to construct orbifolds $M$ as above out of a \emph{stable parabolic structure} of a smooth ruled complex surface $\pi: \bbP(E) \to \Sigma$,  where $\Sigma$ is a compact complex curve and $E$ is a rank $2$ holomorphic vector bundle. The data introducing a \emph{parabolic structure} on $\bbP(E)$ consist of $k$ points $p_1=(\sigma_1, f_1), \ldots, p_k=(\sigma_k, f_k) \in \bbP(E)$ with $\sigma_i=\pi(p_i)\in \Sigma$ and $f_i \in \pi^{-1}(\sigma_i)$, such that $\sigma_i$ are pairwise distinct, and  associated positive rational numbers $v_i=r_i/q_i, \, \gcd(r_i, q_i)=1$. The parabolic structure $(\bbP(E); p_1, \ldots p_k;  v_1, \ldots, v_k)$ is then called \emph{stable} if for any holomorphic line sub-bundle $F\subset E$, we have
 \[ S_F^2 - \sum_{f_i \in S_F}v_i + \sum_{f_j\notin S_F}v_j >0, \qquad S_F := \bbP(F) \subset \bbP(E). \]
Using  a theorem of Metha--Seshadri~\cite{mehta-seshadri}, the following result is established in \cite[Thm.3.3.1]{rollin-singer1} and \cite[Prop.3.2.1]{rollin-singer2}
\begin{theorem}\label{thm:rollin-singer}\cite{rollin-singer1,rollin-singer2} Let $\pi: \bbP(E) \to \Sigma$ be a smooth ruled surface endowed with a stable parabolic structure $(p_1, \ldots, p_k; v_1, \ldots, v_k)$ where $v_j= r_j/q_j$, $r_j, q_j \in \bbN, \, \gcd(r_j, q_j)=1$. Suppose that the orbifold Riemann surface $\bar\Sigma$ obtained from $\Sigma$ by introducing an orbifold singularity of order $q_i$ at each $\sigma_i=\pi(p_i)$ is a good orbifold surface, i.e. is an  orbifold quotient of $\widetilde{\Sigma}=\bbP, \bbC, \bbH$. Then there exists a semi-free isometric group action of $G$ on  $\widetilde{\Sigma} \times \bbP^1$,  associated to an irreducible representation  $\rho: \pi_1^{orb}(\bar\Sigma) \to PU(2)$ of the orbifold fundamental group $\pi_1^{orb}(\bar\Sigma)$ of $\bar\Sigma$, such that the orbifold  complex surface $M=(\widetilde{\Sigma} \times \bbP^1)/G$ has $k$ pairs of isolated orbifold singular points  $(m_i^+, m_i^-)$ where $m_i^+$ is a cyclic singularity of type $\frac{1}{q_i}\left(1, r_i \right)$ and $m_i^-$ is a  cyclic singularity of type $\frac{1}{q_i}(1, q_i-r_i)$. Furthermore, the minimal resolution $\widehat{M}$ of $M$ is biholomorphic to an iterated blow-up of the points $p_i \in \bbP(E)$,  according to an algorithm described in \cite[Prop.2.1.1]{rollin-singer1}.
\end{theorem}
In the above statement, by a  singularity of type $\frac{1}{q}(r, s)$ we mean $\bbC^2/\Gamma$ where $\Gamma$ is the cyclic subgroup of $U(1)\times U(1)$ generated by ${\rm diag}(e^{\frac{2\pi i r}{q}}, e^{\frac{2\pi i s}{q}} )$.

\subsection{Orbifold examples with isolated $A_1$  singularities.} We now apply Theorem~\ref{thm:rollin-singer} to the smooth ruled surface $\pi: \bbP(\cO_{\bbP^1} \oplus \cO_{\bbP^1})\to \bbP^1$, i.e. to $\pi_2:  \bbP^1 \times \bbP^1 \to \bbP^1$ where $\pi_2$ is the projection on the second factor. We thus take $k$ points $p_i=(f_i, \sigma_i) \in \bbP^1\times \bbP^1$ with $\sigma_i \neq \sigma_j$ when $i\neq j$,  and associate weights $v_i=1/2$. As any line sub-bundle of $\cO_{\bbP^1} \oplus \cO_{\bbP^1}$ is trivial, the stability condition is clearly satisfied provided that $k\geq 2$ and $f_i \neq f_j$ for $i\neq j$. The orbifold surface $\bar\Sigma$ in this case has orbifold points of order $2$, thus is a good orbifold, see Propositions~2.1.1 and 2.1.2 in \cite{rollin-singer2}. Furthermore, the associated orbifold $M= (\widetilde{\Sigma} \times \bbP^1)/G$ has $k$ pairs of cyclic orbifold points $A_1=\frac{1}{2}(1,1)$ to which Theorem~\ref{thm:desingularization} applies with the orbifold K\"ahler structure $(\omega_{\m, \n}, \theta_{\la, \m, \n})$  (assuming $\n s_{\widetilde \Sigma} + 2\m >0$) on $M$, and the Calabi--Eguchi--Hanson~\cite{Calabi,EH} K\"ahler Ricci-flat ALE metric on $\cO_{\bbP^1}(-2) \to A_1$. We notice that on the resulting smooth K\"ahler surface $\bl:\widehat{M}\to M$, the $2k$ $A_1$ singularities of $M$ are replaced with $(-2)$ rational curves $E_i,$ $ i=1, \ldots, 2k$, i.e. $\widehat{M}$ is biholomorphic to the minimal resolution $\widehat{M}_{2k}$ of $M$ in  Theorem~\ref{thm:rollin-singer}. Indeed, in our special case, the resolution is achieved by the two-step iterated blow up of each $p_i$  (as described in Theorem~\ref{thm:secondary}) which creates a string of $(-2)$-$ (-1)$-$(-2)$ curves for each pair of points $(m_i^+, m_i^-)$ on $M$.

\subsection{Proof of Theorem~\ref{thm:secondary}}
 By the discussion in the previous subsection and applying Theorem~\ref{thm:desingularization}, we have obtained  $4$ real parameter families of $\Sc^{\hat\beta_{\la, \m, \n}}$-extremal K\"ahler metrics $\hat \omega_{\la, \m, \n, \epi}$ on $\widehat{M}_{2k}$,  belonging to the  K\"ahler classes 
\[ \hat \alpha_{\m, \n, \epi} := \bl^*[\omega_{\m, \n}] - \epi^2 \sum_{i=1}^{2k}[E_i],\]
where $\hat \beta_{\la, \m, \n}= \bl^*[\theta_{\la,\m, \n}]$, 
$E_i$ are the $(-2)$-curves of the resolution and $[E_i]$ denote the Poincar\'e duals of $E_i$ in $H^2(\widehat{M}, \bbZ)$. {Notice that the second cohomology group of $\cO_{\PP^1}(-2)$ is one dimensional, so that the cohomology class $[\omega_i]$ of the Calabi--Eguchi--Hanson metric on $\cO_{\PP^1}(-2)$ is a negative multiple of $[E_i]$.}
In order to complete the proof of Theorem~\ref{thm:secondary}, it remains to show that $\Sc^{\hat\beta_{\la, \m, \n}}(\hat\omega_{\la, \m, \n, \epi})$ is constant and to determine its sign. To this end, we analyse  below the cases $k=2,3,4$ separately,  in the Examples~\ref{zero-scalar}, \ref{zero-scalar1} and \ref{elliptic},  providing  an alternative more explicit construction of $M$.

\bigskip

Suppose $k=2$: the orbifold Riemann surface $\bar\Sigma$ is then $\bbP^1$ with two singularities  of order $2$. This is a good orbifold $\bar\Sigma = \bbP^1/\bbZ_2$ where the action of $\bbZ_2$ on $\bbP^1$ is given in homogeneous coordinates by $[z_0, z_1] \to [z_0, -z_1]$. It is not hard to see that in this case $M$ is given by the following
\begin{ex}\label{zero-scalar} Let $M= (\bbP^1 \times \bbP^1)/\bbZ_2$ be the orbifold quotient  under  the  diagonal $\bbZ_2$  action
\[([z_0, z_1], [w_0, w_1]) \to  ([z_0, - z_1], [w_0, -w_1]).\]
$M$ is a toric complex orbifold surface with  4 isolated $A_1:=\bbC^2/\langle\pm {\rm Id}\rangle$ quotient singularities. The product  K\"ahler metric $\tilde\omega_{\m,\n}:= \m\omega_{\bbP^1} + \n \omega_{\bbP^1}$ $(\m>0, \n>0)$  and the corresponding harmonic primitive $(1,1)$-form $\tilde \beta_{\la, \m, \n}:= \la(-\m\omega_{\bbP^1} + \n\omega_{\bbP^1}), \, \la \in \bbR$  on $\bbP^1\times \bbP^1$  descend to $M$ and define a (toric) K\"ahler orbifold metric $\omega_{\m, \n}$ of constant $\Sc^{\beta_{\la,\m, \n}}$-scalar curvature. Let $\widehat{M}$ be the toric resolution of $M$,  obtained by replacing the $A_1$-singularities with 4 exceptional $(-2)$-curves or, equivalently,  seen as the toric blow-up of $\bbP^1\times \bbP^1$ at the four fixed points of the torus action.
By Theorem~\ref{thm:desingularization} we obtain $\Sc^{{\hat \beta}_{\la, \m, \n}}$-extremal toric K\"ahler metrics $\hat \omega_{\la, \m,\n, \epi}$ on $\widehat{M}$. Furthermore, using   the invariance of $(\widehat{M}, [\hat\alpha_{\m, \n, \epi}], [\hat\beta_{\la, \m, \n}])$ under the two reflectional $\bbZ_2$-symmetries induced by the permutations of $[z_0, z_1]$ and $[w_0,w_1]$-variables,   it follows that $\xi_{\hat\alpha_{\m, \n, \epi}, \hat \beta_{\la, \m, \n}}^{\ext}=0$,  whereas  \eqref{chern-constraint} yields that  $\Sc^{\hat \beta_{\la, \m, \n}}(\hat \omega_{\la, \m,\n,\epi})$ is  a  positive multiple of $\Sc^{\beta_{\la, \m, \n}}(\omega_{\m,\n})=\left(\frac{2}{\m \n} -4\la^2\right)$. In particular, according to the sign of $(1-\la^2 \m\n)$, we get toric K\"ahler metrics on $\widehat{M}$ with  positive scalar curvature and constant positive, zero or negative $\Sc^{\hat\beta_{\la,\m, \n}}(\hat\omega_{\la,\m, \n, \epi})$-scalar curvature. 
\end{ex}
We next analyse the case $k=3$. The orbifold Riemann surface $\bar\Sigma$ is now a $\bbP^1$ with 3 orbifold points of order $2$. One can realize (see \cite[Prop. 1.2.2] {rollin-singer2}) $\bar\Sigma = \bbP^1/D_2$ where $D_2$ is the dihedral group of order $4$ acting on $\bbP^1$. Using this, we can now describe $M$ as follows:
\begin{ex}\label{zero-scalar1} Let $M=(\bbP^1 \times \bbP^1)/D_2$ be the orbifold quotient under the diagonal action of the dihedral group $D_2$. Then $(\omega_{\m, \n}, \theta_{\la, \m, \n})$ defines a $\Sc^{\beta_{\la, \m, \n}}$-extremal orbifold metric with constant positive scalar curvature and $\Sc^{\beta_{\la, \m, \n}}(\omega_{\m, \n})= \left(\frac{2}{\m\n}-4\la^2 \right)$ on $M$; furthermore,  $M$ has 6 isolated $A_1$ singularities and  no non-trivial holomorphic vector fields~\cite[Prop.2.1.2]{rollin-singer2}. The resolution $\widehat{M}_6$ of $M$ has no non-zero holomorphic vector fields either. Theorem~\ref{thm:desingularization} then yields a family $\hat \omega_{\la,\m, \n, \epi}$ of K\"ahler metrics on $\widehat{M}$  which have constant $\Sc^{\hat\beta_{\la,\m, \n}}$-scalar curvature with the sign of $(1-\la^2\m\n)$ and positive scalar curvature.
\end{ex}
We now describe the case $k=4$. Then $\bar\Sigma$ is $\bbP^1$ with $4$ orbifold points of order $2$.  It follows that  $\bar\Sigma= \bbT^2/\bbZ_2$, where $\bbT^2:= \bbC/(\bbZ \oplus \sqrt{-1}\bbZ)$ and the $\bbZ_2$-action is defined on $\bbC$ by $z \to -z$. This $\bbZ_2$-action descends to $\bbT^2$ where it has 
$4$ fixed points and the quotient $\bbT^2/\bbZ_2$ is the pillow case.  We can now describe $M$ as follows.
\begin{ex}\label{elliptic} Let $M= (\bbT^2\times \bbP^1)/\bbZ_2$ be the quotient by the diagonal $\bbZ_2$-action $(z, [w_0, w_1]) \to (-z, [w_0, -w_1])$. The resulting orbifold $M$ has 8 isolated $A_1$-singularities and the pair $(\omega_{\m, \n}, \theta_{\la, \m, \n}), \m>0, \n>0$  (induced from  the product structure on  the universal cover $\bbC \times \bbP^1$) defines K\"ahler metrics  on $M$ of constant positive scalar curvature  and $\Sc^{\beta_{\la, \m, \n}}(\omega_{\m, \n}) = -4\la^2$. Furthermore, $\omega_{\m, \n}$ has an $S^1$ Hamiltonian isometry corresponding to the $S^1$-action on $\bbP^1$ which descends to $M$. By Theorem~\ref{thm:desingularization}, 
the  resolution $\widehat{M}$ of $M$  with $8$ copies of the Eguchi--Hanson ALE metric--- which is equivariantly biholomorphic to the minimal resolution $\widehat{M}_8$ by Theorem~\ref{thm:rollin-singer}---  admits a family of $S^1$-invariant $\Sc^{\hat\beta_{\la, \m, \n}}$-extremal K\"ahler metrics $\hat \omega_{\la, \m, \n, \epi}$ with positive scalar curvature. Using further the invariance of $(\widehat{M}_8, \hat\alpha=[\hat\omega_{\la,\m, \n, \epi}], \hat \beta= [\hat\theta_{\la, \m, \n}], S^1)$ by the $\bbZ_2$-symmetry corresponding to the reflection $[w_0, w_1] \to [w_1, w_0]$ on $\bbP^1$ (which in turn induces a reflection on the Lie algebra ${\rm Lie}(S^1)$), we have  $\xi^{\ext}_{\hat\alpha, \hat \beta}=0$, i.e. $\Sc^{\hat\beta}(\hat\omega_{\la, \m, \n, \epi})$ is constant. By \eqref{chern-constraint} this constant is proportional to $-\la^2$, i.e. can be either negative or zero. The latter case  occurs iff $\la=0$, showing that  the K\"ahler metrics $\hat{\omega}_{0, \m, \n, \epi}$ is a solution  \eqref{PDE-0} with $\hat\beta=0$. 
\end{ex}

\begin{remark}
Taking the product of $(\widehat{M}_8, \hat{\omega}_{0, \m, \n, \epi})$  with a flat elliptic curve $T^2$  we obtain complex $3$-dimensional solutions of the ``Box equation'' of \cite{CGMS, GK}.
\end{remark}

We finally analyze the case $k\geq 5$. In this case, the orbifold Euler characteristic $\chi^{orb}(\bar\Sigma)= 2- k/2 <0$, showing that $\bar\Sigma$ is a good orbifold covered by $\bbH$, i.e. $M= (\bbH\times \bbP^1)/\pi_1^{orb}(\bar \Sigma)$. As observed in  \cite[Thm.3.4.1]{rollin-singer1}, $M$ (and hence its minimal resolution $\widehat{M}_{2k}$), has no non-trivial holomorphic vector fields. Furthermore, using the product structure of $\bbH^2\times \bbP^1$, $M$ has induced K\"ahler metrics $\omega_{\m, \n}$  and corresponding closed primitive $(1,1)$-forms $\theta_{\la, \m, \n}$ with $R_{\omega_{\m, \n}}= \frac{2}{\m\n}(\m-\n)$ and $\Sc^{\beta_{\la, \m, \n}}(\omega_{\m, \n})= -\left(\frac{2}{\m\n} +4\right)<0$.  As $M$ has $2k$ $A_1$ singularities,  Theorem~\ref{thm:desingularization} yields K\"ahler metrics $\hat\omega_{\la, \m, \n, \epi}$ (with $\m>\n$)  on $\widehat{M}_{2k}$
with positive scalar curvatures and constant negative $\Sc^{\hat\beta_{\la, \m, \n}}$-scalar curvature. 

This concludes the proof of Theorem~\ref{thm:secondary}, modulo Theorem~\ref{thm:desingularization}. 
\hfill{$\square$}

\begin{remark} The connected component of identity $\Aut_\circ(Bl_p(\bbP^1 \times \bbP^1))$ of the automorphism group  of the blow-up of $\bbP^1 \times \bbP^1$ at a single point $p$ is not reductive, nor is the group $\Aut_\circ(\widehat{M}_2) \cong  \Aut_\circ(Bl_p(\bbP^1 \times \bbP^1))$ of the two-step iterated blow-up.
By the Lichnerowicz--Matsushima obstruction established in \cite{ALL2026}, $\widehat{M}_2$ cannot admit a K\"ahler metric with positive scalar curvature and constant $\Sc^{\beta}$-scalar curvature. From the view point 
of parabolic structures on $\bbP^1 \times \bbP^1$, notice that if such a structure is  defined by a single point, then it is unstable. Moreover,  the corresponding orbifold $\bar\Sigma$ is not a good orbifold in this case.
\end{remark} 

\subsection{$\Sc^{0}$-extremal K\"ahler surfaces from delPezzo orbi-surfaces} Here we provide examples of smooth K\"ahler surfaces with  positive scalar curvature and (positive) constant $\Sc^0$ scalar curvature,  obtained  as desingularizations of K\"ahler--Einstein orbifold surfaces with crepant singularities. These are examples of $Z$-critical metrics in the sense of \cite{Dervan}.

\begin{ex}  Let $M = \bbP^2/\bbZ_3$ where $\bbZ_3$ acts in homogeneous coordinates on $\bbP^2$ by  \[ \zeta_3 \cdot [z_0, z_1, z_2] = [z_0, \zeta_3 z_1, \zeta_3^2 z_2], \qquad \zeta_3^3=1.\] Then $M$ is a K\"ahler--Einstein orbifold with 3 isolated $A_2$ singularities  and an induced toric action. By Theorem~\ref{thm:desingularization}, applied with a toric multi-Eguchi-Hanson  metric on resolution of $A_2$~\cite{GibbonsHawking, Hitchin-inst},  the equivariant resolution $\widehat{M}$  of $M$ admits invariant (toric) $\Sc^0$-extremal metrics in the K\"ahler classes $\hat \alpha_{\epi}=\bl^*(c_1(M)) -\epi^2\left(\sum_{i=1}^6 E_i\right)$, where $E_i$ are the $(-2)$ exceptional curves of the resolution. As $\widehat{M}$ is a toric surface with $c_1^2(\widehat{M})=3$ which has a  $\bbZ_3$ rotational symmetry (corresponding to cyclically permuting $[z_0, z_1, z_2]$) preserving $\hat\alpha_{\epi}$, the extremal vector $\xi^{\ext}_{\hat\alpha_{\epi}, 0}=0$, i.e. the metrics have constant positive $\Sc^0$-scalar curvature. \end{ex}

\begin{ex} Let $M=(Bl_{p_0,p_1,p_2}(\bbP^2))/\bbZ_2$ be the Cayley cubic, viewed as an orbifold quotient by the Cremona involution of the K\"ahler--Einstein  delPezzo surface  obtained by blowing up 3 points of $\bbP^2$, see \cite{OSS2016}. $M$ is then a K\"ahler--Einstein (and hence  $\Sc^0$-extremal) orbifold delPezzo surface with trivial connected automorphism group and 4 isolated $A_1$ orbifold singularities. By Theorem~\ref{thm:desingularization}, the minimal  resolution  $\widehat{M} \cong \bbP^2 \sharp 6 \bar{\bbP}^2$ of  $M$, obtained by blowing the 4 singular points of $M$ and introducing $(-2)$ curves $E_i, \, i=1, \ldots, 4$,  admits a 1-parameter family of K\"ahler metrics with constant $\Sc^0$-scalar curvature in the K\"ahler class $\bl^*(c_1(M)) -\epi^2\left(\sum_{i=1}^4 [E_i]\right)$, where the parameter $\epi>0$ is small enough.
\end{ex}

\section{Examples of regular BHE 3-folds: deriving Theorem~\ref{thm:main} from Theorem~\ref{thm:secondary}}\label{s:3}

\subsection{Proof of Theorem~\ref{thm:main}}
By \cite{ABLS},  we can construct  regular non-K\"ahler BHE structures on the total space $N$ of a principal $T^2$-bundle over a smooth K\"ahler surface $(\widehat{M}, \hat\omega, \hat\beta)$  solving \eqref{PDE-0}, provided that 
the  Euler classes $\gamma_1, \gamma_2 \in H^2(\widehat{M}, \bbZ)$ belong to the $\bbR$-span of $c_1(\widehat{M})$ and $\hat \beta$. We next describe the smooth solutions of \eqref{PDE-0} on $\widehat{M}_{2k}$ given by Theorem~\ref{thm:secondary},  underlying the examples of Theorem~\ref{thm:main} via the above construction. 

\smallskip
In the case $k=2$, we take $\n=1, \m \in \bbN, \la=1/\m$,  and  
let $N$ be the principal $T^2$-bundle over $\widehat{M}_4$  with Euler classes  $\gamma_1=c_1(\widehat{M}_4), \, \gamma_2=\frac{1}{4\pi}[\hat\beta_{\m,1,1}] \in H^2(\widehat{M}_4, \bbZ)$. As explained in Example~\ref{zero-scalar}, we can think of $\widehat{M}_4$ as the toric blow-up of $\bbP^1\times \bbP^1$ at the $4$ distinct fixed points of the toric action, so by the blow-up formula 
\[ \gamma_1=2([A]+[B]) - [F_1] - [F_2] -[F_3] -[F_4], \qquad \gamma_2=-\m [A] + [B],\] where  $F_i$ are the $(-1)$-curves of $\widehat{M}_4$ (seen as the toric blow-up of $\bbP^1\times \bbP^1$) and  $[A], [B]$ stand for the pullbacks to $\widehat{M}_4$ of the generators of $H^2(\bbP^1\times \bbP^1, \bbZ)$ (given up to a factor of $\frac{1}{4\pi}$ by the cohomology classes of the Fubini--Study metrics on each factor).  It follows from the above description that $\gamma_1, \gamma_2$ split a factor $\bbZ^2$ in $H^2(\widehat{M}_4)\cong \bbZ^5$. This shows that $N$ is simply connected (see \cite[Prop.11.5]{topology-T2}) and, furthermore, 
by  \cite[Prop.12 \& Thm.13]{grantcharov2008calabi} (see also \cite[Thm.1.3]{streets-topo}),  $N\cong 3(S^2\times S^4) \sharp 4(S^3 \times S^3)$. This  yields Theorem~\ref{thm:main}~(a).

\smallskip
In the case $k=3$, we  take $\n=1, \m\in  \bbN, \la=1/\m$ and argue similarly. Notice that $\widehat{M}_6$ is diffeomorphic to the blow up $\widetilde{M}_6$ of $\bbP^1\times \bbP^1$ at $6$ distinct points, and the underlying complex structures on $(\bbP^1\times \bbP^1) \sharp 6 \bar{\bbP}^2$ are in the same connected component of almost complex structures. Thus, $c_1(\widehat{M}_6)= c_1(\widetilde{M}_6) \in H^2(\widehat{M}_6, \bbZ)$ and we can consider a principal $T^2$-bundle $N$ over $\widehat{M}_6$ with Euler classes 
\[ \gamma_1 = c_1(\widehat{M}_6)=2([A]+[B]) -\sum_{i=1}^6 [\tilde F_i], \qquad \gamma_2 =-\m [A] + [B],\]
where $\tilde{F_i}$ are the $(-1)$-curves of $\widetilde{M}_6$. Using \cite{grantcharov2008calabi, topology-T2, streets-topo} again, we conclude that $N\cong 5(S^2\times S^4) \sharp 6(S^3 \times S^3)$ and  Theorem~\ref{thm:main}~(b) follows.

\smallskip
Finally, we consider the case $k=4$ and $\la=0$. We then have a family of solutions $\hat\omega_{0,\m, \n, \epi}$ of \eqref{PDE-0} on $\widehat{M}_8$ with $\hat\beta=0$. As $\widehat{M}_8$ is simply connected, $c_1(\widehat{M}_8)$ is primitive in $H^2(\widehat{M}_8, \bbZ)$ and $c_1(\widehat{M}_8)=w_2(\widehat{M}_8) \, {\rm mod}  \, 2$, the principal $S^1$-bundle $Y$ over $\widehat{M}_8$ with Euler class $c_1(\widehat{M}_8)$ is 
diffeomorphic to $\sharp 9 (S^2\times S^3)$ (see e.g. \cite{topology-S1}). By \cite{ABLS},  
 $N=S^1\times 9(S^2 \times S^3)$  admits a family of  BHE structures and   Theorem~\ref{thm:main}~(c) follows.

\smallskip
To compute the dimension $h^{1,1}_{\rm BC}(N)$  of the Bott-Chern cohomology group $H^{1,1}_{\rm BC}(N, \bbC)$ in each case,  we use that $h^{1,1}_{\rm BC}(N) = b_2(\widehat{M}_{2k})=2 +2k$ (see \cite[Prop.2.21]{ABLS}).
\hfill{$\square$}
\section{A general gluing result and  Proof of Theorem~\ref{thm:desingularization}}\label{s:4}
\subsection{Analytic preliminaries: the $b$-calculus of Melrose} \label{b.0}

In this subsection, we will review the main definitions and properties of the $b$-calculus of Melrose \cite{MelroseAPS}.  This will be the occasion to adapt some of the results to the specific setting that we will need in this paper. In particular, we will extend the regularity result~\cite[Prop.~5.61]{MelroseAPS} for elliptic differential $b$-operators to the case of  a class of elliptic \emph{pseudodifferential} $b$-operators. In the interest of space, this review is not entirely self-contained.  We will in particular assume some familiarity of the reader with manifolds with corners, see e.g.\cite{Grieser,Melrose1992,MelroseAPS}  and \cite[\S~2]{HHM1995} for further details.

Let $W$ be a compact manifold with boundary $\pa W$.  We assume that $W$ is connected, but $\pa W$ may have several connected components.  Let $x\in\CI(W)$ be a boundary defining function for $W$, i.e. $x\ge 0$ on $W$, $x^{-1}(0)=\pa W$ and $dx$ is nowhere vanishing on $\pa W$.  The space of \emph{$b$-vector fields} on $W$ is the Lie subalgebra of vector fields given by 
$$
  \cV_b(W):= \{ \xi\in \CI(W;TW)\; | \; {\xi_{|_{\partial W}}} \; \mbox{is tangent to} \; \pa W\}.
$$ 
Near a point $p\in \pa W$, let $(x,y_1,\ldots, y_{\dim W-1})$ be coordinates with $p$ corresponding to the origin and $\pa W$ corresponding to the hypersurface $x=0$.  In these coordinates, a $b$-vector field takes the form
$$
    \xi= ax\frac{\pa}{\pa x} + \sum_{i=1}^{\dim W-1} b_i\frac{\pa }{\pa y_i}
$$
for some smooth functions $a$ and $b_1,\ldots, b_{\dim W-1}$.  In particular, $\cV_b(W)$ is a locally free projective $\CI(W)$-module of rank $\dim W$.  By the Serre-Swan theorem \cite{Swan1962}, there is a vector bundle ${}^{b}TW\to W$ on $W$, called the \emph{ $b$-tangent bundle}, which comes with a canonical identification
$$
   \CI(W;{}^{b}TW)= \cV_b(W).  
$$
This identification is in fact induced by a map of vector bundles $a_b: {}^{b}TW \to TW$ which can be seen as an anchor map giving  $({}^{b}TM,a_b)$ the structure of a Lie algebroid.  On the interior of $W$, the anchor map induces an isomorphism
\begin{equation}
   a_b: {}^{b}TW|_{W\setminus\pa W} \to T(W\setminus \pa W).
\label{b.1}\end{equation}
\begin{definition}   
A \emph{$b$-metric} is a Riemannian metric on $W\setminus\pa W$ which induces via \eqref{b.1} a bundle metric on ${}^{b}TW$. 
\label{b.2}\end{definition}
By \cite{ALN04, Bui}, any $b$-metric is complete of infinite volume with bounded geometry.  We will not assume that the corresponding bundle metric is \emph{smooth} all the way to the boundary, \ie that $g_b\in \CI(W; S^2({}^{b}T^*W))$ where ${}^{b}T^*W$ is the vector bundle dual to ${}^{b}TW$. Instead,  we will assume that $g_b$ is \emph{polyhomogeneous}, meaning
$$
    g_b\in \cA^{\cE}_{\phg}(W; S^2({}^bT^*W)= \cA_{\phg}^\cE(W)\otimes_{\CI(W)} \CI(W; S^{2}({}^bT^*W))
$$ 
for some index set $\cE$, where $\cA_{\phg}^{\cE}(W)$ denotes the space of smooth functions $f$ on $W\setminus \pa W$ admitting a polyhomogeneous expansion 
$$
     f\sim \sum_{(z,k)\in \cE} a_{\lambda,k} x^{z}(\log x)^{k}
$$
near the boundary in the sense of \cite[(5.67), (5.73)]{MelroseAPS}, where the coefficients $a_{\lambda,k}$ are smooth functions on $\pa W$.  

Using the convention that $\bbN$ is the set of nonnegative integers, recall from \cite[Definition~5.23]{MelroseAPS} that an \emph{index set} is a discrete subset of $\bbC\times \bbN$ such that 
\begin{gather*}
 \{(z_j,k_j)\}\subset \cE, \quad |z_j|+k_j\to \infty \quad \Longrightarrow \quad \Re z_j\to \infty; \\
 (z,k)\in \cE \quad \Longrightarrow \quad (z+j,\ell) \in \cE \quad \forall j\in \bbN, \; \forall \ell\in \{0,\ldots,k\}.
\end{gather*}
The second condition ensures in particular that $\cA^\cE_{\phg}(W)$ is a $\CI(W)$-module.  Given an index set $\cE$, we define its \emph{real part} by
$$
   \Re(\cE):= \{ (\Re(z),k)\; | \; (z,k)\in \cE\}\subset \bbR\times \bbN.
$$
We will say that an index set $\cE$ is real if $\Re(\cE)=\cE$.
Notice that the space $\bbR\times \bbN$ is totally ordered by the order relation
%\footnote{\color{blue} VA: is the sense of the second inequality correct? Should we write ``either .... or ....'' as logic dictates?}
$$
       (z_1,k_1)<(z_2,k_2) \; \Longleftrightarrow \; \mbox{either} \, z_1< z_2 \; \mbox{or else}  \; z_1=z_2 \; \mbox{and} \; k_1>k_2.
$$
For an index set $\cE$, we will use the notation $\Re(\cE)> (r,k)$ (or $\cE>(r,\ell)$ if $\cE$ is real) for some $(r,\ell)\in \bbR\times \bbN$ to mean that
$$
          (\Re z,k)> (r,\ell) \quad \forall (z,k)\in \cE.
$$
Similarly, we will write $\Re (\cE)\ge (r,\ell)$ (or $\cE\ge (r,\ell)$ if $\cE$ is real) for some $(r,\ell)\in \bbR\times \bbN$ if 
$$
 (\Re z,k)\ge  (r,\ell) \quad \forall (z,k)\in \cE.
$$
We will also use at times the simpler notation $z$ to denote $(z,0)\in \bbC\times \bbN$. Thus, $\Re(\cE)\ge 0$ will mean that $\Re(\cE)\ge (0,0)$.   

The space $\Diff^k_b(W)$ of \emph{differential $b$-operators of order $k$} corresponds to a finite linear combination of  differential operators generated by multiplication by elements of $\CI(W)$ and the composition of up to $k$ $b$-vector fields.  More generally, as explained in \cite[\S~2.4]{MelroseAPS} or \cite{ALN04}, given vector bundles $E_1\to W$ and $E_2\to W$, we can define the space {$\Diff^k_b(W;E_1,E_2)$} of differential $b$-operators acting on sections of $E_1$ and taking values in sections of $E_2$.  Replacing $\CI(W)$ by $\cA^\cE_{\phg}(W)$ for an index set $\cE$ such that $\Re \cE\ge 0$ yields more generally the space of $b$-differential operators $\Diff^{k}_{b,\cE}(W;E_1,E_2)$ of order $k$ with coefficients in 
$\cA^{\cE}_{\phg}(W;\Hom(E_1,E_2))$.

To construct good parametrices for differential $b$-operators, Melrose introduced in \cite{MelroseAPS} the notion of \emph{pseudodifferential $b$-operators}.  Their Schwartz kernels are defined on the \emph{$b$-double space}%~\footnote{\color{blue} V.A. maybe double b-space?} 
$$
   W_b^2:= [W^2; (\pa W)^2] \quad \mbox{with blow-down map} \quad \bl_b: W^2_b\to W^2
$$   
obtained by blowing up the corner $(\pa W)^2$ of $W^2$ in the sense of \cite[\S~4.1]{MelroseAPS}, see also \cite[\S~2.2]{HHM1995}.  Denote by $\fb=\fb(W):=\bl_b^{-1}((\pa W)^2)$ the new boundary hypersurface created by the blow-up of the corner and let 
 $\lf= \overline{\bl_b^{-1}(\pa W\times (W\setminus\pa W))}$ and $\rf= \overline{\bl_b^{-1}( (W\setminus\pa W)\times \pa W)}$ be the lifts of the boundary hypersurfaces $\pa W\times W$ and $W\times \pa W$ to $W^2_b$.  If $\pr_L: W^2\to W$ and $\pr_R: W^2\to W$ denote respectively the projections on the left and on the right factors, let $\pi_L:=\pr_L\circ\be_b$ and $\pi_R:= \pr_R\circ \be_b$ be the corresponding maps on $W^2_b$.  Let 
 $$
     {}^{b}\Omega(W)= |\Lambda^{\dim W}({}^{b}T^*W)|
 $$ 
be the bundle of $b$-densities on $W$ and set 
$$
   {}^{b}\Omega_R(W):= \pi_R^*({}^{b}\Omega(W)).
$$
Similarly, if $E_1$ and $E_2$  are vector bundles on $W$, we consider the vector bundle 
$$
  \Hom_b(E_1,E_2):= \pi_L^{*}E_2\otimes \pi_R^{*}E_1^*
$$
on $M^2_b$ and   let 
$$
   \cD_b:= \overline{\be_b^{-1}(\cD\setminus \pa \cD)}
$$
be the lift of the diagonal $\cD$ in $W^2$ to $W^2_b$. We then consider the space 
$$
    I^q(W^2_b; \cD_b; \Hom_b(E_1,E_2)\otimes {}^{b}\Omega_R(W))
$$
of conormal distributions of order $q$ with respect to $\cD_b$ in the sense of \cite[Definition~18.2.6]{Hormander3} taking values in the vector bundle $\Hom_b(E_1,E_2)\otimes {}^{b}\Omega_R(W)$.  Thus, away from $\cD_b$, this space corresponds to smooth sections of $\Hom_b(E_1,E_2)\otimes {}^{b}\Omega_R(W)$, while along $\cD_b$, it corresponds to distributions with singularities conormal to $\cD_b$.  

Recall from \cite{MelroseAPS, Melrose1992} that an \emph{index family} $\cE$ for $W^2_b$ is a triple of index sets $(\cE(\fb),\cE(\lf),\cE(\rf))$ to which we can associate the ring $\cA_{\phg}^{\cE}(W^2_b)$ of smooth functions on $W^2_b\setminus \pa W^2_b$ admitting  polyhomogeneous expansions at $\fb$, $\lf$ and $\rf$ respectively specified by $\cE(\fb),\cE(\lf)$ and $\cE(\rf)$.  We will say that the index family $\cE$ is \emph{real} if the corresponding index sets are real.  For an index family $\cE$ for $W^2_b$, set 
$$
    I^{q,\cE}_{\phg}(W^2_b; \cD_b; \Hom_b(E_1,E_2)\otimes {}^{b}\Omega_R(W)):= \cA^{\cE}_{\phg}(W^2_b)\otimes_{\CI(W^2_b)} I^q(W^2_b; \cD_b; \Hom_b(E_1,E_2)\otimes {}^{b}\Omega_R(W)).
$$
\begin{definition}
For $\cE$ an index family for $W^2_b$ and $E_1,E_2$ vector bundles on $W$, there is an associated space of pseudodifferential $b$-operators of order $q$ acting on sections of $E_1$ and taking values in sections of $E_2$,  defined by
$$
   \Psi^{q,\cE}_b(W;E_1,E_2):= I^{q,\cE}_{\phg}(W^2_b; \cD_b; \Hom_b(E_1,E_2)\otimes {}^{b}\Omega_R(W)).
$$
\label{b.3}\end{definition}
\begin{remark}
Compared to \cite{MelroseAPS} or \cite{MazzeoEdge}, in this definition we allow the singular part of the conormal distributions to have a polyhomogeneous expansion at $\fb$ instead of requiring that the singular conormal part has a smooth expansion at $\fb$.  This is to ensure that 
$$
    \Diff^m_{b,\cH}(W;E_1,E_2)\subset \Psi^{m,\cE}_b(W;E_1,E_2)
$$
for an index family $\cE$ such that $\cE(\fb)=\cH$. 
\label{b.4}\end{remark}

The pushforward theorem of Melrose \cite[Theorem~5]{Melrose1992} yields the following result.
\begin{proposition}
Let $\cE$ be an index family for $W^2_b$ and $\cF$ and index set for $W$.  If $\Re(\cE(\rf)+ \cF)>0$, then
$$
    A\sigma:= (\pi_L)_*(A\pi_R^*\sigma)
$$
is well-defined for $A\in \Psi^{q,\cE}_b(W;E_1,E_2)$ and $\sigma\in \cA^{\cF}_{\phg}(W;E_1)$ with $A\sigma\in \cA^{\cG}_{\phg}(W;E_2)$ for the index set
$$
    \cG:= \cE(\lf)\overline{\cup} (\cE(\fb)+\cF),
$$
where $\overline{\cup}$ denotes the extended union of index sets \cite[(5.120)]{MelroseAPS}.
\label{b.5}\end{proposition}
\begin{proof}
This well-known result follows from the pushforward theorem, see for instance \cite[(5.158)]{MelroseAPS} or \cite[Proposition~3.28]{MazzeoEdge} (where $\cE(\fb)=\bbN$ and a different convention for densities is used),  or else \cite[Theorem~3.6]{KR1} when $M$ in the statement of that theorem is the manifold with boundary $W$.  The proof still works when we take into account the weaker condition we imposed on the conformal distributions as indicated in Remark~\ref{b.4}.
\end{proof}

Using the $b$-triple space %~\footnote{\color{blue} VA: Or triple $b$-space?} 
\cite[(5.1)]{Mazzeo-MelrosePhi}, the pushforward theorem of Melrose also yields the following composition result.
\begin{proposition}\label{b.6}
If $\cE$ and $\cF$ are two index families on $W^2_b$ such that 
$$
       \Re(\cE(\rf)+ \cF(\lf))>0,
$$
then
$$
      \Psi^{q,\cE}_b(W;E_2,E_3)\circ \Psi^{q',\cF}_b(W;E_1,E_2)\subset \Psi^{q+q',\cG}_{b}(W;E_1,E_3)
$$
for the index family $\cG$ given by
$$
\begin{aligned}
   \cG(\fb)&= (\cE(\fb)+\cF(\fb))\overline{\cup} (\cE(\lf)+ \cF(\rf)), \\
    \cG(\lf)&= (\cE(\fb)+\cF(\lf))\overline{\cup} \cE(\lf), \\
   \cG(\rf)&= (\cE(\rf)+\cF(\fb))\overline{\cup} \cF(\rf).
\end{aligned}   
$$
\end{proposition}
\begin{proof}
This is the particular case of \cite[Theorem~5.3]{KR1} where $M$ in the statement of that theorem is the manifold with boundary $W$ with boundary fibration the projection onto a point.  Again, the proof still works when we take into account Remark~\ref{b.4}.
\end{proof}

If the index family $\cE$ is such that $\Re(\cE(\fb))\ge 0$, the restriction to $\fb$ is well-defined.  This allows us to define the indicial family as follows.  For our choice of boundary defining function $x\in \CI(W)$, let us also denote by $x$ the pull-back $\pi_L^*x$ and by $x'$ the pull-back $\pi_R^* x$.  Then $s:= \frac{x}{x'}$ induces a smooth function on $W^2_b\setminus (\lf\cup \rf)$.  On the interior of $\mathring{\fb}$, there is also an identification
$$
   \mathring{\fb}\cong \pa W\times ]0,\infty[
$$  
with $s$ corresponding to the projection on the factor $]0,\infty[$.
\begin{definition}
The \emph{indicial family} of an operator $A\in \Psi^{q,\cE}_b(W;E_1,E_2)$ for $\cE$ such that $\Re(\cE(\fb)\setminus\{(0)\})> 0$  is the family of operators $\bbC\ni \lambda\mapsto I(A,\lambda)$ defined by 
$$
   (I(A,\lambda)f)(w)= \int_{\pa W\times ]0,\infty[} A|_{\fb}(w,w',s) s^{-\lambda} f(w')
$$
for $f\in \CI(\pa W;E_1)$, where the integral on the left is in the variables $(w',s)$ and is using the density coming from $A$ and is in the sense distributions.  
\label{b.7}\end{definition}
Using the variable $u=\log s$, the indicial family corresponds essentially to the weighted Fourier transform of $A|_{\fb}$ in the $u$ variable, in fact a Fourier transform when $\Re \lambda=0$.  When $A\in \Diff^m_{b,\cE}(W;E_1,E_2)$ is a differential $b$-operator with index set such that $\Re(\cE)\ge 0$, then the indicial family can also be defined by 
$$
  I(A,\lambda)f:=( x^{-\lambda}A x^{\lambda} \widetilde{f})|_{x=0}
$$
for $f\in \CI(\pa W)$, where $\widetilde{f}\in \CI(W)$ is any smooth function with $\widetilde{f}|_{\pa W}=f$.

\begin{definition}
An \emph{indicial root} of the indicial family $I(A,\lambda)$ for $A\in \Psi^{q,\cE}_{b}(W;E_1,E_2)$ with  $q\in \bbN$ and $\Re(\cE(\fb)\setminus\{0\})\ge 0$ is a complex number $\zeta$ such that 
$$
    I(A,\zeta): L^{2,q}(\pa W; E_1)\to L^2(\pa W;E_2)
$$ 
is not invertible, where $L^2(\pa W;E_2)$ is the space of $L^2$-integrable sections of $E_2$ and  $L^{2,q}(\pa W; E_1)$ is the $L^2$-Sobolev space of order $q$ of sections of $E_1$.  A \emph{critical weight} of $I(A,\lambda)$ is the real part of an indicial root of $I(A,\lambda)$.
\label{b.8}\end{definition}

Using the notion of principal symbol for conormal distributions, there is a corresponding principal symbol for pseudodifferential $b$-operators.  To describe it, let $\mathcal{Q}: {}^{b}T^*W\to W$ be the bundle map and let $\cS^q({}^{b}T^*W; \mathcal{Q}^*\Hom(E_1,E_2))$ be the space of sections $\sigma$ of $\mathcal{Q}^*\Hom(E_1,E_2)$ such that
\begin{equation}
    \sup_{u,\xi} \frac{|\pa^{a}_u \pa^{b}_{\xi} \sigma|}{(1+|\xi|^2)^{\frac{q-|\be|}2}}< \infty \quad \forall {a, b}\in\bbN^{\dim W}
\label{b.9}\end{equation}
in any local trivialization over some open set $\cU$ of $E_1$, $E_2$ and ${}^{b}T^*W$ with 
$$
    {}^{b}T^*W|_{\cU}\cong \cU\times \bbR^{\dim W} \ni (u,\xi),
$$
and where the norm in \eqref{b.9} is with respect to any choice of fiberwise norm on $\mathcal{Q}^*\Hom(E_1,E_2)$ coming from the pull-back of a choice of fiberwise norm of $\Hom(E_1,E_2)$ on $W$.  If we set 
$$
 \cS^{[q]}({}^{b}T^*W; \mathcal{Q}^*\Hom(E_1,E_2)):= \cS^q({}^{b}T^*W; \mathcal{Q}^*\Hom(E_1,E_2))/ \cS^{q-1}({}^{b}T^*W; \mathcal{Q}^*\Hom(E_1,E_2)),
$$
then the principal symbol map is of the form
$$
     {}^{b}\sigma_q: \Psi^{q,\cE}_b(W;E_1,E_2)\to  (\mathcal{Q}^*\cA^{\cE(\fb)}_{\phg}(W))\otimes_{q^*\CI(W)}  \cS^{[q]}({}^{b}T^*W; \mathcal{Q}^*\Hom(E_1,E_2)).
$$

\begin{definition}
An operator $P\in \Psi^{q,\cE}_b(W;E_1,E_2)$ is $b$-elliptic if ${}^{b}\sigma_q(P)$ is invertible, namely if there exists $p\in \cS^{[-q]}({}^{b}T^*W; q^*\Hom(E_1,E_2))$ such that 
$$
      p \cdot {}^{b}\sigma_q(P)= {}^b\sigma_0(\Id_{E_1}) \quad \mbox{and} \quad   {}^{b}\sigma_q(P)\cdot p= {}^b\sigma_0(\Id_{E_2}).
$$
\label{b.10}\end{definition}

For an elliptic differential $b$-operator $P$, the set of indicial roots is discrete by analytic Fredholm theory \cite[\S~5.3]{MelroseAPS}.  In fact, the inverse of the indicial family is a meromorphic family of operators with poles precisely at the indicial roots.  In this case, the \emph{order} of an indicial root is the order of the corresponding pole.  We also define the \emph{order} $\ord(\delta)$ of a critical weight $\delta$ to be the maximal order among the indicial roots with real part equal to $\delta$.     If $z$ is an indicial root of order $p$, the \emph{rank} of the indicial root is
$$
  \rank(z):= \dim \left\{  u= \sum_{j=1}^p \frac{u_j}{(\lambda-z)^j} \quad \Big| \quad u_j\in \CI(\pa W;E_1), \quad I(P,\lambda)u \; \mbox{is holomorphic near $\lambda=z$}  \right\}.
$$
We also define the \emph{rank} $\rank(\delta)$ of a critical weight $\delta$ to be the sum of the ranks of the indicial roots with real part equal to $\delta$.

For elliptic pseudodifferential $b$-operators, we can construct a good parametrix as follows.

\begin{theorem}\label{b.11}
Let $P\in \Psi^{q,\cE}_b(W;E_1,E_2)$ be an elliptic pseudodifferential $b$-operator of order $q\in \bbN$ with index family $\cE$ such that $\Re(\cE(\fb)\setminus\{0\})\ge 0$.  Suppose that $I(P,\lambda)$ is also the indicial family of an elliptic differential $b$-operator of order $q$.  If $\delta\in \bbR$ is not a critical weight of $I(P,\lambda)$ and $]\delta-\mu_R, \delta+\mu_L[$ contains no critical weight for some $\mu_L>0$ and $\mu_R>0$ such that
$$
      \Re(\cE(\lf))>\mu_L+\delta\quad \mbox{and } \quad \Re(\cE(\rf))>\mu_R-\delta, 
$$ 
then there exist $Q\in \Psi^{-q,\cQ}_b(W;E_2,E_1)$ and $R\in \Psi^{-\infty,\cR}_b(W;E_2,E_2)$ such that
$$
    PQ=\Id -R
$$
for index families $\cQ$ and $\cR$ such that
$$
\begin{gathered}
    \Re(\cQ(\fb)\setminus\{0\})> 0, \quad \Re(\cQ(\lf))\ge (\mu_L+\delta, \ord(\mu_L+\delta)-1), \quad \Re(\cQ(\rf))\ge (\mu_R-\delta, \ord(\delta-\mu_R)-1), \\
    \cR(\fb)=\cR(\lf)=\emptyset, \quad \Re(\cR(\rf))\ge (\mu_R-\delta,\ord(\delta-\mu_R)-1).
\end{gathered}    
$$
\end{theorem}
\begin{proof}
Replacing $P$ by $x^{-\delta}Px^{\delta}$, we can assume that $\delta=0$.  Since the operator $P$ is elliptic and $0$ is not a critical weight, we can proceed as in \cite{MelroseAPS} to find an operator $Q_0\in \Psi^{-q,\cQ_0}_b(W;E_2,E_1)$ with 
$$
\Re(\cQ_0(\fb)\setminus\{0\})> 0, \quad \Re(\cQ_0(\lf))\ge (\mu_L, \ord(\mu_L)-1), \quad \Re(\cQ_0(\rf))\ge (\mu_R, \ord(-\mu_R)-1),
$$
and 
$$
      PQ_0=\Id -R_0
$$
for $R_0\in \Psi^{-\infty,\cR_0}_b(W;E_2,E_2)$ with 
$$
\cR_0(\fb)>0, \quad   \cR_0(\lf)> \mu_L \quad \mbox{and} \quad \cR_0(\rf)\ge (\mu_R, \ord(-\mu_R)-1).
$$
Proceeding as in \cite[Lemma~5.44]{MelroseAPS} and using our assumption that the indicial family of $P$ is also the indicial family of a differential $b$-operator, we can choose $Q_0$ more carefully and suppose that $\cR_0(\lf)=\emptyset$. Using Proposition~\ref{b.6}, notice then that there is $\mu>0$ such that for $k\in \bbN$, 
$$
    R_0^k\in \Psi^{-\infty,\cR_k}_b(W;E_2)
$$
with
$$
       \Re(\cR_k(\fb))>k\mu, \quad \cR_k(\lf)=\emptyset \quad \mbox{and} \quad \Re(\cR_k(\rf))\ge (\mu_R,\ord(-\mu_R)-1)
$$
and
$$
    \cR_k(\rf)= \cR_{k-1}(\rf)\overline{\cup} (\cR_{k-1}(\fb)+\cR_{0}(\rf)).
$$
This means that there exists an asymptotic sum
$$
  S\sim \sum_{k=1}^{\infty} R_0^k
$$
with $S\in \Psi^{-\infty,\cS}(W;E_2)$ such that
$$
     (\Id-R_0)(\Id+S)= \Id -R
$$
with
$$
\cS(\lf)=\emptyset, \quad \Re(\cS(\fb))>0 \quad \mbox{and} \quad \Re{\cS(\rf)}\ge (\mu_R,\ord(-\mu_R)-1),
$$
where $R\in \Psi^{-\infty,\cR}_b(W;E_2)$ is as in the statement of the theorem (with $\delta=0$).  Hence it suffices to take $Q=Q_0(\Id+S)$ and use Proposition~\ref{b.6} to conclude that $Q\in \Psi^{-q,\cQ}_b(W;E_2,E_1)$ as in the statement of the theorem.  
\end{proof}

This parametrix can be used to determine when an elliptic pseudodifferential $b$-operator is Fredholm.  To describe this,  let us choose connections and bundle metrics on $E_1$ and $E_2$ as well as a $b$-metric for $W$. This induces $L^2$-spaces $L^2_b(W;E_i)$ and associated $L^2$-Sobolev spaces $L^{2,k}_b(W;E_i)$.  
\begin{corollary}
Let $P\in\Psi^{q,\cE}_b(W;E_1,E_2)$ and $\delta$ be as in Theorem~\ref{b.11}.  Then $P$  induces a Fredholm operator
\begin{equation}
        P: x^{\delta}L^{2,q+k}_b(W;E_1)\to x^{\delta}L^{2,k}_b(W;E_2)
\label{b.15a}\end{equation}
for each $k\in \bbN$.
\label{b.15}\end{corollary}
\begin{proof}
Replacing $P$ by $x^{-\delta}Px{\delta}$, we can assume that $\delta=0$.    
Then the parametrix $Q$ of Theorem~\ref{b.11}  inverts $P$ on the right modulo the error term $R$.  Using Proposition~\ref{b.6} and the fact that the indicial family of the parametrix $Q$ is the inverse of the one of $P$, we see that 
$$
   QR=\Id - R'  \quad \mbox{for} \; R'\in\Psi^{-\infty, \cR'}_{b}(W;E_2)
$$
for an index family $\cR'$ such that $\Re(\cR'(\fb))>0$, $\Re(\cR'(\lf))>0$ and $\Re(\cR'(\rf))>0$.  By \cite[Lemma~5.33]{MelroseAPS}, see also \cite[Corollary~7.6]{KR1} in the case $M=W$ with boundary fiber bundle mapping $\pa W$ onto a point, the operators $R$ and $R'$ are compact when acting on $x^{\delta}L^{2,k}_b(W;E_2)$ and $x^{\delta}L^{2,q+k}_b(W;E_1)$.  This shows that \eqref{b.15a} is invertible modulo compact operators, hence that the map \eqref{b.15a} is Fredholm. 
\end{proof}

We can also construct a finer left parametrix.  Given $P\in \Psi^{q,\cE}_b(W;E_1,E_2)$, we can use bundle metrics on $E_1$ and $E_2$ and a choice of $b$-metric to define its formal adjoint $P^*\in \Psi^{q,\overline{\cE}}_b(W;E_2,E_1)$ with index family $\overline{\cE}$ defined by
$$
    \overline{\cE}(\fb)= \overline{\cE(\fb)}, \quad  \overline{\cE}(\lf)= \overline{\cE(\rf)} \quad \mbox{and} \quad  \overline{\cE}(\rf)= \overline{\cE(\lf)},
$$  
where we use the notation
$$
   \overline{\cF}= \{ (\overline{z},k)\; | \; (z,k)\in \cF\}
$$
for $\cF$ an index set.  Applying Theorem~\ref{b.11} to the formal adjoint yields the following.
\begin{corollary}\label{b.12}
Let $P\in \Psi^{q,\cE}_b(W;E_1,E_2)$ be an elliptic pseudodifferential $b$-operator of order $q\in \bbN$ with index family $\cE$ such that $\Re(\cE(\fb)\setminus\{0\})> 0$.  Suppose that $I(P,\lambda)$ is also the indicial family of an elliptic differential $b$-operator of order $q$.  If $\delta\in \bbR$ is not a critical weight of $I(P^*,\lambda)$ and $]\delta-\mu^*_R, \delta+\mu^*_L[$ contains no critical weight for some $\mu^*_L>0$ and $\mu^*_R>0$ such that
$$
      \Re(\cE(\rf))>\mu^*_L+\delta\quad \mbox{and } \quad \Re(\cE(\lf))>\mu^*_R-\delta, 
$$ 
then there exists $Q\in \Psi^{-q,\cQ}_b(W;E_2,E_1)$ and $R\in \Psi^{-\infty,\cR}_b(W;E_2,E_2)$ such that
\begin{equation*}
    QP=\Id -R
%\label{b.12a}
\end{equation*}
for index families $\cQ$ and $\cR$ such that
$$
\begin{gathered}
    \Re(\cQ(\fb)\setminus\{0\})> 0, \quad \Re(\cQ(\rf))\ge (\mu^*_L+\delta, \ord(\mu^*_L+\delta)-1), \quad \Re(\cQ(\lf))\ge (\mu^*_R-\delta, \ord(\delta-\mu^*_R)-1), \\
    \cR(\fb)=\cR(\rf)=\emptyset, \quad \Re(\cR(\lf))\ge (\mu^*_R-\delta,\ord(\delta-\mu^*_R)-1).
\end{gathered}    
$$
\end{corollary}
\begin{proof}
By assumption, Theorem~\ref{b.11} can be applied to the formal adjoint $P^*$ to obtain a right parametrix $\widetilde{Q}$ such that
$$
    P^*\widetilde{Q}=\Id- \widetilde{R}.
$$
Taking the formal adjoint in this equation then yields
$$
     \widetilde{Q}^*P=\Id- \widetilde{R}^*,
$$
so that it suffices to take $Q=\widetilde{Q}^*$ and $R=\widetilde{R}^*$.
\end{proof}

This can be used to obtain the following regularity result.
\begin{corollary}\label{b.13}
Let $P$ and $\delta$ be as in Corollary~\ref{b.12}.  If $\sigma\in x^{-\mu^*_L-\delta}L^2_b(W;E_1)$ is such that $P\sigma=f\in \cA^{\cF}_{\phg}(W;E_2)$ for an index set $\cF$ such that
$\Re(\cF)>-\mu^*_L-\delta$, then in fact $\sigma\in\cA^{\cG}_{\phg}(W;E_1)$ with index set 
\begin{equation*}
    \cG= \left( \cQ(\lf)\overline{\cup} (\cQ(\fb)+\cF) \right)\cup \cR(\lf).
%\label{b.13b}
\end{equation*}
\end{corollary}
\begin{proof}
The fact $\sigma\in x^{-\mu^*_L-\delta}L^2_b(W;E_1)$ ensures that $P\sigma$ is well-defined, while the condition on $\cF$ ensures that $Qf$ is well-defined when $Q$ is the left parametrix of Corollary~\ref{b.12}.   Applying the left parametrix $Q$ on both sides of $P\sigma=f$ yields
$$
    \sigma= Qf + R\sigma.
$$
Since $R$ is very residual in the sense of \cite{MazzeoEdge} and $\cR(\fb)=\cR(\rf)=\emptyset$, the term $R\sigma$ is well-defined and is an element of $\cA^{\cR(\lf)}_{\phg}(W;E_1)$.  The result thus follows by applying Proposition~\ref{b.5} to the term $Qf$.
\end{proof}

When $f=0$, this yields the following.
\begin{corollary}
Let $P\in\Psi^{q,\cE}_b(W;E_1,E_2)$ be an elliptic pseudodifferential $b$-operator of order $q\in \bbN$ with $\Re(\cE(\fb)\setminus\{0\})\ge 0$ such that its indicial family is also the indicial family of differential $b$-operator. Then for each $\kappa\in \bbR$ such that $\Re(\cE(\rf))>-\kappa$ and $\Re(\cE(\lf))>\kappa$, the weighted $L^2$-kernel $\ker_{x^{\kappa}L^2_b}P:= (\ker P)\cap  x^{\kappa}L^2_b(W;E_1)$ is finite dimensional with 
$$
\ker_{x^{\kappa}L^2_b}P\subset \cA^{\cG}_{\phg}(W;E_1)
$$
for some index set $\cG$ such that $\Re(\cG)>\kappa$.  More precisely, $\Re(\cG)\ge (-r,\ord(r)-1)$ if there exists a highest critical weight $r$ of $P^*$ such that $\Re(\cE(\lf))>-r>\kappa$, while otherwise $\Re(\cG)> -r$ for all $r\in\bbR$ such that $\Re(\cE(\lf))>-r>\kappa$.
\label{b.14}\end{corollary}
\begin{proof}
The  condition $\Re(\cE(\rf))>-\kappa$ is to ensure that $P$ acts on $x^{\kappa}L^2_b(W;E_1)$.  Choose   $\delta\in \bbR$ such that $\Re(\cE(\rf))>\delta>-\kappa$ and such that $\delta$ is not a critical weight of $P^*$.  
For this choice of $\delta$, we can then apply  Corollary~\ref{b.13} with $f=0$ to conclude that the elements of $\ker_{x^{\kappa}L^2_b}P$ are polyhomogeneous for some index set $\cG$.  Since all these elements must be in $x^{\kappa}L^2_b(W;E_1)$, we can suppose that $\Re(\cG)>\kappa$.  Thus, there exists $\nu>0$ such that $\Re(\cE(\lf))>\kappa+\nu$ and 
$$
    \ker_{x^{\kappa}L^2_b}P \subset x^{\kappa+\nu}L^2_b(W;E_1)\cap \ker P.
$$
Choosing $\nu$ such that $\kappa+\nu$ is not a critical weight of $P$, we conclude from Corollary~\ref{b.15} that $ \ker_{x^{\kappa}L^2_b}P$ is finite dimensional.  
Finally, pick  $\delta\in ]-\kappa-\nu,-\kappa[$ such that $\delta$ is not a critical weight of $P^*$.  Applying Corollary~\ref{b.13}, we see that  $\Re(\cG)\ge (-r,\ord(r)-1)$ if there exists a highest critical weight $r$ of $P^*$ such that $\Re(\cE(\lf))>-r>\kappa$, while otherwise $\Re(\cG)> -r$ for all $r\in\bbR$ such that $\Re(\cE(\lf))>-r>\kappa$.
\end{proof}

When the pseudodifferential operator is invertible, we can also provide a pseudodifferential characterization.  We will need the following one.
\begin{corollary}\label{b.16}
Suppose that $P\in \Psi^{q,\cE}_b(W;E_1,E_1)$ is formally self-adjoint with $\cE$ a real index family and that $P$ satisfies the assumptions of Theorem~\ref{b.11} with $\delta=0$ and $\mu_L=\mu_R=\mu>0$.  If $P$ is invertible as a Fredholm operator from $L^{2,q}_b(W;E_1)$ to $L^2_b(W;E_1)$, then 
$$
   P^{-1}\in \Psi^{-q,\cG}_{b}(W;E_1,E_1)
$$
with $\cG$ a real index family such that
$$
  \cG(\lf)=\cG(\rf)\ge (\mu,\ord(\mu)-1) \quad \mbox{and} \quad \cG(\fb)\ge 0.
$$
\end{corollary}
\begin{proof}
By Theorem~\ref{b.11} and Corollary~\ref{b.12}, there exist parametrices $Q_i\in \Psi^{-q,\cQ_i}_b(W;E_1,E_1)$ for $i\in\{1,2\}$ such that
$$
  PQ_2=\Id-R_2 \quad \mbox{and} \quad Q_1P=\Id-R_1
$$
for $R_i\in\Psi^{-\infty,\cR_i}(W;E_1,E_1)$, where $\cQ_i$ and $\cR_i$ are real index families such that
$$
   \cQ_i(\lf)\ge (\mu,\ord(\mu)-1), \quad \cQ_i(\rf)\ge (\mu,\ord(\mu)-1), \quad \cQ_i(\fb)\ge 0
$$
and 
$$
   \cR_2(\lf)=\cR_1(\fb)=\cR_2(\fb)=\cR_1(\rf)=\emptyset, \quad \cR_2(\rf)\ge (\mu,\ord{\mu}-1), \quad \cR_1(\lf)\ge (\mu,\ord{\mu}-1).
$$
Using these parametrices, we find that 
$$
   P^{-1}= P^{-1}\Id= P^{-1}(PQ_2+R_2)= Q_2 + P^{-1}R_2 \quad \mbox{and}  \quad P^{-1}=\Id P^{-1}= (Q_1P+R_1)P^{-1}= Q_1+ R_1P^{-1}.
$$ 
Substituting the first equation in the second one, we find that
\begin{equation}
   P^{-1}= Q_2 + (Q_1+R_1 P^{-1})R_2= Q_2+ Q_1R_2 + R_1P^{-1}R_2.
\label{b.17}\end{equation}
Since $R_1$ and $R_2$ are very residual in the sense of \cite{MazzeoEdge} and $P^{-1}$ is a bounded operator when acting on $L^2_{b}(W;E_1)$, notice that $R_1P^{-1}R_2\in \Psi^{-\infty,\cR}_b(W;E_1)$ for some real index family $\cR$ suc that
$$
      \cR(\lf)\ge (\mu,\ord{\mu}-1), \quad \cR(\rf)\ge (\mu,\ord{\mu}-1) \quad \mbox{and} \quad \cR(\fb)>0. 
$$
Hence, the result follows from \eqref{b.17} by applying Proposition~\ref{b.6}, the fact that  $\cG(\lf)=\cG(\rf)$ following a posteriori from the fact that $P$ is formally self-adjoint. 
\end{proof}

We will also need a less precise but more robust version of this corollary from \cite[Theorem~4.20]{MazzeoEdge}.
\begin{corollary}\label{b.18}
Let $P$ be as in Theorem~\ref{b.11} with $\delta=0$.  Suppose as well that $0$ is not a critical weight of $I(P^*,\lambda)$.  Let $\Pi_1$ and $\Pi_2$ be the $L^2_b$-orthogonal projections on the $L^2_b$-nullspaces of $P$ and $P^*$ respectively.  Then the unique operator $G: L^2_b(W;E_2)\to L^{2,q}_b(W;E_1)$ such that
$$
     GP=\Id-\Pi_1 \quad \mbox{and} \quad PG=\Id-\Pi_2
$$
is an element of $\Psi^{-q,\cG}_b(W;E_2,E_1)$ for an index family $\cG$ such that
$$
   \Re(\cG(\lf))>0, \quad \Re(\cG(\rf))>0 \quad \mbox{and} \quad \Re(\cG(\fb)\setminus\{0\})>0.
$$
\end{corollary}
\begin{proof}
By Corollary~\ref{b.14}, the projections $\Pi_1$ and $\Pi_2$ are very residual operators in the sense of \cite{MazzeoEdge}.  The result therefore follow using the  parametrices of Theorem~\ref{b.11} and Corollary~\ref{b.12} as in the proof Corollary~\ref{b.16}.
\end{proof}

\subsection{Linear analysis for orbifold $\Sc^{\beta}$-extremal metrics } \label{la.0}
We suppose that $(M,J,g, \omega)$ is a compact K\"ahler orbifold of complex dimension $m\ge 2$  with isolated singularities, $\alpha= \frac{1}{2\pi}[\omega] \in H^{1,1}(M, \bbR)$ is the corresponding K\"ahler class and $\beta\in H^{1,1}(M, \bbR)$ is another $(1,1)$ deRham class on $(M, J)$, such that
\begin{equation*}
   \beta\cdot \alpha^{m-1}[M]=0.
\label{la.1}\end{equation*}
We further suppose that the K\"ahler metric $(g, J, \omega)$ is $\Sc^{\beta}$-extremal in the sense of \cite[Definition~7.1]{ALL2026}, i.e.  the smooth function 
\begin{equation}
 \Sc^{\beta}(\omega):=  \Delta_{\omega}R_{\omega}+ \frac{R^2_{\omega}}2- \|\Ric_{\omega}\|^2_{\omega}-  \|\theta_{\omega}\|_{\omega}^2
\label{la.3}\end{equation}
is a potential of a  Killing vector field of $(g, \omega)$, i.e. $J\nabla^{\omega}\Sc^{\beta}(\omega)$ preserves $(g, \omega)$. 
In \eqref{la.3}, $R_\omega$ and $\Ric_{\omega}$ denote respectively the scalar curvature and Ricci tensor of $g$, $\Delta_{\omega}$  stands for the (non-negative) Laplace-Beltrami of $g$ acting on functions,  $\theta_{\omega}$ is  the harmonic representative of $2\pi\beta$ with respect to $g$, $\|\cdot\|_{\omega}$ is the pointwise tensorial norm with respect to $g$ and $\nabla^\omega$ is the Levi-Civita connection of $g$. 

By \cite{ALL2026}, we know that  $(g, J, \omega)$ is invariant under a maximal compact (connected) subgroup $K$ of the reduced automorphism group $\Aut_{\rm red}(M,J)$ of $(M,J)$. Let us fix a maximal torus $\bbT$ of $K$.  If $\gt$ denotes the Lie algebra of $\bbT$, then by \cite[Lemma~7.6]{ALL2026} $(g, \omega)$ is extremal if and only if 
\begin{equation*}
           J\nabla^{\omega}\Sc^{\beta}(\omega)\in \gt.
\label{la.2}\end{equation*}
In what follows the complex structure $J$ on $M$   will be fixed and we shall tacitly identify the K\"ahler metric $g$ and its K\"ahler form $\omega$ trough the pointwise relation $g=-\omega J$.

Let $\K^{\bbT}_{\alpha}$ denote the space of $\bbT$-invariant K\"ahler forms in the deRham class $2\pi \alpha \in H^2(M, \bbR)$. For $\omega\in \K^{\bbT}_\alpha$, we let $P_{\omega}$ denote the finite dimensional vector space of all Killing potentials with respect to $\omega$ for vector fields in $\gt$.  By \cite[Lemma~7.4]{ALL2026}, the $L^2$-orthogonal projection of $\Sc^{\beta}(\omega)$ with respect to $\omega$ on $P_{\omega}$ yields a Killing potential for a vector field $\xi^{\ext}_{\alpha, \beta}\in \gt$  only depending on $\alpha, \beta$ and not on the choice of representative $\omega\in \K^{\bbT}_{\alpha}$.  On the other hand, integration by parts shows that the $L^2$-orthogonal projection of $\Sc^{\beta}(\omega)$ onto the constant functions is given by the topological constant
\begin{equation*}
 c_{\alpha,\beta}:= 2m(m-1)\left(  \frac{(c_1(M)^2+\beta^2)\cdot\alpha^{m-2}}{\alpha^m} \right).
\label{la.4}\end{equation*}

Now, for each $\omega\in \K^{\bbT}_{\alpha}$, there is an associated normalized moment map 
$$
\mu_{\omega}: M\to \gt^*,\quad \int_M \mu_{\omega}(\omega)^{[m]}=0.
$$  
The corresponding $\bbT$-momentum polytope $\Pol_{\alpha}:= \mu_{\omega}(M)\subset \gt^*$ is independent of the choice of $\omega\in \K^{\bbT}_{\alpha}$ (see e.g. \cite{lahdili}). With respect to the linear affine function
$$
   \ell^{\ext}_{\alpha,\beta}(x):= \langle \xi_{\alpha,\beta}^{\ext},x\rangle + c_{\alpha,\beta}
$$ 
on $\Pol_{\alpha}$, a K\"ahler metric $\omega\in \K^{\bbT}_{\alpha}$ is extremal if and only if  
\[ \Sc^{\beta}(\omega)= \ell^{\ext}(\mu_{\omega}).\]  Writing $\omega_{\varphi} = \omega_0 + dd^c \varphi, \, \varphi\in C^{\infty}(M)^{\bbT}$ for the metrics in $\K^{\bbT}_{\alpha}$, we thus reduce the search of $\Sc^{\beta}$-extremal metrics in $2\pi \alpha$ to a $6$th order PDE problem  \eqref{PDE} for the  $\bbT$-invariant smooth K\"ahler potential $\varphi$:
\begin{equation*}\label{Z-extremal-PDE} \Sc^{\beta}(\omega_{\varphi})= \ell^{\ext}_{\alpha,\beta}(\mu_{\varphi})=\langle \xi_{\alpha,\beta}^{\ext},x\rangle + c_{\alpha,\beta}, \qquad \mu_{\varphi}= \mu_0+d^c\varphi. \end{equation*}
We introduce the \emph{reduced} $\Sc^{\beta}$-scalar curvature by
$$
  \mathring{\Sc}^{\beta}(\omega_{\varphi}):= \Sc^{\beta}(\omega_{\varphi})-\ell^{\ext}_{\alpha,\beta}(\mu_{\varphi}).
$$

\begin{convention}
In what follows, we shall assume that $M$ admits a $\bbT$-invariant $\Sc^{\beta}$-extremal
K\"ahler metric $\omega_0$. 
To simplify the notation,  we will denote by underscrip $0$ the geometric objects corresponding to $\omega_0$, by underscript $\varphi$ those with respect to $\omega_{\varphi} \in\K^{\bbT}_{\alpha}$ and by underscript $\omega$ these quantities for a generic element $\omega \in \K^{\bbT}_{\alpha}$. For instance, we shall write  $R_{\varphi}, {\rm Rc}_{\varphi}, \Delta^H_{\varphi}, \theta_{\varphi}$ and $R_0, {\rm Rc}_0, \Delta^H_0, \theta_0$ for the scalar curvature, Ricci tensor, Hodge Laplace operator, and harmonic representative of $\beta$ with respect to $\omega_{\varphi}$ and $\omega_0$. We shall also refer to the Hodge Laplacian acting of functions as the \emph{Laplace--Beltrami operator} and denote by $\Delta_{\varphi}$ the one associated to $\omega_{\varphi}$,  in accordance with tradition in the analysis literature.
\end{convention}
According to \cite[\S~7]{ALL2026}, the linearization of $ \mathring{\Sc}^{\beta}$ at $\varphi=0$ is given by the $6$th order elliptic pseudodifferential operator
\begin{multline}
\bbS(\dot{\varphi}):= (D_{0} \mathring{\Sc}^{\beta})(\dot{\varphi})=
-4\delta_{0}\delta_{0}(\nabla^{0,-}(\delta_{0}\nabla^{0,-}d\dot{\varphi}))-2\delta_{0}\delta_{0}(R_{0}\nabla^{0,-}d\dot{\varphi}) \\ 
+ 32\delta_{0}\delta_{0}(\theta_{0}\circ \bbG^{H}_{0}(\theta_{0}\circ \nabla^{0,-}d\dot{\varphi})^{\Skew}),
\label{la.6}\end{multline}
where $\bbG^{H}_{0}$ is the Green operator of the Friedrichs extension the Hodge Laplacian of $g_0$.   

The next lemma shows that the Friedrichs extension is the natural self-adjoint extension to consider for the definition of $\bbG^{H}_{0}$.  
\begin{lemma}
The Friedrichs extension of the Hodge Laplacian $\Delta^H_{0}$ is the orbifold $L^2$-Sobolev space of order $2$ $L^{2,2}_{\omega_0}(\Lambda^*(T^*M))$ associated to the metric $g_0$. 
\label{Fr.1}\end{lemma}
\begin{proof}
The Hodge-deRham operator $\eth_{0}=d+\delta_{0}$ of $g_0$ is essentially self-adjoint with unique self-adjoint extension corresponding to the orbifold $L^2$-Sobolev space $L^{2,1}_{\omega_0}(M;\Lambda^*(T^*M))$, see for instance \cite[Remark~58]{DR}.  In particular, the Friedrichs extension of $\Delta^H_{0}= \eth_{0}^2$ is then given by 
$$
   \operatorname{Fr}(\Delta^H_{0})= \{ \eta\in L^{2,1}_{\omega_0}(M;\Lambda^*(T^*M)) \; | \; \eth_{0}\eta \in L^{2,1}_{\omega_0}(M;\Lambda^*(T^*M))\},
$$
which by elliptic regularity coincides with $L^{2,2}_{\omega}(\Lambda^*(T^*M))$.
\end{proof}
\begin{remark}
When $m\ge 3$, the Hodge Laplacian is in fact essentially self-adjoint.  Indeed, the positive part of the spectrum of the Hodge Laplacian on the sphere $S^{2m-1}$ with its canonical metric does not contain $]0,4[$ when $m\ge 3$ by \cite{GM1975}, so the Hodge Laplacian $\Delta_{\omega}^H$ is essentially self-adjoint by \cite[Corollary~2.7]{ARS3}. 
\label{Fr.2}\end{remark}

The operator in \eqref{la.6} is also formally self-adjoint with respect to the $L^2$-inner product induced by $\omega_0$.  Let 
\begin{equation}
    \bM:= [M;p_1,\ldots,p_{\ell}]
\label{bm.1}\end{equation} 
be the manifold with boundary obtained from $M$ by blowing up the singular points in the sense of Melrose \cite{MelroseAPS}.  Thus, if $p\in M$ is a singular point with singularity modelled on $\bbC^{m}/\Gamma$ for some finite subgroup $\Gamma\subset \operatorname{U}(m)$, then the boundary component $\pa_p \bM$ of $\bM$ corresponding to $p$ is identified with $S^{2m-1}/\Gamma$ with collar neighbourhood  $S^{2m-1}/\Gamma\times [0,\delta)$ for some $\delta>0$ corresponding to spherical coordinates on $\bbC^m/\Gamma$.  If $M$ has $\ell$ singular points $p_1,\ldots,p_{\ell}$ locally modelled respectively on $\bbC^m/\Gamma_1,\ldots,\bbC^m/\Gamma_{\ell}$ for $\Gamma_1,\ldots, \Gamma_{\ell}$ finite subgroups of $\operatorname{U}(m)$, then $\pa \bM$ has $\ell$ connected components $\pa_{p_1}\bM\cong S^{2m-1}/\Gamma_1, \ldots, \pa_{p_\ell}\bM\cong S^{2m-1}/\Gamma_\ell$.  Let $r\in\CI(\bM)$ be a boundary defining function for $\bM$.  Since the action of $\bbT$ on $M$ fixes the singular points, it lifts to a smooth action on $\bM$.  Averaging with respect to this action if necessary, we can assume without loss of generality that $r$ is a $\bbT$-invariant  function and that in the local neighborhood identified with $\bbC^{m}/\Gamma_i$ near $p_i$, $r$ corresponds to the Euclidean distance to the origin.  Consider then the weighted operator
\begin{equation*}
 \bbS_b:= r^6\bbS.
\label{la.7}\end{equation*}

To describe this operator as a pseudodifferential $b$-operator, let us first rewrite the las term on the right hand side of \eqref{la.6} in terms of the Green operator $\bbG_{0}$ of the Laplace-Beltrami operator (with nonnegative spectrum) $\Delta_{0}$ of the  K\"ahler metric $\omega_0$, that is, the operator such that 
\begin{equation}
     \Delta_{0}\bbG_{0}= \bbG_{0}\Delta_{0}= \Id - \Pi_{0},
\label{la.7b}\end{equation}
where 
$$
     \Pi_{0}u= \frac{\int_M u \omega_0^{[m]}}{\int_M \omega_0^{[m]}}
$$
is the $L^2$-orthogonal projection onto constant functions with respect to the metric $\omega_0$.

\begin{lemma}
The last term on the right hand side of \eqref{la.6} is also taking the form
$$
     -2\langle \theta_{0}, dd^c \bbG_{0}(E_{\theta_{0}}(dd^c \dot{\varphi})\rangle - F_{\theta_{0}}(dd^c\dot{\varphi}),
$$
where $E_{\theta_{0}}$ and $F_{\theta_{0}}$ are $(2,0)$ tensors depending on $\theta_{0}$ that are contracted with $dd^c\dot{\varphi}$.
\label{green.1}\end{lemma}
\begin{proof}
If $\theta_{\varphi}$ is the harmonic representative of $\beta$ with respect to $\omega_{\varphi}$, then by the $dd^c$-lemma on compact  K\"ahler orbifolds (see e.g. \cite[Lemma~5.4]{BBFMT}  or \cite[Sec.4]{Boyer-Galicki-book}),  there exists a unique function $v$ such that
$$
 \theta_{{\varphi}}= \theta_{0}+dd^c v \quad \mbox{and}   \quad \int_M v\omega^{[m]}_{\varphi}=0.
$$
Since $\beta$ is primitive with respect to the K\"ahler class of $\omega_{\varphi}$, notice that 
$$
0= \Lambda_{\varphi}\theta_{{\varphi}}= \Lambda_{\varphi}\theta_{0}- \Delta_{{\varphi}}v,
$$
where $\Lambda_{\varphi}$ is the adjoint of 
$$
   \begin{array}{lccc}
    L_{\varphi}: & \Omega^p(M) &\to & \Omega^{p+2}(M) \\
                & \eta & \mapsto & \omega_{\varphi}\wedge \eta
   \end{array}
$$
with respect to the K\"ahler metric $\omega_{\varphi}$.  Hence, 
$$
     v= \bbG_{{\varphi}} \Lambda_{\varphi}\theta_{0}
$$
and
$$
   \theta_{{\varphi}}= \theta_{0} + dd^c\bbG_{{\varphi}}\Lambda_{\varphi}\theta_{0},
$$
where $\bbG_{{\varphi}}$ is the Green operator of (the Friedrichs extension of) the Laplace-Beltrami operator of $\omega_{\varphi}$.  Since $\beta$ is primitive with respect to the K\"ahler class of $\omega_0$, notice that 
$$
    \Lambda_{\epio\varphi}\theta_{0}= \cO(\epio) \quad \mbox{as} \; \epio\searrow 0.
$$
This means that 
$$
     \theta_{{\epio\varphi}}= \theta_{0}+ dd^c\bbG_{0}\Lambda_{\epio\varphi}\theta_{{0}}+ \cO(\epio^2)
     \quad \mbox{as} \; \epio\searrow 0.
$$
Using this, we get that 
$$
\|\theta_{{\epio\varphi}}\|_{{\epio\varphi}}^2= \|\theta_{0}\|^2_{0}+ \epio\lrp{2\lrp{\theta_{0}, dd^c\bbG_{0}E_{ \theta_0}(dd^c\varphi)}_{0}+ F_{\theta_0}(dd^c\varphi)} + \cO(\epio^2) \quad \mbox{as} \; \epio\searrow 0
$$
for $(2,0)$-tensors $E_{\theta_{0}}$ and $F_{\theta_{0}}$ depending on $\theta_{0}$ that are contracted with $dd^c\varphi$, from which the result follows.
\end{proof}

To describe $\bbG_{0}$ as a pseudodifferential $b$-operator, notice that near each connected component $\pa_{p_i}\bM$ of $\pa \bM$, we have that 
$$
        r^2\Delta_{0}- r^{2}\Delta_{\bbC^m/\Gamma_i}\in r\Diff^2_b(\bM),
$$
where $\Delta_{\bbC^m/\Gamma_i}$ is the Laplace-Beltrami operator of $\bbC^m/\Gamma_i$ with the quotient Euclidean metric.   In spherical coordinates, recall that
$$
   r^2\Delta_{\bbC^m/\Gamma_i}= \Delta_{S^{2m-1}/\Gamma_i}-\left(r \frac{\pa}{\pa r} \right)^2 -(2m-2)r\frac{\pa}{\pa r},
$$
where $\Delta_{S^{2m-1}/\Gamma_i}$ is the Laplace-Beltrami operator of $S^{2m-2}/\Gamma_i$ with its induced metric, so that
$$
  I(r^2\Delta_{\bbC^m/\Gamma_i},\lambda)= \Delta_{S^{2m-1}/\Gamma_i}-\lambda^2-(2m-2)\lambda.
$$
Hence, the set of indicial roots $\Spec_b(r^2\Delta_{\bbC^m/\Gamma_i})$ of $I(r^2\Delta_{\bbC^m/\Gamma_i},\lambda)$ is given by
\begin{equation}
  \Spec_b(r^2\Delta_{\bbC^m/\Gamma_i})= \left\{  -(m-1)\pm \sqrt{(m-1)^2+\nu} \; | \; \nu\in \Spec(\Delta_{S^{2m-1}/\Gamma_i})  \right\}.
\label{la.13}\end{equation}
Since elements of $\Spec(\Delta_{S^{2m-1}/\Gamma_i})$ are of the form $k(k+2(m-1))$ for $k\in\bbN$ (see e.g.  \cite{BGM}),  we see in particular that
\begin{equation}
  \Spec_b(r^2\Delta_{\bbC^m/\Gamma_i})\subset \bbZ\setminus \left( \bbZ\cap  ]-2(m-1),0[ \right).
\label{la.13a}\end{equation}
\begin{lemma}\label{green.2}
For $m\ge 2$, the Green operator $\bbG_{0}$ is an element of $\Psi^{-2,\cG}_{b}(\bM)$ for some real index family $\cG$ such that
$$
  \min \cG|_{\lf}\ge 0, \quad \min \cG|_{\rf}\ge 2m \quad \mbox{and} \quad \min\cG|_{\fb}\ge 2.
$$
\end{lemma}
\begin{proof}
Thanks to \eqref{la.13a}, we know by \cite[Corollary3.17]{Gil-Mendoza}, see also \cite[Proposition~1.5]{ARS3}, that the Laplace-Beltrami operator $\Delta_{0}$ is essentially self-adjoint with respect to the orbifold metric $g_0$.  In terms of the conformally related $b$-metric, this means that 
\begin{equation}
    P_b:= r^m(r\Delta_{0}r)r^{-m}= r^{m-1}(r^2\Delta_{0})r^{-(m-1)}
\label{green.2b}\end{equation}
is a formally self-adjoint $b$-operator.  This suggests to rewrite \eqref{la.7b} as
\begin{equation}
  P_b G_m= \Id -r^{m+1}\Pi_{0}r^{-m-1} \quad \mbox{and} \quad G_mP_b=\Id- r^{m-1}\Pi_{0}r^{-(m-1)},
\label{green.3}\end{equation}
where $G_m=r^{m-1}\bbG_{0}r^{-m-1}$.  On the other hand, by Theorem~\ref{b.11} and Corollary~\ref{b.12}, there exists parametrices $Q_i\in \Psi^{-2,\cQ_i}_b(\bM)$ and error terms $R_i\in \Psi^{-\infty,\cR_i}(\bM)$ such that
\begin{equation}
 P_bQ_1= \Id- R_1 \quad \mbox{and} \quad Q_2P_b=\Id-R_2
\label{green.4}\end{equation}
with $\cQ_i$ and $\cR_i$ real index families such that 
$$
     \cQ_i(\fb)\ge 0, \quad \cQ_i(\lf)\ge m-1, \quad \cQ_i(\rf)\ge m-1,
$$
and
$$
 \cR_1(\fb)=\cR_1(\lf)= \cR_2(\fb)=\cR_2(\rf)=\emptyset,\quad \cR_1(\rf)\ge m-1 \quad \mbox{and} \quad \cR_2(\lf)\ge m-1.
$$
From \eqref{green.3} and \eqref{green.4}, we see that
\begin{equation}
 G_m= G_m\Id= G_m(P_bQ_1+R_1)= (\Id- r^{m-1}\Pi_{0}r^{-(m-1)})Q_1 +G_mR_1
\label{green.5}\end{equation}
and
\begin{equation}
 G_m= \Id G_m= (Q_2P_b+R_2)G_m= Q_2(\Id-r^{m+1}\Pi_{0}r^{-m-1})+ R_2G_m.
\label{green.6}\end{equation}
Substituting \eqref{green.6} in \eqref{green.5}, we find that
$$
\begin{aligned}
G_m&= (\Id- r^{m-1}\Pi_{0}r^{-(m-1)})Q_1 + \left[ Q_2(\Id-r^{m+1}\Pi_{0}r^{-m-1})+ R_2G_m\right]R_1 \\
  &= (\Id- r^{m-1}\Pi_{0}r^{-(m-1)})Q_1+ Q_2(\Id-r^{m+1}\Pi_{0}r^{-m-1})R_1 + R_2G_mR_1.
\end{aligned}
$$
Since $G_m$ is bounded on $L^2_b(M)$ and $R_1$ and $R_2$ are very residual in the sense of \cite{MazzeoEdge}, the term $R_2G_mR_1$ is very residual and is an element of $\Psi^{-\infty,\cR}_b(\bM)$ for some real index family $\cR$ such that
$$
   \cR(\lf)\ge m-1, \quad \cR(\rf)\ge m-1\quad \mbox{and} \quad \cR(\fb)\ge 2(m-1).
$$
Using Proposition~\ref{b.6} for the other terms, we deduce that $G_m\in \Psi^{-2,\cH}_b(\bM)$ for an index family $\cH$ with real index sets such that
$$
    \cH(\fb)\ge 0, \quad \cH(\lf)\ge m-1 \quad \mbox{and} \quad \cH(\rf)\ge m-1.
$$
Since $\bbG_{0}= r^{-(m-1)}G_m r^{m+1}$, the result follows.
\end{proof}

\begin{lemma}\label{la.8}
The operator $\bbS_b$ is an elliptic pseudodifferential $b$-operator of order $6$, more precisely
$$
    \bbS_b\in \Psi^{6,\cE}_b(\bM)
$$
with real index family $\cE$ such that $\cE(\lf)\ge 4$, $\cE(\fb)\ge 0$ and $\cE(\rf)\ge 2m-2$.
\end{lemma}
\begin{proof}
Using Lemmas~\ref{green.1} and \ref{green.2} as well as the $r^6$ factor in \eqref{la.7}, the  
 composition formula for $b$-operators and the fact that $\theta_{0}$ is smooth in the sense of orbifolds, we see that the last term in the definition of $\bbS_b$ is a $b$-operator of order 2 in $\Psi^{2,\cE'}_b(\bM)$ for an index family $\cE'$ such that $ \cE'(\lf)\ge 4$, $ \cE'(\fb)\ge 4$ and $ \cE'(\rf)\ge 2m-2$.  Since the first term in the definition of $\bbS_b$ gives an elliptic differential $b$-operator of order 6, while the second terms gives a differential $b$-operator of order $4$, the result follows. \end{proof}

Since the term involving the Green function $\bbG_{0}$ in $\bbS_b$ restricts to zero on the boundary hypersurface $\fb$, notice that it does not contribute to the indicial family of $\bbS_b$.  For a similar reason, the second term in the right hand side of \eqref{la.6} does not contribute to the indicial family.  
Since the linearization of the scalar curvature at $\varphi=0$ is
$$
 (D_{0}R)(\dot{\varphi})= -\Delta_{0}^2\dot{\varphi} - 2(dd^c\dot{\varphi},\rho_{0}),
$$
where $\rho_{0}$ is the Ricci form of $\omega_0$, we see that the indicial family of $\bbS_b$ at $\pa_{p_i}\bM$ is given by
\begin{equation*}
    I(\bbS_b,\lambda)= \frac12 I(-r^6\Delta^3_{0},\lambda).
\label{la.11}\end{equation*}

Near each connected component of $\pa_{p_i}\bM$ of $\pa \bM$, we have that 
$$
        r^6\Delta_{0}^3- r^{6}\Delta^3_{\bbC^m/\Gamma_i}\in r\Diff^6_b(\bM),
$$
which means that 
$$
     I(r^6\Delta_{0})= \bigoplus_{i=1}^\ell I(r^6\Delta^3_{\bbC^m/\Gamma_i}).
$$

Since
$$
   r^6\Delta_{\bbC^m/\Gamma_i}^3= r^4(r^2\Delta_{\bbC^m/\Gamma_i})r^{-4}r^2(r^2\Delta_{\bbC^m/\Gamma_i})r^{-2}(r^2\Delta_{\bbC^m/\Gamma_i}),
$$ 
we see on the other hand that
\begin{equation}
   I(r^6\Delta_{\bbC^m/\Gamma_i}^3,\lambda)= I((r^2\Delta_{\bbC^m/\Gamma_i}),\lambda-4)I((r^2\Delta_{\bbC^m/\Gamma_i}),\lambda-2)I((r^2\Delta_{\bbC^m/\Gamma_i}),\lambda),
\label{ind.1}\end{equation}
which implies that the set of indicial roots of $\bbS_b$ is given by
\begin{equation}
  \Spec_b(\bbS_b)=\bigcup_{i=1}^{\ell}\left( \bigcup_{k=0}^{2} \left( \Spec_b((r^2\Delta_{\bbC^m/\Gamma_i}))+2k\right) \right).
\label{ind.2}\end{equation}
This yields in particular the following.
\begin{lemma}
The indicial set $\Spec_b(\bbS_b)$ consists only of integers and is symmetric with respect to $\lambda=-(m-3)$.  For $m>3$, $0$ and $6-2m$ are indicial roots of order $1$ and rank $\ell$ and 
$$
    \Spec_b(\bbS_b)\cap ]6-2m,0[=\emptyset,
$$
while for $m=3$, $0$ is an indicial root of order $2$ and rank $2\ell$.  For $m=2$, $0$ and $2$ are indicial roots of order $2$ with the part of order $2$ coming from locally constant functions on $\pa \bM$, namely the direct sum of the nullspaces of $\Delta_{S^{2m-1}/\Gamma_i}$ for $i\in\{1,\ldots,\ell\}$.
\label{la.14}\end{lemma}

On the other hand, the operator $\bbS$ is formally self-adjoint with respect to the $L^2$-inner product induced by the volume form of $\omega_0$.  If $L^2_w(\bM)$ is the corresponding $L^2$-space, then 
$$
        L^2_w(\bM)= r^{-m}L^2_b(\bM),
$$
which means that $r^m(r^3\bbS r^3)r^{-m}$ is formally self-adjoint as a $b$-operator acting formally on $L^2_b(\bM)$.  In particular, this means that 
$$
     r^{-3+m}\bbS_br^{3-m}= \left(  r^{-3+m} \bbS_b r^{3-m} \right)^*= r^{3-m}\bbS_b^*r^{-(3-m)},
$$
so that as a $b$-operator, the formal adjoint of $\bbS_b$ is given by
\begin{equation}
\bbS_b^*= r^{-(6-2m)}\bbS_b r^{6-2m}.
\label{la.18}\end{equation}
In particular, we deduce that 
\begin{equation}
    \Spec_b(\bbS^*_b)= \Spec(\bbS_b)+2m-6.
\label{dual.1}\end{equation}
We can use these results to characterize the nullspace of $\bbS_b$.

\begin{lemma}\label{la.15}
Suppose that $R_{0}\ge 0$ and $m>3$.  Then for $\nu\in ]6-2m,0[$, 
$$
     (\ker\bbS_b)\cap r^{\nu}L^2_b(M)^{\bbT}= P_{0},
$$  
where $L^2_b(M)^{\bbT}$ is the subspace of $\bbT$-invariant functions in $L^2_b(\bM)$.  
\end{lemma}
\begin{proof}
Let $u\in r^\nu L^2_b(\bM)^{\bbT}$ be such that $\bbS_bu=0$.  By Corollary~\ref{b.14}, Lemma~\ref{la.8} and \eqref{dual.1}, we deduce from Lemma~\ref{la.14} that
$$
    u\in \cA^{\cE}(\bM)
$$ 
for some nonnegative index set $\cE$.  Moreover, the term of order $0$ is a locally constant function on $\pa \bM$.  In particular,
$$
      du\in r^{\delta-1}L^{2,\infty}_b(\bM; {}^{w}T^*\bM)
$$
for some $\delta>0$.  More generally this implies that 
$$
  \nabla^{0,-}du\in r^{\delta-2}L^{2,\infty}_b(\bM;{}^{w}T^*\bM\otimes {}^{w}T^*\bM) \quad \mbox{and} \quad 
     \delta_{0}\nabla^{0,-}du\in r^{\delta-3}L^{2,\infty}_b(\bM;{}^{w}T^*\bM).
$$
This ensures that integration by parts in $\langle u, \bbS_b u\rangle_{L^2_w}$ can be used to obtain that
$$
0=\langle u,\bbS u\rangle_{L^2_w}= -4 \|\delta_{0}\nabla^{0,-}du\|^2_{L^2_w}- 2\langle \nabla^{0,-}du, R_{0} \nabla^{0,-}du\rangle_{L^2_w}
-32 \langle (\theta_{0}\circ\nabla^{0,-}du)^{\Skew}, \bbG_{0} (\theta_{0}\circ\nabla^{0,-}du)^{\Skew}\rangle_{L^2_w}.
$$
Since each term on the right is nonpositive, we get that
$$
  \|\delta_{0}\nabla^{0,-}du\|^2_{L^2_w}=0 \quad \mbox{and} \quad \langle \nabla^{0,-}du, R_{0} \nabla^{0,-}du\rangle_{L^2_w}=0.
$$
From the first equation, we deduce that $\delta_{0}\nabla^{0,-}du=0$, hence integrating by parts yields
$$
   0=\langle du, \delta_{0}\nabla^{0,-}du\rangle= \| \nabla^{0,-}du \|^2_{L^2_w},
$$
from which we conclude that $\nabla^{0,-}du=0$.  Since $u$ is $\bbT$-invariant, this means that $u\in P_{0}$ by the maximality of $\bbT$.  On the other hand, 
$$
     P_{0} \subset \ker \bbS_b\cap r^{\nu}L^2_b(\bM)^{\bbT},
$$
so the result follows.
\end{proof}

Setting 
$$
  (r^{\nu}L^{2,k}_b(\bM)^{\bbT})^{\perp}= \left\{ u \in r^{\nu}L^{2,k}_b(\bM)^{\bbT} \quad \Big| \quad \int_M uh \omega_0^{[m]} \; \, \, \forall h\in P_{0}^{\bbT} \right\}
$$
for $\nu>-2m$ has the following consequence for the mapping properties of the operator $\bbS$.
\begin{proposition}\label{la.16}
Suppose $m>3$ and $R_0\ge 0$.  Then for  $\nu \in]6-2m,0[$ and $k\in \bbN$, the operator $\bbS_b$ induces a continuous linear isomorphism
\begin{equation*}
\bbS: (r^{\nu}L^{2,k+6}_b(\bM)^{\bbT})^{\perp}\to (r^{\nu-6}L^{2,k}_b(\bM)^{\bbT})^{\perp}.
\label{la.16a}\end{equation*}
\end{proposition}
\begin{proof}
 Since $\nu\in]2m-6,0[$, we know by Lemma~\ref{la.14} that it is not an indicial root of $\bbS_b$.
 By Lemma~\ref{la.8} and Corollary~\ref{b.15}, we know that the map
 \begin{equation}
\bbS_b: r^{\nu}L^{2,k+6}_b(M)^{\bbT}\to r^{\nu}L^{2,k}_b(\bM)^{\bbT}
\label{la.17}\end{equation}
is Fredholm and that its nullspace is $P_{0}$ by Lemma~\ref{la.15}.

By \eqref{dual.1}, the formal adjoint of $r^{-\nu}\bbS_b r^{\nu}$ with respect to the $b$-metric $g_b= r^{-2}g_0$ is
\begin{equation}
     (r^{-\nu}\bbS_br^{\nu})^*= r^{\nu}\bbS_b^*r^{-\nu}= r^{\nu-(6-2m)}\bbS_b r^{-\nu+ 6-2m}.
\label{la.19}\end{equation}
Now, the cokernel of the Fredholm operator
\begin{equation}
  r^{-\nu}\bbS_b r^{\nu}: L^{2,k+6}_b(\bM)^{\bbT}\to L^{2,k}_b(\bM)^{\bbT}
\label{la.20}\end{equation}
is identified with the $L^2_b$-nullspace of its formal adjoint.  Since 
$$
  \nu \in ]6-2m,0[ \quad \Longrightarrow \quad 6-2m-\nu\in ]6-2m,0[,
$$
it follows from Lemma~\ref{la.15} that this nullspace is $r^{\nu-(6-2m)}P_{0}$.  In particular, the Fredholm indices of \eqref{la.17} and \eqref{la.20} are zero.  Moreover, as a subspace of $L^2_b(\bM)$, $r^{\nu-(6-2m)}P_{0}$ is in fact $L^2_b(\bM)$-orthogonal to the range of $r^{-\nu}\bbS_br^{\nu}$, which means that $P_{0}$ is $L^2_w(\bM)$-orthogonal to the range of 
\begin{equation}
 \bbS: r^{\nu}L^{2,k+6}_b(\bM)^{\bbT}\to r^{\nu-6}L^{2,k}_b(\bM)^{\bbT}.
\label{la.21}\end{equation}
On the other hand, the discussion above shows that the operator \eqref{la.21} is Fredholm of index zero, so its cokernel is of dimension $\dim P_{0}$.  Since its range is $L^2_w(\bM)$-orthogonal to $P_{0}$, the result follows.
\end{proof}

We can have a weaker result when $\nu>0$.  To formulate it, for $i\in\{1,\ldots,\ell\}$, let $\chi_i\in \CI(M)^{\bbT}$ be a cut-off function equal to $1$ near $p_i\in M$ and to $0$ near $p_j\in M$ for $j\ne i$.  Consider the spaces of functions
$$
    \cC_i= \{ c_i\chi_i\; | \; c_i\in \bbC\} \subset \CI(M)^{\bbT}
$$ 
and 
$$
   \cP_i= \{ \chi_i P_i\; | \; P_i \in \bbP_{1,\Gamma_i}^{\bbT}\}\subset \CI(M)^{\bbT},
$$
where $\bbP_{1,\Gamma_i}^{\bbT}$ is the space of $\Gamma_i$-invariants polynomials of degree $1$ on $\bbC^m$ which are also $\bbT$-invariant with respect to the linear action of $\bbT$ on $\bbC^m/\Gamma_i$  induced by the action of $\bbT$ on $M$ near $p_i$.
\begin{corollary}\label{la.16b}
Suppose $m>3$.  Then for $\nu\in]0,1[$, the operator $\bbS_b$ induces a surjective continuous linear map
\begin{equation}
  \bbS: (r^{\nu}L_b^{2,k}(\bM)^{\bbT})^{\perp}\oplus \lrp{\bigoplus_{i=1}^{\ell}\cC_i}\to  (r^{\nu-6}L_b^{2,k}(\bM)^{\bbT})^{\perp},
\label{la.16c}\end{equation}
while for $\nu\in]1,2[$ it induces a surjective continuous linear map
\begin{equation}
  \bbS: (r^{\nu}L_b^{2,k}(\bM)^{\bbT})^{\perp}\oplus \lrp{\bigoplus_{i=1}^{\ell}\cP_i}\to  (r^{\nu-6}L_b^{2,k}(\bM)^{\bbT})^{\perp}.
\label{la.16d}\end{equation}
\end{corollary}
\begin{proof}
 Since $1$ is the only indicial root of $\bbS_b$ in $]0,2[$, the maps \eqref{la.16c} and \eqref{la.16d} are therefore Fredholm by Lemma~\ref{la.8}, \eqref{ind.2} and Corollary~\ref{b.15}.  To show that they are surjective, we can use  
 Proposition~\ref{la.16} and the regularity result of Corollary~\ref{b.13}.
\end{proof}

When $m=3$, the result we obtain is of different nature.  To describe it, set 
\begin{equation}
   P_{0,+}:= P_{0}\cap r^{\nu}L^2_b(\bM)^{\bbT} \quad \mbox{for} \quad \nu\in ]0,1[.
\label{la.16e}\end{equation}
Using the same integrations by parts that are used in the proof of Lemma~\ref{la.15}, we can also deduce that 
$$P_{\omega,+}= \ker \bbS\cap r^{\nu}L^2_b(\bM)^{\bbT} \quad \mbox{for} \quad  \nu\in ]0,1[.
$$  

By Lemmas~\ref{la.8}, \ref{la.14} and Corollary~\ref{b.14}, this shows in particular that the definition of $P_{0,+}$ is independent of the choice of $\nu$.  
Since $P_{0,+}$ does not contain the constant functions, notice also that it is strictly included in $P_0$. 
 On the other hand, by \eqref{la.18} (with $m=3$), notice that $\bbS_b$ is formally self-adjoint as a $b$-operator.  By Lemmas~\ref{la.8}, \ref{la.14} and Corollary~\ref{b.15}, for $\nu\in ]0,1[$ fixed, the operators $\bbS_{b,\pm\nu}:=r^{\mp\nu}\bbS_b r^{\pm\nu} $ induce  Fredholm operators
$$
    \bbS_{b,\pm\nu}: L^{2,k+6}_b(\bM)^{\bbT}\to L^{2,k}_b(\bM)^{\bbT}.
$$
Since $\bbS_b$  is formally self-adjoint as a $b$-operator, we know that there are canonical identifications
$$
    \coker_{L^2_b}(\bbS_{b,\pm\nu})= \ker_{L^2_b}\bbS_{b,\mp\nu},
$$
which implies that
\begin{equation}
 \ind (\bbS_{b,-\nu})= -\ind(\bbS_{b,\nu}).  
 \label{la.22}\end{equation}
 Set 
 $$
     \widetilde{\bbS}u:= -4\delta_{0}\delta_{0}(\nabla^{0,-}(\delta_{0}\nabla^{0,-}du))-2\delta_{0}\delta_{0}(R_{0}\nabla^{0,-}du),
 $$
 so that $\widetilde{\bbS}=\bbS$ when $\beta=0$. Set also $\widetilde{\bbS}_b= r^6\widetilde{\bbS}$ and $\widetilde{\bbS}_{b,\pm\nu}:=r^{\mp\nu}\widetilde{\bbS}_b r^{\pm\nu} $.  By Lemma~\ref{la.8}, for $\nu\in ]0,1[$, the operator $\widetilde{\bbS}_{b,\pm\nu}$ can be seen as a compact perturbation of $\bbS_{b,\pm \nu}$.  
Using Lemma~\ref{la.14} and the fact that $\widetilde{\bbS}_b$ has the same indicial family as $\bbS_b$, we see from the relative index theorem of Melrose \cite[Theorem~6.5]{MelroseAPS} that
 $$
     \ind(\bbS_{b,-\nu})= \ind(\widetilde{\bbS}_{b,-\nu})= \ind(\widetilde{\bbS}_{b,\nu})+ 2\ell= \ind(\bbS_{b,\nu})+ 2\ell.
 $$
Plugging this inside \eqref{la.22}, we thus find that 
\begin{equation*}
  \ind(\bbS_{b,\nu})=-\ell.
\label{la.22a}\end{equation*}
Since $\ker_{L^2_b}\bbS_{b,\nu}= r^{-\nu}P_{0,+}$, this means that 
\begin{equation}
\dim \ker_{L^2_b}\bbS_{b,-\nu}= \dim \coker_{L^2_b} \bbS_{b,\nu}= \dim \ker_{L^2_b}\bbS_{b,\nu}- \ind(\bbS_{b,\nu})= \dim P_{0,+}+ \ell.
\label{la.23}\end{equation}
Since $r^{\nu}P_{0}\subset \ker_{L^2_b}\bbS_{b,-\nu}$, we consider  the $L^2$-orthogonal subspace $K_0\subset r^{-\nu}\ker_{L^2_b}\bbS_{b,-\nu}$  to $P_{0}$ with respect to the $L^2$-inner product induced by the metric $g_0$, which yields the decomposition 
\begin{equation}
   \ker_{L^2_b} \bbS_{b,-\nu}= r^{\nu} (P_{0}\oplus K_{0}).
\label{la.24}\end{equation}
By \eqref{la.23} and \eqref{la.24}, the dimension of $K_{0}$ is given by 
$$
    \dim K_{0}=\ell + \dim P_{0,+} -\dim P_{0}.
$$
\begin{lemma}
For each non-trivial element $u\in K_0$, there is a non-vanishing term of order $\log r$ in its polyhomogeneous expansion at $\pa\bM$, namely there exists a non-vanishing locally constant function $a_{0,1}: \pa\bM\to \bbR$ such that $u-a_{0,1}\log r$ is a bounded function near $\pa \bM$.  
\label{la.25}\end{lemma}
\begin{proof}
Let $\nu\in ]0,1[$ be given.  By Lemma~\ref{la.14} and the regularity result of Lemma~\ref{b.14}, we know that an element $u\in \ker_{L^2_b}\bbS_{b,-\nu}$ admits at the boundary component $\pa_{p_i}\bM$ of $\pa \bM$ a polyhomogeneous expansion of the form
$$
  r^{\nu}\left(a_{p_i,0,1}\log r + a_{p_i,0}+ \cO(r\log r)  \right)
$$
for some constants $a_{p_i,0,1}$ and $a_{p_i,0}$.  If $a_{p_i,0,1}=0$ for each boundary component of $\pa \bM$, then 
$$
     d(r^{-\nu}u)\in r^{\delta-1}L^2_b(\bM;{}^{w}T^*\bM) \quad \mbox{and} \quad \bbS (r^{-\nu}u)\in r^{-6+\delta}L^2_b(\bM)=r^{\delta-3}L^2_w(\bM)
$$
for all $\delta\in]0,1[$.  This shows that  integrations by parts as in the proof of Lemma~\ref{la.15} are valid and show that $r^{-\nu}u\in P_{0}$, namely that $u\in r^{\nu}P_{0}$.  Since $K_0$ is complementary to $P_{0}$ in $r^{-\nu}\ker_{L^2_b}\bbS_{b,-\nu}$, the result follows.    
\end{proof}

This can be used to deduce the following mapping properties of the operator $\bbS$ when $m=3$.  For $\nu\in ]0,1[$ and $k\in\bbN$, let us consider the spaces
$$
  (r^{\nu}L^{2,k}_b(\bM)^{\bbT})^{\perp_+}:= \left\{ u\in r^{\nu}L^{2,k}_b(\bM)^{\bbT}\quad \Big| \quad \int_{M} u h\omega^{[m]}=0 \quad \forall h\in P_{0,+} \right\}
$$
and
$$
   (\CI(M)^{\bbT})^{\perp}:= \left\{ u\in \CI(M)^{\bbT}\quad \Big| \quad \int_{M} u h\omega^{[m]}=0 \quad \forall h\in P_{0} \right\}.
$$

\begin{proposition}\label{la.26}
If $m=3$ and $R_0\ge 0$, there exist  $\bbT$-invariant functions $f_1,\ldots, f_{\dim K_0}\in (\CI(M)^{\bbT})^{\perp}$,  smooth {on $M$} in the sense of orbifolds,  such that for $\nu\in ]0,1[$, the operator $\bbS$ induces a continuous linear isomorphism
\begin{equation}
  \bbS: (r^{\nu}L^{2,k+6}_b(\bM)^{\bbT})^{\perp_+}\oplus \cD\to (r^{\nu-6}L^{2,k}_{b}(\bM)^{\bbT})^{\perp},
\label{la.26b}\end{equation}
where $\cD\subset (\CI(M)^{\bbT})^{\perp}$ is the finite dimensional subspace generated by $f_1,\ldots, f_{\dim K_{0}}$.  For $\nu\in]1,2[$, the operator $\bbS$ induces instead a surjective continuous linear map
\begin{equation}
  \bbS: (r^{\nu}L^{2,k+6}_b(\bM)^{\bbT})^{\perp_+}\oplus \lrp{\bigoplus_{i=1}^{\ell} \cP_i}\to (r^{\nu-6}L^{2,k}_{b}(\bM)^{\bbT})^{\perp}.
\label{la.26c}\end{equation}
\end{proposition}
\begin{proof}
By the discussion above, we know that for $\nu\in ]0,1[$, the operator
\begin{equation*}
  \bbS_b: r^{\nu}L^{2,k+6}_b(\bM)^{\bbT}\to r^{\nu}L^{2,k}_b(\bM)^{\bbT}
\label{la.27}\end{equation*}
is Fredholm with  nullspace  $P_{0,+}$, index $-\ell$ and cokernel canonically identified with 
$$
    \ker \bbS_b \cap r^{-\nu}L^2_b(\bM)^{\bbT}= \ker \bbS_b \cap r^{-\frac12}L^2_b(\bM)^{\bbT}= P_0\oplus K_{0}.
$$
In particular, the operator $\bbS_b$ induces an injective Fredholm operator
$$
   \bbS_b: (r^{\nu}L^{2,k+6}_b(\bM)^{\bbT})^{\perp_+}\to (r^{\nu}L^{2,k}_b(\bM)^{\bbT})^{\perp}.
$$
with cokernel identified with $K_{0}$.  By Lemma~\ref{la.25}, the map
$$
    \psi: K_{0}\to \bbR^{\ell}
$$
which to $u\in K_{0}$ associate $\psi(u)=(a_{p_1,0,1},\ldots,a_{p_{\ell},0,1})$ is injective, where $a_{p_i,0,1}\in \bbR$ is the coefficient such that $u-a_{p_i,0,1}\log r$ is bounded near $\pa_{p_i}\bM$.  On the other hand, since $\bbS_b$ is formally self-adjoint as a $b$-operator, the sesquilinear form \cite[(6.3)]{MelroseAPS} with $P=\bbS_b$ and $z=z'=0$ is anti-symmetric when restricted to the real part of the formal nullspace $F(\bbS_b,0)$ in the sense of \cite[(5.166)]{MelroseAPS}.  On the other hand this sesquilinear pairing is non-degenerate when $z=z'=0$ by \cite[Proposition~6.2]{MelroseAPS}.  Technically speaking, \cite[Proposition~6.2]{MelroseAPS} is formulated for elliptic differential $b$-operators, but the pseudodifferential part of $\bbS_b$ does not compromise this result as can be readily checked from the proof provided in \cite{MelroseAPS}.   This means using \cite[(6.9)]{MelroseAPS} that there exist functions $f_1,\ldots,f_{\dim K_0}\in \CI(M)^{\bbT}$ spanning a $\dim K_{0}$ subspace $\cD$ of $\CI(M)^{\bbT}$ such that the pairing 
\begin{equation}
         \begin{array}{lccc}
         B: & \cD\times K_{0}& \to & \bbC \\
         & (f,u) & \mapsto & \int_{\bM}((\bbS_b f)u)dg_b= \int_{\bM}(\bbS f)u \omega_0^{[m]}
         \end{array}
\label{la.27d}\end{equation}
is non-degenerate, where $dg_b=r^{-2m}\omega_0^{[m]}= r^{-6}\omega_0^{[3]}$ is the $b$-density naturally associated to the $b$-metric $g_b= r^{-2}g$.  This ensures in particular that the operator \eqref{la.26b} is an isomorphism.  As is clear from \eqref{la.27d}, we can assume without loss of generality that the functions $f_1,\ldots,f_{\dim K_{0}}$ are $L^2$-orthogonal to $P_{0}$ with respect to the metric $\omega_0$, so that $\cD\subset (\CI(M)^{\bbT})^{\perp}$. 

Since $]1,2[$ contains no indicial root, the map \eqref{la.26c} is Fredholm.  The surjectivity of \eqref{la.26b} and the regularity results of Corollary~\ref{b.13} then ensures that \eqref{la.26c} is surjective.
\end{proof}

In the case $m=2$, the following observation will be important.
\begin{lemma}
If $m=2$, then $\bbS_b$ has no indicial root in $]0,2[$ if $\Gamma_i\ne \{\Id\}$ for each $i\in\{1,\ldots,\ell\}.$  
\label{la2.2}\end{lemma}
\begin{proof}
By \eqref{ind.2} and \eqref{la.13}, the only possible indicial root in $]0,2[$ is at $\nu=1$ and corresponds to $\Delta_{S^{2m-1}/\Gamma_i}$ having $3$ as an eigenvalue for some $i\in \{1,\ldots,\ell\}$.  By \cite{BGM}, this will be the case if and only if the space of $\Gamma_i$-invariant homogeneous polynomials of degree $1$ is non-trivial for some $i\in\{1,\ldots,\ell\}$.  However, since $\Gamma_i$ acts freely on $\bbC^{m}\setminus\{0\}$, there is no non-trivial $\Gamma_i$-invariant homogeneous polynomial of degree $1$. 
\end{proof}
\begin{remark}
Notice that the conclusion of  Lemma~\ref{la2.2} is not true when we allow some of the points $p_i$ to be smooth, that is, with $\Gamma_i=\{\Id\}$.  This is the technical reason why our gluing construction does not work in complex dimension $2$ when smooth points are blown up.
\label{nbu.1}\end{remark}

For $\nu\in ]0,2[$, this suggests to set 
\begin{equation}
   P_{0,+}:= P_{0}\cap r^{\nu}L^2_b(\bM)^{\bbT}.
\label{la2.2b}\end{equation}
The next lemma confirms, as the notation suggests, that $P_{0,+}$ does not depend on the choice of $\nu\in]0,2[$.
\begin{lemma}\label{la2.3}
Suppose that $m=2$ and $R_{0}\ge 0$.  Then for $\nu\in ]0,2[$,
$$
    \ker \bbS\cap r^{\nu}L^2_b(\bM)^{\bbT}= P_{0,+} \subset r^\lambda L^2_b(\bM)^{\bbT}\quad \forall \lambda<2.
$$ 
\end{lemma}
\begin{proof}
It is clear  that $P_{0,+}\subset \ker \bbS\cap r^{\nu}L^2_b(\bM)^{\bbT}$, so we need to show that 
\begin{equation}
        \ker \bbS\cap r^{\nu}L^2_b(\bM)^{\bbT}\subset P_{0,+} \subset r^\lambda L^2_b(\bM)^{\bbT} \quad \forall \lambda<2
\label{la2.4}\end{equation}
to conclude that these two spaces coincide.  Since $\bbS_b$ has no indicial root in $]0,2[$ by Lemma~\ref{la2.2}, we know from Corollary~\ref{b.14} that
\begin{equation*}
     \ker \bbS\cap r^{\nu}L^2_b(\bM)^{\bbT}\subset r^{\lambda}L^{2,\infty}_b(\bM)^{\bbT} \quad \forall \; \lambda<2.
\label{la2.5}\end{equation*}
This ensures that the integrations by parts in the proof of Lemma~\ref{la.15} are still justified and yield the inclusion \eqref{la2.4}, from which the result follows.
\end{proof}

Thanks to this lemma, the space 
\begin{equation*}
 (r^\nu L^{2,k}_b(\bM)^\bbT)^{\perp_+}:= \left\{ u\in r^{\nu}L^{2,k}_b(\bM)^{\bbT}\; \Big| \; \int_M uh\omega^{[m]}=0 \quad \forall \; h\in P_{0,+} \right\}
\label{la2.6}\end{equation*}
is well-defined for $\nu>-6$.  This yields the following mapping property for the operator $\bbS$.
\begin{proposition}
Suppose that $m=2$, $R_{0}\ge 0$ and $\Gamma_i\ne \{\Id\}$ for each $i\in \{1,\ldots,\ell\}$.  Then  the operator $\bbS$ induces a continuous linear isomorphism
$$
     \bbS: (r^{\nu}L^{2,k+6}_b(\bM)^{\bbT})^{\perp_+} \to  (r^{\nu-6}L^{2,k}_b(\bM)^{\bbT})^{\perp_+}
$$
for $\nu\in ]0,2[$ and $k\in\bbN$.
\label{la2.7}\end{proposition}
\begin{proof}
By Lemma~\ref{la2.2}, the operator $\bbS_b$ has no indicial root in $]0,2[$.  Moreover, by \eqref{la.19}, when $\nu=1$, $r^{-\nu}\bbS_br^\nu$ is formally self-adjoint when acting formally on $L^2_b(\bM)^{\bbT}$. Proceeding as in the proof of Proposition~\ref{la.16}, we can check therefore that $\bbS$ induces a Fredholm operator 
\begin{equation}
  \bbS: r^{\nu}L^{2,k+6}_b(\bM)^{\bbT}\to r^{\nu-6}L^{2,k}_b(\bM)^{\bbT}
\label{la2.8}\end{equation}
with kernel and cokernel identified with $P_{0,+}$.  In fact, $P_{0,+}$ is $L^2(\bM)$-orthogonal to the range of \eqref{la2.8}, from which the result follows.
\end{proof}

\subsection{Linear analysis for ${\Sc}^{0}$-extremal $\ALE$-metrics}

We will suppose through this subsection that the K\"ahler orbifold $(M,\omega)$ admits a $\bbT$-equivariant K\"ahler resolution %\footnote{\color{blue} We use $\pi$ to denote at least 3 different things: a real number, bundle projections and  resolutions. We should probably be a bit more cautious about this! I favour using $\bl$ instead of $\pi$ for resolutions and blowups.}
\begin{equation}
    \bl: \hM\to M.
\label{la.27a}\end{equation}
  In particular, the singular models $\bbC^m/\Gamma_i$ for $i\in \{1,\ldots,\ell\}$ admit $\bbT$-equivariant K\"ahler resolutions
\begin{equation}
    \bl_i: X_i \to \bbC^{m}/\Gamma_i.
\label{la.27b}\end{equation}
Let $\bX_i$ the radial compactification of $X_i$.  In this paper, we will use the following definition of K\"ahler asymptotically locally euclidean (ALE) metrics on $X_i$. %\vesti{We seems to use interchangably $\|z\|^2, \rho^2, r^2$ to denote essentially the same geometric quantity, possibly appearing on different ``parts'' before the gluing.  Is this reasonable? Maybe introduce a convention glossary explaining  each notation.} \fr{They are not quite the same and not use in the same context. I think it is reasonable and necessary.}
\begin{definition}
    A K\"ahler metric $\omega_i$ on $X_i$ is  called \emph{asymptotically locally Euclidean} ($\ALE$ for short)  if,  outside a compact set $K$ of $X_i$,  it takes the form 
$$
 \omega_i= dd^c\bl_i^*\lrp{ \frac14\|z\|^2 + \|z\|^2 f_i},
$$    
where $f_i=\cO(\|z\|^{-\sigma_i})$ for some $\sigma_i>0$  is a smooth function admitting  a \emph{polyhomogeneous expansion} at infinity in terms of the radial compactification of $\bbC^m/\Gamma_i$. 
\label{ALE.1}\end{definition}
\begin{remark}
There are different definitions of $\ALE$ K\"ahler metrics in the literature, see for instance \cite{Arezzo-Pacard2006, Joyce}.  Compared to these, Definition~\ref{ALE.1} requires lesser decay but more regularity (polyhomogeneity) for the potential function $f_i$.  However, by \cite{Goto}, see also \cite[(4.4)]{CH2013}, combined with \cite[Theorem~5.3]{CMR2015},  the potential function $f_i$ corresponding to a \emph{Ricci-flat} K\"ahler ALE metric in the sense of \cite{Joyce} is automatically polyhomogeneous.  
\label{ALE.2}\end{remark}

{ \begin{assumption}\label{a:overall} We will suppose  from now on that each model $X_i$ admits a $\bbT$-invariant  ALE K\"ahler metric $g_i$ with K\"ahler form $\omega_i$ (in the sense of Definition~\ref{ALE.1}), further  satisfying 
\begin{equation}
  \Sc^{0}(\omega_i)=0, \qquad R_{\omega_i}\ge 0.
\label{van.1}\end{equation}
\end{assumption}}
Let $\bX_i$ be the radial compactification of $X_i$, so that the resolution \eqref{la.27b} extends to a resolution
 \begin{equation}
    \bl_i: \bX_i \to \overline{\bbC^{m}/\Gamma_i}.
\label{la.27c}\end{equation}

    In particular, $\pa \bX_i$ is naturally identified with $S^{2m-1}/\Gamma_i$.  Let $\rho_i\in\CI(\bX_i)$ be a choice of boundary defining function for $\bX_i$ such that $\rho_i= \|z\|^{-1}$ outside a compact set.  
\begin{lemma}
On $(X_i,g_i, \omega_i)$, the only Killing potentials contained in $\rho_i^{\nu}L^2_b(\bX_i)$ for $\nu\in ]-1,0[$ are the constant functions.  
\label{la.28}\end{lemma}
\begin{proof}
Using Hartogs' theorem as in the proof of \cite[Lemma~6.4]{Najafpour}, we see that there is no non-trivial holomorphic vector fields decaying at infinity, from which the result follows.  
\end{proof}

Before continuing, let us recall the following mapping properties for the Laplace-Beltrami operator of the metric $g_i$.

\begin{lemma}
For $\nu\in ]0,2(m-1)[$, the Laplace-Beltrami operator $\Delta_i$ induces an isomorphism
\begin{equation}
  \rho^{-2}_i\Delta_i: \rho_i^{\nu}L^{2,k+2}_b(\bX_i)\to  \rho_i^{\nu}L^{2,k}_b(\bX_i)
\label{la.28b}\end{equation}
for all $k\in \bbN$.  Moreover, for $\nu=m-1$, the inverse $G_i$ of the map
\begin{equation}
 \rho_i^{-(m-1)}(\rho^{-2}_i\Delta_i)\rho_i^{m-1}: L^{2,k+2}_b(\bX_i)\to  L^{2,k}_b(\bX_i)
\label{la.28c}\end{equation}
is an element $\Psi^{-2,\cG_i}_b(\bX_i)$ for a real index family $\cG_i$ such that 
$$
   \cG_i(\lf)\ge m-1, \quad \cG_i(\rf)\ge m-1 \quad \mbox{and} \quad \cG_i(\fb)\ge 0.
$$
Finally, for $\nu\in ]-1,0[$, the map \eqref{la.28b} is still surjective with kernel corresponding to the space of constant functions.  
\label{la.28a}\end{lemma}
\begin{proof}
By Corollary~\ref{b.15}, we know that the map \eqref{la.28b} is Fredholm for $\nu\in ]0,2m-2[$ and $\nu\in ]-1,0[$ since in this case $\nu$ is not a critical weight of the indicial family of $\rho_i^{-2}\Delta_i$.  A simple integration by parts argument then shows that \eqref{la.28b} has trivial kernel when $\nu\in ]0,2m-2[$.  Since the formal adjoint of  the $b$-operator $\rho_i^{-\nu}(\rho_i^{-2}\Delta_i)\rho_i^{\nu}$ is $\rho_i^{\nu-2(m-1)}(\rho_i^{-2}\Delta_i)\rho_i^{2(m-1)-\nu}$, we see that the cokernel of \eqref{la.28b} is also trivial when $\nu\in]0,2m-2[$, so that the map is invertible.  For the same reason, the map \eqref{la.28b} is surjective when $\nu\in ]-1,0[$.  Clearly, the constant functions are included in the kernel of this map.  By the relative index theorem of Melrose \cite[Theorem~6.5]{MelroseAPS}, this kernel must also be one dimensional, so must consist precisely of the constant functions.  Finally, the pseudodifferential characterization of the inverse \eqref{la.28c} can be deduced from Corollary~\ref{b.16}.
\end{proof}

Considering the linearization of the operator $\varphi\mapsto \Sc^{0}(\omega_i+dd^c\varphi)$ yields in this case a scattering pseudodifferential operator in the sense of \cite{MelroseGST}
taking the form
\begin{equation}
\bbS_{i,\sc}(u)=
-4\delta_{i}\delta_{i}(\nabla^{i,-}(\delta_{i}\nabla^{i,-}du))-2\delta_{i}\delta_{i}(R_{i}\nabla^{i,-}),
\label{la.30}\end{equation}
where $\delta_i$, $R_i$ and $\nabla^{i}$ denotes the divergence, scalar curvature and Levi-Civita connection of $\omega_i$.
\begin{lemma}
For $m\ge 2$, the operator 
$$
 \bbS_{i,b}:= \rho_i^{-6}\bbS_{i,\sc}
$$
is an elliptic differential $b$-operator of order $6$.
For $m>3$, the indicial family of $\bbS_{b,i}$  is contained in 
$$
   \bbZ\setminus ( \bbZ\cap ]0,2m-6[)
$$
with $0$ and $2m-6$ indicial roots of order $1$ and rank $1$.  If $m=3$, then $\Spec_b(\bbS_{i,b})\subset \bbZ$ with $0$ an indicial root of order $2$ and rank $2$.  In all cases, $\Spec_b(\bbS_{i,b})$ is symmetric with respect to $\lambda=m-3$.  
\label{la.31}\end{lemma}
\begin{proof}
The first statement is clear.  Notice that
\begin{equation}
      R_{i}\in \rho_i^{2+\sigma_i}L^{2,\infty}_{b}(\bX_i),
\label{la.32}\end{equation}
where $R_i$ is the scalar curvature of $\omega_i$ and its  decay in \eqref{la.32} is a consequence of our assumption that $g_i$ is an $\ALE$-metric, so is asymptotic at infinity to $g_{\bbC^m/\Gamma_i}$ at a rate $\rho_i^{\sigma_i}$ for the  $\sigma_i>0$ of Definition~\ref{ALE.1}. 

This shows that the term involving $R_{i}$ in the definition of $\bbS_{i,b}$ decays to order $\sigma_i$ at the $b$-front face $\fb$.  In particular, this term does not contribute to the indicial family of $\bbS_{i,b}$.  The indicial family therefore depends only from the first term in the right hand side of \eqref{la.30}.  Proceeding as in the proof of Lemma~\ref{la.14}, we deduce that the indicial family of $\bbS_{i,b}$ is the same as the one of $\rho_i^{-6}\Delta^3_{\bbC^m/\Gamma_i}$ at $\rho_i=0$, where $\Delta_{\bbC^m/\Gamma_i}$ is the Laplace-Beltrami operator of the canonical metric on $\bbC^m/\Gamma_i$.  Since we can assume $\rho_i= r^{-1}$ outside a compact set, we have that
$$
    \rho_i^{-\lambda}\rho_i^{-6}\Delta^3_{\bbC^m/\Gamma_i}\rho_i^{\lambda}= r^{\lambda} r^6\Delta^3_{\bbC^m/\Gamma_i}r^{-\lambda},
$$
so that in terms of \eqref{ind.1},
\begin{equation}
    I(\bbS_{i,b},\lambda)= I(r^6\Delta_{\bbC^m/\Gamma_i}^3,-\lambda).
\label{la.32c}\end{equation}
The results about the indicial family of $\bbS_{i,b}$ corresponds therefore to those of Lemma~\ref{la.14} with the sign of $\lambda$ reversed.
\end{proof}

For $m> 3$, this can be used to deduce the following mapping properties.
\begin{proposition}
For $m>3$, the operator $\bbS_{i,b}$ induces a continuous linear isomorphism
\begin{equation}
\bbS_{i,b}: \rho_i^{\nu}L^{2,k+6}_b(\bX_i)\to \rho_i^{\nu}L^{2,k}_b(\bX_i)
\label{la.33a}\end{equation}
for $\nu\in ]0,2m-6[$ and $k\in\bbN$.  Alternatively, this says that $\bbS_{i,\sc}$ induces a continuous linear isomorphism
\begin{equation}
\bbS_{i,\sc}: \rho_i^{\nu}L^{2,k+6}_b(\bX_i)\to \rho_i^{\nu+6}L^{2,k}_b(\bX_i)
\label{la.33b}\end{equation}
for $\nu\in ]0,2m-6[$ and $k\in\bbN$.  
\label{la.33}\end{proposition}
\begin{proof}
By Lemma~\ref{la.31} and Corollary~\ref{b.15}, the operator
\begin{equation}
     \rho_i^{-\nu}\bbS_{i,b}\rho_i^{\nu}: L^{2,k+6}_b(\bX_i)\to L^{2,k}_b(\bX_i)
\label{la.33c}\end{equation}
is Fredholm.  Moreover, by the relative index theorem of Melrose \cite[Theorem~.6.5]{MelroseAPS}, its index does not depend on the choice of $\nu\in]0,2m-6[$.

Clearly,  the proposition will follow provided we can show that its index vanishes and its nullspace is trivial.
 Now, using the fact that $\bbS_{i,\sc}$ is formally self-adjoint as a scattering operator, we can deduce that the formal adjoint of $\bbS_{i,b}$ as a $b$-operator is 
\begin{equation}
    \bbS_{i,b}^*= \rho_i^{-(2m-6)}\bbS_{i,b}\rho_i^{2m-6}.
\label{adj.1}\end{equation}
In particular, the formal adjoint of $\rho_i^{-\nu}(\bbS_{i,b})\rho_i^{\nu}$ as a $b$-operator is 
$$
      (\rho_i^{-\nu}\bbS_{i,b}\rho_i^{\nu})^*= \rho_i^{\nu}\rho_i^{-(2m-6)}\bbS_{i,b}\rho_i^{2m-6}\rho_i^{-\nu}.  
$$
For $\nu=m-3$, this implies that $(\rho_i^{-\nu}\bbS_{i,b}\rho_i^{\nu})$ is a formally self-adjoint $b$-operator, which implies that the index of \eqref{la.33a} vanishes for $\nu=m-3$, so for all $\nu\in ]0,2m-6[$.  Thus, it remains to show that the nullspace of \eqref{la.33a} is trivial.  Now, by the regularity result of Corollary~\ref{b.14}, the kernel of \eqref{la.33b} is the same for all $\nu\in ]0,2m-6[$ and $k\in\bbN$ and is contained in $\rho_i^{2m-6-\delta}L^{2,\infty}_b(\bX_i)$ for all $\delta>0$.  If $u$ is an element of this kenrel, this implies  that we can integrate by parts as in the proof of Lemma~\ref{la.15} to deduce that 
\begin{equation}
0= -4\| \delta_{i}\nabla^{i,-}du\|^2_{L^2_{\sc}}-2\|R_{i}^{\frac12}\nabla^{i,-}du\|^2_{L^2_{\sc}}.
\label{la.34}\end{equation}
  In particular, we deduce from \eqref{la.34} that 
$$
   \| \delta_{i}\nabla^{i,-}du\|^2_{L^2_{\sc}}=0,
$$
which implies that $\delta_{i}\nabla^{i,-}du=0$.  This in turn implies that
$$
   0= \langle du, \delta_{i}\nabla^{i,-}du\rangle_{L^2_{\sc}}= \|\nabla^{i,-}du \|^2_{L^2_{\sc}}=0,
$$
so that
$$
     \nabla^{i,-}du=0.
$$
This means that $u$ must be a Killing potential, but by Lemma~\ref{la.28} and the fact $u$ decays at infinity, this implies that $u=0$.  This means that the nullspace of \eqref{la.33a} is trivial, from which the result follows.
\end{proof}

\begin{corollary}
For $\nu\in ]-2,-1[\cup ]-1,0[$, the maps \eqref{la.33a} and \eqref{la.33b} are no longer injective, but they remain surjective.  
\label{la.33e}\end{corollary}
\begin{proof}
For $\nu\in  ]-2,-1[\cup ]-1,0[$, the map \eqref{la.33c} remains Fredholm since $\nu$ is not an indicial root.  However, by the relative index theorem \cite[Theorem~6.5]{MelroseAPS}, its index is positive.    Its formal adjoint will also be Fredholm and its kernel will be trivial by the same argument used in the proof of Proposition~\ref{la.33}, implying that  \eqref{la.33c} is surjective.
\end{proof}

For $m=3$, we obtain a slightly different statement.

\begin{proposition}
For $m=3$, the operator $\bbS_{i,b}$ induces a surjective Fredholm operator
\begin{equation}
    \bbS_{i,b}: \rho_i^{\nu} L^{2,k+6}_b(\bX_i)\to \rho_i^{\nu}L^{2,k}_b(\bX_i)
\label{la.35a}\end{equation}
for $\nu\in ]-1,0[$ with nullspace corresponding to constant functions.  For $\nu\in]0,1[$, it induces instead an injective Fredholm operator with cokernel identified with the constant functions.
\label{la.35}\end{proposition}
\begin{proof}
For $m=3$ and $\nu>0$, the integrations by parts involved in the proof of Proposition~\ref{la.33} can still be used to show that \eqref{la.35a} has trivial nullspace when $\nu>0$.  For $\nu\in ]0,1[$ fixed, we know also from Lemma~\ref{la.31}, Corollary~\ref{b.15}  and \cite[Theorem~6.5]{MelroseAPS} that the operators
\begin{equation}
     \bbS_{i,b,\pm\nu}:= \rho_i^{\mp\nu}\bbS_{i,b}\rho_{i}^{\pm\nu}: L^{2,k+6}_b(\bX_i)\to L^{2,k}_b(\bX_i)
\label{la.35b}\end{equation}
are Fredholm with 
$$
   \ind \bbS_{i,b,-\nu}= \ind \bbS_{i,b,\nu}+2.
$$
Since $\bbS_{i,b,-\nu}$ is the adjoint of $\bbS_{i,b,\nu}$, this means that
$$
     \ind \bbS_{i,b,\nu}=-1 \quad \mbox{and} \quad \ind\bbS_{i,b,-\nu}=1.
$$

This implies that $\bbS_{i,b,-\nu}$ is surjective, since its cokernel is identified with the kernel of $\bbS_{i,b,\nu}$, which is trivial.  Hence, the kernel of $\bbS_{i,b,-\nu}$ has dimension $1$.  Since constant functions are obviously contained in it, this implies that $\ker \bbS_{i,b,-\nu}$ only contains constant functions.  Correspondingly, the cokernel of $\bbS_{i,b,\nu}$, which is identified with $\ker \bbS_{i,b,-\nu}$, is therefore identified with the space of constant functions.
\end{proof}

Finally, to treat the case $m=2$, let us first make the following observation.  
\begin{lemma}
If $m=2$ and $\Gamma_i\ne \{\Id\}$ for all $i\in\{1,\ldots,\ell\}$, then $\bbS_{i,b}$ has no indicial root in $]-2,0[$. 
\label{lasc.1}\end{lemma}
\begin{proof}
This follows from Lemma~\ref{la2.2} and \eqref{la.32c}.
\end{proof}
This can be used to identify the kernel of $\bbS_{i,\sc}$ with the constant functions.  
\begin{lemma}
Suppose that $m=2$ and $R_{\omega_i}\ge 0$.  Suppose also that $\Gamma_i\ne \{\Id\}$ for all $i\in\{1,\ldots,\ell\}$.   Then for $\nu\in ]-2,0[$,
$$
    \ker \bbS_{i,\sc}\cap \rho_i^{\nu}L^2_b(\bX_i)^{\bbT}
$$
coincides with the space of constant functions on $X_i$.  
\label{lasc.2}\end{lemma}
\begin{proof}
Suppose that $u\in \ker \bbS_{i,\sc}\cap \rho_i^{\nu}L^2_b(\bX_i)$.  By the regularity results of Corollary~\ref{b.14} and the description of the indicial set of $\bbS_{i,b}$, we know that 
\begin{equation}
   u- a_0 + a_{0,1}\log \rho_i \in \rho_i^{\lambda}L^{2,\infty}(\bX_i)
\label{lasc.3}\end{equation}
for some $a_0\in \CI(\pa \bX_i)$, $a_{0,1}\in\bbR$ and $\lambda>0$.  This implies that we can integrate by parts as in \eqref{la.34} to conclude that 
$$
    \delta_{i}\nabla^{i,-}du=0.
$$
Using \eqref{lasc.3}, we can use integration parts once more to check that
\begin{equation}
0 = \langle du, \delta_{i}\nabla^{i,-}du\rangle_{L^2_{\sc}}= \| \nabla^{i,-}du\|^2_{L^2_{\sc}}-c
\label{lasc.4}\end{equation}
for some constant $c$ corresponding to the limiting behaviour of the boundary term.  The constant $c$ comes from the top order term
$$
  a_0+ a_{0,1}\log \rho_i
$$
in the asymptotic expansion of $u$ at $\pa X_i$.  In particular, we deduce from \eqref{lasc.4} that 
\begin{equation}
   \nabla^{i,-}du\in L^2_{\sc}(\bX_i; S^2({}^{\sc}T^*\bX_i)).
\label{lasc.5}\end{equation}
Without loss of generality, $\lambda>0$ in \eqref{lasc.3} can be chosen as small as we want. For $\lambda>0$ small enough, this means that \eqref{lasc.5} is only possible if 
$$
   \nabla^{i,-}du\in \rho_i^{2+\lambda} \lrp{\cA_{\phg}(\bX_i; S^2({}^{\sc}T^*\bX_i)) \cap L^{\infty}(\bX_i; S^2({}^{\sc}T^*\bX_i))}.
$$
In that case,
$$
\delta_{i}\nabla^{i,-}du\in \rho_i^{3+\lambda} \lrp{\cA_{\phg}(\bX_i; {}^{\sc}T^*\bX_i) \cap L^{\infty}(\bX_i; {}^{\sc}T^*\bX_i)}.
$$
Since on the other hand \eqref{lasc.3} implies that
$$
  du\in \rho_i \lrp{\cA_{\phg}(\bX_i; {}^{\sc}T^*\bX_i) \cap L^{\infty}(\bX_i; {}^{\sc}T^*\bX_i)},
$$
this means that integration by parts can be used to conclude that
$$
  0 =\langle du, \delta_{i}\nabla^{i,-}du\rangle_{L^2_{\sc}}= \| \nabla^{i,-}du\|_{L^2_{\sc}},
$$
that is, that
$$
\nabla^{i,-}du=0.
$$
This implies that $u$ is a Killing potential.  By Lemma~\ref{la.28} and \eqref{lasc.3}, this means that $u$ must be a constant function.
\end{proof}

This allows us to obtain the following mapping property of the operator $\bbS_{i,\sc}$.  

\begin{proposition}
Suppose that $m=2$ and $R_{\omega_i}\ge 0$ with $\Gamma_i\ne \{\Id\}$ for all $i\in\{1,\ldots,\ell\}$.  Then for $\nu\in ]-2,0[$, the operator $\bbS_{i,\sc}$ induces a Fredholm operator
\begin{equation}
 \bbS_{i,\sc}: \rho_i^{\nu}L^{2,k+6}_b(\bX_i)^{\bbT}\to \rho_i^{\nu+6}L^{2,k}_b(\bX_i)^{\bbT}
\label{lasc.6a}\end{equation}
with kernel and cokernel both identified with the space of constant functions.
\label{lasc.6}\end{proposition}
\begin{proof}
Since $\nu\in]-2,0[$ is not an indicial root by Lemma~\ref{lasc.1}, we know that the map \eqref{lasc.6a} is Fredholm.  By Lemma~\ref{lasc.2}, its kernel corresponds to constant functions.  Its cokernel is identified with the $L^2_b$-kernel of the formal adjoint of $\rho_i^{-\nu}\bbS_{i,b}\rho_i^{\nu}$, namely the $L^2_b$-kernel of  
$$
(\rho_i^{-\nu}\bbS_{i,b}\rho_i^{\nu})^*= \rho_i^{2+\nu} \bbS_{i,b} \rho_i^{-2-\nu}.
$$
Since $-2-\nu\in ]-2,0[$, Lemma~\ref{lasc.2} implies that its kernel is given by 
$
\{ \rho_i^{2+\nu}c \; | \; c\in \bbR\}.
$
\end{proof}
Setting 
$$
  (\rho_i^{\nu+6}L^{2,k}_b(\bX_i)^{\bbT})^{\perp}= \left\{ u\in  \rho_i^{\nu-6}L^{2,k}_b(\bX_i)^{\bbT}\; \Big| \; \int_{X_i} u \omega_i^{m}=0\right\},
$$
we obtain in particular the following.
\begin{corollary}
Suppose that $m=2$ and  $R_{\omega_i}\ge 0$.  Then for $\nu\in ]-2,0[$, the operator $\bbS_{i,\sc}$ induces a surjective Fredholm operator
\begin{equation*}
 \bbS_{i,\sc}: \rho_i^{\nu}L^{2,k+6}_b(\bX_i)^{\bbT}\to (\rho_i^{\nu+6}L^{2,k}_b(\bX_i)^{\bbT})^{\perp}
\label{lasc.7a}\end{equation*}
with kernel identified with the space of constant functions.
\label{lasc.7}\end{corollary}

The following result will also be useful in the construction of a formal solution.

\begin{corollary}
Suppose that $m=2$ and $R_{\omega_i}\ge 0$ with $\Gamma_i\ne \{\Id\}$ for all $i\in\{1,\ldots,\ell\}$.   Then for $\nu\in ]-3,-2[$, the operator $\bbS_{i,\sc}$ induces a surjective Fredholm operator
\begin{equation}
 \bbS_{i,\sc}: \rho_i^{\nu}L^{2,k+6}_b(\bX_i)^{\bbT}\to \rho_i^{\nu+6}L^{2,k}_b(\bX_i)^{\bbT}.
\label{lasc.8a}\end{equation}
Moreover, given $f\in \rho_i^{\nu+7}(\cA_{\phg}(\bX_i)\cap L^{\infty}(X_i))^{\bbT}$, there exists an index set $\cE_i$ with $\cE_i\ge (-2,1)$ and a $\bbT$-invariant function $u\in\cA_{\phg}^{\cE_i}(\bX_i)$ such that
$\bbS_{i,\sc}u=f$.
\label{lasc.8}\end{corollary}
\begin{proof}
Since $\bbS_{i,b}$ has no indicial root in $]-3,-2[$, we know that \eqref{lasc.8a} is Fredholm for $\nu\in]-3,-2[$.  On the other hand, for such a $\nu$, we know by Lemma~\ref{lasc.2} that
$$
      \ker\bbS_{i,\sc}\cap \rho_i^{-2-\nu}L^2_b(\bX_i)^{\bbT}=\{0\} 
$$
since $-2-\nu>0$.  By \eqref{adj.1} this corresponds to the cokernel of \eqref{lasc.8a}, showing that \eqref{lasc.8a} is indeed surjective.  The second part of the assertion follows from the surjectivity  of \eqref{lasc.8a}, the regularity results of Corollary~\ref{b.13}  and the description of the indicial set of $\bbS_{i,b}$ given by \eqref{la.32c} and  Lemma~\ref{la.14}. \end{proof}

\subsection{The $b$-surgery space} \label{fs.0}

Proceeding as in \cite[(2.10)]{CDR2019}, consider the orbifold with corners
\begin{equation*}
 \cM:= [M\times [0,1[_{\epio}; \{p_1\}\times \{0\},\ldots, \{p_{\ell}\}\times \{0\}]
\label{fs.1}\end{equation*}
obtained from $M\times [0,1[_{\epio}$ by blowing up the singular points of $M$ at $\epio=0$ in the sense of Melrose.  Let $\cB:\cM\to M\times[0,1[$ be the corresponding blow-down map.  This has $\ell+1$ boundary hypersurfaces.  Denote by $H_0$ the boundary hypersurface corresponding to the lift of $M\times \{0\}$ and by $H_i$ the boundary hypersurface arising from the blow-up of $\{p_i\}\times \{0\}$.  Clearly, there are natural diffeomorphisms
$$
    H_{0}\cong \bM \quad \mbox{and} \quad H_i= \overline{\bbC^m/\Gamma_i} \quad \mbox{for}\; i>0,
$$
where $\overline{\bbC^m/ \Gamma_i}$ is the radial compactification of $\bbC^{m}/ \Gamma_i$ obtained by adding $S^{2m-1}/\Gamma_i$ at infinity.  The resolutions \eqref{la.27a} and \eqref{la.27c} naturally combine to induce a resolution 
\begin{equation}
 \cbl: \chM \to \cM
\label{fs.2}\end{equation}
with $\chM$ a manifold with corners with boundary hypersurfaces $\hH_0,\ldots,\hH_{\ell}$ coming with natural identifications $\hH_0\cong \bM$ and $\hH_{i}\cong \bX_i$ for $i\in\{1,\ldots,\ell\}$ inducing commutative diagrams 
$$
\xymatrix{
     \hH_0 \ar[r]^{\cbl} \ar[d] & H_0 \ar[d] \\
     \bM \ar[r]^{\Id} & \bM 
}
\quad \quad \mbox{and} \quad \quad  \xymatrix{
     \hH_i \ar[r]^{\cbl} \ar[d] & H_i \ar[d] \\
      \bX_i \ar[r]^{\bl_i} & \overline{\bbC^m\setminus\Gamma_i} 
}
$$ 
respectively.  For $\epio>0$, the natural identifications 
$$
   \chM\setminus \left(\bigcup_{i=0}^{\ell}\hH_i \right)= \hM\times ]0,1[_{\epio} \quad \mbox{and} \quad \cM\setminus \left(\bigcup_{i=0}^{\ell}H_i\right)= M\times ]0,1[_{\epio}
$$
also induce the commutative diagram 
$$
\xymatrix{
  \chM\setminus \left(\bigcup_{i=0}^{\ell}\hH_i\right) \ar[r]^{\cbl} \ar[d] & \cM\setminus \left(\bigcup_{i=0}^{\ell}H_i\right) \ar[d] \\
  \hM\times ]0,1[\ar[r]^{\bl\times \Id_{]0,1[}} & M\times ]0,1[.
}
$$
As a manifold with corners, $\chM$ admits a Lie algebra of $b$-vector fields $\cV_b(\chM)$.  Denoting also by $\epio$ the function 
$$
    \pr_2\circ \mathcal{B}\circ\cbl: \chM\to [0,1[,
$$
where $\pr_2: M\times [0,1[\to [0,1[$ is the projection on the second factor, we can proceed as in \cite{Mazzeo-MelroseETA} and consider the Lie subalgebra of $b$-surgery vector fields 
$$
   \cV_{b,s}(\chM):= \{ \xi\in \cV_b(\chM)\; | \; \cL_{\xi}\epio=0\}.
$$
There is a corresponding Lie algebroid ${}^{b,s}T\chM\to \chM$ with anchor map $\iota_{b,s}: {}^{b,s}T\chM \to T\chM$ inducing an isomorphism of Lie algebras 
$$
   (\iota_{b,s})_*: \CI(\chM;{}^{b,s}T\chM) \to \cV_{b,s}(\chM).
$$
It comes with natural identifications
$$
     {}^{b,s}T\chM|_{\hH_i} = {}^{b}T\hH_i\quad \mbox{for} \quad i\in\{0,1,\ldots,\ell\}.
$$

The associated space of differential $b$-surgery operators $\Diff^k_{b,s}(\chM)$ of order $k$ corresponds to a finite linear combination of differential operators generated by multiplication by elements of $\CI(\chM)$ and the composition of up to $k$ $b$-surgery vector fields.  As for differential $b$-operators, given vector bundles $E_1\to \chM$ and $E_2\to \chM$, we get a corresponding space $\Diff^k_{b,s}(\chM;E_1,E_2)$ of \emph{differential $b$-surgery operators}.  If $\cE$ is an index family for $\chM$ assigning the index set $\cE_{\hH_i}$ to the boundary hypersurface $\hH_i$, we can also consider the space of differential $b$-surgery operators
$$
  \Diff^k_{b,s,\cE}(\chM;E_1,E_2)= \cA_{\phg}^{\cE}(\chM)\otimes_{\CI(\chM)}\Diff^k_{b,s}(\chM;E_1,E_2).
$$

Let $\rho_+\in \CI(\chM)$ be a nonnegative function vanishing only on $\hH_i$ for $i>0$ in such a way that its differential is nowhere zero on $\hH_i$ for $i>0$.  In other words, $\rho_+$ is a boundary defining function for the (disconnected) boundary hypersurface
$$
   \hH_+:= \bigcup_{i=1}^{\ell} \hH_i.
$$
The action of $\bbT$ on $\hM$ naturally lifts to an action on $\chM$.  Averaging with respect to this action if necessary, we can assume that $\rho_+$ is a $\bbT$-invariant function.  We can also assume that $\rho_+|_{\hH_0}=r$, where $r\in \CI(\bM)$ is the boundary defining function that we previously used on $\bM\cong \hH_0$.  A $\bbT$-invariant boundary defining function for $\hH_0$ is given by  $\rho_0:= \frac{\epio}{\rho_+}$.

Away from $\pa\chM:= \bigcup_{i=0}^{\ell} \hH_i$ and under the identification $\chM\setminus \pa \chM = \hM\times ]0,1[_{\epio}$, notice that the anchor map induces an isomorphism 
$$
  \iota_{b,s}: {}^{b,s}T\chM|_{\chM\setminus \pa\chM} \to \pr_1^*T\hM,
$$
where $\pr_1: \hM\times ]0,1[_{\epio}\to \hM$ is the projection on the first factor.  From \cite{Mazzeo-MelroseETA}, a $b$-surgery metric is a family of metrics $g_{\epio}$ on $\hM$ parametrized by $\epio\in ]0,1[$ such that 
$$
     \iota_{b,s}^*g_{\epio}= g_{b,\epio} 
$$
for some bundle metric $g_{b,\epio}\in \CI(\chM; S^2({}^{b,s}T^*\chM))$ for the vector bundle ${}^{b,s}T\chM\to \chM$.  Trusting this will lead to no confusion, we will usually directly denote such a family of degenerating metrics by $g_{b,\epio}$.  More generally, we can relax the regularity of $g_{b,\epio}$ and only require that 
$$
    g_{b,\epio}\in \cA_{\phg}(\chM; S^{2}({}^{b,s}T^*\chM))
$$
with well-defined restrictions to $\hH_i$ inducing a $b$-metric for all $i\in \{0,1,\ldots,\ell\}$.  

Notice that a $b$-metric on $\chM$ induces by restriction on ${}^{b,s}T\chM$ a $b$-surgery metric.  If $g_b$ is a choice of $b$-metric on $\chM$ and $E_1\to \chM$ is a vector bundle  equipped with a bundle metric, then there are associated $b$-Sobolev spaces
$$
   L^{2,k}_{b}(\chM;E_1):=\{ u\in L^2(\chM;E_1)\; | \; Pu \in L^2_{b}(\chM;E_1)\quad \forall \; P\in \Diff^k_{b}(\chM;E_1,E_1)\},
$$
where $L^2_b(\chM;E_1)$ is the space of square integrable sections of $E_1$ with respect to $g_b$ and the bundle metric.  However, fixing a $b$-surgery metric $g_{\epio}$, it will be more useful to us to consider the family of $L^2$-spaces 
$$
L^2_{b,\epio}(\hM;E_1):= L^2(\hM; E_1|_{\hM\times \{\epio\}})
$$
parametrized by $\epio>0$ under the identification $\chM\setminus \pa\chM=\hM\times ]0,1[_{\epio}$ with norm induced by $g_{\epio}$ and the bundle metric of $E_1$, as well as the corresponding \emph{$b$-surgery Sobolev spaces}
$$
 L^{2,k}_{b,\epio}(\hM;E_1):= \{ u\in L^2_{b,\epio}(\hM;E_1)\; | \; Pu\in L^2_{b,\epio}(\hM;E_1) \quad \forall P\in \Diff^k_{b,s}(\chM;E_1)\}
$$
with norm induced by $g_{\epio}$ and the bundle metric of $E_1$.  Notice that the $L^2$-norm of $L^{2,k}_{b,\epio}(\hM;E_1)$ degenerates as $\epio\searrow 0$ with $L^{2,k}_{b,\epio}(\hM;E_1)$
corresponding to 
$$
  L^{2,k}_{b,0}(\hM;E_1):= \bigoplus_{i=0}^\ell L^{2,k}_b(\hH_i;E_1|_{\hH_i})
$$
in the limite $\epio\searrow 0$.  The Sobolev spaces $L^{2,k}_{b,\epio}(\hM;E_1)$ differ from the Sobolev spaces of \cite[(85)]{Mazzeo-MelroseETA}, those being more an hybrid between $L^{2,k}_b(\chM;E_1)$ and $L^{2,k}_{b,\epio}(\chM;E_1)$.  Thinking of the Sobolev spaces $L^{2,k}_{b,\epio}(\hM;E_1)$ as forming a vector bundle $L^{2,k}_{b,s}(\chM;E_1)\to ]0,1[$ of infinite rank with fiber above $\epio\in]0,1[$ given by $L^{2,k}_{b,\epio}(\hM;E_1)$, we can for $\epio_0\in]0,1]$ also consider the space
\begin{equation*}
 \cC^q(]0,\epio_0[;L^{2,k}_{b,s}(\chM;E_1))
\label{sob.1}\end{equation*}
of $q$-differentiable sections $\zeta$  of $L^{2,k}_{b,s}(\chM;E_1)$ with $q$th differential continuous in $\epio$ (as a section of $L^{2,k}_{b,s}(\chM;E_1)$) such that
\begin{equation}
\|\zeta\|_{\cC^qL^{2,k}_{b,s}}:= \sup_{\epio\in ]0,\epio_0[} \lrp{\sum_{j=0}^q \left\| \lrp{\epio\frac{\pa}{\pa\epio}}^j\zeta \right\|_{L^{2,k}_{b,\epio}}}<\infty,
\label{sob.2}\end{equation}
where $\zeta$ in the expression above is seen as a section of $E_1$ on $\chM$ and the derivatives in $\epio$ are with respect to a choice of connection for $E_1$.  Clearly the norm \eqref{sob.2} makes $\cC^q(]0,\epio_0[;L^{2,k}_{b,s}(\chM;E_1)])$ a Banach space.

As in \cite{ARS3, Najafpour, Benabida}, the class of families of degenerating metrics we need to consider is not quite the one of $b$-surgery metrics, but the one of conical surgery metrics.
\begin{definition}
A \emph{conical surgery metric} is a family of metrics $g_{c,\epio}$ on $\hM$ for $\epio\in ]0,1[$ of the form
$$
      g_{c,\epio}= \rho_+^2 g_{b,\epio}
$$
for some $b$-surgery metric $g_{b,\epio}$.
\label{fs.3}\end{definition}
We define a corresponding vector bundle ${}^{c,s}T\chM= \rho_+^{-1}\lrp{{}^{b,s}T\chM}$, so that a conical surgery metric $g_{c,\epio}$ can be seen as a bundle metric for ${}^{c,s}T\chM$.  On $\hH_0=\bM$, a conical surgery metric restricts to a metric $g_{c,0}$ with conical singularities at $\pa\bM$.  An example of such a metric with conical singularities is given by a metric smooth in the sense of orbifolds as considered in the previous section.  On the other hand, on $\hH_+$, the restriction of $g_{b,\epio}$ induces a $b$-metric $g_{b,+}$.  In terms of $g_{c,\epio}$, we can instead consider the restriction 
$$
     g_{c,+}:=\epio^{-2}g_{c,\epio}|_{\hH_+}= \rho_0^{-2}g_{b,+}.
$$
This is a scattering metric on $\hH^+$ in the sense of \cite{MelroseGST}.  Restricted to one of the connected components $\hH_i$ of $\hH_+$, this is in fact an $\ALE$-metric.

Given the extremal orbifold metric $g_0$ on $M$ and the extremal $\ALE$-metrics $g_i$ on $X_i$ for each $i\in\{1,\ldots,\ell\}$, we want to find a conical surgery metric $g_{c,\epio}$ which is K\"ahler for $\epio>0$ and such that
$$
     g_{c,\epio}|_{\hH_0}=g_0, \quad \epio^{-2}g_{c,\epio}|_{\hH_i}=g_i, \quad i\in\{1,\ldots,\ell\}.
$$  
To do so, we can proceed as in \cite[\S~5.3.1]{Benabida}.  First, since $g_0$ is smooth in the sense of orbifolds, we can assume (see e.g. \cite[p.190]{Arezzo-Pacard2006}) that its K\"ahler form $\omega_0$ is
$$
     \omega_0= dd^c\lrp{\frac{r^2}4+f_0}
$$
near $\pa\bM$ with $f_0\in r^{4}\CI(M)\subset r^4\CI_b(\bM)$.  For the K\"ahler form $\omega_i$, recalling that $\rho_i= \frac{1}{\|z\|}$ outside a compact set of $X_i$, we have from Definition~\ref{ALE.1} 
that
$$
   \omega_i= dd^c\left(\rho_i^{-2}\left(\frac14+f_i\right)\right)
$$
outside a compact set. In the above expression 
$f_i$ is a $\bbT$-invariant polyhomogeneous function of order $\mathcal{O}(\rho_i^{\sigma_i})$ for some $\sigma_i>0$. Since we can assume that $\rho_i$ is the restriction of $\rho_0$ to $\hH_i$, notice that
$$
      \epio^2\omega_i= \epio^2 dd^c\left(\rho_i^{-2}\left(\frac14+f_i\right)\right)=dd^c\left(\rho_+^2\left(\frac14+f_i\right)\right),
$$
so that the top order term of $\epio^2\omega_i$ at $\hH_i\cap\hH_0$ agrees with the one of $\omega_0$.  In a small neighbourhood of $\hH_i\cap \hH_0$ in $\chM$, we can thus find a polyhomogeneous function $\psi_i$ such that
$$
 \frac{\psi_i}{\rho_+^2}|_{\hH_i}= \frac14+f_i \quad \mbox{and} \quad \psi_i|_{\hH_0}= \frac{r^2}{4}+ f_0.
$$
Indeed, if $\chi\in\CI_c([0,\infty[)$ is a function such that $\chi(x)=1$ for $x<\frac12$ and $\chi(x)=0$ for $x>1$, then setting $\chi_{\delta}(x):=\chi(\frac{x}{\delta})$ for $\delta>0$, it suffices to take
\begin{equation}
    \psi_i= \frac{\rho_+^2}{4}+ \chi_{\delta}(\rho_+)\rho_+^2f_i + \chi_{\delta}(\rho_0)f_0
\label{fs.4a}\end{equation}
for $\delta>0$ small enough, where $f_i$ is extended smoothly off $\hH_i$  and $f_0$ is extended smoothly off $\hH_0$.  Averaging with respect to the action of $\bbT$, we can assume without loss of generality that $\psi_i$ is $\bbT$-invariant. Proceeding as in \cite[\S~5.3.1]{Benabida}, we can therefore consider a degenerating family of closed $(1,1)$-forms $\omega_{c,\epio}$ on $\hM$ corresponding on $\chM$ to a section of $\Lambda^2({}^{c,s}T^*\chM)$ given by
\begin{equation}
  \omega_{c,\epio}= \left\{ \begin{array}{ll}
     dd^c\psi_i, & \mbox{near $\hH_i\cap \hH_0$ for $i>0$}, \\
     \epio^2\omega_i, & \mbox{near $\hH_i$ for $i>0$, but away from $\hH_0$}, \\
     \omega_0, & \mbox{near $\hH_0$, but away from $\hH_i$ for $i>0$}.
     \end{array}
   \right.
\label{fs.5}\end{equation}  
In particular, $\omega_{c,\epio}$ is automatically $\bbT$-invariant.  By construction, 
$$
   \omega_{c,\epio}|_{\hH_0}=\omega_0 \quad \mbox{and} \quad \frac{\omega_{c,\epio}}{\epio^2}|_{\hH_i}=\omega_i,  \quad i\in \{1,\ldots,\ell\}.
$$
Hence, if $\hat{J}$ is the complex structure on $\hM$, this means that 
$$
    g_{c,\epio}:= \omega_{c,\epio}(\cdot, \hat{J}\cdot)
$$
is a bounded polyhomogeneous section of $S^2({}^{c,\epio}T^*\chM)$ such that 
$$
     g_{c,\epio}|_{\hH_0}=g_0 \quad \mbox{and} \quad    \frac{g_{c,\epio}}{\epio^2}|_{\hH_i}=g_i \quad i\in\{1,\ldots,\ell\}.
$$
This ensures that $g_{c,\epio}$ is a positive definite section of $S^2({}^{c,\epio}T^*\chM)$ when restricted to $\hH_{+}$ and $\hH_0$.  Since being positive definite is an open condition, this implies that $g_{c,\epio}\in \cA_{\phg}(\chM; S^2({}^{c,\epio}T^*\chM))\cap L^{\infty}(\chM; S^2({}^{c,\epio}T^*\chM))$ defines a K\"ahler metric on $\hM$ for $\epio>0$ small enough with K\"ahler form $\omega_{c,\epio}$.  In the remaining of this section, we will always implicitly assume that $\epio$ is small enough so that $\omega_{c,\epio}$  is the K\"ahler form of a conical surgery metric.  Similarly, whenever we will consider a $2$-form $\omega_{c,\epio}+dd^c\phi$ for some potential $\phi$ decaying at $\hH_0$ and $\hH_+$, we will assume that  $\epio>0$ is small enough (in a way depending on $\phi$) so that    $\omega_{c,\epio}+dd^c\phi$ is the K\"ahler form of a conical surgery metric.

We explain now how to extend classes in $H^{1,1}(M)$   and $H^{1,1}(X_i)$ to  classes  in $H^{1,1}(\hM)$.  
\begin{lemma}
There is a canonical isomorphism of de Rham cohomologies
$$
    H^k(\hM)= \left\{ \begin{array}{ll}
    H^k(M)\oplus \left( \bigoplus_{i=1}^{\ell} H^k(X_i)\right) = H^k(\bM)\oplus \left( \bigoplus_{i=1}^{\ell} H^k(X_i)\right), & \mbox{if} \; k\in\{1,\ldots,2m-1\}, \\
    H^k(M), & \mbox{if} \; k\in\{0,2m\},
\end{array}\right.    
$$
where $\bM$ is the manifold with boundary defined in \eqref{bm.1}.
\label{fs.6}\end{lemma}
\begin{proof}
For $k\in\{0,2m\}$, the isomorphism is induced by the pull-back by $\bl$, so we can assume that $k\in\{1,\ldots,2m-1\}$.
By the Poincar\'e lemma applied to $\Gamma_i$-invariant forms on $\bbC^m$, we have  that $H^k(M)= H^k(\bM,\pa \bM)$.  Since $H^k(\pa \bM)=\{0\}$, we see from the long exact sequence in cohomology associated to the pair $(\bM,\pa \bM)$ that 
$$
    H^k(M)= H^k(\bM,\pa \bM)= H^k(\bM),
$$
so it remains to show that $H^k(\hM)= H^k(\bM)\oplus \left( \bigoplus_{i=1}^{\ell} H^k(X_i)\right)$.  This follows from the Mayer-Vietoris long exact sequence associated to the decomposition $\hM= \cU \cup \cV$, where $\cU= \pi^{-1}(M\setminus \{p_1,\ldots,p_{\ell}\})$ and $\cV$ is an open neighbourhood of $\pi^{-1}(\{p_1,\ldots,p_{\ell}\})$ with $\cU\cap \cV$ homotopic to the disjoint union 
$$
   \bigsqcup_{i=1}^{\ell} \lrp{S^{2m-1}/\Gamma_i}.
$$
\end{proof}

\begin{remark}%\footnote{VA: Should we  rather say wedge product?}
Since $H^k(\cU)=H^k_c(\cU)$ and $H^k(X_i)=H^k_c(X_i)$ for $k\in \{1,\ldots,2m-1\}$ and $i\in \{1,\ldots,\ell\}$, notice that in terms of the decomposition of Lemma~\ref{fs.6}, the cup product on $\hM$ takes the form 
\begin{equation}
  \left( \sum_{i=0}^{\ell} \gamma_i \right)\cdot  \left( \sum_{i=0}^{\ell} \gamma_i' \right) = \sum_{i=0}^{\ell} \gamma_i \cdot \gamma_i',
\label{fs.6b}\end{equation}
where $\gamma_0\in H^k(\bM)$, $\gamma_0'\in H^{k'}(\bM)$, $\gamma_i\in H^{k}(X_i)$ and $\gamma_i'\in H^{k'}(X_i)$ for $k,k'\in \{1,\ldots,2m-1\}$ and $i\in \{1,\ldots,\ell\}$.  When $k+k'=2m$, notice that \eqref{fs.6b} still makes sense with $\gamma_i\cdot\gamma_i'$ represented by a compactly supported closed form on $\bM\setminus \pa\bM$ when $i=0$ and on $X_i$ when $i>0$.  The resulting cup product is then obtained by pushing forward those classes to $H^{2m}(\hM)$ and taking the sum.
\label{fs.7}\end{remark}

Using the decomposition of Lemma~\ref{fs.6}, we consider the class
\begin{equation}
    \hbeta:= \beta\in H^2(\hM).
\end{equation}
Since $\beta\cdot\alpha^{m-1}=0$, notice by Remark~\ref{fs.7} that $\hbeta$ is automatically primitive with respect to the K\"ahler class
\begin{equation}
    2\pi \halpha_{\epio} := [\omega_{c,\epio}]= [\omega_0]+\epio^2\sum_{i=1}^{\ell} [\omega_i].
\label{fs.9b}\end{equation}

%{\color{blue} \begin{convention} Using the decomposition of Lemma~\ref{fs.6}, and with a slight abuse of notation, in the reminder of this section we shall drop the hat uperscripts for deRham calsses in $H^2(\widehat{M})$. Furthermore, for  deRham classes $\alpha, \beta \in H^{2}(M)$,  we shall continue to denote by $\alpha, \beta$ the induced deRham class on $\widehat{M}$ via the embedding $H^2(M)\subset H^2(\hM)$; similarly, for deRham classes $\alpha_i \in H^2(X_i)$  we shall denote the same way the induced deRham classes on $\widehat{M}$ via the embedding $H^2(X_i) \subset H^2(\widehat{M})$. Thus, we shall write \[  2\pi \alpha_{\epio}:= [\omega_{c,\epio}]= [\omega_0]+ \epio^2\sum_{i=1}^{\ell} [\omega_i]  = 2\pi \alpha + 2\pi \epsilon^2 \sum_{i=1}^k\alpha_i. \] for the decomposition of the corresponding K\"ahler classe on $\widehat{M}$.  \end{convention} Using the decomposition of Lemma~\ref{fs.6},  and  since $\beta\cdot\alpha^{m-1}=0$ on $M$, we see by Remark~\ref{fs.7} that  $\beta$ is automatically primitive on $\widehat{M}$ with respect to $\alpha_{\epsilon}$. } \fr{We can remove this blue part and save space!  I will add all the hats...}

\subsection{The $b$-surgery calculus of Mazzeo-Melrose}

The $b$-surgery calculus of Mazzeo-Melrose \cite{Mazzeo-MelroseETA} will be a key tool in our gluing construction.  This will require however some small adaptations to our setting.  For this reason, we will make a quick review of the $b$-surgery calculus and prove the results that we will need.

Compared to \cite{Mazzeo-MelroseETA}, the main difference is that we will need the $b$-surgery calculus for the space $\chM$ of \eqref{fs.2} instead of the single surgery space of \cite{Mazzeo-MelroseETA}.  This requires introducing a suitable $b$-surgery double space.  We can first consider the double space
\begin{multline*}
\cM^2_{b,s}:= [ M^2\times [0,1[; \{(p_i,p_j,0)\} \; \mbox{for} \; i,j\in \{1,\ldots,\ell\}, \{p_1\}\times M\times \{0\}, \ldots, \{p_\ell\}\times M\times \{0\}, \\
 M\times \{p_1\}\times \{0\},\ldots M\times \{p_\ell\}\times \{0\}]
\end{multline*}
with blow-down map 
$$
     \cB_{b,s}^2: \cM^2_{b,s}\to M^2\times [0,1[
$$
obtained from $M^2\times [0,1[$ by blowing up the singular strata of $M\times M$ at $\epio=0$ in the sense of Melrose in an order compatible with inclusion.  This is an orbifold with corners in the sense of \cite[\S~2]{CDR2019}.  Let us denote by $H_{ij}$ the boundary hypersurface created by the blow-up of $\{(p_i,p_j)\}\times \{0\}$ for $i,j\in \{1,\ldots,\ell\}$, by $H_{i0}$ and $H_{0i}$ the boundary hypersurfaces created respectively by the blow-ups of $\{p_i\}\times M\times \{0\}$ and $M\times \{p_i\}\times \{0\}$ for $i\in \{1,\ldots,\ell\}$ and by $H_{00}$ the lift of the boundary hypersurface $M^2\times \{0\}$ of $M^2\times [0,1[$ to $\cM^2_{b,s}$.  Notice that there is a natural identification 
$$
     H_{ij}\cong [ \overline{(\bbC^m/\Gamma_i)\times (\bbC^m/\Gamma_j) }; \{0\}\times \pa \overline{\bbC^m/\Gamma_j}, \pa\overline{\bbC^m/\Gamma_i}\times \{0\}]
$$
for $i,j\in \{1,\ldots,\ell\}$, where 
$$
    \overline{(\bbC^m/\Gamma_i)\times (\bbC^m/\Gamma_j) }= \overline{\bbC^m\times \bbC^m}/(\Gamma_i\times \Gamma_j)
$$  
is the radial compactification of 
$$
  \bbC^m/\Gamma_i \times \bbC^m/\Gamma_j= (\bbC^m\times \bbC^m)/(\Gamma_i\times \Gamma_j).
$$
Similarly, there are natural identifications 
$$
    H_{i0}= \overline{\bbC^m/\Gamma_i}\times \bM \quad \mbox{and} \quad H_{0i}= \bM\times \overline{\bbC^m/\Gamma_i}
$$
for $i\in \{1,\ldots,\ell\}$.  

Now, the resolution \eqref{la.27a} induces a resolution $\bl^2: \hM^2\to M^2$ while for $i\in\{1,\ldots,\ell\}$,  the resolution \eqref{la.27b} induces a resolution
$$
    \bl_i\times \bl_j: [ \overline{X}_i\times \overline{X}_j; \pa\overline{X}_i\times \pa \overline{X}_j ] \to [\overline{\bbC^m/\Gamma_i}\times \overline{\bbC^m/\Gamma_j}; \pa(\overline{\bbC^m/\Gamma_i})\times \pa(\overline{\bbC^m}/\Gamma_j)]\cong  H_{ij}.
$$ 
These resolutions combine to give a resolution 
\begin{equation}
 \cbl^2_{b,s}: \chM^2_{b,s}\to \cM^2_{b,s}
\label{bs.1}\end{equation}
with $\chM^2_{b,s}$ a manifold with corners with boundary hypersurfaces $\hH_{ij}$ for $i,j\in\{0,1,\ldots,\ell\}$ corresponding respectively to resolutions of $H_{ij}$ via the map \eqref{bs.1}.  In particular, for $i\in\{1,\ldots,\ell\}$, there are natural identifications
\begin{equation}
\hH_{ij}= [\bX_i\times \bX_j; \pa \bX_i\times \pa\bX_j], \quad \hH_{i0}= \bX_i\times \bM, \quad \hH_{0i}=\bM\times \bX_i \quad \mbox{and} \quad \hH_{00}=\bM^2_b.
\label{bs.2}\end{equation}

To strengthen the analogy with the $b$-double space, it will be convenient to combine these boundary hypersurfaces into disjoint unions as follows,
\begin{equation}
\hH_{\lf}:= \bigsqcup_{i=1}^{\ell} \hH_{i0}, \quad \hH_{\rf}:= \bigsqcup_{i=1}^{\ell} \hH_{0i}, \quad \hH_{\fb}:= \bigsqcup_{i,j\ge 1} \hH_{ij} \quad \mbox{and} \quad
\hH_{\mf}:=\hH_{00} .
\label{bs.3}\end{equation}
As in \cite[(71)]{Mazzeo-MelroseETA}, the projections 
$$
\begin{array}{lccc}
\pr_L: & M\times M\times [0,1[ & \to & M\times [0,1[ \\
 & (q,q',\epio) & \mapsto & (q,\epio)
 \end{array}
 \quad \mbox{and} \quad 
 \begin{array}{lccc}
\pr_R: & M\times M\times [0,1[ & \to & M\times [0,1[ \\
 & (q,q',\epio) & \mapsto & (q',\epio)
 \end{array}
$$
naturally lift to maps 
$$
     \pi_{L,s}: \cM^2_{b,s}\to \cM \quad \mbox{and} \quad \pi_{R,s}: \cM^2_{b,s}\to\cM.
$$
As can be readily checked in local coordinates, these maps further lift on the resolution $\chM^2_{b,s}$ to $b$-fibrations
$$
     \widehat{\pi}_{L,s}: \chM^2_{b,s}\to \chM \quad \mbox{and} \quad  \widehat{\pi}_{R,s}: \chM^2_{b,s}\to \chM
$$
inducing the commutative diagram %\vesti{Change $\beta_{b,s}$ with $\bl_{b,s}$ in the diagram below?} 
$$
\xymatrix{
\chM \ar[d]^{\cbl} & \ar[l]_{\widehat{\pi}_{L,s}} \chM^2_{b,s}\ar[d]^{\cbl^2_{b,s}} \ar[r]^{\widehat{\pi}_{R,s}} & \chM \ar[d]^{\cbl} \\
\cM \ar[d]^{\cB} & \ar[l]_{\pi_{L,s}} \cM^2_{b,s} \ar[r]^{\pi_{R,s}} \ar[d]^{\cB_{b,s}^2} & \cM \ar[d]^{\cB} \\
M\times [0,1[ & \ar[l]_{\pr_L^2} M^2\times [0,1[  \ar[r]^{\pr_R^2} & M\times [0,1[.
}
$$

If we denote by
$$
   \pa \chM^2_{b,s}:= \hH_{\lf}\cup \hH_{\rf}\cup \hH_{\fb}\cup \hH_{\mf}
$$
the boundary of $\chM^2_{b,s}$, then notice that there is a natural identification
$$
    \chM_{b,s}^2\setminus\pa \chM^2_{b,s}= \hM^2\times ]0,1[.
$$
If $\hD$ is the diagonal in $\hM^2$, then let $\chD_{b,s}$ be the closure of $\hD\times ]0,1[$ in $\chM_{b,s}$.  Notice that $\widehat{\pi}_{L,s}$ and $\widehat{\pi}_{R,s}$ induce identifications between $\chD_{b,s}$ and $\chM$.  If 
$$
{}^{b,s}\Omega(\chM)= |\Lambda^{2m}({}^{b,s}T^*\chM)|
$$
is the bundle of $b$-surgery densities on $\chM$, then set
$$
  {}^{b,s}\Omega_R(\chM):= (\widehat{\pi}_{R,s})^*({}^{b,s}\Omega(\chM)).
$$
Similarly, if $E_1$ and $E_2$ are vector bundles on $\chM$, consider the vector bundle 
$$
      \Hom_{b,s}(E_1,E_2):= \widehat{\pi}_{L,s}^*E_2 \otimes \widehat{\pi}^*_{R,s}(E_1^*).
$$
Following \cite[(76)]{Mazzeo-MelroseETA}, we will consider on $\chM^2_{b,s}$ the space of conormal distributions 
$$
  I^{q}(\chM^2_{b,s};\chD_{b,s}; \Hom_{b,s}(E_1,E_2)\otimes {}^{b,s}\Omega_R(\chM))
$$
of order $q$ with respect to $\chD_{b,s}$ in the sense of \cite[Definition~18.2.6]{Hormander3}.  Given an index family $\cE$ on $\chM^2_{b,s}$, one can then define the associated space of pseudodifferential $b$-surgery operators of order $q$ acting from sections of $E_1$ to sections of $E_2$ by
$$
  \Psi^{q,\cE}_{b,s}(\chM;E_1,E_2):= \cA^{\cE}_{\phg}(\chM^2_{b,s})\otimes_{\CI(\chM^2_{b,s})}  I^{q-\frac14}(\chM^2_{b,s};\chD_{b,s}; \Hom_{b,s}(E_1,E_2)\otimes {}^{b,s}\Omega_R(\chM)).
$$

Using the principal symbol for conormal distributions, there is a corresponding notion of principal symbol for pseudodifferential $b$-surgery operators.  If $\widehat{\cQ}: {}^{b,s}T^*\chM\to \chM$ is the bundle map, then one can define the space 
$$
  \cS^q({}^{b,s}T^*\chM; \widehat{\cQ}^*\Hom(E_1,E_2))
$$
proceeding as in \eqref{b.9} using local trivializations of $E_1,E_2$ and ${}^{b,s}T^*\chM$.  If $\iota: \chD_{b,s}\to \chM^2_{b,s}$ is the natural inclusion, then there is a principal symbol map
$$
  {}^{b,s}\sigma_q: \Psi^{q,\cE}_{b,s}(\chM;E_1,E_2)\to \cA_{\phg}^{\iota^\#\cE}(\chM)\otimes_{\CI(\chM)} \cS^{[q]}({}^{b,s}T^*\chM; \widehat{\cQ}^*\Hom(E_1,E_2)),
$$
where $\iota^\#\cE$ is the index family of $\cD_{b,s}\cong \chM$ such that 
$$
   \iota^\#\cE(\hH_i)= \cE(\hH_{ii}) \quad \mbox{for} \; i \in \{0,1,\ldots,\ell\},
$$
and where 
$$
 \cS^{[q]}({}^{b,s}T^*\chM; \widehat{\cQ}^*\Hom(E_1,E_2)):= \cS^q({}^{b,s}T^*\chM; \widehat{\cQ}^*\Hom(E_1,E_2))/ \cS^{q-1}({}^{b,s}T^*\chM; \widehat{\cQ}^*\Hom(E_1,E_2)).
$$

Using the pushforward theorem of Melrose \cite[Theorem~5]{Melrose1992}, we can describe the action of pseudodifferential $b$-surgery operators on polyhomogeneous sections.  To ease the comparison with Proposition~\ref{b.5}, we will only consider index families on $\chM^2_{b,s}$ of the form 
$$
    \cE=(\cE(\lf), \cE(\rf), \cE(\fb), \cE(\mf))
$$ 
with $\cE(\bullet)$ denoting an index set for boundary hypersurface $\hH_{\bullet}$ described in \eqref{bs.3}.  Similarly, we will treat $\chM$ as a manifold with corners with two (possibly disconnected) boundary surfaces, namely $\hH_+$ and $\hH_0$.  

\begin{proposition}\label{bs.4}
Let $\cE=(\cE(\lf), \cE(\rf), \cE(\fb), \cE(\mf))$ and $\cF= (\cF(\hH_+),\cF(\hH_0))$ be index families for $\chM^2_{b,s}$ and $\chM$ respectively.  Then for $A\in \Psi^{q,\cE}_{b,s}(\chM;E_1,E_2)$ and $\zeta\in \cA^{\cF}_{\phg}(\chM;E_1)$,
$$
         A\zeta:= (\widehat{\pi}_{L,s})_*(A\widehat{\pi}^*_{R,s}\zeta)
$$ 
is well-defined as an element of $\cA^{\cG}_{\phg}(\chM;E_2)$ for the index family $\cG$ given by
$$
  \cG(\hH_0)=(\cE(\mf)+\cF(\hH_0)) \overline{\cup} (\cE(\rf)+ \cF(\hH_+))) \quad \mbox{and} \quad \cG(\hH_+)=(\cE(\lf)+\cF(\hH_0)) \overline{\cup} (\cE(\fb)+ \cF(\hH_+)). 
$$
\end{proposition}
\begin{proof}
It suffices to use the pullback and pushforward theorems of Melrose and proceed as in \cite[\S~8.2]{Mazzeo-MelroseETA}.
\end{proof}

We can also obtain the following composition result.
\begin{proposition}\label{bs.5}
If $\cE=(\cE(\lf), \cE(\rf), \cE(\fb), \cE(\mf))$ and $\cF=(\cF(\lf), \cF(\rf), \cF(\fb), \cF(\mf))$ are indicial families for $\chM^2_{b,s}$, then 
$$
    \Psi^{q,\cE}_{b,s}(\chM;E_2,E_3)\circ \Psi^{q',\cF}_{b,s}(\chM;E_1,E_2) \subset \Psi^{q+q',\cG}_{b,s}(\chM;E_1,E_3)
$$
where $\cG$ is the index family given by
$$
\begin{aligned}
\cG(\fb)= (\cE(\fb)+\cF(\fb))\overline{\cup} (\cE(\lf)+ \cF(\rf)), &\quad
\cG(\lf)= (\cE(\fb)+\cF(\lf))\overline{\cup}(\cE(\lf)+ \cF(\mf)), \\
\cG(\rf)= (\cE(\rf)+ \cF(\fb))\overline{\cup} (\cE(\mf)+ \cF(\rf)), & \quad  
\cG(\mf)= (\cE(\mf)+ \cF(\mf))\overline{\cup}(\cE(\rf)+ \cF(\lf)).
\end{aligned}
$$
\end{proposition}
\begin{proof}
We can proceed as in \cite[\~8.3]{Mazzeo-MelroseETA} by introducing a suitable $b$-surgery triple space $\chM^3_{b,s}$ corresponding to the natural resolution of the $b$-surgery triple space $\cM^3_{b,s}$ of $\cM$ obtained from $M^3\times [0,1[$ by blowing up in the sense of Melrose the singular strata of $M^3\times \{0\}$ in an order compatible with inclusion.  We leave the details to the reader.
\end{proof}

As in \cite[\S~4]{Mazzeo-MelroseETA}, one can check via Schur's test and the pushforward theorem as well as H\"ormander $L^2$-trick that  a pseudodifferential $b$-surgery operator $P\in \Psi^{q,\cE}_{b,s}(\chM;E_1,E_2)$ induces a uniformly  bounded family of operators
$$
   P: L^{2,q+j}_{b,\epio}(\hM;E_1)\to L^{2,j}_{b,\epio}(\hM;E_2)
$$
in $\epio\in ]0,\epio_0[$ for $\epio_0\in]0,1[$ fixed
provided $\Re(\cE(\lf))>0$, $\Re(\cE(\rf))>0$, $\Re(\cE(\fb)\setminus\{0\})>0$ and $\Re(\cE(\mf)\setminus\{0\})>0$.  Similarly, since  pseudodifferential $b$-surgery operators are stable with respect to taking $\epio\frac{\pa}{\pa \epio}$ derivatives, we see that for $\epio_0\in ]0,1[$ fixed, an operator $P\in \Psi^{q,\cE}_{b,s}(\chM;E_1,E_2)$ also induces a bounded operator
$$
   P: \cC^q(]0,\epio_0[;L^{2,q+j}_{b,s}(\chM;E_1))\to \cC^q(]0,\epio_0[;L^{2,j}_{b,s}(\chM;E_2))
$$
provided $\Re(\cE(\lf))>0$, $\Re(\cE(\rf))>0$, $\Re(\cE(\fb)\setminus\{0\})>0$ and $\Re(\cE(\mf)\setminus\{0\})>0$.

Notice from \eqref{bs.2} and \eqref{bs.3} that there are natural identifications 
$$
  \hH_{\mf}= \bM^2_b \quad \mbox{and} \quad \hH_{\fb}= (\hH_+)^2_b= [\hH^2_+; (\pa \hH_+)^2].
$$
In fact, for an index family $\cE$ such that $\Re(\cE(\fb)\setminus\{0\})>0$ and $\Re(\cE(\mf)\setminus\{0\})>0$, restrictions to $\hH_{\fb}$ and $\hH_{\mf}$ induce normal homomorphisms
$$
N_{\fb}: \Psi^{q,\cE}_{b,s}(\chM;E_1,E_2)\to \Psi_b^{q, \cE_{\fb}}(\hH_+;E_1,E_2) \quad \mbox{and} \quad N_{\mf}: \Psi^{q,\cE}_{b,s}(\chM;E_1,E_2)\to \Psi_b^{q, \cE_{\mf}}(\bM;E_1,E_2), 
$$
where $\cE_{\fb}$ and $\cE_{\mf}$ are the index families for $(\hH_+)^2_b$ and $\bM^2_b$ given by
$$
    \cE_{\fb}(\lf)=\cE_{\mf}(\lf)=\cE(\lf), \quad \cE_{\fb}(\rf)=\cE_{\mf}(\rf)=\cE(\rf), \quad \cE_{\fb}(\fb)= \cE(\mf), \quad \mbox{and} \quad \cE_{\mf}(\fb)=\cE(\fb).
$$
These normal homomorphisms are compatible with composition in the sense that
$$
   N_{\fb}(A\circ B)= N_{\fb}(A)\circ N_{\fb}(B) \quad \mbox{and} \quad N_{\mf}(A\circ B)= N_{\mf}(A)\circ N_{\mf}(B)
$$
whenever both sides are defined.    Using these normal homomorphisms, we can obtain the following invertibility result.
\begin{proposition}\label{bs.6}
Let $P\in \Psi^{q,\cE}_{b,s}(\chM;E_1,E_2)$ be an elliptic pseudodifferential $b$-surgery operator for a real index family $\cE$ such that 
$$
\cE(\lf)> \tau, \quad \cE(\rf)> \tau, \quad \cE(\fb)\ge 0 \quad \mbox{and} \quad \cE(\mf)\ge 0
$$
for some $\tau >0$.  Suppose that $N_{\fb}(P)$ and $N_{\mf}(P)$ are invertible with inverses in $\Psi^{-q,\cF}_{b}(\hH_+;E_2,E_1)$ and $\Psi^{-q,\cH}_{b,s}(\bM;E_2,E_1)$ for real index families $\cF$ and $\cH$ such that 
$$
   \cF(\lf)\ge \tau, \quad \cF(\rf)\ge\tau, \quad, \cF(\fb)\ge 0, \quad  \cH(\lf)\ge \tau, \quad \cH(\rf)\ge\tau \quad \mbox{and} \quad \cH(\fb)\ge 0.
$$
Then there exists $\epio_0>0$ such that for $\epio<\epio_0$, $P$ admits an inverse $G \in \Psi^{-q,\cG}_{b,s}(\chM;E_2,E_1)$ for some real index family $\cG$ such that
$$
      \cG(\lf)\ge\tau, \quad \cG(\rf)\ge\tau, \quad \cG(\fb)\ge 0 \quad \mbox{and} \quad \cG(\mf)\ge 0.
$$
\end{proposition}
\begin{proof}
By ellipticity and the invertibility of $N_{\fb}(P)$ and $N_{\mf}(P)$, there exists a parametrix 
$$
Q_0\in \Psi^{-q,\cQ_0}_{b,s}(\chM;E_2,E_1)
$$ 
such that 
$$
    N_{\fb}(Q_0)=N_{\fb}(P)^{-1}, \quad N_{\mf}(Q_0)= N_{\mf}(P)^{-1}
$$
with 
$$
      PQ_0=\Id-R_2 \quad \mbox{and} \quad Q_0P=\Id-R_1
$$
for operators $R_i\in \Psi^{-\infty,\cR_i}_{b,s}(\chM;E_i)$, where the $\cQ_0$, $\cR_1$ and $\cR_2$ are real index families such that
$$
\cQ_0(\lf)\ge \tau, \quad \cQ_0(\rf)\ge \tau, \quad \cQ(\fb)\ge 0, \quad \cQ(\mf)\ge 0
$$
and
$$
\cR_i(\lf)\ge \tau, \quad \cR_i(\rf)\ge \tau, \quad \cR_i(\fb)> 0, \quad \cR(\mf)> 0.  
$$
Using Proposition~\ref{bs.5}, we see in particular that there exists $\delta>0$ such that for $j\in\bbN$ with $j\ge2$, $R_i^j\in \Psi^{-\infty,\cR_{i,j}}_{b,s}(\chM;E_i)$ for a real index family $\cR_{i,j}$ such that
$$
   \cR_{i,j}(\lf)>\tau+ j\delta, \quad \cR_{i,j}(\rf)>\tau+j\delta, \quad \cR_{i,j}(\fb)>j\delta \quad \mbox{and} \quad \cR_{i,j}(\mf)>j\delta.
$$
This means that there exists an asymptotic sum 
$$
   S_i\sim \sum_{j=1}^{\infty} R_{i}^j,
$$
with $S_i\in \Psi^{-\infty,\cS_i}_{b,s}(\chM;E_i)$ for a real index family $\cS_i$ such that
$$
\cS_i(\lf)\ge\tau, \quad \cS_i(\rf)\ge \tau, \quad \cS_{i}(\fb)>0 \quad \mbox{and} \quad \cS_i(\mf)>0.
$$
In particular,
$$
   (\Id+S_1)(\Id-R_1)= \Id-T_1 \quad \mbox{and} \quad (\Id-R_2)(\Id+S_2)=\Id -T_2 
$$
for $T_i\in \dot{\Psi}^{-\infty}(\chM;E_i):= \Psi^{-\infty,\cH}_{b,s}(\chM;E_i)$ for $\cH$ the index family corresponding to the empty set for each boundary hypersurface of $\chM$.  Setting
$$
  Q_2= Q_0(\Id+S_2) \quad \mbox{and} \quad Q_1= (\Id+S_1)Q_0,
$$
we see that 
$$
      PQ_2=\Id-T_2 \quad \mbox{and} \quad Q_1P=\Id-T_1.
$$
Furthermore, by Proposition~\ref{bs.5}, $Q_i\in\Psi^{-q,\cQ_i}_{b,s}(\chM;E_2,E_1)$ with real index families $\cQ_i$ such that
$$
    \cQ_i(\lf)\ge \tau, \quad \cQ_i(\rf)\ge \tau, \quad \cQ(\fb)\ge 0 \quad \mbox{and} \quad \cQ(\mf)\ge 0.
$$

Since $T_i\in \dot{\Psi}^{-\infty}(\chM;E_i)$, its norm as an operator acting on $L^2_{b,s}(\chM;E_i)$ can be taken as small as we want by restricting $\epio$ to be smaller or equal to $\epio_0$ for some small $\epio_0>0$.  Choosing such an $\epio_0$, we can thus assume that $\Id-T_i$ is invertible, in which case its inverse is of the form $\Id+U_i$ for some $U_i\in \dot{\Psi}^{-\infty}(\chM;E_i)$.  Hence, finally, we find that
$$
    P^{-1}= Q_2(\Id+U_2)= (\Id+U_1)Q_1
$$
with $P^{-1}\in \Psi^{-q,\cG}_{b,s}(\chM;E_2,E_1)$ with real index family $\cG$ as claimed in the statement of the proposition.
\end{proof}

This can be used to  give a pseudodifferential characterization of the Green operator of the Laplace-Beltrami operator $\Delta_{c,\epio,\phi}$  of a K\"ahler conical surgery metric $g_{c,\epio,\phi}$ with K\"ahler form
\begin{equation}
\omega_{c,\epio,\phi}:= \omega_{c,\epio}+ dd^c\phi
\label{fs.10}\end{equation}
for a potential $\phi\in \cA_{\phg}(\chM)\cap L^{\infty}(\chM)$ with $dd^c\phi$ bounded as a section of $\Lambda^{2}({}^{c,\epio}T^*\chM)$.  Let 
$$
    \Pi_{c,\epio,\phi}: L^2_{c,\epio}(\hM)\to L^2_{c,\epio}(\chM)
$$     
be the orthogonal $L^2$-projection onto constant functions with respect to the metric $g_{c,\epio,\phi}$, where  $L^2_{c,\epio}(\hM)$ is the space of $L^2$-functions on $\hM$ with respect to the metric $g_{c,\epio,\phi}$.

\begin{corollary}\label{MM.1}
 The Green operator $\bbG_{c,\epio,\phi}$ such that
 \begin{equation}
  \Delta_{c,\epio,\phi}\bbG_{c,\epio,\phi}= \bbG_{c,\epio,\phi}  \Delta_{c,\epio,\phi}= \Id-\Pi_{c,\epio,\phi}
 \label{MM.1a}\end{equation}
 is an element of $\Psi^{-2,\cG}_{b,s}(\chM)$ for an index family $\cG$ such that
 $$
   \cG|_{\lf}\ge 0, \quad  \cG_{\rf}\ge 2m, \quad  \cG|_{\fb}\ge 2\quad \mbox{and} \quad \cG|_{\mf}\ge 0.
 $$
\end{corollary}
\begin{proof}
To use the previous proposition, it is convenient to replace \eqref{MM.1a} by
$$
    P_{b,s}G_{m,s}=G_{m,s}P_{b,s}=\Id,
$$
where 
$$
  P_{b,s}:= \rho_+^{m+1}(\Delta_{c,\epio,\phi}+ \Pi_{c,\epio,\phi})\rho_{+}^{-(m-1)}= \rho_+^{m-1}(\rho_+^2(\Delta_{c,\epio,\phi}+ \Pi_{c,\epio,\phi}))\rho_{+}^{-(m-1)}
$$
and
$$
  G_{m,s}:=\rho_+^{m-1}(\bbG_{c,\epio,\phi}+ \Pi_{c,\epio,\phi})\rho_+^{-m-1}.
$$
Now, $P_{b,s}\in \Psi^{2,\cE}_{b,s}(\chM)$ for a real indicial family $\cE$ such that 
$$
  \cE(\lf)\ge m+1, \quad \cE(\rf)\ge m+1, \quad \cE(\fb)\ge 0 \quad \mbox{and} \quad \cE(\mf)\ge 0.
$$
Moreover, 
$$
  N_{\mf}(P_{b,s})= P_b+ r^{m+1}\Pi_{0}r^{-(m-1)},
$$
where $P_b$ is the $b$-operator of \eqref{green.2b}.  By the proof of  Lemma~\ref{green.2}, its inverse is 
$$
  G_m+ r^{m-1}\Pi_{0}r^{-m-1}\in \Psi^{-2,\cH}_b(\bM)
$$
for a real index family $\cH$ such that
$$
  \cH(\lf)\ge m-1, \quad \cH(\rf)\ge m-1 \quad \mbox{and} \quad \cH(\fb)\ge 0.
$$
Similarly, 
$$
N_{\fb}(P_{b,s}) = N_{\fb}(\rho_+^{m-1}(\rho_+^2\Delta_{c,\epio,\phi})\rho_+^{-(m-1)}) 
 = N_{\fb}(\rho_0^{-(m-1)} (\rho_0^{-2} \epio^2 \Delta_{c,\epio,\phi})\rho_0^{m-1} ) 
$$
with
$$
 N_{\fb}(P_{b,s})|_{\hH_{ii}}= \rho_0^{-(m-1)}(\rho_0^{-2}\Delta_i)\rho_0^{m-1},
$$
where $\Delta_i$ is the Laplace-Beltrami operator of the $\ALE$-metric $g_i$ on $X_i$.  By Lemma~\ref{la.28a},  the $b$-operator $\rho_0^{-(m-1)}(\rho_0^{-2}\Delta_i)\rho_0^{m-1}$  is invertible when acting on $b$-Sobolev spaces with  inverse in $\Psi^{-2,\cG_i}_{b}(\bX_i)$ for some real index family $\cG_i$ such that
$$
     \cG_i(\lf)\ge m-1, \quad \cG_i(\rf)\ge m-1 \quad \mbox{and} \quad \cG_i(\fb)\ge 0.
$$
This means that we can apply Proposition~\ref{bs.6} to $P_{b,s}$ taking $\tau=m-1$.  The result then follows from
$$
  \bbG_{c,\epio,\phi}= \rho_{+}^{-(m-1)} P_{b,s}^{-1}\rho_+^{m+1}- \Pi_{c,\epio,\phi}.
$$
\end{proof}

This allows us to obtain  a regularity result closely related to the one of \cite[\S~5.2.1]{Benabida}.  To formulate it, it will be convenient to use the notation $(\epio)_*(f\omega^{[m]}_{c,\epio,\phi})$ to denote the pushforward by $\epio: \cM\to [0,1[$ of the volume form $f\omega_{c,\epio,\phi}^{[m]}$ for $f$ a function, that is, $(\epio)_*(f\omega_{c,\epio,\phi}^{[m]})$ is the function on $[0,1[$ obtained by integrating $f\omega^{[m]}_{c,\epio,\phi}$ along the level sets of $\epio$ (when the integral is defined).  In particular, for $x\in ]0,1[$, we have that
$$
       (\epio)_*(f\omega_{c,\epio,\phi}^{[m]})(x)= \int_{\epio^{-1}(x)} f\omega_{c,\epio,\phi}^{[m]}.
$$
\begin{corollary}
Let $\phi\in \left( \cA_{\phg}(\chM)\cap L^{\infty}(\chM) \right)$ be such that $\phi$ vanishes at $\hH_0$ and $\hH_+$ with $dd^c\phi$ vanishing on $\hH_0$ and $\hH_+$ as  a section of $\Lambda^2(T^*\chM)$.  Let $f\in\cA_{\phg}(\chM)$ be such that $(\epio)_*(f\omega_{c,\epio,\phi}^{[m]})=0$ with $\rho_+^2f$ bounded and vanishing  on $\hH_+$.  Then the unique function $u$ such that
$$
      \Delta_{c,\epio,\phi} u=f \quad \mbox{and} \quad (\epio)_*u=0
$$
is in $\cA_{\phg}(\chM)\cap L^{\infty}(\chM)$.  If instead $\rho_+^2f$ does not necessarily vanish on $\hH_+$, but is such that   $\rho^2_+f=\cO(\epio^{\kappa}\rho_0^{\delta})$ for some $\kappa\ge 0$ and  $\delta\in ]0,2m-2[$, then $u$ is polyhomogeneous and $u=\cO(\epio^{\kappa}\rho_0^{\delta})$, while if $\rho^2_+f=o(\epio^{\kappa})$ for some $\kappa>0$, then $u=o(\epio^{\kappa})$.
\label{fs.11}\end{corollary}
\begin{proof}
 The results follows from the fact that $u=\bbG_{c,\epio,\phi} f$,  Corollary~\ref{MM.1}  and Proposition~\ref{bs.4}.
\end{proof}

We will also need later on the following more robust version of Proposition~\ref{bs.6}.

\begin{proposition}\label{rob.1}
Let $P\in \Psi^{q,\cE}_{b,s}(\chM;E_1,E_2)$ be an elliptic pseudodifferential $b$-surgery operator for a real index family $\cE$ such that 
$$
\cE(\lf)> \tau, \quad \cE(\rf)> \tau, \quad \cE(\fb)\ge 0 \quad \mbox{and} \quad \cE(\mf)\ge 0
$$
for some $\tau >0$.  For $i\in \{1,2\}$, let $\Pi_i\in\Psi^{0,\cP_i}_{b,s}(\chM;E_i)$ be a finite rank projection (for each fixed $\epio$) for some  real index family $\cP_i$ such that
$$
   \cP_i(\lf)> \tau, \quad \cP_i(\rf)> \tau, \quad \cP_i(\fb)\ge 0, \quad \cP_i(\mf)\ge 0.
$$
Suppose that there exist  $Q_{\fb}\in\Psi^{-q\cF}_{b}(\hH_+;E_2,E_1)$ and $Q_{\mf}\in\Psi^{-q,\cH}_{b,s}(\bM;E_2,E_1)$ such that
\begin{multline*}
       Q_{\fb}N_{\fb}(P)= \Id -N_{\fb}(\Pi_1), \quad N_{\fb}(P)Q_{\fb}= \Id -N_{\fb}(\Pi_2), \quad  Q_{\mf}N_{\mf}(P)= \Id -N_{\mf}(\Pi_1), \\ 
       \mbox{and} \quad  N_{\mf}(P)Q_{\mf}= \Id -N_{\mf}(\Pi_2),
\end{multline*}
where $\cF$ and $\cH$ are real index families such that 
$$
   \cF(\lf)> \tau, \quad \cF(\rf)>\tau, \quad, \cF(\fb)\ge 0, \quad  \cH(\lf)> \tau, \quad \cH(\rf)>\tau \quad \mbox{and} \quad \cH(\fb)\ge 0.
$$
Then there exists $\epio_0>0$ such that for $\epio<\epio_0$, there exists  $G \in \Psi^{-q,\cG}_{b,s}(\chM;E_2,E_1)$ such that
\begin{equation}
      G(\Id-\Pi_2)P(\Id-\Pi_1)= \Id-\Pi_1 \quad \mbox{and} \quad (\Id-\Pi_2)P(\Id-\Pi_1)G=\Id-\Pi_2
\label{rob.2}\end{equation}
for some real index family $\cG$ such that
$$
      \cG(\lf)>\tau, \quad \cG(\rf)>\tau, \quad \cG(\fb)\ge 0,  \quad \mbox{and} \quad \cG(\mf)\ge 0.
$$
\end{proposition}
\begin{proof}
By ellipticity and our assumptions on $N_{\fb}(P)$ and $N_{\mf}(P)$, there exists a parametrix  \\ 
$Q_0\in \Psi^{-q,\cQ_0}_{b,s}(\chM;E_2,E_1)$ such that 
$$
    N_{\fb}(Q_0)=Q_{\fb}, \quad N_{\mf}(Q_0)= Q_{\mf}
$$
with 
$$
      (\Id-\Pi_2)P(\Id-\Pi_1)Q_0= (\Id-\Pi_2)(\Id-R_2) \quad \mbox{and} \quad Q_0(\Id-\Pi_2)P(\Id-\Pi_1)=(\Id-R_1)(\Id-\Pi_1)
$$
for operators $R_i\in \Psi^{-\infty,\cR_i}_{b,s}(\chM;E_i)$, where $\cQ_0$, $\cR_1$ and $\cR_2$ are real index families such that
$$
\cQ_0(\lf)> \tau, \quad \cQ_0(\rf)> \tau, \quad \cQ(\fb)\ge 0, \quad \cQ(\mf)\ge 0
$$
and
$$
\cR_i(\lf)> \tau, \quad \cR_i(\rf)> \tau, \quad \cR_i(\fb)> 0, \quad \cR(\mf)> 0.  
$$
Proceeding as in the proof of Proposition~\ref{bs.6}, we see that there exists 
 $S_i\in \Psi^{-\infty,\cS_i}_{b,s}(\chM;E_i)$ for a real index family $\cS_i$ with 
$$
\cS_i(\lf)>\tau, \quad \cS(\rf)> \tau, \quad \cS_{i}(\fb)>0 \quad \mbox{and} \quad \cS_i(\mf)>0
$$
such that
$$
   (\Id+S_1)(\Id-R_1)= \Id-T_1 \quad \mbox{and} \quad (\Id-R_2) (\Id+S_2)=\Id -T_2 
$$
for $T_i\in \dot{\Psi}^{-\infty}(\chM;E_i)$.  Setting
$$
  Q_2= Q_0(\Id+S_2) \quad \mbox{and} \quad Q_1= (\Id+S_1)Q_0,
$$
we see that 
$$
      (\Id-\Pi_2)P(\Id-\Pi_1)Q_2=(\Id-\Pi_2)(\Id-T_2) \quad \mbox{and} \quad Q_1(\Id-\Pi_2)P(\Id-\Pi_1)=(\Id-T_1)(\Id-\Pi_1).
$$
Furthermore, by Proposition~\ref{bs.5}, $Q_i\in\Psi^{-k,\cQ_i}_{b,s}(\chM;E_2,E_1)$ with real index families $\cQ_i$ such that
$$
    \cQ_i(\lf)> \tau, \quad \cQ_i(\rf)> \tau, \quad \cQ(\fb)\ge 0 \quad \mbox{and} \quad \cQ(\mf)\ge 0.
$$

Since $T_i\in \dot{\Psi}^{-\infty}(\chM;E_i)$, its norm as an operator acting on $L^2_{b,s}(\chM;E_i)$ can be taken as small as we want by restricting $\epio$ to be smaller or equal to $\epio_0$ for some small $\epio_0>0$.  Choosing such an $\epio_0$, we can thus assume that $\Id-T_i$ is invertible, in which case its inverse is of the form $\Id+U_i$ for some $U_i\in \dot{\Psi}^{-\infty}(\chM;E_i)$.  Hence, finally, we find that \eqref{rob.2} holds with 
$$
    G= Q_2(\Id+U_2)= (\Id+U_2)Q_1 \in \Psi^{-q,\cG}_{b,s}(\chM;E_2,E_1)
$$
with real index family $\cG$ as claimed in the statement of the proposition.
\end{proof}

\subsection{Construction of a formal solution}
We want to find $\phi\in \cA_{\phg}(\chM)\cap L^{\infty}(\cM)$ such that 
$\mathring{\Sc}^{\hbeta}(\omega_{c,\epio,\phi})$ vanishes rapidly at $\hH_0$ and $\hH_+$ for the K\"ahler form $\omega_{c,\epio,\phi}$ of \eqref{fs.10}.  
To do this, we first need to check that $\mathring{\Sc}^{\hbeta}(\omega_{c,\epio,\phi})$ is polyhomogeneous whenever $\phi$ is.   
Using Corollary~\ref{fs.11}, we can first obtain the following regularity result for the harmonic representative $\theta_{c,\epio,\phi}$ of $2\pi\hbeta$ with respect to $\omega_{c,\epio,\phi}$.
\begin{lemma}
The harmonic representative $\theta_{c,\epio,\phi}$ of $2\pi\hbeta$ is polyhomogeneous with $\rho_+^2\theta_{c,\epio,\phi}$ bounded as a section of $\Lambda^2({}^{c,\epio}T^*\chM)$.  
\label{fs.12}\end{lemma}
\begin{proof}
The class  $\hbeta$ is represented by a closed $2$-form $\widetilde{\eta}$ supported away from $\hH_+$ 
with $\widetilde{\eta}|_{\hH_0}\in \CI_c(\hH_0\setminus \pa \hH_0; \Lambda^2(T^*\hH_0))$ representing the class $2\pi\beta\in H^2_c(\hH_0\setminus \pa \hH_0)= H^2(\hH_0)$.  In particular, $\rho_+^2\widetilde{\eta}$ is clearly polyhomogeneous and bounded as a section of $\Lambda^2({}^{c,\epio}T^*\chM)$.  By the $dd^c$-lemma, there exists a unique function $\widetilde{u}$ such that
$$
      dd^c\widetilde{u}= \widetilde{\eta}-\theta_{c,\epio,\phi} \quad \mbox{with} \quad (\epio)_*(\widetilde{u}\omega^{[m]}_{c,\epio,\phi})=0.
$$
Since $\hbeta$ is primitive, this implies that 
$$
    \Delta_{c,\epio,\phi} \widetilde{u}= -\Lambda_{c,\epio,\phi} dd^c\widetilde{u}= - \Lambda_{c,\epio,\phi} \widetilde{\eta},
$$ 
where $\Lambda_{c,\epio,\phi}$ is the adjoint of the operator $L_{c,\epio,\phi}$ defined by 
$$
    L_{c,\epio,\phi}(\psi)= \psi\wedge \omega_{c,\epio,\phi}
$$
for $\psi$ a form.  Now, $\rho_+^2\Lambda_{c,\epio,\phi} \widetilde{\eta}$ is bounded, polyhomogeneous and vanishes at $\hH_+$. 
By Corollary~\ref{fs.11}, it follows that $\widetilde{u}$ is bounded and polyhomogeneous.

This implies that
$$
   \theta_{c,\epio,\phi}= \widetilde{\eta}- dd^c\widetilde{u}
$$
is polyhomogeneous with $\rho^2_+\theta_{c,\epio,\phi}$ bounded as a section of $\Lambda^2({}^{c,\epio}T^*\chM)$.
\end{proof}
We will need also the following estimate.
\begin{lemma}\label{fs.12a}
Let $\phi_1, \phi_2 \in (\cA_{\phg}(\chM)\cap L^{\infty}(\chM))$ be two potentials vanishing at $\hH_+$ and $\hH_0$ such that $dd^c\phi_1$ and $dd^c\phi_2$ vanish on $\hH_+$ and $\hH_-$.  If $dd^c(\phi_1-\phi_2)=o(\epio^\kappa)$ as a section of $\Lambda^{2}({}^{c,\epio}T^*\chM)$ near $\hH_+$ and $\hH_0$ for some $\kappa>0$, then
$$
     \rho_+^2(\theta_{c,\epio,\phi_1}-\theta_{c,\epio,\phi_2})= o(\epio^{\kappa})
$$
near $\hH_+$ and $\hH_0$.  If instead $dd^c(\phi_1-\phi_2)$ does not necessarily vanish at $\hH_+$ but is such that $dd^c(\phi_1-\phi_2)= \cO(\epio^{\kappa}\rho_0^\delta)$ for some $\kappa\ge 0$ and $\delta\in]0,2m-2[$ as a section of $\Lambda^{2}({}^{c,\epio}T^*\chM)$, then 
$$
   \rho_+^2(\theta_{c,\epio,\phi_1}-\theta_{c,\epio,\phi_2} )= \cO(\epio^{\kappa}\rho_0^{\delta}).
$$
\end{lemma}
\begin{proof}
By the $dd^c$-lemma, there exists a unique function $u$ such that
$$
    dd^cu= \theta_{c,\epio,\phi_1}-\theta_{c,\epio,\phi_2} \quad \mbox{and} \quad (\epio)_* \lrp{u \omega^{[m]}_{c,\epio,\phi_2}}.
$$
Since $\hbeta$ is primitive, this implies that
$$
  \Delta_{\omega_{c,\epio,\phi_2}}u= -\Lambda_{c,\epio,\phi_2}dd^c u= -\Lambda_{c,\epio,\phi_2}\theta_{c,\epio,\phi_1}
$$
with $\rho_+^2\Lambda_{c,\epio,\phi_2}\theta_{c,\epio,\phi_1}\in \cA_{\phg}(\chM)\cap L^{\infty}(\chM)$. If $dd^c(\phi_2-\phi_1)=o(\epio^\kappa)$ as a section of $\Lambda^{2}({}^{c,\epio}T^*\chM)$, we see on the other hand that 
$$
    \Lambda_{c,\epio,\phi_2}-\Lambda_{c,\epio,\phi_1}=o(\epio^\kappa)
$$ 
as a section of $(\Lambda^2({}^{c,\epio}T^*\chM))^{*}$, so $\rho^2_+\Lambda_{c,\epio,\phi_2}\theta_{c,\epio,\phi_1}=o(\epio^{\kappa})$ near $\hH_0$ and $\hH_1$.  By Corollary~\ref{fs.11}, this means that $u$ is polyhomogeneous and bounded with $u=o(\epio^{\kappa})$, so that
$$
   \rho^2_+\lrp{\theta_{c,\epio,\phi_2}-\theta_{c,\epio,\phi_2}}= -\rho^2_+(dd^cu)=o(\epio^{\kappa})
$$
as a section of $\Lambda^2({}^{c,\epio}T^*\chM)$.  If instead $dd^c(\phi_1-\phi_2)=  \cO(\epio^{\kappa}\rho_0^\delta)$ for some $\kappa\ge 0$ and $\delta\in]0,2m-2[$, then $\rho^2_+\Lambda_{c,\epio,\phi_2}\theta_{c,\epio,\phi_1}=\cO(\epio^{\kappa}\rho_0^{\delta})$ as a section of $\Lambda^2({}^{c,\epio}T^*\chM)$, so Corollary~\ref{fs.11} shows that $u=\cO(\epio^{\kappa}\rho_0^{\delta})$, so that
$$
   \rho^2_+\lrp{\theta_{c,\epio,\phi_2}-\theta_{c,\epio,\phi_2}}= -\rho^2_+(dd^cu)=\cO(\epio^{\kappa}\rho_0^{\delta})
$$
as a section of $\Lambda^2({}^{c,\epio}T^*\chM)$.
\end{proof}

For $\epio>0$ fixed, let $\mu_{c,\epio,\phi}$ be the moment map of the $\bbT$-action on $\chM$ with respect to $\omega_{c,\epio,\phi}$ with the normalization $(\epio)_*(\mu_{c,\epio,\phi}\omega^{[m]}_{c,\epio,\phi})=0$.  This induces an application $\mu_{c,\cdot,\phi}: \chM\to \mathfrak{t}^*$ when we allow $\epio$ to vary.
\begin{lemma}
The moment map  $\mu_{c,\cdot,\phi}$ is a bounded and polyhomogeneous $\mathfrak{t}^*$-valued function on $\chM$.
\label{fs.13}\end{lemma} 
\begin{proof}
Given $\tau\in \mathfrak{t}$, we need to show that $u:=\langle \mu_{c,\epio,\phi},\tau\rangle$ is a bounded polyhomogeneous function on $\chM$.  If $\xi_{\tau}$ is the Hamiltonian Killing field associated to $\tau$, then the function $u$ is obtained by solving 
$$
      \Delta_{c,\epio,\phi}u= \delta_{c,\epio,\phi}\xi_\tau \quad \mbox{with normalization} \quad (\epio)_*(u \omega^{[m]}_{c,\epio,\phi})=0,
$$ 
where $\delta_{c,\epio,\phi}$ is the divergence operator of $\omega_{c,\epio,\phi}$.  Since $\xi_{\tau}$ is smooth as a section of ${}^{c,\epio}T\chM$, we see that $\delta_{c,\epio,\phi}\xi_\tau$ is a polyhomogeneous function with $\rho_+\delta\xi_{\tau}$ bounded.  Hence, we can use Corollary~\ref{fs.11} to conclude that $u$ is polyhomogeneous and bounded.  
\end{proof}
Similarly, we can estimate the convergence of $\mu_{c,\cdot,\phi}$ near $\hH_0$ and $\hH_+$ as follows.
\begin{lemma}
Let $\phi_1,\phi_2\in \lrp{\cA_{\phg}(\chM)\cap L^{\infty}(\chM)}$ be potentials. If for some $\kappa\ge 0$ and $\delta>0$,
\begin{equation}
    \phi_1-\phi_2= \cO(\rho_0^{\kappa}\rho_+^{\delta})
\label{fs.13b}\end{equation}
near $\hH_0$ and $\hH_+$, then
$$
    \mu_{c,\epio,\phi_1}-\mu_{c,\epio,\phi_2}= \cO(\rho_0^{\kappa}\rho_+^{\delta})
$$
near $\hH_0$ and $\hH_+$ as a function taking values in $\mathfrak{t}^*$.
\label{fs.13a}\end{lemma}
\begin{proof}
Fix $\tau\in \mathfrak{t}$ and let $\xi_\tau$ be the corresponding Killing field on $\chM$. By a standard result, see for instance \cite[Lemma~4.5.1]{Gauduchon}, we know that
$$
  \langle \mu_{c,\epio,\phi_2}-\mu_{c,\epio,\phi_1},\tau\rangle= \cL_{\hat{J}\xi_\tau}(\phi_2-\phi_1)=\cO(\rho_0^{\kappa}\rho_+^{\delta})
$$
with the estimate following from \eqref{fs.13b} and the fact $\hat{J}\xi_{\tau}$ is a $b$-surgery vector fields, so decay to order one at $\hH_+$ as a section of ${}^{c,\epio}T\chM$.  Since $\tau\in \mathfrak{t}$ was arbitrary, the result follows.
\end{proof}

We can show that the $\Sc^{\hbeta}$-scalar curvature is polyhomogeneous.
\begin{lemma}
   Let $\phi\in (\cA_{\phg}(\chM)\cap L^{\infty}(\chM))$ be a function vanishing at $\hH_0$ and $\hH_+$ such that $dd^c\phi$ vanishes on $\hH_+$ and $\hH_0$ as a section of $\Lambda^2({}^{c,\epio}T^*\chM)$. Then the function $\rho_+^4\mathring{\Sc}^{\hbeta}(\omega_{c,\epio,\phi})$ is polyhomogeneous and bounded.  Moreover, it vanishes on $\hH_+$.
\label{fs.14}\end{lemma}
\begin{proof}
Notice first that $\rho_+^4\Sc^{\hbeta}(\omega_{c,\epio,\phi})$ takes the form
\begin{equation}
\rho_+^4\Sc^{\hbeta}(\omega_{c,\epio,\phi})= \rho^4_+\Delta_{c,\epio,\phi}R_{c,\epio,\phi}+
\frac12 \left( \rho_+^2R_{c,\epio,\phi}\right)^2 - \| \rho_+^2 \Ric_{c,\epio,\phi}\|^2_{\omega_{c,\epio,\phi}}-  \| \rho^2_+ \theta_{c,\epio,\phi}\|_{\omega_{c,\epio,\phi}}^2,
\label{fs.15}\end{equation}
so that 
\begin{equation}
\rho_+^4\mathring{\Sc}^{\hbeta}(\omega_{c,\epio,\phi})= \rho_+^4\Sc^{\hbeta}(\omega_{c,\epio,\phi})- \rho_+^4 \ell^{\ext}_{\halpha_\epio,\hbeta}(\mu_{c,\epio,\phi})
\label{fs.16}\end{equation}
with
$$
   \ell^{\ext}_{\halpha_\epio,\hbeta}(x):= \langle \xi^{\ext}_{\halpha_{\epio}, \hbeta},x\rangle + c_{\halpha_{\epio},\hbeta},
$$
where $\xi^{ext}_{\halpha_{\epio},\hbeta}$ is the Killing vector fields whose Killing potential is the $L^2$-projection of $\Sc^{\hbeta}(\omega_{c,\epio,\phi})$ onto $P_{\omega_{c,\epio,\phi}}$,  and
$$
    c_{\halpha_{\epio},\hbeta}= 2m(m-1)\lrp{\frac{(c_1(\hM)^2+\hbeta^2)\cdot \halpha_{\epio}^{m-2}}{\halpha_{\epio}^m}}.
$$ 
Let us first show that $\rho_+^4\Sc^{\hbeta}(\omega_{c,\epio,\phi})$ is a bounded and polyhomogeneous function on $\cM_+$.  For the first three terms in \eqref{fs.15}, polyhomogeneity follows from the polyhomogeneity of $\phi$, while boundedness is ensured by the factor of $\rho_+^4$ and the fact that $dd^c\phi$ is bounded (in fact even decay as $\epio\searrow 0$) as a section of $\Lambda^2({}^{c,\epio}T^*\chM)$.  For the last term, polyhomogeneity and boundedness are a consequence of Lemma~\ref{fs.12}.  Thus, the polyhomogeneity and boundedness of $\rho_+^4\mathring{\Sc}^{\hbeta}(\omega_{c,\epio,\phi})$ will follow from the one of $\rho_+^4 \ell^{\ext}_{\halpha_\epio,\hbeta}(\mu_{c,\epio,\phi})$.  Moreover, since
\begin{equation}
\epio^4\Sc^{\hbeta}(\omega_{c,\epio,\phi})|_{\hH_i}= \Sc^{0}(\omega_i)=0 \quad \forall \; i\in \{1,\ldots,\ell\}
\label{fs.16b}\end{equation}
by \eqref{van.1}, we see that it vanishes on $\hH_+$.  

Now,  the boundedness of $\rho_+^4\Sc^{\hbeta}(\omega_{c,\epio,\phi})$, its polyhomogeneity, its vanishing at $\hH_+$ and Lemma~\ref{fs.13} ensure that the $L^2$-projection of $\Sc^{\hbeta}(\omega_{c,\epio,\phi})$  on $P_{c,\epio,\phi}$ is well-defined and is bounded and polyhomogeneous as a function taking values in $\mathfrak{t}$, where $P_{c,\epio,\phi}$ is the space of Killing potentials with respect to $\omega_{c,\epio,\phi}$ for the $\bbT$-action on $\hM$.   Invoking Lemma~\ref{fs.13} again, this ensures that $\ell^{\ext}_{\halpha_\epio, \hbeta}(\mu_{c,\epio,\phi})$ is bounded and polyhomogeneous.  In particular, $\rho_+^4\ell^{\ext}_{\halpha_{\epio},\hbeta}(\mu_{c,\epio,\phi})$ vanishes at $\hH_+$.  Hence, we see that
$$
   \rho_+^4\mathring{\Sc}^{\hbeta}(\omega_{c,\epio,\phi})|_{\hH_+}= \rho_+^4\Sc^{\hbeta}(\omega_{c,\epio,\phi})|_{\hH_+}= 0
$$
by \eqref{fs.16b}.  The result follows.
\end{proof}

The rate of decay of $\rho_+^4\mathring{\Sc}^{\hbeta}(\omega_{c,\epio,\phi})|_{\hH_i}$ as $\epio\searrow 0$ can be estimated as follows.
\begin{lemma}\label{fs.16a}
Let $\phi_1,\phi_2\in \lrp{\cA_{\phg}(\chM)\cap L^{\infty}(\chM)}$ be two potentials vanishing on $\hH_+$ and $\hH_0$ with $dd^c(\phi_1)$ and $dd^c\phi_2$ vanishing on $\hH_+$ and $\hH_0$ as sections of $\Lambda^2({}^{c,\epio}T^*\chM)$.  If for some $0< \kappa\le 4$, $\phi_1-\phi_2=o(\rho_0^{\kappa})$ and  
$$
    dd^c(\phi_1-\phi_2)= o(\epio^{\kappa})
$$
as a section of $\Lambda^2({}^{c,\epio}T^*\chM)$, then 
$$
    \rho_+^4\mathring{\Sc}^{\hbeta}(\omega_{c,\epio,\phi_1}) -    \rho_+^4\mathring{\Sc}^{\hbeta}(\omega_{c,\epio,\phi_2})  = o(\epio^{\kappa})
$$
near $\hH_0$ and $\hH_+$.  If instead for $\kappa>4$ and $0<\delta\le \kappa -4$, 
$$
    dd^c(\phi_1-\phi_2)= o(\epio^{\kappa})
$$
and 
$$
  \phi_1-\phi_2\in o(\rho_0^{\kappa}\rho_+^{\delta}),
$$
then 
$$
    \rho_+^4\mathring{\Sc}^{\hbeta}(\omega_{c,\epio,\phi_1}) -    \rho_+^4\mathring{\Sc}^{\hbeta}(\omega_{c,\epio,\phi_2})  = o(\rho_0^{\kappa}\rho_+^{\delta+4}).
$$
Finally, if $dd^c(\phi_1-\phi_2)$ and $\phi_1-\phi_2$ do not necessarily vanish on $\hH_+$, but
$$
  dd^c(\phi_1-\phi_2)= \cO(\epio^{\kappa}\rho_0^{\delta})
$$
and 
$$
 \phi_1-\phi_2\in \cO(\epio^{\kappa}\rho_0^{\delta}\rho_+^{-4})
$$
for some $\kappa\ge 0$  and $\delta\in ]0,2m-2[$, then
$$
  \rho_+^4\mathring{\Sc}^{\hbeta}(\omega_{c,\epio,\phi_1}) -    \rho_+^4\mathring{\Sc}^{\hbeta}(\omega_{c,\epio,\phi_2})  = \cO(\epio^{\kappa}\rho_0^{\delta}).
$$
\end{lemma}
\begin{proof}
The result is clear for the parts of $\rho_+^4\mathring{\Sc}^{\hbeta}(\omega_{c,\epio,\phi_1})-\rho_+^4\mathring{\Sc}^{\hbeta}(\omega_{c,\epio,\phi_2})$ coming from the first three terms of \eqref{fs.15} and only depends on the decay of $dd^c(\phi_1-\phi_2)$. It follows from Lemma~\ref{fs.12a} for the part coming from the fourth term of \eqref{fs.15}, again in this case just relying on the decay of $dd^c(\phi_1-\phi_2)$. For the part coming from the second term on the right hand side of \eqref{fs.16}, it follows from Lemma~\ref{fs.13a} using the decay of $\phi_1-\phi_2$.
\end{proof}

Given these results, we want to find $\phi$ such that $\mathring{\Sc}^{\hbeta}(\omega_{c,\epio,\phi})\in \dot{\cC}^{\infty}(\chM)$. Our strategy to construct such a $\phi$ is a step by step elimination of all the terms in the expansion of $\mathring{\Sc}^{\hbeta}(\omega_{c,\epio,0})$ at $\hH_0$ and $\hH_+$.  Let us first explain how the decay at $\hH_0$ can be improved.

\begin{lemma}\label{c2.1}
Suppose that $m\ge 2$ and that there exists $\phi\in\cA_{\phg}(\chM)\cap L^{\infty}(\chM)$ vanishing on $\cH_0$ and $\cH_+$ with $dd^c\phi\in \cA(\chM;\Lambda^2({}^{c,\epio}T^*\chM))\cap L^{\infty}(\chM;\Lambda^2({}^{c,\epio}T^*\chM))$ vanishing on $\hH_0$ and $\hH_+$  such that
\begin{equation*}
    \rho_+^4\mathring{\Sc}^{\hbeta}(\omega_{c,\epio,\phi})= o(\epio^\kappa)
\label{c2.1a}\end{equation*}
for some $\kappa\ge 0$. 
Then there exists $u\in\cA_{\phg}(\chM)\cap L^{\infty}(\chM)$ vanishing on $\cH_0$ and $\cH_+$ with $dd^c u\in \cA(\chM;\Lambda^2({}^{c,\epio}T^*\chM))\cap L^{\infty}(\chM;\Lambda^2({}^{c,\epio}T^*\chM))$ vanishing on $\hH_0$ and $\hH_+$ such that
\begin{equation*}
    \rho_+^4\mathring{\Sc}^{\hbeta}(\omega_{c,\epio,\phi+\epio^{\kappa}u})= o(\rho_0^{\lambda}\epio^\kappa) \quad \forall \;\lambda<2.
\label{c2.1b}\end{equation*}
\end{lemma}

\begin{proof}
By the previous result, 
$$
   \frac{ \rho_+^4\mathring{\Sc}^{\hbeta}(\omega_{c,\epio,\phi})}{\epio^{\kappa}}\in\cA_{\phg}(\chM)\cap L^{\infty}(\chM)
$$
and its restrictions to $\hH_+$ and $\hH_0$ vanish by assumption.  In the polyhomogeneous expansion of  
$$
  \frac{\mathring{\Sc}^{\hbeta}(\omega_{c,\epio,\phi})}{\epio^{\kappa}}
$$
at $\hH_0$, suppose that the top order term is 
$$
     \sum_{j=0}^k v_{\delta,j}\rho_0^{\delta}(\log\rho_0)^j
$$
for some $\delta\in]0,2[$ and $k\in\bbN$, so that
$$
 \rho_+^4\left( \frac{\mathring{\Sc}^{\hbeta}(\omega_{c,\epio,\phi})}{\epio^{\kappa}}- \sum_{j=0}^k v_{\delta,j}\rho_0^{\delta}(\log\rho_0)^j\right)= o(\rho_0^{\delta})
$$
near $\hH_0$.  This top order term can be rewritten
$$
\begin{aligned}
 \sum_{j=0}^k v_{\delta,j}\rho_0^{\delta}(\log\rho_0)^j &=  \sum_{j=0}^k v_{\delta,j}\lrp{\frac{\epio}{\rho_+}}^{\delta}(\log\epio- \log\rho_+)^j \\
  &= \sum_{j=0}^k v_{\delta,j}\lrp{\frac{\epio}{\rho_+}}^{\delta} \sum_{i=0}^j (-1)^i \lrp{\begin{matrix} j \\ i \end{matrix}} (\log\epio)^{j-i}(\log\rho_+)^i \\
  &= \sum_{j=0}^k \widetilde{v}_{\delta,j}\epio^{\delta}(\log\epio)^{j},
 \end{aligned}
$$
where 
$$
 \widetilde{v}_{\delta,j}= \sum_{p=0}^{k-j} (-1)^p \lrp{\begin{matrix} j+p \\ p \end{matrix}} v_{\delta,p+j}\rho_+^{-\delta} (\log\rho_+)^{p}.
$$
Since $\frac{\rho_+^4\mathring{\Sc}^{\hbeta}(\omega_{c,\epio,\phi})}{\epio^{\kappa}}$ vanishes at $\hH_+$, there exists $\lambda>0$ such that
$$
    v_{\delta,j}\in \rho_+^{\lambda-4} L^{2,\infty}_b(\hH_0) \quad \forall \; j,
$$
hence such that 
$$
   \widetilde{v}_{\delta,j}\in \rho_+^{b}L^{2,\infty}_b(\hH_0) \quad \forall \; j \quad \mbox{and} \quad \forall \; b\in ]-\delta-4, -\delta-4+\lambda[.
$$
  Since $\omega_{c,\epio,0}^{[m]}\in \rho_+^{2m}\CI(\chM; \Lambda^{2m}({}^{b,\epio}T^*\chM))$ decays to order $\rho_+^{2m}$ at $\hH_+$ as a $b$-surgery density, we can proceed as in \cite[\S~5.2.1]{Benabida} and use Melrose's pushforward theorem \cite[Theorem~5]{Melrose1992} to see that
$$
   (\epio)_*\left( \mathring{\Sc}^{\hbeta}(\omega_{c,\epio,0})\mu_{\epio,\phi}(\tau)\omega_{c,\epio,0}^{[m]}  \right)
$$ 
is polyhomogeneous for all $\tau\in\mathfrak{t}$.  Furthermore, when $m\ge 3$,  its top order term at $\epio=0$ is given by
\begin{equation}
   \sum_{j=0}^{k} \lrp{ \int_{\hH_0} \widetilde{v}_{\delta,j} \mu_{\epio,\phi}(\tau) \omega_0^{[m]}  } \epio^{\delta}(\log\epio)^j.
\label{c2.3}\end{equation}
For $m=2$, the top order term is still given by \eqref{c2.3} provided $\tau\in \mathfrak{t}$ is such that $\mu_{\epio,\phi}(\tau)|_{\hH_0}=\mu_0(\tau)\in P_{0,+}$.
Since  on the other hand, by the definition of $\mathring{\Sc}^{\hbeta}(\omega_{c,\epio,\phi})$, 
$$
   (\epio)_*\left( \mathring{\Sc}^{\hbeta}(\omega_{c,\epio,\phi})\mu_{\epio,\phi}(\tau)\omega_{c,\epio,\phi}^{[m]}  \right)=0,
$$
we have that 
$$
  \int_{\hH_0} \widetilde{v}_{\delta,j} \mu_{\epio,0}(\tau) \omega_0^{[m]}=0 \quad \forall \tau\in \mathfrak{t}, \quad, \forall \; j
$$
when $m\ge 3$. For $m=2$, this holds for $\tau\in \mathfrak{t}$ such that $\mu_{\epio,\phi}(\tau)|_{\hH_0}=\mu_0(\tau)\in P_{0,+}$, 
which implies that $\widetilde{v}_{\delta,i}\in (r^bL^{2,\infty}_{b}(\hH_0)^{\bbT})^{\perp_+}$ for all $b\in ]-\delta-4, -\delta-4+\lambda[$.  Since $\delta\le 2$, we can apply 
Corollary~\ref{la.16b} when $m>3$, Proposition~\ref{la.26} when $m=3$ and Proposition~\ref{la2.7} when $m=2$  with   $\nu\in]0,1[\cup ]1,2[$  with
\begin{equation*}
    -\delta-4<\nu-6< -\delta-4+ \lambda
\label{sol.2}\end{equation*}
to find  $w_{\delta,j}$ such that 
\begin{equation}
   \bbS w_{\delta,j}= \widetilde{v}_{\delta,j}.
\label{sol.1}\end{equation}
By the regularity results of Corollary~\ref{b.13}, $w_{\delta,j}$ is also polyhomogeneous.
   Moreover, near $\pa_{p_i}\bM= \hH_i\cap \hH_0$, we can find a $\Gamma_i$-invariant polynomial $P_{\delta,j,i}$ of degree $1$ on $\bbC^m/\Gamma_i$ such that 
$$
        w_{\delta,j}- \chi_iP_{\delta,j,i}\in r^b L^{2,\infty}_b(\hH_0) \quad \mbox{for any } \quad -\delta+2<b<\min\{2,-\delta+2+\lambda\},
$$  
where $\chi_i\in \CI(\hH_0)$ is a cut-off function equal to $1$ near $\hH_0\cap\hH_i$ and $0$ outside some small neighbourhood of $\hH_0\cap\hH_i$ in $\hH_0$.
Extending $w_{\delta,j}$ smoothly off $\hH_0$, this means, using the fact that $dd^cP_{\delta,j,i}=0$, that $dd^c (\epio^{\delta}(\log\epio)^j w_{\delta,j})$ vanishes on $\hH_+$ and $\hH_0$ as a section of $\Lambda^2({}^{c,\epio}T^*\chM)$.

Setting 
$$
    u_{\delta,j}:= w_{\delta,j}\epio^{\delta}(\log\epio)^j \quad \mbox{and} \quad u_{\delta}:= -\sum_{j=0}^k u_{\delta,j},
$$
we then obtain using Lemma~\ref{fs.16a} that
$$
\begin{aligned}
\frac{\rho_+^4\Sc^{\hbeta}(\omega_{c,\epio,\phi+\epio^{\kappa}u_{\delta}})}{\epio^{\kappa}}&= \frac{\rho_+^4\Sc^{\hbeta}(\omega_{c,\epio,\phi})}{\epio^{\kappa}}- \sum_{j=0}^{k} \rho_+^4\epio^{\delta}(\log\epio)^{j} \bbS w_{\delta,j} + o(\rho_0^{\delta}) \\
              &= \frac{\rho_+^4\Sc^{\hbeta}(\omega_{c,\epio,\phi})}{\epio^{\kappa}}- \rho_+^4 \lrp{\sum_{j=0}^{k} \widetilde{v}_{\delta,j} \epio^{\delta}(\log \epio)^{k} } +o(\rho_0^{\delta}) \\
              &=  o(\rho_0^{\delta})
\end{aligned}
$$
near $\hH_0$.  Replacing $\omega_{c,\epio,\phi}$ by $\omega_{c,\epio,\phi+\epio^{\kappa}u_{\delta}}$, we can iterate this argument and find $u_0\in \cA_{\phg}(\chM)\cap L^{\infty}(\chM)$ vanishing on $\hH_0$ and $\hH_+$  with $dd^cu_0\in \cA_{\phg}(\chM;\Lambda^2({}^{c,\epio}T^*\chM))\cap L^{\infty}\chM;\Lambda^2({}^{c,\epio}T^*\chM))$ vanishing on $\hH_0$ and $\hH_+$ such that
\begin{equation*}
   \frac{\rho_+^4\mathring{\Sc}^{\hbeta}(\omega_{c,\epio,\phi+\epio^{\kappa} u_0})}{\epio^{\kappa}}\in \rho_0^{\lambda}\cA_{\phg}(\chM)\cap L^{\infty}(\chM)  \quad \forall \;  \lambda<2
\label{fs.18}\end{equation*}
 vanishes on $\hH_+$.  
\end{proof}

To improve the decay at $\hH_+$, we can instead proceed as follows.
\begin{lemma}\label{c2.4}
Suppose that $m\ge 2$ and that there exists $\phi\in\cA_{\phg}(\chM)\cap L^{\infty}(\chM)$ vanishing on $\cH_0$ and $\cH_+$ with $dd^c\phi\in \cA(\chM;\Lambda^2({}^{c,\epio}T^*\chM))\cap L^{\infty}(\chM;\Lambda^2({}^{c,\epio}T^*\chM))$ vanishing on $\hH_0$ and $\hH_+$  such that
\begin{equation}
    \rho_+^4\mathring{\Sc}^{\hbeta}(\omega_{c,\epio,\phi})= o(\epio^\kappa \rho_0^{\frac32})
\label{c2.4a}\end{equation}
for some $\kappa\ge 0$.  
Then there exists $u\in\cA_{\phg}(\chM)\cap L^{\infty}(\chM)$ vanishing on $\cH_0$ and $\cH_+$ with $dd^c u\in \cA(\chM;\Lambda^2({}^{c,\epio}T^*\chM))\cap L^{\infty}(\chM;\Lambda^2({}^{c,\epio}T^*\chM))$ vanishing on $\hH_0$ and $\hH_+$ such that
\begin{equation*}
    \rho_+^4\mathring{\Sc}^{\hbeta}(\omega_{c,\epio,\phi+\epio^{\kappa}u})= o(\epio^{\kappa+\frac12}).
\label{c2.4b}\end{equation*}
\end{lemma}
\begin{proof}
To improve the decay at $\hH_i$ for $i>0$, suppose that for some $\delta\in ]-4,-\frac72]$, the top order term in the polyhomogeneous expansion of $\frac{\mathring{\Sc}^{\hbeta}(\omega_{c,\epio,\phi})}{\epio^{\kappa}}$ at $\hH_i$ is of the form
$$
 \begin{aligned}
 \sum_{j=0}^k v_{\delta,j}\rho_+^{\delta}(\log\rho_+)^{j}&= \sum_{j=0}^k v_{\delta,j} \lrp{\frac{\epio}{\rho_0}}^\delta (\log \epio- \log \rho_0)^j \\
 &= \sum_{j=0}^{k} \widetilde{v}_{\delta,j} \epio^{\delta}(\log\epio)^j,
 \end{aligned}
$$
where 
$$
    \widetilde{v}_{\delta,j}:= \sum_{p=0}^{k-j} (-1)^p \lrp{\begin{matrix} j+p \\ p\end{matrix}}v_{\delta,p+j} \rho_0^{-\delta}(\log\rho_0)^p.
$$
By \eqref{c2.4a}, we know that
$$
      v_{\delta,j}\in \rho_0^{\frac32+\delta'}L^{2,\infty}_b(\hH_i) \quad \mbox{for some} \quad \delta'>0,
$$
so
$$
   \widetilde{v}_{\delta,j}\in \rho_0^{\frac32-\delta+\frac{\delta'}{2}}L^{2,\infty}_b(\hH_i).
$$
Without loss of generality, we can assume that $\delta'<\delta+4$.  Suppose first that $m\ge 3$.  Then
taking $\nu\in ]-1,0[$ such that $\nu<-\frac92-\delta+\frac{\delta'}2$, we know by Corollary~\ref{la.33e} when $m>3$ and by Proposition~\ref{la.35} when $m=3$,  combined with the regularity result of Corollary~\ref{b.13}, that there exists 
$$
   w_{\delta,j}\in \cA_{\phg}(\hH_i)\cap \rho_0^{-\frac92-\delta+\frac{\delta'}4}L^{2,\infty}_b(\hH_i)
$$
such that 
\begin{equation}
   \bbS_{i,\sc}w_{\delta,j}= \widetilde{v}_{\delta,j}.
\label{c2.7}\end{equation}
Extending $w_{\delta,j}$ smoothly off $\hH_i$, notice that
\begin{equation}
   \epio^{\delta+6}(\log\epio)^jw_{\delta,j}=\cO(\rho_+^{\delta+6-\sigma}\rho_0^{\frac32+\frac{\delta'}4}) \quad \forall \sigma>0
\label{c2.7b}\end{equation}
near $\hH_i$,
which ensures that $dd^c(\epio^{\delta+6}(\log\epio)w_{\delta,j})$ vanishes on $\hH_i$ and $\hH_0$ as a section of $\Lambda^2({}^{c,\epio}T^*\chM)$.    Setting
$$
u_{\delta,j}:= \epio^{\delta+6}(\log\epio)^j w_{\delta,j} \quad \mbox{and} \quad u_\delta:= -\sum_{j=0}^k u_{\delta,j},
$$
we see from Lemma~\ref{fs.16a} that there exists $\sigma>0$ sufficiently small such that
\begin{equation}
\begin{aligned}
\frac{\mathring{\Sc}^{\hbeta}(\omega_{c,\epio,\phi_0+\epio^{\kappa}u_{\delta}})}{\epio^{\kappa}}&= \frac{\mathring{\Sc}^{\hbeta}(\omega_{c,\epio,\phi_0})}{\epio^{\kappa}}- \sum_{j=0}^k \epio^{\delta}(\log\epio)^j \bbS_{i,\sc}w_{\delta,j} + \cO(\rho_+^{\delta+\sigma}\rho_0^{\frac32+\frac{\delta'}4}) \\
&= \frac{\mathring{\Sc}^{\hbeta}(\omega_{c,\epio,\phi_0})}{\epio^{\kappa}}- \sum_{j=0}^k \epio^{\delta}(\log\epio)^j \widetilde{v}_{\delta,j} + \cO(\rho_+^{\delta+\sigma}\rho_0^{\frac32+\frac{\delta'}4}) \\
&= \cO(\rho_+^{\delta+\sigma}\rho_0^{\frac32+\frac{\delta'}4})
\end{aligned}
\label{fs.19}\end{equation}
near $\hH_i$.  Performing a similar argument for each connected component of $\hH_+$ ensures that \eqref{fs.19} holds near $\hH_+$.  Iterating this argument, we can find $u\in \cA_{\phg}(\chM)\cap L^{\infty}(\chM)$ with $\frac{u}{\rho_+^2}$ vanishing at $\hH_+$ and $\hH_0$ such that 
$$
    \frac{\rho_+^4\mathring{\Sc}^{\hbeta}(\omega_{c,\epio,\phi+\epio^{\kappa}u})}{\epio^{\kappa}}\in \epio^{\frac12}\lrp{ \cA_{\phg}(\chM)\cap L^{\infty}(\chM) }
$$ 
with $ \epio^{-\kappa-\frac12}\rho_+^4\mathring{\Sc}^{\hbeta}(\omega_{c,\epio,\phi+\epio^{\kappa}u})$ vanishing at $\hH_+$ and $\hH_0$.   

When $m=2$, we need to work directly with $\hH_+$. To obtain $w_{\delta,j}$ as in \eqref{c2.7} for each $i$, we can apply Corollary~\ref{lasc.8}, which yields such a $w_{\delta,j}\in \cA^{\cE}_{\phg}(\hH_+)$ for an index set $\cE$ such that $\min\cE\ge (-2,1)$.  In this case, the top order term of $\epio^{\delta+6} (\log\epio)^j w_{\delta,j}$ at $\hH_0$ is of order $\rho_0^{\delta+4}(\log\rho_0)^{j+1}$, so that \eqref{fs.19} is replaced by
\begin{equation*}
 \frac{\mathring{\Sc}^{\hbeta}(\omega_{c,\epio,\phi_0+\epio^{\kappa}u_{\delta}})}{\epio^{\kappa}}= \cO(\rho_+^{\delta+\sigma}\rho_0^{\delta+4}(\log\rho_0)^{k+1})
\label{c2.8}\end{equation*}
for some $\sigma>0$.  This improves the decay at $\hH_+$, but since $\delta+4\le \frac12$, it worsen the decay at $\hH_0$ by adding a term of order $\rho_0^{\delta+4}(\log\rho_0)^{k+1}$.

However, we can eliminate this term by applying the argument of Lemma~\ref{c2.1}.  Doing so, this add terms of order $\rho_+^{\delta}(\log\rho_+)^j$ for $j\in\{0,\ldots,k+1\}$ at $\hH_+$.  We can also eliminate higher order terms in the expansion at $\hH_0$ up to order $\rho_0^{\frac32}$ at the cost of adding terms of order $o(\rho_+^{\delta})$ at $\hH_+$.  Now, the new terms of order $\rho_+^{\delta}(\log\rho_0)^j$ comes from the linearization of $\mathring{\Sc}^{\hbeta}$ at $\hH_+$, so they lie in the image of the operator in Corollary~\ref{lasc.7}.  This ensures that these new terms of order $\rho_+^{\delta}(\log\rho_+)^j$ can be eliminated using Corollary~\ref{lasc.7} instead of Corollary~\ref{lasc.8} to obtain \eqref{fs.19}.  The whole argument can then be iterated to find $u\in \cA_{\phg}(\chM)\cap L^{\infty}(\chM)$ with $dd^cu$ vanishing at $\hH_+$ and $\hH_0$ such that 
$$
    \frac{\rho_+^4\mathring{\Sc}^{\hbeta}(\omega_{c,\epio,\phi+\epio^{\kappa}u})}{\epio^{\kappa}}\in \epio^{\frac12}\lrp{ \cA_{\phg}(\chM)\cap L^{\infty}(\chM) }
$$ 
with $ \epio^{-\kappa-\frac12}\rho_+^4\mathring{\Sc}_{\hbeta}(\omega_{c,\epio,\phi+\epio^{\kappa}u})$ vanishing at $\hH_+$ and $\hH_0$.  \end{proof}

The last two lemmas can be combined as follows to construct a formal solution to the equation
$$
\mathring{\Sc}^{\hbeta}(\omega_{c,\epio,\phi})=0.
$$
\begin{theorem}\label{fs.17}
Suppose that $m\ge 2$.  Then there exists $\phi\in\cA_{\phg}(\chM)\cap L^{\infty}(\chM)$ vanishing on $\cH_0$ and $\cH_+$ with $dd^c\phi\in \cA(\chM;\Lambda^2({}^{c,\epio}T^*\chM))\cap L^{\infty}(\chM;\Lambda^2({}^{c,\epio}T^*\chM))$ vanishing on $\hH_0$ and $\hH_+$  such that
\begin{equation}
    \mathring{\Sc}^{\hbeta}(\omega_{c,\epio,\phi})\in \dot{\cC}^{\infty}(\chM).
\label{fs.17a}\end{equation}
\end{theorem}
\begin{proof}
For $k\in \bbN$, we want to proceed by induction to show that we can find $\phi_k\in \cA_{\phg}(\chM)\cap L^{\infty}(\chM)$ vanishing at $\hH_+$ and $\hH_0$ with $dd^c\phi_k\in\cA_{\phg}(\chM;\Lambda^2({}^{c,\epio}T^*\chM))\cap L^{\infty}(\chM;\Lambda^2({}^{c,\epio}T^*\chM))$ vanishing at $\hH_+$ and $\hH_-$ such that
$$
    \rho_+^4\mathring{\Sc}^{\hbeta}(\omega_{c,\epio,\phi_k})\in \epio^{\frac{k}{2}}\lrp{\cA_{\phg}(\chM)\cap L^{\infty}(\chM) }
$$ 
with $\epio^{-\frac{k}{2}}\rho_+^4\mathring{\Sc}^{\hbeta}(\omega_{c,\epio,\phi_k})$ vanishing at $\hH_+$ and $\hH_-$. The case $k=0$ is trivial, since it suffices to take $\phi_0=0$. So suppose we have found $\phi_k$ and let us show how to construct $\phi_{k+1}$.  The idea is to consider a function  
$$
   u\in  \cA_{\phg}(\chM)\cap L^{\infty}(\chM) 
$$
with $u$ and $dd^cu$ vanishing at $\hH_0$ and $\hH_+$ and  the functional 
$$
    u\mapsto \frac{\rho_+^4\mathring{\Sc}^{\hbeta}(\omega_{c,\epio,\phi_k+\epio^{\frac{k}{2}}u})}{\epio^{\frac{k}{2}}}.
$$
But by Lemmas~\ref{c2.1} and \ref{c2.4}, we can find $u$ as above such that
$$
  \frac{\rho_+^4\mathring{\Sc}^{\hbeta}(\omega_{c,\epio,\phi_k+\epio^{\frac{k}{2}}u})}{\epio^{\frac{k+1}{2}}}  \in\lrp{\cA_{\phg}(\chM)\cap L^{\infty}(\chM)}
$$
and is vanishing at $\hH_+$ and $\hH_0$.  It suffices therefore to take $\phi_{k+1}= \phi_k+ \epio^{\frac{k}{2}}u$ to complete the inductive step of this proof by induction on $k$.  Finally, taking an asymptotic sum
$$
   \phi\sim  \phi_1+ \sum_{k=2}^{\infty} (\phi_k-\phi_{k-1}),
$$
we obtain a function $\phi\in \lrp{\cA_{\phg}(\chM)\cap L^{\infty}(\chM)}$  vanishing at $\hH_+$ and $\hH_0$ with $dd^c\phi$ vanishing on $\hH_+$ and $\hH_0$ such that \eqref{fs.17a} holds. \end{proof}

\subsection{Linear analysis on the $b$-surgery space} \label{MM.0}
Consider as in \eqref{fs.10} a K\"ahler conical surgery metric $\omega_{c,\epio,\phi}$ for some  $\bbT$-invariant function $\phi\in \cA_{\phg}(\chM)\cap L^{\infty}(\chM)$ vanishing at $\hH_0$ and $\hH_+$ with $dd^c\phi$ vanishing at $\hH_0$ and $\hH_+$ as a section of $\Lambda^2({}^{c,\epio}T^*\chM)$.    In this subsection, we want to study various linear elliptic geometric operators associated to this metric in the limit $\epio\searrow 0$.  More precisely, using the $b$-surgery calculus of Mazzeo-Melrose \cite{Mazzeo-MelroseETA}, we will invert them uniformly in $\epio$ by providing a suitable pseudodifferential characterization of their inverses.

Notice first that by  Lemma~\ref{green.1} and Corollary~\ref{MM.1}, the family of operators 
\begin{multline*}
\bbS_{\hbeta,\epio,\phi}(u):= 
-2\delta_{c,\epio,\phi}\delta_{c,\epio,\phi}(\nabla^{\omega_{c,\epio,\phi},-}(\delta_{c,\epio,\phi}\nabla^{\omega_{c,\epio,\phi},-}du))-\delta_{c,\epio,\phi}\delta_{c,\epio,\phi}(R_{c,\epio,\phi}\nabla^{\omega_{c,\epio,\phi},-}du) \\ 
+ 16\delta_{c,\epio,\phi}\delta_{c,\epio,\phi}(\theta_{c,\epio,\phi}\circ \bbG^{H}_{c,\epio,\phi}(\theta_{c,\epio,\phi}\circ \nabla^{\omega_{c,\epio,\phi},-}du)^{\Skew})
\label{MM.2}\end{multline*}
is in $\Psi^{6,\cS}_{b,s}(\chM)$ for an index family $\cS$ such that
$$
\cS|_{\lf}\ge -2, \quad \cS|_{\rf}\ge 2m-2, \quad \cS|_{\fb}\ge-6, \quad \mbox{and} \quad \cS|_{\mf}\ge 0.
$$

Let $\Pi_{P_{c,\epio,\phi}}$ denote the $L^2$-orthogonal projection onto the space $P_{c,\epio,\phi}$ of Killing potentials for vector fields in $\mathfrak{t}$  with respect to the metric $\omega_{c,\epio,\phi}$.

\begin{proposition}\label{MM.3}
Suppose that $m>3$.  Suppose also that $dd^c\phi$ vanishes on $\hH_0$ and $\hH_+$.  Then for $\epio>0$ sufficiently small, the $\bbT$-invariant $L^2$-kernel of $\bbS_{\hbeta,\epio,\phi}$ is $P_{\omega_{c,\epio,\phi}}$.  Moreover, there exists an operator $\bbG_{\hbeta,\epio,\phi}\in \Psi^{-6,\cG}_{b,s}(\chM)$ with 
$$
\cG(\lf)\ge 0, \quad \cG(\rf)\ge2m, \quad \cG(\fb)\ge 6, \quad \mbox{and} \quad \cG(\mf)\ge 0
$$
such that
$$
  \bbG_{\hbeta,\epio,\phi}\bbS_{\hbeta,\epio,\phi}=\bbS_{\hbeta,\epio,\phi}  \bbG_{\hbeta,\epio,\phi} = \Id - \Pi_{P_{c,\epio,\phi}}.
$$
In particular, for $\nu\in ]6-2m,0[$, the operator $\bbS_{\hbeta,\epio,\phi}$ induces a bijection
\begin{equation*}
   \bbS_{\hbeta,\epio,\phi}: (\rho_+^\nu L^{2,k+6}_{b,\epio}(\hM)^{\bbT})^{\perp}\to (\rho_+^{\nu-6} L^{2,k}_{b,\epio}(\hM)^{\bbT})^{\perp}
\label{MM.3b}\end{equation*}
with inverse uniformly bounded as $\epio\searrow 0$, where 
$$
  (\rho_+^\nu L^{2,k+6}_{b,\epio}(\hM)^{\bbT})^{\perp}=\{ u\in \rho_+^\nu L^{2,k+6}_{b,\epio}(\hM)^{\bbT}\; | \;  \Pi_{P_{c,\epio,\phi}}u=0\}.
$$
\end{proposition}
\begin{proof}
The $\bbT$-invariant kernel of $\bbS_{\hbeta,\epio,\phi}$ contains $P_{c,\epio,\phi}$.  Since $\phi|_{\hH_0}=0$, it is precisely $P_{0}$ at $\epio=0$.  So using a parametrix, we see that the $\bbT$-invariant part of the kernel of $\bbS_{\hbeta,\epio,\phi}$ is precisely  $P_{c,\epio,\phi}$ for $\epio>0$ sufficiently small. 

To obtain the result, it suffices then to apply Proposition~\ref{bs.6} to the operator
$$
   P_{b,s}:= \rho_+^{m-3}(\rho_+^6)(\bbS_{\hbeta,\epio,\phi}+P_{c,\epio,\phi})\rho_+^{-(m-3)}.
$$
Indeed, by Propositions~\ref{la.16} and \ref{la.33} and by Corollary~\ref{b.16}, we see that Proposition~\ref{bs.6} applies to $P_{b,s}$ with $\tau=m-3$.
\end{proof}

However, even if $\phi|_{\hH_0}=0$,  $\omega_{c,\epio,\phi}$ is not necessarily extremal for $\epio>0$, so the linearization of 
$$
    u\mapsto \mathring{\Sc}^{\hbeta}(\omega_{c,\epio,\phi}+dd^c u),
$$
is not quite the operator $\bbS_{\hbeta,\epio,\phi}$, but instead the operator
\begin{equation}
\begin{aligned}
\check{\bbS}_{\hbeta,\epio,\phi}(u)&:= \bbS_{\hbeta,\epio,\phi}(u) + \langle d\Sc^{\hbeta}(\omega_{c,\epio,\phi}), du\rangle_{\omega_{c,\epio,\phi}}-d^cu(V^{\ext}_{\halpha_{\epio}}) \\
&= \bbS_{\hbeta,\epio,\phi}(u) + \langle d(\Id-\Pi_{P_{c,\epio,\phi}}){ \Sc^{\hbeta}}(\omega_{c,\epio,\phi}), du\rangle_{\omega_{c,\epio,\phi}}.
\end{aligned}
\label{MM.4}\end{equation}
The extra term $u\mapsto \langle d(\Id-\Pi_{P_{c,\epio,\phi}}){\Sc^{\hbeta}}(\omega_{c,\epio,\phi}), du\rangle_{\omega_{c,\epio,\phi}}$ is a differential operator of order $1$.  In particular, for $\epio>0$ fixed, it can be seen as a compact perturbation of $\bbS_{\hbeta,\epio,\phi}$.    When $\phi$ is chosen as in Theorem~\ref{fs.17}, this extra terms decays rapidly as $\epio\searrow 0$ since $\mathring{\Sc}^{\hbeta}(\omega_{c,\epio,\phi})\in \dot{\cC}^{\infty}(\chM)$.  In particular, from Proposition~\ref{MM.3}, we see that
$$
       \bbG_{\hbeta,\epio,\phi}\check{\bbS}_{\hbeta,\epio,\phi}= \Id-\Pi_{P_{c,\epio,\phi}}+ R_1 \quad \mbox{and} \quad \check{\bbS}_{\hbeta,\epio,\phi}\bbG_{\hbeta,\epio,\phi}=\Id-\Pi_{P_{c,\epio,\phi}}+R_2
$$ 
with error terms $R_1,R_2\in \dot{\Psi}^{-5}(\chM)$ vanishing rapidly as $\epio\searrow 0$.  In particular, we deduce the following from  Proposition~\ref{MM.3}. 
\begin{corollary}\label{MM.6}
Suppose that $m>3$.  For $\phi$ as in Theorem~\ref{fs.17} and $\nu\in ]6-2m,0[$, the composition
\begin{equation*}
(\Id-\Pi_{P_{c,\epio,\phi}})\circ\check{\bbS}_{\hbeta,\epio,\phi}: (\rho_+^\nu L^{2,k+6}_{b,\epio}(\hM)^{\bbT})^{\perp}\to (\rho_+^{\nu-6} L^{2,k}_{b,\epio}(\hM)^{\bbT})^{\perp}
%\label{MM.5}
\end{equation*}
is a bijection for $\epio>0$ small enough with inverse uniformly bounded as $\epio\searrow 0$.  When acting on $\bbT$-invariant functions, there is in fact an inverse $\check{\bbG}_{\hbeta,\epio,\phi}\in \Psi^{-6,\check{\cG}}_{b,\epio}(\chM)$ such that
\begin{equation*}
   \check{\bbG}_{\hbeta,\epio,\phi}\check{\bbS}_{\hbeta,\epio,\phi}=  \check{\bbS}_{\hbeta,\epio,\phi} \check{\bbG}_{\hbeta,\epio,\phi}= \Id - \Pi_{P_{c,\epio,\phi}}
%\label{MM.6}
\end{equation*}
 with real index family $\check{\cG}$ such that
 $$
 \check{\cG}(\lf)\ge 0, \quad \check{\cG}(\rf)\ge 2m, \quad \check{\cG}(\fb)\ge 6, \quad \mbox{and} \quad \check{\cG}(\mf)\ge 0.
 $$  
 \end{corollary}

For $m\in\{2,3\}$, we need to proceed slightly differently.  Suppose first that $m=3$.  If $f_1,\ldots,f_{\dim K_0}\in (\CI(M)^{\bbT})^{\perp}$ are the functions of Proposition~\ref{la.26}, let $\widehat{f}_1,\ldots, \widehat{f}_{\dim K_0}\in (\cA_{\phg}(\chM)\cap L^{\infty}(\chM)^{\bbT})^{\perp}$ be polyhomogeneous extensions off $\hH_0$ to $\chM$ that are locally constant on $\hH_+$, where
$$
(\cA_{\phg}(\chM)\cap L^{\infty}(\chM)^{\bbT})^{\perp}:= \{ u\in \cA_{\phg}(\chM)\cap L^{\infty}(\chM)^{\bbT}\; | \;  \Pi_{P_{c,\epio,\phi}}u=0\}.
$$  Denote by $\chD$ the finite dimensional vector space generated by $\widehat{f}_1,\ldots, \widehat{f}_{\dim K_0}$.  On the level set $\epio=\lambda$, let also $\chD_{\lambda}$ be the finite dimensional vector space generated by $\widehat{f}_1|_{\epio^{-1}(\lambda)
},\ldots, \widehat{f}_{\dim K_0}|_{\epio^{-1}(\lambda)
}$.   For $\nu\in ]0,1[$ and $k\in \bbN$, consider the space 
\begin{equation}
\lrp{  \rho_+^{\nu} L^{2,k}_{b,\epio}(\hM)^{\bbT} }^{\perp_{\chD+ P_{c,\epio,\phi}}}= \left\{  u\in (\rho_+^{\nu} L^{2,k}_{b,\epio}(\hM)^{\bbT})^{\perp}\; | \;
     (\epio)_*\lrp{ uv \omega^{[m]}_{c,\epio,\phi}}=0 \quad \forall \;v\in \chD \right\}.
\label{c3.1}\end{equation}
\begin{proposition}\label{c3.2}
Suppose that $m=3$ and $\phi$ is such that $dd^c\phi$ vanishes on $\hH_0$ and $\hH_+$ as a section of $\Lambda^2({}^{c,\epio}T^*\chM)$.   Then  for  $\epio>0$ sufficiently small, the operator $\bbS_{\hbeta,\epio,\phi}$ induces an isomorphism
\begin{equation*}
   \bbS_{\hbeta,\epio,\phi}: (\rho_+^\nu L^{2,k+6}_{b,\epio}(\hM)^{\bbT})^{\perp_{\chD+P_{c,\epio,\phi}}}\oplus \chD_{\epio} \to (\rho_+^{\nu-6} L^{2,k}_{b,\epio}(\chM)^{\bbT})^{\perp}
%\label{v3.2a}
\end{equation*}
for $\nu\in ]0,1[$ and $k\in\bbN$ with inverse $\bbG_{\hbeta,\epio,\phi}\in\Psi^{-6,\cG}_{b,\epio}(\chM)$ with $\cG(\lf)\ge 0$, $\cG(\rf)>6-\nu$, $\cG(\fb)\ge 6$ and $\cG(\mf)\ge 0$ such that
$$
  \bbG_{\hbeta,\epio,\phi}\bbS_{\hbeta,\epio,\phi}=  \bbS_{\hbeta,\epio,\phi} \bbG_{\hbeta,\epio,\phi}= \Id - \Pi_{P_{\omega,\epio,\phi}}.
$$
\end{proposition}
\begin{proof}
Fix $\nu\in ]0,1[$ and consider the $b$-surgery operator
$$
   P_{b,s}:= \rho_+^{-\nu}(\rho_+^6\bbS_{\hbeta,\epio,\phi})\rho_+^{\nu}.
$$
By Proposition~\ref{la.26}, $N_{\mf}(P_{b,s})$ induces a Fredholm operator
$$
    N_{\mf}(P_{b,s}): L^{2,k+6}_b(\bM)^{\bbT}\to L^{2,k}_{b}(\bM)^{\bbT}
$$
with kernel corresponding to $\rho_+^{-\nu} P_{0,+}$ and cokernel corresponding to $\rho_+^{\nu}(P_{0}\oplus K_0)$.  Similarly, by Proposition~\ref{la.35}, $N_{\fb}(P_{b,s})$ induces a surjective Fredholm operator
$$
    N_{\fb}(P_{b,s}): L^{2,k+6}_b(\hH_+)^{\bbT}\to L^{2,k}_{b}(\hH_+)^{\bbT}
$$
with nullspace corresponding to $\rho_0^{\nu}V$, where $V$ is the space locally constant functions on $\hH_+$.  Extends elements of $P_{0,+}$ polyhomogeneously off $\hH_0$ so that they generate a subspace $\widehat{P}_{c,\epio,\phi,+}\subset P_{c,\epio,\phi}$ for each $\epio>0$.  Similarly, extends elements of $V$ polyhomogeneously off $\hH_+$ so that they generate a subspace $\widehat{V}= \widehat{P}_{\omega_{c,\epio,-}}\oplus \chD$ for $\epio>0$, where   $\widehat{P}_{c,\epio,\phi,-}$ is such that 
$$
      \widehat{P}_{c,\epio,\phi,-}\subset  \widehat{P}_{c,\epio,\phi} \quad \mbox{and} \quad  \widehat{P}_{c,\epio,\phi,+}\oplus  \widehat{P}_{c,\epio,\phi,-}= \widehat{P}_{c,\epio,\phi}
$$
for $\epio>0$.  Let $\Pi_1$ be the $L^2$-projection (for each fixed $\epio$) onto $\rho_0^{\nu}\widehat{V}\oplus \rho_+^{-\nu}P_{0,+}$ with respect to the $b$-surgery metric $\rho_+^{-2}g_{c,\epio,\phi}$.  Similarly, let $\Pi_2$ be the $L^2$-projection (for each fixed $\epio$) onto $\rho_+^{\nu}\widehat{P}_{c,\epio,\phi}\oplus  \rho_+^{-\nu+6}\bbS_{\hbeta,\epio,\phi}\chD$. By Proposition~\ref{rob.1} and Corollary~\ref{b.18}, there is an operator $G\in \Psi^{-6,\cH}_{b,s}(\chM)$ such that
$$
    (\Id-\Pi_2)P_{b,s}(\Id-\Pi_1)G= \Id-\Pi_2 \quad \mbox{and} \quad G(\Id-\Pi_2)P_{b,s}(\Id-\Pi_1)= \Id-\Pi_1
$$
for some real index family $\cH$ such that
$$
  \cH(\lf)>0, \quad \cH(\rf)>0, \quad \cH(\fb)\ge 0 \quad \mbox{and} \quad \cH(\mf)\ge 0.
$$
Hence, the result follows by taking 
$\bbG_{\hbeta,\epio,\phi}= \rho_{+}^{\nu}G\rho_+^{-\nu+6}+R_{\epio} 
$
with $R_{\epio}$ the inverse of the map
$ \bbS_{\hbeta,\epio,\phi}: \chD_{\epio}\to \bbS_{\hbeta,\epio,\phi}(\chD_{\epio}).
$
\end{proof}

As in the case $m>3$, this has the following consequence.
\begin{corollary}
For $m=3$ and $\phi$ as in Theorem~\ref{fs.17} and $\nu\in ]0,1[$, the composition
\begin{equation*}
(\Id-\Pi_{P_{c,\epio,\phi}})\circ\check{\bbS}_{\hbeta,\epio,\phi}: (\rho_+^\nu L^{2,k+6}_{b,\epio}(\hM)^{\bbT})^{\perp_{\chD+P_{c,\epio,\phi}}}\oplus \chD_{\epio}  \to (\rho_+^{\nu-6} L^{2,k}_{b,\epio}(\hM)^{\bbT})^{\perp}
\end{equation*}
is a bijection for $\epio>0$ small enough with inverse uniformly bounded as $\epio\searrow 0$.  When acting on $\bbT$-invariant functions, there is in fact an inverse $\check{\bbG}_{\hbeta,\epio,\phi}\in \Psi^{-6,\check{\cG}}_{b,\epio}(\chM)$ such that
\begin{equation*}
   \check{\bbG}_{\hbeta,\epio,\phi}\check{\bbS}_{\hbeta,\epio,\phi}=\check{\bbS}_{\hbeta,\epio,\phi}  \check{\bbG}_{\hbeta,\epio,\phi} = \Id - \Pi_{P_{c,\epio,\phi}}
%\label{c3.4}
\end{equation*}
 with index family $\check{\cG}$ such that
 $$
 \check{\cG}(\lf)\ge 0, \quad \check{\cG}(\rf)> 6-\nu, \quad \check{\cG}(\fb)\ge 6, \quad \mbox{and} \quad \check{\cG}(\mf)\ge 0.
 $$  
 \label{c3.5}\end{corollary}
 
 Finally, for $m=2$, let us consider  the map
 $$
 \begin{array}{lccl}
  \iota_+: & P_{0} &\to & \bbR^{\ell} \\
   & v &\mapsto & (v(p_1),\ldots, v(p_{\ell})),
 \end{array}
 $$
so that $P_{0,+}$ defined in \eqref{la2.2b} corresponds to the kernel of the map $\iota_+$ by Lemma~\ref{la2.3}.  Let $\cD$ be a complement of $\iota_+(P_{0})$ in $\bbR^{\ell}$ such that $\cD$ is spanned by 
$$(f_1(p_1),\ldots, f_1(p_{\ell})),\ldots (f_{\dim\cD}(p_1),\ldots, f_{\dim\cD}(p_{\ell}))
$$ 
for  functions $f_1,\ldots f_{\dim \cD}\in \CI(M)^{\bbT}$ that are constant near $p_i$ for $i\in\{1,\ldots,\ell\}$.  Notice that
$$
    \dim P_{0}+ \dim \cD= \dim P_{0,+}+\ell.
$$
Let  $\widehat{f}_{1},\ldots, \widehat{f}_{\dim\cD}\in \CI(\chM)^{\bbT}$ be smooth extensions of $f_1,\ldots, f_{\dim \cD}$ to $\chM$ that are locally constant near $\hH_+$.
 Let $\chD$ be the finite dimensional vector space generated by $\widehat{f}_1,\ldots, \widehat{f}_{\dim \cD}$.  Let also $\chD_{\lambda}$ be the corresponding finite dimensional vector space generated by $\widehat{f}_1|_{\epio^{-1}(\lambda)},\ldots, \widehat{f}_{\dim \cD}|_{\epio^{-1}(\lambda)}$ on the level set $\epio=\lambda$.  Extends elements of $P_{0,+}$ polyhomogeneously off $\hH_0$ so that they generate a subspace $\hP_{c,\epio,\phi,+}\subset P_{c,\epio,\phi}$ of dimension $\dim P_{0,+}$ for each $\epio$.  Let $\hP_{c,\epio,\phi,-}$ denote a complementary subspace so that 
 $$
      P_{c,\epio,\phi}= \hP_{c,\epio,\phi,+}\oplus \hP_{c,\epio,\phi,-}.
 $$

 Consider the spaces
$ (\rho_+^\nu L^{2,k+6}_{b,\epio}(\hM)^{\bbT})^{\perp_{\chD+P_{c,\epio,\phi}}}$ defined as in \eqref{c3.1}, as well as the spaces 
\begin{multline*}
 (\rho_+^\nu L^{2,k}_{b,\epio}(\hM)^{\bbT})^{\perp}= \left\{ u\in (\rho_+^\nu L^{2,k}_{b,\epio}(\hM)^{\bbT}\; | \; (\epio)_*\lrp{ u v\omega^{[m]}_{c,\epio,\phi}}= (\epio)_*\lrp{ u (\epio^{2-\nu}w)\omega^{[m]}_{c,\epio,\phi}}  =0  \right.\\
  \left. \forall \;v\in \hP_{c,\epio,\phi,+}, \; \forall \;  w \in \hP_{c,\epio,\phi,-} \right\}
\end{multline*}
and 
\begin{equation}
(\rho_+^{\nu-6} L^{2,k}_{b,\epio}(\hM)^{\bbT})^{\perp_{\bbS_{\hbeta,\epio,\phi}(\chD)+P_{c,\epio,\phi}}}:= \left\{  u\in (\rho_+^{\nu-6} L^{2,k}_{b,\epio}(\hM)^{\bbT})^{\perp}\; | \;
     (\epio)_*\lrp{ \rho_+^{8-2\nu} u \bbS_{\hbeta,\epio,\phi}(v) \omega^{[m]}_{c,\epio,\phi}}=0 \quad \forall \;v\in \chD_{\epio} \right\}.
\label{c2.77}\end{equation}
Even though an element $v\in \hD$ is supported away from $\hH_+$, notice that it is not necessarily the case for $\bbS_{\hbeta,\epio,\phi}v$ when $\hbeta\ne 0$ due to the pseudodifferential term in $\bbS_{\hbeta,\epsilon,\phi}$.  The factor $\rho^{8-2\nu}$ in \eqref{c2.77} is to ensure that $(\epio)_*\lrp{ \rho_+^{8-2\nu} u \bbS_{\hbeta,\epio,\phi}(v) \omega^{[m]}_{c,\epio,\phi}}=0$ makes sense also on $\hH_0$ when $\epio=0$ and corresponds to the $L^2_{b,s}$-projection of $\rho_+^{6-\nu}u$ onto $\rho_+^{6-\nu}\bbS_{\hbeta,\epio,\phi}(\chD_{\epio})$ with respect to the $b$-surgery metric $\rho_+^{-2}g_{c,\epio,\phi}$.
  This yields the following analog of Proposition~\ref{c3.2}.
 
\begin{proposition}\label{c2.9}
Suppose that $m=2$ and  that $\phi$ is such that $dd^c\phi$ vanishes on $\hH_0$ and $\hH_+$ as a section of $\Lambda^2({}^{c,\epio}T^*\chM)$.   Then  for  $\epio>0$ sufficiently small, the operator $\bbS_{\hbeta,\epio,\phi}$ induces an isomorphism
\begin{equation*}
  \begin{array}{lccc} \bbS_{\hbeta,\epio,\phi}: & (\rho_+^\nu L^{2,k+6}_{b,\epio}(\chM)^{\bbT})^{\perp_{\chD+P_{c,\epio,\phi}}}\oplus \chD & \to & (\rho_+^{\nu-6} L^{2,k}_{b,\epio}(\chM)^{\bbT})^{\perp_{\bbS_{\hbeta,\epio,\phi}(\chD)+P_{c,\epio,\phi}}} \oplus \bbS_{\hbeta,\epio,\phi}\chD\\
  &(u,f) & \mapsto &  \bbS_{\hbeta,\epio,\phi}(u+f)
 \end{array} 
%\label{v3.2a}
\end{equation*}
for $\nu\in ]0,2[$ and $k\in\bbN$ with inverse $\bbG_{\hbeta,\epio,\phi}\in\Psi^{-6,\cG}_{b,\epio}(\chM)$ with $\cG(\lf)\ge 0$, $\cG(\rf)> 6-\nu$, $\cG(\fb)\ge 6$ and $\cG(\mf)\ge 0$ such that
$$
  \bbG_{\hbeta,\epio,\phi} \bbS_{\hbeta,\epio,\phi}=  \bbS_{\hbeta,\epio,\phi} \bbG_{\hbeta,\epio,\phi}= \Id - \Pi_{P_{c,\epio,\phi}}.
$$
\end{proposition}
\begin{proof}
  Fix $\nu\in ]0,2[$ and consider the operator 
$$
    P_{b,s}:= \rho_+^{-\nu}(\rho_+^{6}\bbS_{\hbeta,\epio,\phi})\rho_+^{\nu}.
$$
By Proposition~\ref{la2.7}, the operator $N_{\mf}(P_{b,s})$ induces a Fredholm operator
$$
    N_{\mf}(P_{b,s}): L^{2,k+6}_b(\bM)^{\bbT}\to L^{2,k}_b(\bM)^{\bbT}
$$
with kernel corresponding to $\rho_+^{-\nu}P_{0,+}$ and cokernel corresponding to $\rho_+^{2-\nu}P_{0,+}$.  Similarly, by Proposition~\ref{lasc.6}, $N_{\fb}(P_{b,s})$ induces a Fredholm operator 
$$
   N_{\fb}(P_{b,s}): L^{2,k+6}_b(\hH_+)^{\bbT}\to L^{2,k}_b(\hH_+)^{\bbT}
$$
with kernel and cokernel identified with $\rho_0^{\nu} V$ and $\rho_0^{2-\nu}V$ respectively, where $V$ is again the space of locally constant functions on $\hH_+$.  Extends elements of $V$ polyhomogeneously off $\hH_+$ so that they generate the space
$$
   \hV= \hP_{c,\epio,\phi,-}\oplus \chD
$$
for $\epio>0$.  Let $\Pi_1$ be the $L^2$-projection onto 
$$
\rho_0^{\nu}\chD\oplus \rho_0^{\nu}\hP_{c,\epio,\phi,-} \oplus \rho_+^{-\nu}P_{c,\epio,\phi,+}
$$
for each $\epio>0$ with respect to the $b$-surgery metric $\rho_+^{-2}g_{c,\epio,\phi}$ associated to the conical surgery metric $g_{c,\epio,\phi}$ with K\"ahler form $\omega_{c,\epio,\phi}$.  Similarly, let $\Pi_2$ be the $L^2_{b,s}$-projection (for each fixed $\epio$)  onto 
$$
     \rho_+^{\nu-2}\hP_{c,\epio,\phi,+}\oplus \rho_0^{2-\nu}\hP_{c,\epio,\phi,-}\oplus \rho_+^{-\nu+6}\bbS_{\hbeta,\epio,\phi}(\chD)
$$
with respect to $\rho_+^{-2}g_{c,\epio,\phi}$.  Notice that  $\bbS_{\hbeta,\epio,\phi}(\chD)$ restricts on $\hH_0$ to a space of functions of dimension $\dim \cD$.   Indeed, if 
$$
     \bbS(u)=0 \quad \mbox{for some} \; u\in \chD|_{\hH_0},
$$
then since $u$ is compactly supported on $\bM\setminus \pa \bM$ and $u$ is smooth in the sense of orbifold, we can integrate by parts as in the proof of Lemma~\ref{la.15} to deduce that $u\in P_{0}$.  Since $\cD$ is the complement of $\iota_+(P_{0})$ in $\bbR^{\ell}$, this means that $u=0$.  Hence,  $\bbS_{\hbeta,\epio,\phi}(\chD)|_{\hH_0}$ does define a subspace of functions of dimension $\dim \cD$.  

By Proposition~\ref{rob.1}, there is an operator $G\in \Psi^{-6,\cH}_{b,s}(\chM)$ such that 
$$
 (\Id-\Pi_2)P_{b,s}(\Id-\Pi_1)G= \Id-\Pi_2 \quad \mbox{and} \quad   G(\Id-\Pi_2)P_{b,s}(\Id-\Pi_1)= \Id-\Pi_1
$$
for some real index family $\cH$ such that
$$
   \cH(\lf)>0, \quad \cH(\rf)>0, \quad \cH(\fb)\ge 0 \quad \mbox{and} \quad \cH(\mf)\ge 0.
$$
Hence, the result follows by taking 
$$
  \bbG_{\hbeta,\epio,\phi}= \rho_+^{\nu}G \rho_+^{6-\nu}+R_{\epio}
$$
with $R_{\epio}$ the inverse of the map
$$
    \bbS_{\hbeta,\epio,\phi}: \chD_{\epio}\to \bbS_{\hbeta,\epio,\phi}(\chD_{\epio}).
$$
\end{proof}

As in the case $m\ge 3$, this has the following consequence.
\begin{corollary}
Suppose that $m=2$. If $\phi$ is as in Theorem~\ref{fs.17} and $\nu\in ]0,2[$, then the operator
\begin{equation*}
(\Id-\Pi_{P_{c,\epio,\phi}})\circ\check{\bbS}_{\hbeta,\epio,\phi}: (\rho_+^\nu L^{2,k+6}_{b,\epio}(\hM)^{\bbT})^{\perp_{\chD+P_{c,\epio,\phi}}}\oplus \chD_{\epio}  \to (\rho_+^{\nu-6} L^{2,k}_{b,\epio}(\hM)^{\bbT})^{\perp_{\bbS_{\hbeta,\epio,\phi}(\chD)+P_{c,\epio,\phi}}} \oplus \bbS_{\hbeta,\epio,\phi}(\chD_{\epio})
\end{equation*}
is a bijection for $\epio>0$ small enough.  When acting on $\bbT$-invariant functions, there is in fact an inverse $\check{\bbG}_{\hbeta,\epio,\phi}\in \Psi^{-6,\check{\cG}}_{b,\epio}(\chM)$ such that
\begin{equation*}
   \check{\bbG}_{\hbeta,\epio,\phi}\check{\bbS}_{\hbeta,\epio,\phi}= \check{\bbS}_{\hbeta,\epio,\phi} \check{\bbG}_{\hbeta,\epio,\phi}= \Id - \Pi_{P_{c,\epio,\phi}}
%\label{c3.4}
\end{equation*}
 with index family $\check{\cG}$ such that
 $$
 \check{\cG}(\lf)\ge 0, \quad \check{\cG}(\rf)\ge2m, \quad \check{\cG}(\fb)\ge 6, \quad \mbox{and} \quad \check{\cG}(\mf)\ge 0.
 $$  
 \label{c2.10}\end{corollary}

\subsection{Non-linear analysis and the extremal solution on the resolution}\label{nla.0}

  Let $\phi$ be the formal solution of Theorem~\ref{fs.17} and  set 
$$
      \homega_{c,\epio}= \omega_{c,\epio,\phi},
$$
so that 
$$
   \mathring{\Sc}^{\hbeta}(\homega_{c,\epio})\in \dot{\cC}^{\infty}(\chM).
$$
If we set 
$$
  \homega_{c,\epio,u}:= \homega_{c,\epio}+ dd^cu
$$
for $u$ a function, we want to find $u$ such that 
\begin{equation}
   \mathring{\Sc}^{\hbeta}(\homega_{c,\epio,u})=0.
\label{nla.1a}\end{equation}
To solve this equation, let us first notice that
\begin{equation}
   \mathring{\Sc}^{\hbeta}(\homega_{c,\epio,u})=  \mathring{\Sc}^{\hbeta}(\homega_{c,\epio})+ \check{\bbS}_{\hbeta,\epio,\phi}(u) +Q(u),
\label{nla.1}\end{equation}
where $Q(u)$ is nonlinear in $u$ and $\check{\bbS}_{\hbeta,\epio,\phi}$ is defined in \eqref{MM.4}.
\begin{lemma}
The nonlinear term $Q(u)$ is of the form 
\begin{multline*}
Q(u)=   B_{2,2}(\nabla^2u,\nabla^2u; \nabla^2u) +
\underset{i+j+k\le 10}{\sum_{i,j,k\ge 2}} B_{i,j,k}(\nabla^iu,\nabla^ju,\nabla^ku) 
+ B_{3,3,3,3}(\nabla^3u,\nabla^3u,\nabla^3u,\nabla^3u)
\label{nla.2b}\end{multline*}
where the functions $B_{2,2}$, $B_{i,j,k}$ and $B_{3,3,3,3}$ are smooth in the last entry and linear in the other entries.  
\label{nla.2}\end{lemma}
\begin{proof}
For the terms in \eqref{la.3} coming from the scalar curvature, this comes from \cite[(21),(22)]{Arezzo-Pacard2009}.  The term involving the Ricci tensor has a nonlinear behaviour in $u$ similar to the one coming from the square of the scalar curvature.  For the last term, notice that
$$
     \theta_{c,\epio,\phi+u}= \theta_{c,\epio,\phi}+ dd^c v
$$
for some function $v$ such that 
$(\epio)_*(v \homega^{[m]}_{c,\epio,u})=0.
$
Since $\hbeta$ is a primitive class, 
$$
   0= \Lambda_{c,\epio,\phi+u}\theta_{c,\epio,\phi+u}= \Lambda_{c,\epio,\phi+u}\theta_{c,\epio,\phi}- \Delta_{c,\epio,\phi+u}v,
$$
so
$ v= \bbG_{c,\epio,\phi+u} \Lambda_{c,\epio,\phi+u}\theta_{c,\epio,\phi}
$
and 
$$
    \theta_{c,\epio,\phi+u}= \theta_{\homega_{c,\epio}}+ dd^c\bbG_{c,\epio,\phi+u} \Lambda_{c,\epio,\phi+u}\theta_{\homega_{c,\epio}}.
$$
Now, one can check that 
\begin{multline*}
\bbG_{c,\epio,\phi+u}=  \lrp{\Id -\Pi_{c,\epio,\phi+u} -\lrp{\Id-\Pi_{c,\epio,\phi+u}}\lrp{\Id-\Pi_{c,\epio,\phi}}  }^{-1}(\Id-\Pi_{c,\epio,\phi}) \\ \lrp{\Id - \Delta_{c,\epio,\phi}^{-1}(\Delta_{c,\epio,\phi}-\Delta_{c,\epio,u,\phi})  }^{-1}
 \bbG_{c,\epio,\phi} \lrp{\Id -\Pi_{c,\epio,\phi}},
\end{multline*}
so $dd^cu\mapsto \bbG_{c,\epio,\phi+u}$ can be seen as an operator valued functional  smooth in terms of $d d^c u$.  Hence, the contribution to $Q$ in the last term in \eqref{la.3} is included in the term
$$
 B_{2,2}(\nabla^2u,\nabla^2u; \nabla^2u).
$$
\end{proof}

Now, if we have $u$ such that \eqref{nla.1a} holds, then
\begin{equation}
  \lrp{\Id -\Pi_{P_{\homega_{c,\epio}}}} \mathring{\Sc}^{\hbeta}(\homega_{c,\epio,u})=0,
\label{nla.3}\end{equation}
so that from \eqref{nla.1}, we obtain that
\begin{equation*}
\lrp{\Id -\Pi_{P_{\homega_{c,\epio}}}} \check{\bbS}_{\hbeta,\epio,\phi}(u)=  -\lrp{ \Id -\Pi_{P_{\homega_{c,\epio}} }}  \lrp{ \mathring{\Sc}^{\hbeta}(\homega_{c,\epio})+ Q(u) }.
\label{nla.4}\end{equation*}
Applying the inverse $\check{\bbG}_{\hbeta,\epio,\phi}$ of Corollaries~\ref{MM.6}, \ref{c3.5} or \ref{c2.10} depending on whether $m>3$, $m=3$ or $m=2$, we obtain that 
\begin{equation*}
  u= -\check{\bbG}_{\hbeta,\epio,\phi}  \lrp{ \mathring{\Sc}^{\hbeta}(\homega_{c,\epio})+ Q(u) },
\label{nla.5}\end{equation*}
so that a solution to \eqref{nla.3} corresponds to a fixed point of the operator
\begin{equation*}
   \widehat{\cN}(u):= -\check{\bbG}_{\hbeta,\epio,\phi}  \lrp{ \mathring{\Sc}_{\hbeta}(\homega_{c,\epio})+ Q(u) }.
\label{nla.6}\end{equation*}
To obtain a solution, let us therefore show  that $  \widehat{\cN}$ is a contraction when acting on a suitable space.  
\begin{lemma}
For $k>m$, multiplication on continuous functions extends to a uniformly continuous (in $\epio$)  bilinear map 
$$
          \begin{array}{ccc}
            L^{2,k}_{b,\epio}(\hM)\times L^{2,\epio}_{b,\epio}(\hM) & \to & L^{2,k}_{b,\epio}(\hM) \\
            (u,v) & \mapsto & uv
          \end{array}
$$
such that 
$$
    \| uv\|_{L^{2,k}_{b,\epio}(\hM)} \le C_k \| u\|_{L^{2,k}_{b,\epio}(\hM)} \| v\|_{L^{2,k}_{b,\epio}(\hM)}
$$
for a constant $C_k>0$ (independent of $\epio$), as well as to a continuous bilinear map
$$
          \begin{array}{ccc}
            \cC^q(]0,\epio_0[;L^{2,k}_{b,s}(\chM))\times \cC^q(]0,\epio_0[;L^{2,k}_{b,s}(\chM)) & \to & \cC^q(]0,\epio_0[; L^{2,k}_{b,s}(\chM)) \\
            (u,v) & \mapsto & uv
          \end{array}
$$
such that 
$$
    \| uv\|_{\cC^qL^{2,k}_{b,s}} \le C_{q,k} \| u\|_{\cC^qL^{2,k}_{b,s}} \| v\|_{\cC^qL^{2,k}_{b,s}}
$$
for a constant $C_{q,k}$ depending on $q,k\in\bbN$.
\label{nla.7}\end{lemma}
\begin{proof}
On $\bbR^{2m}$, this is well-known and follows from the Sobolev embedding and the Sobolev inequality.  Using coordinate charts and a partition of unity on $\hM$, the first assertion can be reduced to this case.  Moreover the constant $C_k$ can be chosen to be independent of $\epio>0$ as long as $\epio$ is bounded above.  The second assertion follows from the first and induction on $q$ using Leibniz rule.
\end{proof}
If $m>3$, consider for $q,k\in \bbN$, $\mu\ge 0$, $\nu\in ]6-2m,0[$ and $\epio_0\in]0,1[$ the Banach space
$$
      \cB_{q,k,\mu,\nu,\epio_0}:=\{ u\in \epio^{\mu}\rho_+^{\nu}\cC^q(]0,\epio_0[;L^{2,k}_{b,s}(\chM)) \; |\;  \epio^{-\mu}\rho_+^{-\nu}u|_{\epio=\lambda}\in (L^{2,k}_{b,\lambda}(\hM)^{\bbT} )^{\perp} \quad \forall \lambda\in]0,\epio_0[\}
$$
with norm 
$$
    \|u\|_{\cB_{q,k,\mu,\nu,\epio_0}}= \| \epio^{-\mu}\rho_+^{-\nu}u\|_{\cC^qL^{2,k}_{b,s}}.
$$
For $m\in\{2,3\}$, consider instead for $\nu\in]0,1[$, $q,k\in\bbN$, $\mu\ge 0$ and $\epio_0\in]0,1[$ the Banach spaces
$$
  \cB^1_{q,k,\mu,\nu,\epio_0}:= \{ u\in \epio^{\mu}\rho_+^{\nu}\cC^q(]0,\epio_0[;L^{2,k}_{b,s}(\chM)) \; |\;  \epio^{-\mu}\rho_+^{-\nu}u|_{\epio=\lambda}\in (L^{2,k}_{b,\lambda}(\hM)^{\bbT} )^{\perp_{\chD+P_{\widehat{\omega}_{c,\epio}}}} \quad \forall \lambda\in]0,\epio_0[\}
$$
with norm
$$
    \|u\|_{\cB^1_{q,k,\mu,\nu,\epio_0}}= \| \epio^{-\mu}\rho_+^{-\nu}u\|_{\cC^qL^{2,k}_{b,s}},
$$
$$
\cB^2_{q,k,\mu,\nu,\epio_0}:= \{ u\in \epio^{\mu}\cC^q(]0,\epio_0[;L^{2,k}_{b,s}(\chM)) \; |\;  \epio^{-\mu}u|_{\epio=\lambda}\in \chD_{\lambda} \quad \forall \lambda\in]0,\epio_0[\}
$$
with norm
$$
   \|u\|_{\cB^2_{q,k,\mu,\nu,\epio_0}}= \| \epio^{-\mu}u\|_{\cC^qL^{2,k}_{b,s}}  
$$
and 
$$
\cB_{q,k,\nu,\epio_0}:= \cB^1_{q,k,\nu,\epio_0}\oplus \cB^2_{q,k,\nu,\epio_0}
$$
with norm induced by those of $\cB^1_{q,k,\nu,\epio_0}$ and $\cB^2_{q,k,\nu,\epio_0}$.

For $c>0$, set 
\begin{equation*}
  \cU_{q,k,\mu,\nu, \epio_0,c}:= \{ u\in \cB_{q,k,\mu,\nu,\epio_0} \; | \;  \| u\|_{\cB_{q,k+6,\mu,\nu,\epio_0}}<c\}.
\label{nla.9}\end{equation*}
The Banach algebra property of Lemma~\ref{nla.7} ensures, as in \cite{LeBrun-Simanca1994},  that the functional
$$
\begin{array}{lccc}
          Q: &  \cU_{q,k,\mu,\nu,\epio_0,c} & \to & \epio^{\mu}\rho_+^{\nu-6}\cC^q(]0,\epio_0[;L^{2,k-6}_{b,s}(\chM)) \\
        &     q & \mapsto & Q(u)
          \end{array} 
          $$
is well-defined for 
$$\nu\in  \left\{ \begin{array}{ll} ]6-2m,0[, & m>3, \\ ]0,1[, & m\in\{2,3\},
\end{array}\right.
$$ and $k>m+6$ with $\mu+\nu\ge 2$ provided that $c>0$ is small enough.  For $m\in\{2,3\}$, we also use the fact that the derivatives of elements of $\chD$ are bounded with respect  to $\omega_{c,\epio}$.  From Corollaries~\ref{MM.6}, \ref{c3.5} and \ref{c2.10} as well as from Lemma~\ref{nla.2}, we also see that there is a constant $C_{q,k,\mu,\nu}>0$ such that 
\begin{equation}
\| \check{\bbG}_{\hbeta,\epio,\phi} \lrp{  Q(u_1)-Q(u_2)}\|_{\cB_{q,k,\mu,\nu,\epio_0}}\le C_{q,k,\mu,\nu} \lrp{ 
\|u_1\|_{\cB_{q,k,\mu,\nu,\epio_0}} +   \|u_2\|_{\cB_{q,k,\mu,\nu,\epio_0}} } \lrp{ \|u_1-u_2\|_{\cB_{q,k,\mu,\nu,\epio_0}} }
\label{nla.8}\end{equation}
for all $u_1, u_2\in \cU_{q,k,\mu,\nu,\epio_0,c}$.  
\begin{proposition}
For $m>3$, pick $\nu\in ]6-2m, 0[$, while for $m\in\{2,3\}$, pick $\nu\in ]0,1[$.  Then for $k>m+6$ and $\mu>2-\nu$, there exists $c>0$ and $\epio_0>0$ such that $\widehat{\cN}$ induces a contraction
$$
   \widehat{\cN}: \cU_{q,k,\mu,\nu,\epio_0,c}\to \cU_{q,k,\mu,\nu,\epio_0,c}
$$ 
for $\epio<\epio_0$.  
\label{nla.10}\end{proposition} 
\begin{proof}
If $C_{q,k,\mu, \nu}$ is the constant in \eqref{nla.8}, then choose $c>0$ such that
\begin{equation}
       c<\frac{1}{4C_{q,k,\mu,\nu}}.
\label{nla.10a}\end{equation}
Then by \eqref{nla.8}, we see that for $u_1,u_2\in \cU_{q,k,\mu,\nu,\epio_0,c}$, 
\begin{equation}
\begin{aligned}
   \| \widehat{\cN}(u_1)- \widehat{\cN}(u_2) \|_{\cB_{q,k,\mu,\nu,\epio_0}} &= \| \check{\bbG}_{\hbeta,\epio,\phi}\lrp{ Q(u_1)- Q(u_2)} \|_{\cB_{q,k,\mu,\nu,\epio_0}} \\ 
   & \le C_{q,k,\mu,\nu}(2c)  \|u_1-u_2\|_{\cB_{q,k,\mu,\nu,\epio_0}} \\
     & \le \frac12  \|u_1-u_2\|_{\cB_{q,k,\mu,\nu,\epio_0}}.
\end{aligned}   
\label{nla.11}\end{equation}
On the other hand, from the rapid decay of $\mathring{\Sc}^{\hbeta}(\homega_{c,\epio})$, we can find $\epio_0>0$ such that
\begin{equation}
   \|  \check{\bbG}_{\hbeta,\epio,\phi} \mathring{\Sc}_{\hbeta}(\homega_{c,\epio}) \|_{\cB_{q,k,\mu,\nu}}<\frac{c}{2}
\label{nla.12}\end{equation}
for $\epio<\epio_0$.  This means that for such a $\epio_0$ and  $u\in \cU_{q,k,\mu,\nu,\epio_0,c}$, 
\begin{equation}
\begin{aligned}
\| \widehat{\cN}(u)  \|_{\cB_{q,k,\mu,\nu,\epio_0}} & \le  \|  \check{\bbG}_{\hbeta,\epio,\phi} \mathring{\Sc}^{\hbeta}(\homega_{c,\epio}) \|_{\cB_{q,k,\mu,\nu,\epio_0}} + \|  \check{\bbG}_{\hbeta,\epio,\phi}\lrp{ Q(u)- Q(0)} \|_{\cB_{q,k,\mu,\nu,\epio_0}} \\
 & < \frac{c}{2}+ C_{q,k,\mu, \nu} c^2 \quad \mbox{by \eqref{nla.8} and \eqref{nla.12}}, \\
    &< \frac{c}{2}+ \frac{c}{4}< c, \quad \mbox{by \eqref{nla.10a}}.
\end{aligned}
\label{nla.13}\end{equation}
Thus, \eqref{nla.11} and \eqref{nla.13} confirms that we have found $c$ and $\epio_0>0$ such that 
$$
   \widehat{\cN}: \cU_{q,k,\mu,\nu,\epio_0,c}\to \cU_{q,k,\mu,\nu,\epio_0,c}
$$ 
induces a contraction.  
\end{proof}

This allows us to achieve our gluing construction as follows. 
\begin{theorem}\label{thm:general}{Let $(M, \omega_0, \bbT)$ be a compact $\Sc^{\beta}$-extremal orbifold  of dimension  $m\ge 2$ and of nonnegative scalar curvature, with only isolated singular points,  each of which satisfying Assumption~\ref{a:overall}.} Let $\phi$ be the formal solution of Theorem~\ref{fs.17}.  Then there exists $\epio_0>0$ and $u\in \dot{\cC}^{\infty}(\chM)$ such that $\homega_{c,\epio,u}$ is a polyhomogeneous family of $\Sc^{\hbeta}$-extremal metrics on $\widehat{M}$ for $\epio<\epio_0$.
\label{nla.14}\end{theorem}
\begin{proof}
If $m>3$, pick $\nu\in ]6-2m,0[$, while if $m\in \{2,3\}$, pick $\nu\in]0,1[$.  Pick $k>m+6$ and $\mu>2-\nu$ and choose $q\in\bbN$.
Choosing $c>0$ and $\epio_0>0$ as in Proposition~\ref{nla.10}, we know by the Banach fixed point theorem that there exists a unique $u\in \cU_{q,k,\mu,\nu,\epio_0,c}$ such that
$$
    \widehat{\cN}(u)=u \quad \Longrightarrow \quad  \lrp{\Id -\Pi_{P_{\homega_{c,\epio}}}} \mathring{\Sc}^{\hbeta}(\homega_{c,\epio,u})=0.
$$
Since $\Pi_{P_{\homega_{c,\epio,u}}}$ is asymptotically close to $\Pi_{P_{\homega_{c,\epio}}}$ as $\epio\searrow 0$ and since on the other hand
$$
  \Pi_{P_{\homega_{c,\epio,u}}} \mathring{\Sc}^{\hbeta}(\homega_{c,\epio,u})=0,
$$
this means that taking $\epio_0$ smaller if needed, we have that
$$
    \mathring{\Sc}^{\hbeta}(\homega_{c,\epio,u})=0.
$$
As in \cite{ALL2026}, we conclude that for fixed $\epio$, $u$ is smooth on $\chM$.  By the Banach space implicit function theorem and local uniqueness, the family $u$ is also smooth in $\epio$ for $\epio>0$ small enough.  %\fr{I think this is correct}

Taking $q,k\in\bbN$ and $\mu>2-\nu$ larger, we can get another solution $\widetilde{u}$.  Taking $\epio_0>0$ smaller if needed, we can assume that $\widetilde{u}$ is contained in the original set $\cU_{q,k,\mu,\nu,\epio_0,c}$ used to construct $u$.  By the uniqueness of the fixed point in the Banach fixed point theorem, $u$ and $\widetilde{u}$ must agree.  Since $q$, $k$ and $\mu$ can be taken as large as we want, this implies that the derivatives of $u$ in $\epio$ all exist at $\epio=0$ and vanish.  
This shows that the solution $u$ we have produced is in fact in $\dot{\cC}^{\infty}(\chM)$ (for $\epio<\epio_{0}$).  In particular, the family of extremal metrics $\homega_{c,\epio,u}$ is polyhomogeneous as a conical surgery metric.
\end{proof}

We can show that the scalar curvature of these extremal metrics is positive in two different situations.  The first one is when the $\ALE$-metrics $\omega_i$ are Ricci-flat.

\begin{corollary}\label{c:Ricci-flat ALE}
{Suppose furthermore in the hypothesis of Theorem~\ref{thm:general} that the $\Sc^{\beta}$-extremal metric $\omega_0$ on $M$ has positive scalar curvature and that the $\Sc^0$-extremal metrics $\omega_i$ on $X_i$ are Ricci-flat.}  Then there exists $\epio_0$ so that the corresponding $\Sc^{\hbeta}$-extremal K\"ahler metrics $\homega_{c,\epio,u}$ on $\widehat{M}$  have positive scalar curvature for $\epio<\epio_0$.
\label{pos.1}\end{corollary}  
\begin{proof}
Since $R_{\omega_i}=0$ and $f_0\in r^4\CI(M)$ in \eqref{fs.4a}, notice that the scalar curvature $R_{\omega_{c,\epio}}$ of the family of K\"ahler metrics $\omega_{c,\epio}$ is bounded.  On the other hand, by \cite[Theorem~8.2.3]{Joyce},  $\sigma_i=2m\ge 4$ in Definition~\ref{ALE.1}, so
$$
    \psi_i-\frac{\rho_+^2}{4}= \cO(\epio^4) \quad \mbox{as} \quad \epio\searrow 0,
$$
for the function $\psi_i$ in \eqref{fs.4a}.  This ensures in particular that 
$$
     \rho_+^4\Sc^{\hbeta}(\omega_{c,\epio})= \cO(\rho_+^2\rho_0^4)= \cO(\epio^2\rho_0^2).
$$
This means that the construction of the formal solution in the proof of Theorem~\ref{fs.17} can start by applying Lemma~\ref{c2.4} with $\kappa=\frac32$, which ensures that the formal solution $\phi$ in Theorem~\ref{fs.17} is such that $\phi=o(\epio^{\frac32})$ with 
$$
     \frac{dd^c\phi}{\epio^{\frac32}}\in \cA_{\phg}(\chM;\Lambda^2({}^{c,\epio}T^*\chM))\cap L^{\infty}(\chM;\Lambda^2({}^{c,\epio}T^*\chM))
$$
vanishing on $\hH_0$ and $\hH_+$.
From this, we cannot deduce that $R_{c,\epio,\phi}$ is bounded, but we can deduce at least that
$$
    R_{c,\epio,\phi}= o(\rho_+^{-\frac12}).
$$
Similarly, we find that 
$$
   \| \Ric_{c,\epio,\phi}\|_{c,\epio,\phi}= o(\rho_+^{-\frac12}) \quad \mbox{and} \quad \| \theta_{c,\epio,\phi}\|_{c,\epio,\phi}= o(\rho_+^{-\frac12}).
$$

Since $u\in\dot{\cC}^{\infty}(\chM)$ in Theorem~\ref{nla.14}, notice that
\begin{multline*}
   R_{c,\epio,\phi+u}- R_{c,\epio,\phi}= o(1), \quad \|\Ric_{c,\epio,\phi+u}\|_{\homega_{c,\epio,u}}- \|\Ric_{c,\epio,\phi}\|_{\omega_{c,\epio,\phi}}= o(1) \\
   \mbox{and} \quad \| \theta_{c,\epio,\phi+u}\|_{\homega_{c,\epio,u}}-\| \theta_{c,\epio,\phi}\|_{\omega_{c,\epio,\phi}}=o(1),
\label{pos.2}\end{multline*}
so the same estimates hold for $\homega_{c,\epio,u}$, namely
\begin{equation}
R_{c,\epio,\phi+u}= o(\rho_+^{-\frac12}), \quad \| \Ric_{c,\epio,\phi+u}\|_{\homega_{c,\epio,u}}= o(\rho_+^{-\frac12}) \quad \mbox{and} \quad \| \theta_{c,\epio,\phi+u}\|_{\homega_{c,\epio,u}}= o(\rho_+^{-\frac12}).
\label{pos.3}\end{equation}

Moreover, $R_{c,\epio,\phi+u}$, $\| \Ric_{c,\epio,\phi+u}\|^2_{\homega_{c,\epio,u}}$ and $\| \theta_{c,\epio,\phi+u}\|^2_{\homega_{c,\epio,u}}$ are polyhomogeneous.  On the other hand, by construction,  
$$
     \Sc^{\hbeta}(\homega_{c,\epio,u})- \ell^{\ext}_{\halpha_{\epio}, \hbeta}(\mu_{c,\epio,\phi+u})= \mathring{\Sc}^{\hbeta}(\homega_{c,\epio,u})=0,
$$
so $\Sc^{\hbeta}(\homega_{c,\epio,u})$
is bounded by Lemma~\ref{fs.13}.    Since 
\begin{equation}
\Sc^{\hbeta}(\homega_{c,\epio,u})= \Delta_{c,\epio,\phi+u} R_{c,\epio,\phi+u}+ \frac{R^2_{c,\epio,\phi+u}}{2} -
 \| \Ric_{c,\epio,\phi+u}\|^2_{\homega_{c,\epio,u}} - \| \theta_{c,\epio,\phi+u}\|^2_{\homega_{c,\epio,u}},
\label{pos.4}\end{equation}
this means that the potential singular terms  in the expansions of the terms on the right hand side of \eqref{pos.4} must cancel out.      By \eqref{pos.3}, this can be rewritten more simply as 
\begin{equation}
\Sc^{\hbeta}(\homega_{c,\epio,u})= \Delta_{c,\epio,\phi+u} R_{c,\epio,\phi+u}+ o(\rho_+^{-1}).
\label{pos.5}\end{equation}

This can be used to deduce that $R_{c,\epio,\phi+u}$ is bounded.  Indeed, suppose not.  Then, let $\delta$ be the smallest number in $]-\frac12,0[$ such that $R_{c,\epio,\phi+u}$ has a non-trivial term of the form
$$
  v_{\delta}= \sum_{j=0}^{k} v_{\delta,j}\rho_+^{\delta}(\log\rho_+)^k \quad \mbox{for some} \quad v_{\delta,j}\in \cA_{\phg}(\hH_+)\cap L^{\infty}(\hH_+) \quad \mbox{and} \quad k\in\bbN
$$
in its polyhomogeneous expansion at $\hH_+$.  As in the proof of Lemma~\ref{c2.4}, this term can be rewritten
$$
      v_{\delta}= \sum_{j=0}^k \widetilde{v}_{\delta,j}\epio^\delta(\log\epio)^j
$$
with 
$$
     \widetilde{v}_{\delta,j}:= \sum_{p=0}^{k-j} (-1)^p \lrp{\begin{array}{c} j+p \\ p \end{array}} v_{\delta,p+j} \rho_0^{-\delta} (\log\rho_0)^p\in L^2_b(\hH_+).
$$
In particular, near $\hH_i$ for $i>0$,
$$
\begin{aligned}
\Delta_{c,\epio,\phi+u}v_{\delta}&= \epio^{-2}(\epio^2\Delta_{c,\epio,\phi+u} v_{\delta} )= \sum_{j=0}^{k} \epio^{-2+\delta}(\log\epio)^j (\epio^2\Delta_{c,\epio,\phi+u}\widetilde{v}_{\delta,j}) \\
 &= \sum_{j=0}^k \epio^{-2+\delta}(\log\epio)^j \Delta_{i} \widetilde{v}_{\delta,j}+ o(\rho_+^{\delta-2}).
\end{aligned}
$$
From \eqref{pos.5} and the fact that $\Sc^{\hbeta}(\homega_{c,\epio,u})$ is bounded, we deduce that $\Delta_{i} \widetilde{v}_{\delta,j}=0$ for each $i>0$ and each $j$.  Since $\widetilde{v}_{\delta,j}$ decays at infinity, we deduce from Lemma~\ref{la.28} that in fact $\widetilde{v}_{\delta,j}=0$ for all $j$, contradicting our assumption that $v_{\delta}$ was a non-trivial term in the partial polyhomogeneous expansion of $R_{c,\epio,\phi+u}$ at $\hH_+$.  To avoid a contradiction, we must therefore conclude that $R_{c,\epio,\phi+u}$ is bounded. 

If $v_0:= R_{c,\epio,\phi+u}|_{\hH_+}$ is the restriction of $R_{c,\epio,\phi+u}$ to $\hH_+$, then \eqref{pos.5} and the boundedness of $\Sc^{\hbeta}(\homega_{c,\epio,u})$ also imply that 
$$
     \Delta_{i}\lrp{v_0|_{\hH_i}}=0 \quad \mbox{for each} \; i\in \{1,\ldots,\ell\}.
$$
By Lemma~\ref{la.28}, this implies that $v_0|_{\hH_i}$ must be constant.  By continuity of $R_{c,\epio,\phi+u}$, we conclude that $v_0$ is equal to $R_{0}(p_i)$ on $\hH_i$ since $R_{c,\epio,\phi+u}|_{\hH_0}$ is equal to the lift of $R_0$ to $\bM=\hH_0$.  Since $R_{0}$ is positive by assumption, this means that $R_{c,\epio,\phi+u}|_{\pa\chM}$ is positive.  By continuity of  $R_{c,\epio,\phi+u}$, this means that there exists $\epio_0>0$ such that $R_{c,\epio,\phi+u}$ is positive for $\epio<\epio_0$.
\end{proof}

When the $\ALE$-metrics $\omega_i$ are not Ricci-flat, we can also ensure positivity of the scalar curvature making the following assumption.

\begin{assumption} {Suppose that $m\ge 3$ and  that in the Assumption~\ref{a:overall} the ALE K\"ahler metric $\omega_i$ on each $X_i$ has \emph{positive} scalar curvature.  Suppose also that the function $f_i$ in Definition~\ref{ALE.1}  satisfies 
\[ f_i = \begin{cases} c_i\|z\|^{-2}\log\|z\|+o(\|z\|^{-2}) \,  \text{with} \,  c_i\ne 0 \, \text{when} \; m=3; \\
\cO(\|z\|^{-\sigma_i}) \, \, \text{with}  \, \,\sigma_i=2(m-1)-2 \, \text{when} \; m>3.
\end{cases}
\]
Finally,  suppose that the top order term of the asymptotic expansion of the scalar curvature of $\omega_i$ at infinity is of the form $b_i\|z\|^{-2(m-1)}$ for some positive constant $b_i$. %\fr{change of notation at some point...}
\label{psc.1}}\end{assumption}

However, we need first to derive some preliminary results about the family of metrics $\omega_{c,\epio}$.  
\begin{lemma}
If Assumption~\ref{psc.1} holds, then near $\hH_i\cap\hH_0$, the scalar curvature $R_{c,\epio}$ of $\omega_{c,\epio}$ is such that
$$
 \rho_+^2 R_{c,\epio}= a_i\rho_+^2+ b_i\rho_0^{2m-4}+ o(\rho_+^2+ \rho_0^{2m-4})
$$
with $a_i=R_{0}(p_i)>0$ the value of the scalar curvature of $\omega_0$ at $p_i$ and $b_i$ the positive constant of Assumption~\ref{psc.1}.
\label{psc.2}\end{lemma}
\begin{proof}
By \eqref{fs.5}, we assume that near $\hH_i\cap\hH_0$, $\chi_{\delta}(\rho_+)=\chi_{\delta}(\rho_0)=1$, so that 
\begin{equation}
\begin{aligned}
\omega_{c,\epio}&= dd^c\psi_i= dd^c(\rho_+^2(\frac14+f_i)+ f_0) \\
&= \frac14 dd^c\rho_+^2+ \epio^2dd^c\rho_i^{-2}f_i+ dd^cf_0.
\end{aligned}
\label{psc.3}\end{equation}
By \cite[Proposition~4.1 and (70)]{Arezzo-Pacard2006}, near $\hH_i\cap \hH_0$,
$$
R_{c,\epio}= \frac12\Delta^2_{\Euc}(\epio^2\rho_i^{-2}f_i+ f_0)+ \cO(|\nabla^2(\epio^2\rho_i^{-2}f_i+f_0)|_{g_{Euc}}^2),
$$
where $\Delta_{\Euc}$ is the Laplacian of the Euclidean metric $g_{\Euc}$ with K\"ahler form $dd^c \rho_+^2$.  From Assumption~\ref{psc.1} and the fact that $f_0\in r^4\CI(M)$, 
$$
\frac12\Delta^2_{\Euc}(\epio^2\rho_i^{-2}f_i+ f_0)=  \widetilde{c}_i\rho_+^{-2}\rho_0^{2m-4} +d_i + o(\rho_+^{-2}\rho_0^{2m-4}+1)
$$
for some bounded function $\widetilde{c}_i$ and $d_i$ on $\hH_i\cap\hH_0$.  On the other hand,
$$
|\nabla^2(\epio^2\rho_i^{-2}f_i+f_0)|_{\Euc}^2= \cO((\rho_0^{2m-4}+\rho_+^2)^2),
$$
so 
$$
\rho_+^2R_{c,\epio}= \widetilde{c}_i\rho_0^{2m-4} +d_i\rho_+^2+ o(\rho_0^{2m-4}+\rho_+^{2}).
$$
Since $R_{c,\epio}|_{\hH_0}=R_{0}$, we deduce that $d_i=a_i=R_{0}(p_i)$.  Similarly, since $\epio^2R_{c,\epio}|_{\hH_i}= R_{i}$, we deduce that $\widetilde{c}_i$ is equal to the constant $b_i$ introduced in Assumption~\ref{psc.1}.
\end{proof}

When $\omega_0$ has positive scalar curvature, the previous result allows us to conclude that $\omega_{c,\epio}$ has positive scalar curvature for $\epio>0$ sufficiently small. To establish the same result for the extremal metrics, we will need the following estimate for $\mathring{\Sc}^{\hbeta}(\omega_{c,\epio})$.
\begin{lemma}
If Assumption~\ref{psc.1} holds, then 
$
\rho_+^2\mathring{\Sc}^{\hbeta}(\omega_{c,\epio})= \cO(\epio^2\rho_0^{2m-6}).
$
\label{psc.4}\end{lemma}
\begin{proof}
The proof consists in looking carefully at the contribution of each term occurring in the definition of the functional $\Sc^{\hbeta}(\omega_{c,\epio})$.  First, from Lemma~\ref{psc.2} and \eqref{psc.3}, we see that
$$
\rho_+^2 \Delta_{\omega_{c,\epio}}R_{\omega_{c,\epio}}+ \frac12 (\rho_+^2R_{c,\epio})^2= \cO(\rho_+^2+ \rho_0^{2m-4})
$$
near $\hH_+\cap\hH_0$.  From \eqref{psc.3} and the formula for the Ricci form
$$
  \rho_{c,\epio}=-\frac12 dd^c\log\lrp{\frac{\omega_{c,\epio}^m}{\omega^m_{\Euc}}},
$$
we also see that
$$
 \|\rho_+^2\Ric_{\omega_{c,\epio}}\|^2_{\omega_{c,\epio}}=\cO((\rho_+^2+\rho_0^{2m-4})^2)= \cO(\rho_+^4+ \rho_0^{4m-8})
$$
near $\hH_+\cap\hH_0$.  For the term $\|\theta_{c,\epio}\|_{\omega_{c,\epio}}$, notice first by the local $dd^c$-lemma that near $p_i$ on $M$, 
$$
  \theta_{0}=dd^cu_i
$$
for some smooth function $u_i$ such that $u_i=\cO(\rho_+^2)$.  Doing that for each singular point $p_i$ and choosing a cut-off function $\chi\in\CI_c([0,\infty[)$ such that $\chi\equiv 1$ near $0$, this means that we can extend $\theta_{0}$ to a family of closed $(1,1)$-forms $\widetilde{\theta}$ on $\chM$ with 
$$
  \widetilde{\theta}=dd^c (\chi(\rho_0)u_i)
$$
near $\hH_i$ and $\widetilde{\theta}=\cbl^*\mathcal{B}^*\pr_1^*\theta_{0}$ near $\hH_0$, where $\pr_1: M\times [0,1[$ is the projection on the first factor while $\cbl$ and $\mathcal{B}$ are defined at the beginning of \S~\ref{fs.0}.  Since $\theta_{0}$ is the harmonic representative of $\beta$ with respect to $\omega_0$, we see from \eqref{psc.3} that 
$$
 \Lambda_{c,\epio}\widetilde{\theta}=\cO(\rho_0^{2m-4}).
$$
Hence, when $m>3$, Corollary~\ref{fs.11} with $\kappa=2$ and $\delta=2(m-1)-4$ yields a unique polyhomogeneous function $v$ such that 
$$
\Delta_{c,\epio}v=\Lambda_{c,\epio}\widetilde{\theta}, \quad \epio_*(v)=0 \quad \mbox{and} \quad v=\cO(\epio^{\kappa}\rho_0^{\delta})= \cO(\rho_+^2\rho_0^{2m-4}).
$$
Since $\theta_{c,\epio}= \widetilde{\theta}+dd^cv,$ this means that 
\begin{equation}
    \|\rho_+^2\theta_{c,\epio}\|^2_{\omega_{c,\epio}}=
    \|\rho_+^2\theta_{0}\|^2_{\omega_0}+ \cO(\rho_+^2+\rho_0^{2m-4}).
\label{psc.5}\end{equation}
When $m=3$, we would need to consider the situation where $\kappa=2$ and $\delta=0$ in Corollary~\ref{fs.11}, in which case the conclusion does not quite hold, but we get from Corollary~\ref{MM.1} that
$$
 v=\cO(\epio^2\log\rho_+).
$$
Since $\Delta_{c,\epio}v$ is bounded, we have more precisely that $v=\lambda\epio^2\log\rho_+ +\cO(\epio^2)$ for some constant $\lambda$.
Hence, $dd^cv= \cO(\rho_0^2)$, from which we can proceed as in the case $m>3$ to conclude that \eqref{psc.5} also holds when $m=3$.  Combining these results together yields
\begin{equation}
    \rho_+^4\Sc^{\hbeta}(\omega_{c,\epio})=\cO(\rho_+^2+ \rho_0^{2m-4}).
\label{psc.6}\end{equation}

Using Lemma~\ref{fs.14} and the pushforward theorem, we then see that 
\begin{equation}
    \ell^{\ext}_{\halpha_{\epio},\hbeta}(\mu_{c,\epio})= \ell^{\ext}_{\alpha,\beta}(\mu_{\omega})+ \cO(\epio^{2m-4}).
\label{psc.7}\end{equation}
Combining \eqref{psc.6} and \eqref{psc.7} then shows that 
$$
\rho_+^4\mathring{\Sc}^{\hbeta}(\omega_{c,\epio})=\cO(\rho_+^2 + \rho_0^{2m-4})
$$
since 
$\rho_+^4\mathring{\Sc}^{\hbeta}(\omega_{c,\epio})$ vanishes on $\hH_+$ and $\hH_0$.  In fact, because of that we finally conclude that 
$$
\rho_+^4\mathring{\Sc}^{\hbeta}(\omega_{c,\epio})= \cO(\rho_+^2\rho_0^{2m-4}).
$$
\end{proof}

 We can now check that Assumption~\ref{psc.1} implies that the scalar curvatures of the extremal metrics of Theorem~\ref{nla.14} are positive for $\epio>0$ sufficiently small.

\begin{corollary}\label{c:blow-up} {Suppose furthermore in the hypothesis of Theorem~\ref{thm:general} that the $\Sc^{\beta}$-extremal K\"ahler metric $\omega_0$ on $M$ has positive scalar curvature and   Assumption~\ref{psc.1} holds.}  Then the $\Sc^{\hbeta}$-extremal K\"ahler metrics on $\widehat{M}$ provided by Theorem~\ref{nla.14} have positive scalar curvature for $\epio>0$ small enough.
\label{psc.8}\end{corollary}
\begin{proof}
Since the models on $\hH_+$ and $\hH_0$ have positive scalar curvature, this will be the case as well for the extremal metrics $\homega_{c,\epio,u}$ of Theorem~\ref{nla.14} outside an arbitrary small neighborhood $\cV$ of $\hH_+\cap\hH_0$ for $\epio>0$ small enough (depending on the choice of $\cV$).  Hence, to obtain the result, we need to show that $\homega_{c,\epio,u}$ has positive scalar curvature near $\hH_+\cap\hH_0$.

By Lemma~\ref{psc.2}, this is the case for $\omega_{c,\epio}$.  Let us check that this is also the case for the formal solution $\omega_{c,\epio,\phi}$ of Theorem~\ref{fs.17}.  This requires going carefully through the construction of $\phi$ using Lemma~\ref{psc.4}.  In the construction of the formal solution, one needs first to apply Lemma~\ref{c2.1} to ensure that $\rho^4_+\mathring{\Sc}^{\hbeta}(\omega_{c,\epio,\phi})$ decays at least quadratically at $\hH_0$.  By Lemma~\ref{psc.4}, for $m>3$, this step is not needed since $2m-6\ge 2$.  For $m=3$, it can be used to improve the decay at $H_0$ by applying Lemma~\ref{c2.1} with $\kappa<2$ as close as we want to $2$ and $\delta\in ]0,2[$ with $\delta\ge 2-\kappa$.  The terms added to the K\"ahler metric are then of order
$$
 \cO(\epio^{\kappa+\delta-\sigma}\rho_+^{-\delta})=\cO(\rho_+^{\kappa-\sigma}\rho_0^{\kappa+\delta-\sigma}) \quad \forall \; \sigma>0.
$$
This in turn contributes a term of order negligible with respect to $\rho_+^2+\rho_0^{2m-4}= \rho_+^2+\rho_0^2$.

The next step in the construction is to apply Lemma~\ref{c2.4} to improve the decay at $\hH_+$.  By Lemma~\ref{psc.4}, when $m>3$, we can in fact apply Lemma~\ref{c2.4} with $\kappa\in]0,2[$ as close as we want to $2$, while the initial  $\delta\in ]-4,-\frac72[$ is 
$$
 \delta=-2-\kappa
$$
with $k=0$ (no logarithmic term) and $\nu\in]-1,0[$ with 
$$
\nu<-\frac92-\delta+\delta'=-\frac52+\kappa+\delta',
$$
where $\delta'$ is the small positive number occurring in the proof of Lemma~\ref{c2.4}.  At this stage, the correction term added to the potential is then of order
$$
\cO(\epio^{\kappa+\delta+6}\rho_0^{-\frac92-\delta+\frac{\delta'}4})=\cO(\rho_+^4\rho_0^{\frac32+\kappa+\frac{\delta'}4}),
$$
where we used the fact that we can take $\sigma=0$ in \eqref{c2.7b} since $k=0$ and there is no logarithmic term.  For the scalar curvature multiplied by $\rho_+^2$, this contributes a term of order $\cO(\rho_+^2\rho_0^{\frac32+\kappa+\frac{\delta'}4})$ which is negligible with respect to $\rho_+^2+ \rho_0^{2m-4}$.  Iterating the argument with $\delta\in]-4,-\frac72[$ and $\delta>-2-\kappa$, we complete the improvement of the decay of $\mathring{\Sc}^{\hbeta}$ at $\hH_+$ by adding terms which only change the scalar curvature (multiplied by $\rho_+^2$) by terms of order negligible with resect to $\rho_+^2+\rho_0^{2m-4}$.

If instead $m=3$, we can still apply Lemma~\ref{c2.4} with $\kappa\in]0,2[$ as close as we want to $2$ by the discussion above, but $\delta\in ]-4,-\frac72[$ cannot necessarily be assumed to be greater or equal to $2-\kappa$.  The correction term added to the potential will therefore be of order
$$
\cO(\epio^{\kappa+\delta+6-\sigma}\rho_0^{-\frac92-\delta+\frac{\delta'}4})=\cO(\rho_+^{\kappa+\delta+6-\sigma}\rho_0^{\kappa+\frac32+\frac{\delta'}4-\sigma}) \quad \forall \; \sigma>0,
$$
contributing this time to $\rho_+^2R_{c,\epio,\phi}$ a term of order 
$$
 \cO(\rho_+^{\kappa}\rho_0^{\kappa+\frac32})
$$
which is negligible with respect $\rho_+^2+\rho_0^{2m-4}$ since we assume that $m=3$.

The next step is to use Lemma~\ref{c2.1} to improve the decay of $\mathring{\Sc}^{\hbeta}$ at $\hH_0$.  We can apply this lemma with $\kappa\in]2,\frac52[$.  For $\delta\in ]0,2[$, the construction adds correction terms to the K\"ahler metric of order
$$
\cO(\epio^{\kappa+\delta-\sigma}\rho_+^{-\delta})= \cO(\rho_+^{\kappa-\sigma}\rho_0^{\kappa+\delta-\sigma}) \quad \forall \ \sigma>0.
$$
Again, this contributes to the scalar curvature (multiplied by $\rho_+^2$) a term of order negligible with respect to $\rho_+^2+\rho_0^{2m-4}$ since $\kappa>2$. 

In the proof of Theorem~\ref{fs.17}, the other iterations of Lemmas~\ref{c2.1} and \ref{c2.4} contribute even more negligible terms to $\rho_+^2R_{c,\epio,\phi}$
with respect to $\rho_+^2+\rho_0^{2m-4}$.  This means that the conclusion of Lemma~\ref{psc.2} also holds for the formal solution of Theorem~\ref{fs.17}.  Since $u\in\dot{\cC}^{\infty}(\chM)$ in Theorem~\ref{nla.14},  the conclusion of Lemma~\ref{psc.2} also holds for the solutions $\homega_{c,\epio,u}$ of Theorem~\ref{nla.14}, namely
$$
\rho_+^2R_{c,\epio,\phi+u}= a_i\rho_+^2+ b_i\rho_0^{2m-4}+ o(\rho_+^2+\rho_0^{2m-4})
$$
near $\hH_i\cap\hH_0$ with $a_i$ and $b_i$ positive constants. In particular, $R_{c,\epio,\phi+u}$ is positive near $\hH_+\cap\hH_0$.
\end{proof}

\subsection{Proof of Theorem~\ref{thm:desingularization}} %\fr{More needs to be added?} 
The existence of the extremal metrics of Theorem~\ref{thm:desingularization} is provided by Theorem~\ref{nla.14} and the positivity of their scalar curvature is a consequence of Corollary~\ref{pos.1}. \hfill{$\square$}

\section{A $\Sc^0$-scalar flat ALE model on $\cO_{\bbP^{m-1}}(-(m-2))$ and proofs of Theorems~\ref{thm:blow-up} and \ref{thm:boxed}}\label{s:5}

\subsection{Examples of $\Sc^0$-flat K\"ahler metrics on $\cO_{\bbP^{m-1}}(-(m-2))$}
We use the Calabi ansatz~\cite{Calabi},  expressed in terms of a  \emph{profile function} $\Theta(x)$~\cite{ACG,HS}.  
To this end,  we consider a K\"ahler structure of the form
\begin{equation} \label{ansatz1}
g=x \ell g_{\bbP^{m-1}} + \frac{1}{\Theta(x)} dx^2 + \Theta(z) \xi^2, \qquad \omega = x \ell \omega_{\bbP^{m-1}} + dx\wedge \xi, \qquad d\xi = \ell \omega_{\bbP^{m-1}},
\end{equation}
where:
\begin{itemize}
\item $\Theta(x)>0$ on $(\edp, \infty)$  ($\edp >0$) and satisfies the end-point condition
\begin{equation} \label{boundary-k} \Theta(\edp)=0, \qquad \Theta'(\edp) = 2; \end{equation}
\item $(g_{\bbP^{m-1}}, \omega_{\bbP^{m-1}})$ is a Fubini-Study metric on $\bbP^{m-1}$ of scalar curvature $2m(m-1)$ and  $\ell \geq 1$ is an integer number;
\item $\xi$ is a connection $1$-form on the   $S^1$-principal bundle $P_\ell$ over $\PP^{m-1}$ of degree $\ell$,   with curvature $\ell\omega_{\bbP^{m-1}}$. 
\end{itemize}
The ansatz \eqref{ansatz1} describes an $S^1$-invariant K\"ahler structure on $\mathring{X}:= (\edp, \infty) \times P_{\ell} \to \Sigma$, with corresponding momentum map $x: X \to (\edp, \infty)$. It is well-known (see e.g. \cite{HS}) that the end-point conditions \eqref{boundary-k} for $\Theta$ ensure that this K\"ahler metric smoothly extends on $X$, the total space of  $\cO_{\bbP^{m-1}}(-\ell)$. We also notice that  for each $\nu>0$, the variable change $x := \nu \tilde x, \quad \tilde \Theta(\tilde x):= \frac{1}{\nu}\Theta(x)$ gives rise to a homothetic K\"ahler structure of the form \eqref{ansatz1}, i.e. the parameter $\edp >0$ can be interpreted as the homothety factor of the construction. Further letting 
\begin{equation*}\label{V} V(x):=x^{m-1}\Theta(x),  \end{equation*}
the ansatz \eqref{ansatz1}-\eqref{boundary-k} become
\begin{equation}\label{solution} 
\begin{split} 
&g= x \ell g_{\bbP^{m-1}} + \frac{x^m}{V(x)} dx^2 + \frac{V(x)}{x^m} \xi^2, \qquad \omega= z \ell \omega_{\bbP^{m-1}} + dx\wedge \xi,\qquad d\xi = \ell \omega_{\bbP^{m-1}},\\
& V(x)>0 \, \,  \text{on} \, \,  (\edp, \infty),  \qquad V(\edp)=0, \, V'(\edp) =2\edp^{m-1}. 
\end{split}
\end{equation}
The Ricci form and scalar curvature are given by (see e.g. \cite{ACG})
\begin{equation}\label{Calabi-curvature}
\begin{split}
\rho_\omega &= m\omega_{\PP^{m-2}} -\frac{1}{2}dd^c\log V(x) \\
&= \left(\frac{m}{\ell}-\frac{1}{2}\frac{V'(x)}{x^{m-1}}\right)\ell\omega_{\PP^1}-\frac{1}{2}\left(\frac{V'(x)}{x^{m-1}} \right)' dx \wedge \xi, \\
R_{\omega}&=\frac{2m(m-1)}{x\ell} -\frac{V''(x)}{x^{m-1}},
\end{split}\end{equation}
whereas for any smooth function $fxz)$,  we have
\[dd^c f = \left(\frac{f'(x)V(x)}{x^{m-1}}\right)' dx \wedge \xi +\left(\frac{f'(x)V(x)}{x^{m-1}}\right)\ell\omega_{\PP^{m-1}}, \qquad \Delta_g f  = - \frac{1}{x^{m-1}}\left(V(x)f'(x)\right)'. \]
By a straightforward computation we obtain
\begin{align*}
   \Delta_\omega R_\omega&=\frac{2m(m-1)V'}{\ell x^{m+1}}+\frac{V'V'''}{x^{2m-2}}-(m-1)\frac{V'V''}{x^{2m-1}}-\frac{4m(m-1)V}{\ell x^{m+2}}+\frac{VV''''}{x^{2m-2}}\\ &-2(m-1)\frac{VV'''}{x^{2m-1}}+m(m-1)\frac{VV''}{z^{2m}}
   \\ R_{\omega}^2&=\frac{4m^2(m-1)^2}{\ell^2x^2}-\frac{4m(m-1)V''}{\ell x^m}+\frac{(V'')^2}{x^{2m+2}}
   \\\|\Ric_{\omega}\|_\omega^2&=\frac{2(m-1)}{x^2}\left(\frac{m}{\ell}-\frac{V'}{2x^{m-1}}\right)^2+\frac{1}{2}\left(\left(\frac{V'}{x^{m-1}}\right)'\right)^2,
   %\\\|\theta\|^2&=\frac{2(n-1)}{n}\frac{\lambda^2}{z^{2n}}.
\end{align*}
leading to the following
\begin{lemma} A K\"ahler metric of the form \eqref{solution} has zero $\Sc^0$-scalar curvature iff $V(z)$ satisfies
\begin{equation} \label{ODE-V}
\begin{split}
    & \frac{2m^2(m-1)(m-2)}{\ell^2} x^{2m-2}+\frac{2m(m-1)} {\ell}\Big(2V'x^{m-1}-2Vx^{m-2}-V''x^m\Big) \\ &+\big(V'V'''+VV''''\big)x^2
    -\left(2(m-1)VV'''\right)x+m(m-1)VV''{-\frac{m(m-1)}{2}(V')^2}=0.
    \end{split}    
\end{equation}
In particular, the following polynomial solutions satisfy \eqref{ODE-V} and the boundary conditions in \eqref{solution}:
\begin{itemize}
\item $\ell= m$ and $V(x)= \frac{2}{m}(x^m - \edp^{m})$;
\item $\ell=(m-2), \, m\geq 3$ and $V(x)=\frac{2}{(m-2)}\left(x^m - \edp^{m-2}x^2\right).$
\end{itemize}
\end{lemma}
\begin{proof} By a direct computation of $\Sc^0(\omega)=\Delta_g R_{\omega} + \frac{R_{\omega}^2}{2} - \|\Ric_{\omega}\|^2_g$ 
%\fr{Should there be a + sign in front of the first term and should it be the scalar curvature square?}, 
based on the expressions above.
\end{proof}
If $\ell=m$, $V(x):= \frac{2}{m}(x^m-\edp^m)$ corresponds to Calabi's original construction~\cite{Calabi} of Ricci-flat K\"ahler metrics on the canonical bundle $\cO_{\bbP^{m-1}}(-m)$ of $\bbP^{m-1}$. In complex dimension $m=2$,  this is also the Eguchi--Hanson metric on $\cO_{\bbP^1}(-2)$. It is by now a well-established fact that in this case, \eqref{solution} gives rise to an ALE Ricci-flat K\"ahler metric compatible with the standard complex structure of  $X=\cO_{\bbP^{m-1}}(-m)$ which, suitably rescaled,  has a K\"ahler potential on $\mathring{X} \cong \bbC^m/\bbZ_m$ which decays as  
$\bl^*\left(\frac{1}{4}\|z\|^2 + \cO(\|z\|^{3-2m})\right)$ when  $\|z\|^2 \to \infty$. 

\bigskip
We focus from now on the case $\ell=(m-2)>0$ and $V(x)=\frac{2}{(m-2)}\left(x^m - \edp^{m-2}x^2\right)$ and study the global properties of the K\"ahler metric given by \eqref{solution}.
\begin{proposition}\label{p:new ALE} Let $X=\cO_{\bbP^{m-1}}(-(m-2))\xrightarrow{\bl}\bbC^m/\ \bbZ_{m-2}, \, m\geq 3$ denote the smooth $U(m)$-equivariant resolution of $\bbC^m/\bbZ_{m-2}$, where 
$\bbZ_{m-2}$ is the linear action on $\bbC^m$ of the cyclic group  $\left\langle e^{\frac{2\pi i}{m-2}}{\rm Id} \right\rangle$.  Then the K\"ahler structure \eqref{solution} with $V(x)=\frac{2}{(m-2)}\left(x^m - \edp^{m-2}x^2\right)$ is  isometric to a $U(m)$-invariant K\"ahler metric $\tilde \omega$ compatible with the standard complex structure of $X$. Furthermore, $\tilde \omega$  has zero $\Sc^{0}$-scalar curvature,  everywhere positive scalar curvature and, on the complement $\mathring{X}$ of the zero section of $X$,  admits a K\"ahler potential of the form
\[\tilde\omega = dd^c \Phi, \qquad \Phi= \begin{cases}\bl^*\left(\frac{1}{2}\|z\|^2 + \frac{\edp}{2} \log \|z\|^2\right) \, \text{when}\, \,  m=3,\\ \bl^*\left(\frac{(m-2)}{2}\|z\|^2 +  \frac{\edp^{m-2}}{2}\|z\|^{4-2(m-1)} + \cO(\|z\|^{3-2(m-1)})\right) \, \text{when}\,  m\geq 4, \end{cases} \] %\fr{Should we add a factor of $\frac12$ in the definition of $\Phi$ so that $\widetilde{\omega}$ is asymptotic to the "canonical" Euclidean metric? \vesti{I believe the expression for $\Phi$ is correct. You will need to rescale or change chart to satify the ``canonical'' condition.}}
where $\|z\|$ stands for the Euclidean norm on $\bbC^m$, seen by $\bbZ_{m-2}$ invariance as a function on $\bbC^{m}/\bbZ_{m-2}$.
The scalar curvature of $\tilde{\omega}$  on $\mathring{X}$ is 
\[ R_{\tilde \omega}= \bl^*\left(\frac{4\edp^{m-2}}{(m-2)(\edp^{m-2}+ \|z\|^{2(m-2)})^{\frac{m-1}{m-2}}}\right).\]
\end{proposition}
\begin{proof} We shall tacitly use the identifications  $\bl^{-1} : \bbC^m\setminus\{0\}/\bbZ_{m-2} \cong \cO_{\PP^{m-1}}(-(m-2))^{\times} \cong \mathring{X}$, where $\cO_{\PP^{m-1}}(-(m-2))^{\times}$ is the total space of the bundle minus the zero section. The first identification  can be explicitly written as
\[ \bl^{-1}(z_1, \ldots, z_m) = \left((z_1,\ldots, z_m)^{\otimes (m-2)}; [z_1, \ldots, z_m]\right).\]
Denote by $r_{m-2}$ the norm function on the line bundle $\cO_{\bbP^{m-1}}(-(m-2))$, induced by the Hermitian metric  with curvature $(m-2)\omega_{\bbP^{m-1}}$. It follows from the above identification that 
\begin{equation*}\label{euc-norm} r_{m-2}= \bl^* \| z \|^{m-2}, \qquad dd^c \log r_{m-2}= (m-2) \omega_{\bbP^{m-1}}, \end{equation*}
where the $dd^c$ operator is  with respect to the standard complex structure $J_0$ on $X$.

\smallskip
With the above understood, we shall define a $U(m)$-invariant smooth function $\Phi$ on $ \cO_{\PP^{m-1}}(-(m-2))^{\times}$,  such that $\tilde \omega=dd^c \Phi $ defines a global  $U(m)$-invariant  K\"ahler metric on the complex manifold $\cO_{\PP^{m-1}}(-(m-2))$, equivariantly isometric to the one given by \eqref{solution}. To this end, define functions
\begin{equation}\label{F}
\begin{split}
v(x) &:= \int^{x} \frac{1}{\Theta(t)} dt= \frac{(m-2)}{2}\int^x \frac{t^{m-3}}{(t^{m-2} -\edp^{m-2})} dt \\
&= \frac{1}{2} \log(x^{m-2}-\edp^{m-2});\\
H(x) &:= \int^x \frac{t}{\Theta(t)}dt=\frac{(m-2)}{2}\int^x \frac{t^{m-2}}{(t^{m-2} -\edp^{m-2})} dt \\
&= \frac{(m-2)}{2}(x-\edp) + \frac{(m-2)\edp^{m-2}}{2} \int^x\frac{1}{(t^{m-2} -\edp^{m-2})} dt
\end{split}
\end{equation}
One directly checks that on $\mathring{X} =(\edp, \infty) \times P_{m-2}$:
\[ d^c v = \xi, \qquad dd^c v = (m-2)\omega_{\bbP^{m-1}}, \qquad dd^c H = \omega.\]
Notice that $v$ is strictly monotone and  $v(\mathring{X})= \bbR$: expressing $x=x(v)= \left(e^{2v} + \edp^{m-2}\right)^{\frac{1}{(m-2)}}$, we introduce the function 
\[ \Phi : \bbR \to \bbR, \qquad \Phi(v) := H(x(v)).\]
It is important for us that $\Phi(v)$ is defined on the real line $\bbR$ and  satisfies
\[ dd^c \Phi (v) = \omega.\]
The above holds on $ X \setminus x^{-1}(\edp)$ but the closed  $(1,1)$-form  $dd^c \Phi(v)$ extends smoothly and positively over $E= x^{-1}(\edp)$ by the first order boundary conditions $\Theta(x)$ satisfied at $x=\edp$.

\smallskip
As $\log r_{m-2}$ and $v$ have the same image $\bbR$, we can let
\begin{equation}\label{Kahler-potential} \tilde{\omega} := dd^c \Phi(\log r_{m-2}).\end{equation}
It follows from the above discussion that  $\tilde\omega$ is a $(1,1)$-form on $\cO_{\PP^{m-1}}(-(m-2))^{\times}$ which defines a K\"ahler structure isometric to $(g, \omega)$, under an isometry which sends   $e^v$ to $r_{m-2}$.
In particular,  the K\"ahler metric $\tilde g:=-\tilde \omega J_0$ extends over the zero section of $\cO_{\PP^{m-1}}(-(m-2))$ (as $g$ does over $x^{-1}(\edp)$). Thus \eqref{Kahler-potential}
defines a  K\"ahler metric on $\cO_{\PP^{m-1}}(-(m-2))$ isometric to $(M, g, \omega)$. By the expression of $H(x)=\Phi(v)$ in \eqref{F}   and using
\begin{equation*}\label{x-norm} x^{m-2} = \edp^{m-2} + e^{2v} \, \qquad e^{2v} \sim r_{m-2}^2=\bl^*\|z\|^{2(m-2)},  \end{equation*} we get  the expansion in $\pi^*\|z\|^{m-2}$ of $\Phi(\log r_{m-2})$ near infinity as claimed.

Finally, by \eqref{Calabi-curvature} and \eqref{x-norm}, we compute $R_{\omega} = \frac{4\edp^{m-2}}{(m-2)x^{m-1}}$, 
showing that 
\[ R_{\tilde{\omega}}= \bl^* \left(\frac{4\edp^{m-2}}{(m-2)\left( \edp^{m-2} + \|z\|^{2(m-2)}\right)^{\frac{m-1}{m-2}}} \right) >0,\]
as claimed.
\end{proof} 
\begin{remark} The decay of the K\"ahler potential of the K\"ahler metric on $X=\cO_{\bbP^{m-1}}(-(m-2))$  constructed in Proposition~\ref{p:new ALE} is one dimension ``off'' of the usual decay of potentials of ALE K\"ahler metrics in $m$-complex dimensions, see \cite{Arezzo-Pacard2006,  Arezzo-Pacard2009}. It however yields that the metric is complete, has Euclidean volume growth of balls, $\cO(\|z\|^{-2(m-1)})$ curvature decay at infinity, and has a K\"ahler potential with polyhomogeneous expansion in $\|z\|$ at infinity. This is enough to apply Theorem~\ref{thm:general} and Corollary~\ref{c:blow-up}.
\end{remark}
\subsection{Proof of Theorem~\ref{thm:blow-up}} 

The family of K\"ahler metrics $\omega_a$ defines by restriction a family of Euclidean metrics $\omega_a(p_i)$ on $T_{p_i}M$ for $i\in\{1,\ldots,\ell\}$.  Those metrics are obviously  K\"ahler with respect to the restriction $J_{M}(p_i)$ to $T_{p_i}M$ of the complex structure $J_M$ of $M$.  Taking $\delta_0$ smaller in Theorem~\ref{thm:blow-up} if needed, we can assume that there exists a smooth family $a\mapsto \psi_{i,a}\in \Hom(T_{p_i}M, J_M(p_i))$ such that 
$$
 \omega_a(p_i)=\psi^*_{i,a}\omega_{a_0}(p_i) \quad \mbox{and} \quad \psi_{i,a_0}=\Id.
$$
By Proposition~\ref{p:new ALE}, the blow-up $X_i$ of $T_{p_i}M$ at the origin admits an asymptotically Euclidean  K\"ahler metric $\omega_{i,a_0}$ asymptotic to $\omega_{a_0}(p_i)$ at infinity with zero $\Sc^0$-curvature and positive scalar curvature.  Since $\psi_{i,a}$ lifts to $X_i$, we see that 
$$
 \omega_{i,a}:=\psi_{i,a}^*\omega_{i,a_0}
$$
is a smooth family of $\ALE$ K\"ahler metrics asymptotic to $\omega_a(p_i)$ at infinity with zero $\Sc^0$-curvature and positive scalar curvature.

This ensures that we can construct a smooth family of formal solutions $\omega_{a,c,\epio,\phi}$ using Theorem~\ref{fs.17}.  Indeed, the models $\omega_{i,a}$ have the same polyhomogeneous expansion at infinity, so each step in the construction of the formal solution admits a family version.  More precisely, for the family version to work, we need to check that the results involved in the proof of Lemmas~\ref{c2.1} and \ref{c2.4}, namely Propositions~\ref{la.26} and \ref{la.35} are applied in a systematic way.  

For Proposition~\ref{la.26}, this is possible thanks to the fact that $\Aut_{\rm red}(M)=\{1\}$, since then $P_{\omega_a,+}=\{0\}$ in \eqref{la.16e} does not depend on $a$, while choosing $\cD$ for $a=a_0$ in the statement of Proposition~\ref{la.26}, we will be able to take the same $\cD$ in \eqref{la.26b} for other values of $a$ by taking $\delta_0$ smaller if needed.  For the map \eqref{la.26c}, the dimension of the nullspace does not depend on $a$, so the nullspaces combine to form a smooth vector bundle on $B(a_0,\delta_0)$ as $a$ varies.  Using the Hilbert norms of $(r^{\nu}L^{2,k+6}_b(\bM))^{\bbT}$ induced by $\omega_a$, the orthogonal complements to the various nullspaces vary smoothly in $a$, allowing us to extract a smooth family of isomorphisms from \eqref{la.26c} as $a$ varies.  Similarly, since the nullspace of \eqref{la.35a} is identified with the constant functions, it does not depend on $a$, so we can extract in the same way from Proposition~\ref{la.35} a smooth family of isomorphisms.  This means that we can indeed apply Theorem~\ref{fs.17} to construct a smooth family of formal solutions $\homega_{a,c,\epio}$. Then taking $\delta_0'$ sufficiently small, the constant $C_{q,k,\mu,\,\nu}$ in \eqref{nla.8}, for $q,k,\mu$ and $\nu$ fixed, can be chosen to be independent of $a\in B(a_0,\delta_0')$.  For $q,k,\mu$ and $\nu$ fixed, we can also take the constant $c$ and $\epio_0$ independent of $a\in B(a_0,\delta_0')$ by taking $\delta_0'$ possibly smaller.  This means that we can take $\epio_0$ in Theorem~\ref{nla.14} to be independent of $a\in B(a_0,\delta_0')$, from which the existence of the family of $\Sc^0$-extremal  K\"ahler metrics $\hat\omega_{\la, \epsilon}$ follows. As $\Aut_{\rm red}(M)=\{1\}$ by assumption,  $\Aut_{\rm red}(\widehat{M})=\{1\}$ too and, therefore,  $\hat\omega_{\la, \epsilon}$ have constant $\Sc^0$-scalar curvature. {The K\"ahler classes of these metrics are \[[\hat\omega_{\la, \epsilon}]= \bl_{p_1, \ldots, p_k}^*[\omega_{\la}]+ \epsilon^2\sum_{i=1}^k [\omega_i],\] where $[\omega_i] \in H^2(\cO_{\PP^2}(-1), \bbR)$ are the cohomology classes of the (suitable rescalled) model ALE metrics,  seen as cohomology classes on $\widehat{M}$ via Lemma~\ref{fs.6}. Let $E$ denote the zero section of $X=\cO_{\PP^2}(-1)$. As  $H^2(X, \bbR) =\bbR$ and  $\cO_X(E)_{|_E}\cong\cO_{\PP^2}(-1)$, $[\omega_i]$   is a negative multiple of $[E]=c_1(\cO_X(E))$.   We thus conclude that, after possibly rescaling the parameter $\epsilon$,  
\[ [\hat\omega_{\la, \epsilon}]= \bl_{p_1, \ldots, p_k}^*[\omega_{\la}]- \epsilon^2\sum_{i=1}^k [E_i].\]}
The positivity of the scalar curvatures in that family then follows from Corollary~\ref{psc.8},  by taking $\delta_0'$ and $\epio_0$ possibly smaller.
\hfill{$\square$}
\subsection{Proof of Theorem~\ref{thm:boxed}} (a) Let $(S, \omega_S)$ ba K\"ahler--Einstein del Pezzo surface of scalar curvature $R_S=4$ and $(\Sigma, \omega_{\Sigma})$ a Riemann surface of scalar curvature $R_{\Sigma}= 2(1-{\bf g})$. 
Let $M:= S \times \Sigma$.  By the classification of K\"ahler--Einstein del Pezzo surfaces~\cite{tian-surfaces}, $S$ is a blow-up of $\PP^2$ at $k$ points  $0\leq k \leq 8$ in generic position, and  thus $\Aut_{\rm red}(S)=\Aut_\circ(S)$ is trivial iff $k\geq 4$, i.e. $1\leq c_1^2(S)\leq 5$.  On $M$ we consider the K\"ahler metrics
\[ \omega_a = \omega_S + \la \omega_{\Sigma}, \qquad \la >0.\]
Notice that
\[R_{\omega_\la} = 4 + \frac{2(1-{\bf g})}{\la}, \qquad \Sc^0(\omega_\la)= 4\left(\frac{2(1-{\bf g})}{\la} + 1 \right).\]
Thus, for $\la> -\left(\frac{1-{\bf g}}{2}\right)$, $\omega_\la$ is $\Sc^0$-extremal with constant $\Sc^0$-scalar curvature and positive scalar curvature:  when  $\la \in ]\frac{({\bf g}-1)}{2}, 2({\bf g}-1)[$, $\Sc^0(\omega_{\la})<0$, for $\la_0=  2({\bf g}-1)$,   $\Sc^0(\omega_{\la})=0$ and when $\la >2({\bf g}-1)$, we have $\Sc^0(\omega_{\la})>0$. As $\Aut_\circ(M)= \{ 1\}$, we can apply Theorem~\ref{thm:blow-up} to the family $\omega_{\la}$  to obtain a  family $\hat \omega_{\la, \epi}$ of K\"ahler metrics of constant $\Sc^0$-scalar curvature and positive scalar curvature on the blow up $\widehat M$ of $M$ at $\{p_1, \ldots, p_k\}$ (where the parameter $\la$ belongs to an interval around $\la_0:=2({\bf g}-1)$). Using that
\[ c_1(\widehat M) = Bl_{p_1, \ldots, p_k}^*(c_1(M))- \sum_{i=1}^k [E_i], \qquad [\hat\omega_{\la, \epi}]= Bl_{p_1, \ldots, p_k}^*[\omega_{\la}] - \epi^2 \sum_{i=1}^k [E_i],\]
we find that the sign of $\Sc^0(\hat \omega_{\la, \epi})$ equals the sign of
\[ \la \Sc^0(\omega_{\la}) - k\epi^2.\]
We can make this zero by taking 
\[ \epi^2 := \frac{\la}{k} \Sc^0(\omega_\la)= \frac{4(\la-\la_0)}{k},\]
for suitable $0<\la-\la_0$ small enough.

\smallskip
(b) In the case $M=\bbP(E)$,  we can run a similar argument, noting that $\Aut_{\circ}(M)=\{1\}$ (as the vector bundle $E$ is stable and therefore simple and $\Aut_\circ(\Sigma)=\{1\}$)   and  that  $M$ admits locally symmetric K\"ahler metrics $\omega_{\la}$ which are covered by the product  a Fubini-Study metric on $\bbP^2$ (of scalar curvature $4$) and  a hyperbolic metric of constant scalar curvature $\frac{2(1-{\bf g})}{\la}$ on $\bbH^2$ (by the Narasimhan--Seshadri theorem~\cite{Nar-Sesh}). \hfill{$\square$}

%\bibliography{Gluing}

\begin{thebibliography}{999}

\bibitem{ARS3}
P.~Albin, F.~Rochon, and D.~Sher, \emph{A {C}heeger-{M}\"uller theorem for manifolds with wedge singularities}, Anal. PDE \textbf{15} (2022),  567--642. 

\bibitem{ALN04}
B.~Ammann, R.~Lauter, and V.~Nistor, \emph{On the geometry of {R}iemannian manifolds with a {L}ie structure at infinity}, Internat. J. Math. (2004), 161--193.

\bibitem{ABLS}
V.~Apostolov, G.~Barbaro, K.-H. Lee, and J.~Streets, \emph{Rigidity results for non-K\"ahler Calabi-Yau geometries on threefolds}, Math. Ann. \textbf{393} (2025), 3609--3637.

\bibitem{ACG}
V.~Apostolov, D.M.J. Calderbank, and P.~Gauduchon, \emph{Hamiltonian 2-forms in K\"ahler geometry I: General theory}, J. Differential Geom. \textbf{73} (2006), 359--412.

\bibitem{ALL2026}
V.~Apostolov, A.~Lahdili, and K.-H. Lee, \emph{Pluriclosed 3-folds with vanishing {B}ismut {R}icci form: general theory in the quasi-regular case}, to appear in Comm. Math. Phys., arXiv: 2601.04937v2.

\bibitem{Arezzo-Pacard2006}
C.~Arezzo and F.~Pacard, \emph{Blowing up and desingularizing constant scalar curvature {K}\"ahler manifolds}, Acta Math. \textbf{196} (2006),  179--228. 

\bibitem{Arezzo-Pacard2009}
\bysame, \emph{Blowing up {K}\"ahler manifolds with constant scalar curvature. {II}}, Ann. of Math. (2) \textbf{170} (2009),  685--738. 

\bibitem{arezzo-pacard-singer}
C.~Arezzo, F.~Pacard, and M.~Singer, \emph{Extremal metrics on blowups}, Duke Math. J. \textbf{157} (2011),  1--51.

\bibitem{BBFMT}
G.~Bazzoni, I.~Biswas, M.~Fern\'andez, V.~Mu\~noz, and A.~Tralle, \emph{Homotopic properties of {K}\"ahler orbifolds}, Special metrics and group actions in geometry, Springer INdAM Ser., vol.~23, Springer, Cham, 2017, pp.~23--57. \MR{3751961}

\bibitem{Benabida}
A.O. Benabida, \emph{Asymptotics for resolutions and smoothings of {C}alabi-{Y}au conifolds}, Comm. Math. Phys. \textbf{407} (2026),  Art.71.

\bibitem{BGM}
M.~Berger, P.~Gauduchon, and E.~Mazet, \emph{Le spectre d'une vari\'et\'e riemannienne}, Lecture Notes in Mathematics, vol. 194, Springer Berlin, Heidelberg, 1971.

\bibitem{BJT} S. Boucksom, M. Jonsson, and A. Trusiani, {\it Weighted extremal K\"ahler metrics on resolutions of singularities}, Comm. Pure and Applied Math. {\bf 79} (2025), 1073--1148.

\bibitem{Boyer-Galicki-book}
C.~Boyer and K.~Galicky, \emph{Sasakian geometry}, Oxford Mathematical Monographs, Oxford University Press, 2008.

\bibitem{fino-et-al}
B.~Brienza, A.~Fino, U.~Fowdar, and G.~Grantcharov, \emph{Sasaki with torsion manifolds and string backgrounds}, arXiv:2608.08781 (2026).

\bibitem{Bui}
Q.-T. Bui, \emph{Injectivity radius of manifolds with a {L}ie structure at infinity}, Ann. Math. Blaise Pascal \textbf{29} (2022), no.~2, 235--246. \MR{4552719}

\bibitem{Calabi}
E.~Calabi, \emph{M\'etriques k\"ahl\'eriennes et fibr\'es holomorphes}, Ann. Sci. \'Ecole Norm. Sup. (4) \textbf{12} (1979), no.~2, 269--294.

\bibitem{Chen-Tian1}
X.X Chen and G.~Tian, \emph{Ricci flow on K\"ahler-Einstein surfaces}, Invent. math. \textbf{147} (2002), 487--544.

\bibitem{Chen-Wang}
X.X. Chen and B.~Wang, \emph{K\"ahler–Ricci flow with small initial energy}, GAFA Geom. Funct. Anal. \textbf{18} (2009), 1525--1563.

\bibitem{CDR2019}
R.~Conlon, A.~Degeratu, and F.~Rochon, \emph{Quasi-asymptotically conical {C}alabi--{Y}au manifolds}, Geom. Topol. \textbf{23} (2019),  29--100. 

\bibitem{CH2013}
R.~J. Conlon and H.-J. Hein, \emph{Asymptotically conical {C}alabi-{Y}au manifolds, {I}}, Duke Math. J. \textbf{162} (2013),  2855--2902. 

\bibitem{CMR2015}
R.~J. Conlon, R.~Mazzeo, and F.~Rochon, \emph{The moduli space of asymptotically cylindrical {C}alabi-{Y}au manifolds}, Comm. Math. Phys. \textbf{338} (2015),  953--1009. 

\bibitem{CGMS}
C.~Couzens, J.-P Gauntlett, D.~Martelli, and J.~Sparks, \emph{A geometric dual of c-extremization}, J. High Energ. Phys. (2019), Art. 212.

\bibitem{Dervan}
R.~Dervan, \emph{Stability conditions for polarised varieties}, Forum Math. Sigma \textbf{11} (2024), 1--57.

\bibitem{DS} R. Dervan and L. M. Sektnan, {\it Extremal K\"ahler metrics on blowups},  arXiv:2110.13579.

\bibitem{Dervan-Hallem}
R.~Dervan and M.~Hallem, \emph{The universal structure of moment maps in complex geometry}, J. \'Ecole polytecnique  {\bf 13} (2026), 399--436.

\bibitem{DR}
P.~Dimakis and F.~Rochon, \emph{Asymptotic geometry at infinity of quiver varieties}, arXiv:2410.15424.

\bibitem{donaldson-K3}
S.~Donaldson, \emph{Calabi-Yau metrics on Kummer surfaces as a model glueing problem}, Adv. Geom. Anal. \textbf{21} (2012), 109--118.

\bibitem{donaldson-moment}
S.~K. Donaldson, \emph{Remarks on gauge theory, complex geometry and 4-manifold topology}, Fields Medallists' Lectures, World Sci. Ser. 20th Century Math., vol.~5, World Sci. Publ. River Edge NJ, 1997, pp.~384--403.

\bibitem{topology-S1}
H.~Duan and C.~Liang, \emph{Circle bundles over 4-manifolds}, Arch. Math. (Basel) \textbf{85} (2005), 278--282.

\bibitem{dyrefelt3} Z. Sj\"ostr\"om Dyrefelt, {\it CscK metrics on birational models of projective varieties}, arXiv:2608.03572.

\bibitem{EH}
H.~Eguchi and A.J. Hanson, \emph{Selfdual solutions to Euclidean gravity}, Ann. Phys. \textbf{120} (1979), 82--105.

\bibitem{fernandez2014non}
M.~Fern{\'a}ndez, S.~Ivanov, L.~Ugarte, and D.~Vassilev, \emph{Non-{K}{\"a}hler heterotic string solutions with non-zero fluxes and non-constant dilaton}, Journal of High Energy Physics \textbf{2014} (2014),  1--23.

\bibitem{finosurvey}
A.~Fino, \emph{Canonical metrics in complex geometry}, Bollettino del'Unione Matematics Italiana {\bf 18} (2024), 185--198.

\bibitem{fujiki-moment}
A.~Fujiki, \emph{Moduli space of polarized algebraic manifolds and K\"ahler metrics}, Translations of Sugaku \textbf{42} (1993),  231--243.

\bibitem{GM1975}
S.~Gallot and D.~Meyer, \emph{Op\'erateur de courbure et laplacien des formes diff\'erentielles d'une vari\'et\'e{} riemannienne}, J. Math. Pures Appl. (9) \textbf{54} (1975),  259--284. 

\bibitem{JFS}
M.~Garcia-Fernandez, J.~Jordan, and J.~Streets, \emph{Non-{K}\"{a}hler {C}alabi-{Y}au geometry and pluriclosed flow}, J. Math. Pures Appl. (9) \textbf{177} (2023), 329--367. 

\bibitem{garciafern2018canonical}
M.~Garcia-Fernandez, R.~Rubio, C.~Shahbazi, and C.~Tipler, \emph{Canonical metrics on holomorphic {C}ourant algebroids}, Proc. Lond. Math. Soc. (3) \textbf{125} (2022),  700--758. 

\bibitem{GRFBook}
M.~Garcia-Fernandez and J.~Streets, \emph{Generalized {R}icci flow}, University Lecture Series, vol.~76, American Mathematical Society, Providence, RI, 2021.

\bibitem{Gauduchon}
P.~Gauduchon, \emph{Calabi's extremal {K}\"ahler metrics: {A}n elementary introduction}, Book available upon request.

\bibitem{GauduchonIvanov}
P.~Gauduchon and S.~Ivanov, \emph{Einstein-{H}ermitian surfaces and {H}ermitian {E}instein-{W}eyl structures in dimension {$4$}}, Math. Z. \textbf{226} (1997), 317--326. 

\bibitem{GK}
J.P. Gauntlett and N.~Kim, \emph{Geometries with Killing spinors and supersymmetric ADS solutions}, Comm. Math. Phys. \textbf{284} (2008), 897--918.

\bibitem{GibbonsHawking}
G.~Gibbons and S.~Hawking, \emph{Gravitational multi-instantons}, Physics Letters B. \textbf{78} (1978), 430--432.

\bibitem{Gil-Mendoza}
J.~B. Gil and G.~A. Mendoza, \emph{Adjoints of elliptic cone operators}, Amer. J. Math. \textbf{125} (2003), 357--408.

\bibitem{Goto}
R.~Goto, \emph{Calabi-{Y}au structures and {E}instein-{S}asakian structures on crepant resolutions of isolated singularities}, J. Math. Soc. Japan \textbf{64} (2012),  1005--1052. 

\bibitem{grantcharov2008calabi}
D.~Grantcharov, G.~Grantcharov, and Y.~S. Poon, \emph{Calabi-Yau connections with torsion on toric bundles}, J. Differential Geom. \textbf{78} (2008),  13--32.

\bibitem{Grieser}
D.~Grieser, \emph{Scales, blow-up and quasi-mode construction}, Contemp. Math., AMS \textbf{700} (2017), 207--266.

\bibitem{hallam} M. Hallam, {\it  Weighted Extremal Metrics on Blowups},  J. Geom. Anal. {\bf 35} (2025), Art. 187. 

\bibitem{HHM1995}
A.~Hassell, R.~Mazzeo, and R.~B. Melrose, \emph{Analytic surgery and the accumulation of eigenvalues}, Comm. Anal. Geom. \textbf{3} (1995),  115--222.

\bibitem{HHM2004}
T.~Hausel, E.~Hunsicker, and R.~Mazzeo, \emph{Hodge cohomology of gravitational instantons}, Duke Math. J. \textbf{122} (2004),  485--548.

\bibitem{Hitchin-inst}
N.J. Hitchin, \emph{Polygons and gravitons}, Math. Proc. Cambridge Math. Soc. \textbf{85} (1979), 465--476.

\bibitem{topology-T2}
T.~H\"ofer, \emph{Remarks on torus principal bundles}, J. Math. Kyoto Univ. \textbf{33} (1993),  227--259.

\bibitem{Hormander3}
L.~H{\"o}rmander, \emph{The analysis of linear partial differential operators. vol. 3}, Springer-Verlag, Berlin, 1985.

\bibitem{HS}
A.D. Hwang and M.A. Singer, \emph{A momentum construction for circle-invariant K\"ahler metrics}, Trans. Amer. Math. Soc. \textbf{349} (1997), 4201--4230.

\bibitem{ivanov2010heterotic}
S.~Ivanov, \emph{Heterotic supersymmetry, anomaly cancellation and equations of motion}, Physics Letters B \textbf{685} (2010),  190--196.

\bibitem{Joyce}
D.~D. Joyce, \emph{Compact manifolds with special holonomy}, Oxford Mathematical Monographs, Oxford University Press, 2000. 

\bibitem{kobayashi-K3}
R.~Kobayashi, \emph{Moduli of Einstein metrics on a $K3$ surface and degeneration of type I.}, K\"ahler metric and moduli spaces (Takushiro Ochia, ed.), Advanced Studies in Pure Mathematics 2, vol.~18, Mathe. Soc. Japan, 1990, pp.~257--311.

\bibitem{KR1}
C.~Kottke and F.~Rochon, \emph{Quasi-fibered boundary pseudodifferential operators}, arXiv:2103.16650.

\bibitem{KR0}
\bysame, \emph{Low energy limit for the resolvent of some fibered boundary operators}, Comm. Math. Phys. \textbf{390} (2022),  231--307. 

\bibitem{Kronheimer1989}
P.~B. Kronheimer, \emph{The construction of {ALE} spaces as hyper-{K}\"ahler quotients}, J. Differential Geom. \textbf{29} (1989),  665--683. 

\bibitem{lahdili}
A.~Lahdili, \emph{K\"ahler metrics with constant weighted scalar curvature and weighted K-stability}, Proc. London Math. Soc. \textbf{119} (2019),  1065--1114.

\bibitem{LeBrun-Simanca1994}
C.~LeBrun and S.~R. Simanca, \emph{Extremal {K}\"ahler metrics and complex deformation theory}, Geom. Funct. Anal. \textbf{4} (1994),  298--336. 

\bibitem{lebrun-singer}
C.~LeBrun and M.~Singer, \emph{A Kummer-type construction of self-dual 4-manifolds}, Math. Ann. \textbf{300} (1994), 165--180.

\bibitem{MazzeoEdge}
R.~Mazzeo, \emph{Elliptic theory of differential edge operators. {I}}, Comm. Partial Differential Equations \textbf{16} (1991),  1615--1664.

\bibitem{Mazzeo-MelroseETA}
R.~Mazzeo and R.~B. Melrose, \emph{Analytic surgery and the eta invariant}, Geom. Funct. Anal. \textbf{5} (1995),  14--75.

\bibitem{Mazzeo-MelrosePhi}
R.~Mazzeo and R.~B. Melrose, \emph{Pseudodifferential operators on manifolds with fibred boundaries}, Asian J. Math. \textbf{2} (1999),  833--866.

\bibitem{mehta-seshadri}
V.B. Mehta and C.S. Seshadri, \emph{Moduli of vector bundles of curves with parabolic structures}, Math. Ann. \textbf{248} (1980), 205--239.

\bibitem{Melrose1992}
R.~B. Melrose, \emph{Calculus of conormal distributions on manifolds with corners}, Int. Math. Res. Not. (1992),  51--61.

\bibitem{MelroseAPS}
\bysame, \emph{The {A}tiyah-{P}atodi-{S}inger index theorem}, A. K. Peters, Wellesley, Massachusetts, 1993.

\bibitem{MelroseGST}
\bysame, \emph{Geometric scattering theory}, Cambridge University Press, Cambridge, 1995.

\bibitem{Najafpour}
M.~Najafpour, \emph{Constant scalar curvature {K}\"ahler metrics on resolutions of an orbifold singularity of depth 1}, J. Geom. Anal. \textbf{35} (2025),  Art. 186.

\bibitem{Nar-Sesh}
M.S. Narasimhan and C.S. Seshadri, \emph{Stable and unitary vector bundles on a compact Riemann surface}, Ann. Math. (2) \textbf{82} (1965), 540--567.

\bibitem{OSS2016}
Y.~Odaka, C.~Spotti, and S.~Sun, \emph{Compact moduli spaces of {D}el {P}ezzo surfaces and {K}\"ahler-{E}instein metrics}, J. Differential Geom. \textbf{102} (2016),  127--172. 

\bibitem{phong2019geometric}
D.~H. Phong, \emph{Geometric partial differential equations from unified string theories},  arXiv:1906.03693 (2019).

\bibitem{picard2024strominger}
S.~Picard, \emph{The {S}trominger system and flows by the {R}icci tensor}, arXiv:2402.17770.

\bibitem{rollin-singer1}
Y.~Rollin and M.~Singer, \emph{Non-minimal scalar-flat surfaces and parabolic stability}, Invent. math. \textbf{162} (2005), 235--270.

\bibitem{rollin-singer2}
\bysame, \emph{Construction of K\"ahler surfaces with constant scalar curvature}, J. Eur. Math. Soc. \textbf{11} (2009), 979--997.

\bibitem{Schroers-Singer}
B.~J. Schroers and M.~A. Singer, \emph{{$D_k$} gravitational instantons as
  superpositions of {A}tiyah-{H}itchin and {T}aub-{NUT} geometries}, Q. J.
  Math. \textbf{72} (2021), no.~1-2, 277--337. 

\bibitem{s-sz} R. Seyyedali and G. Sz\'ekelyhidi, {\it Extremal metrics on blowups along submanifolds},
J. Differential Geom. {\bf 114} (2020),  
171--192.

\bibitem{burns-simanca}
S.R. Simanca, \emph{K\"ahler metrics of constant scalar curvature on bundles over ${\mathbb P}^n$}, Math. Ann. \textbf{291} (1991), 239--246.

\bibitem{Song-Weinkove}
J.~Song and B.~Weinkove, \emph{Energy functionals and canonical K\"ahler metrics}, Duke Math. J. \textbf{137} (2007),  159--184.

\bibitem{streets-topo}
J.~Streets, \emph{Topology of low dimensional generalized Ricci solitons}, arXiv:2608.25829.

\bibitem{st-geom}
\bysame, \emph{Pluriclosed flow and the geometrization of complex surfaces}, Geometric analysis, Progr. Math., vol. 333, Birkh\"{a}user/Springer, Cham, 2020, pp.~471--510. 

\bibitem{PCF}
J.~Streets and G.~Tian, \emph{A parabolic flow of pluriclosed metrics}, Int. Math. Res. Not. IMRN (2010),  3101--3133. 

\bibitem{PCFReg}
\bysame, \emph{Regularity results for pluriclosed flow}, Geom. Topol. \textbf{17} (2013), 2389--2429. 

\bibitem{StromingerSST}
A.~Strominger, \emph{Superstrings with torsion}, Nuclear Phys. B \textbf{274} (1986), 253--284. 

\bibitem{Suvaina}
I.~Suvaina, \emph{ALE Ricci-flat K\"ahler metrics and deformations of quotient surface singularities}, Ann. Glob. Anal. Geom. \textbf{41} (2012),  109--123.

\bibitem{Swan1962}
R.~G. Swan, \emph{Vector bundles and projective modules}, Trans. Amer. Math. Soc. \textbf{105} (1962), 264--277. 

\bibitem{sz1}
G.~Sz\'ekelyhidi, \emph{On blowing up extremal {K}\"ahler manifolds}, Duke Math. J. \textbf{161} (2012), no.~8, 1411--1453.

\bibitem{sz2}
\bysame, \emph{Blowing up extremal {K}\"ahler manifolds II}, Invent. math. \textbf{200} (2015),  925--977.

\bibitem{tian-surfaces}
G.~Tian, \emph{On Calabi’s conjecture for complex surfaces with positive first chern class}, Invent. Math. \textbf{101} (1990), 101--172.

\bibitem{tipler} C. Tipler, {\it A note on blow-ups of toric surfaces and CSC K\"ahler metrics}, Tohoku Math. J. (2) {\bf 66} (2014),  15--29.

\bibitem{tosatti2015non}
V.~Tosatti, \emph{Non-{K}{\"a}hler {C}alabi-{Y}au manifolds}, in Analysis, Complex Geometry, and Mathematical Physics: In Honor of Duong H. Phong. (P.M.N. Feehan, J. Song, B. Weinkove, R. A. Wentworth Eds.) Contemporary Mathematics \textbf{644} (2015), Amer. Math. Soc., pp. 261--277.

\bibitem{Zhu2026}
Xuwen Zhu, \emph{A gluing construction of {$D_k$} {ALF} gravitational
  instantons and existence of non-holomorphic minimal spheres}, Proc. Amer.
  Math. Soc. \textbf{154} (2026), no.~5, 2175--2187.

\end{thebibliography}
%\bibliographystyle{amsplain}

{
\def\cprime{$'$}
\providecommand{\bysame}{\leavevmode\hbox to3em{\hrulefill}\thinspace}
\providecommand{\MR}{\relax\ifhmode\unskip\space\fi MR }

\providecommand{\MRhref}[2]{%
  \href{https://mathscinet.ams.org/mathscinet-getitem?mr=#1}{#2}
}

}

\end{document}